\documentclass[pdflatex,sn-mathphys-num]{sn-jnl}

\usepackage[T1]{fontenc}
\usepackage{mathtools}
\usepackage{amssymb}
\usepackage{txfonts}
\usepackage{tgtermes}
\usepackage{mathrsfs}
\usepackage{xcolor}
\usepackage{microtype}
\usepackage{enumitem}
\usepackage{xurl}
\usepackage{bookmark}
\DeclareFontFamily{U}{rsfs}{\skewchar\font127 }
\DeclareFontShape{U}{rsfs}{m}{n}{%
  <-6.5> rsfs5 <6.5-7.5> rsfs7 <7.5-> rsfs10}{}

\makeatletter
\if@referee\else
  \g@addto@macro\normalsize{%
    \abovedisplayskip=8pt plus2pt minus2pt
    \belowdisplayskip=\abovedisplayskip
    \abovedisplayshortskip=0pt plus2pt
    \belowdisplayshortskip=4pt plus2pt minus1pt}
\fi
\makeatother

\setlist[enumerate,1]{label=\textup{(\arabic*)},ref=\textup{(\arabic*)}}

\makeatletter
\renewcommand{\paragraph}{\@startsection{paragraph}{4}%
  \z@{.35\baselineskip\@plus.15\baselineskip}{-.75em}%
  {\normalfont\bfseries}}
\renewcommand{\ps@plain}{%
  \let\@oddhead\@empty\let\@evenhead\@empty
  \def\@oddfoot{\normalfont\normalsize\hfil\thepage\hfil}%
  \let\@evenfoot\@oddfoot}
\renewcommand{\ps@headings}{%
  \let\@oddfoot\@empty\let\@evenfoot\@empty
  \def\@oddhead{\normalfont\fontsize{9bp}{11bp}\selectfont
    \hfil\s@title\hfil\llap{\thepage}}%
  \def\@evenhead{\normalfont\fontsize{9bp}{11bp}\selectfont
    \rlap{\thepage}\hfil J. Long and Z. Zhao\hfil}}
\AtBeginDocument{%
  \immediate\write\@auxout{\string\citation{arma-bibliography-settings}}}

\if@referee\else
  \def\Titlefont{\normalfont\fontsize{17bp}{21bp}\itshape\selectfont\centering}
  \def\Authorfont{\normalfont\fontsize{12bp}{14bp}\scshape\selectfont\centering}
  \def\addressfont{\normalfont\fontsize{9bp}{11bp}\selectfont\centering}
  \def\abstractheadfont{\normalfont\normalsize\bfseries\centering}
  \def\abstractfont{\normalfont\normalsize
    \leftskip=0pt\rightskip=0pt\parindent=1.5em\parfillskip=0pt plus1fil}
  \renewcommand{\abstracthead}{\@startsection{section}{1}{\z@}
    {12pt plus2pt minus1pt}{9pt}{\abstractheadfont}}
  \def\keywordfont{\normalfont\fontsize{9bp}{11bp}\selectfont
    \leftskip=0pt\rightskip=0pt plus1fil}
  \def\sectionfont{\normalfont\normalsize\bfseries\centering}
  \def\subsectionfont{\normalfont\normalsize\itshape\centering}
  \def\@seccntformat#1{\csname the#1\endcsname.\enspace}
  \patchcmd{\@maketitle}{\else\vskip21pt\fi}{\else\vskip6pt\fi}{}%
    {\PackageError{arma-layout}{Title spacing patch failed}{Check sn-jnl.cls.}}
  \patchcmd{\@maketitle}{\removelastskip\vskip20pt\nointerlineskip}%
    {\removelastskip\vskip12pt\nointerlineskip}{}%
    {\PackageError{arma-layout}{Title spacing patch failed}{Check sn-jnl.cls.}}
  \patchcmd{\@maketitle}{\vskip20pt\artauthors}{\vskip12pt\artauthors}{}%
    {\PackageError{arma-layout}{Author spacing patch failed}{Check sn-jnl.cls.}}
  \patchcmd{\@maketitle}{\removelastskip\vskip24pt}{\removelastskip\vskip8pt}{}%
    {\PackageError{arma-layout}{Address spacing patch failed}{Check sn-jnl.cls.}}
  \patchcmd{\@maketitle}{\removelastskip\vskip24pt}{\removelastskip\vskip8pt}{}%
    {\PackageError{arma-layout}{Email spacing patch failed}{Check sn-jnl.cls.}}
  \patchcmd{\@maketitle}{\removelastskip\vskip36pt}{\removelastskip\vskip18pt}{}%
    {\PackageError{arma-layout}{Front matter spacing patch failed}{Check sn-jnl.cls.}}
  \apptocmd{\maketitle}{\thispagestyle{plain}}{}%
    {\PackageError{arma-layout}{Title page style patch failed}{Check sn-jnl.cls.}}
  
\fi

\if@referee\else
  \newtheoremstyle{thmstyleone}
    {8pt plus2pt minus1pt}{8pt plus2pt minus1pt}
    {\normalsize\itshape}{0pt}{\normalsize\bfseries}{}{.5em}
    {\thmname{#1}\thmnumber{\@ifnotempty{#1}{ }\@upn{#2}.}%
      \thmnote{ {\the\thm@notefont(#3)}}}
  \newtheoremstyle{thmstyletwo}
    {8pt plus2pt minus1pt}{8pt plus2pt minus1pt}
    {\normalsize\normalfont}{0pt}{\normalsize\itshape}{}{.5em}
    {\thmname{#1}\thmnumber{\@ifnotempty{#1}{ }{#2}.}%
      \thmnote{ {\the\thm@notefont(#3)}}}
  \newtheoremstyle{thmstylethree}
    {8pt plus2pt minus1pt}{8pt plus2pt minus1pt}
    {\normalsize\normalfont}{0pt}{\normalsize\bfseries}{}{.5em}
    {\thmname{#1}\thmnumber{\@ifnotempty{#1}{ }\@upn{#2}.}%
      \thmnote{ {\the\thm@notefont(#3)}}}
  \AtBeginDocument{%
    \renewenvironment{proof}[1][\proofname]{%
      \par\removelastskip\pushQED{\qed}%
      \normalsize\normalfont\topsep6\p@\@plus2\p@\@minus1\p@\relax
      \trivlist
      \item[\hskip\labelsep\itshape #1\@addpunct{.}]\ignorespaces
    }{\popQED\endtrivlist\@endpefalse}}
\fi
\makeatother

\theoremstyle{thmstyleone}
\newtheorem{theorem}{Theorem}[section]
\newtheorem{lemma}[theorem]{Lemma}
\newtheorem{proposition}[theorem]{Proposition}
\newtheorem{corollary}[theorem]{Corollary}
\theoremstyle{thmstylethree}
\newtheorem{assumption}[theorem]{Assumption}
\theoremstyle{thmstyletwo}
\newtheorem{eexample}[theorem]{Example}
\numberwithin{equation}{section}

\newcommand{\R}{\mathbb R}
\DeclareMathOperator{\Law}{Law}
\DeclareMathOperator{\supp}{supp}
\DeclareMathOperator{\dist}{dist}
\DeclareMathOperator{\Lip}{Lip}
\newcommand{\1}{\mathbf 1}
\AtBeginDocument{\allowdisplaybreaks[2]}
\newcommand{\manuscripttitle}[2][]{%
  \title[#1]{#2}%
  \hypersetup{pdftitle={#2}}%
}
\hypersetup{
  hidelinks,
  pdfauthor={Jihao Long; Zhenhua Zhao},
  pdfsubject={Manuscript prepared for Archive for Rational Mechanics and Analysis}
}

\begin{document}

\manuscripttitle[Weighted Hessian estimates]{%
Optimal Weighted \texorpdfstring{$L^2$}{L2} Hessian Estimates for Parabolic Equations
under General Diffusion Marginals}
\author*[1]{\fnm{Jihao} \sur{Long}}\email{longjh1998@gmail.com}
\author[2]{\fnm{Zhenhua} \sur{Zhao}}\email{zhenhuazhao@stu.pku.edu.cn}
\affil*[1]{\orgname{Institute for Advanced Algorithms Research},
  \orgaddress{\city{Shanghai}, \country{China}}}
\affil[2]{\orgname{Peking University},
  \orgaddress{\city{Beijing}, \country{China}}}

\abstract{
We study weighted $L^2$ Hessian estimates for parabolic terminal equations
under the marginals of a prescribed diffusion, allowing singular initial
laws and different principal coefficients in the equation and the diffusion.
Under coefficient regularity, uniform ellipticity and bounded reference
drift, we characterize the limiting optimal constant uniformly over
shrinking subintervals with zero terminal data. Its finiteness is equivalent
to linear well-posedness in the weighted parabolic Sobolev space on the
full interval. For initial laws satisfying weighted doubling and coercivity
conditions, this constant is the maximum of contributions from heat
evolutions of rescaled initial measures, constant weights and exponential
mixtures. The phase contributions are bounded above in terms of the
coefficients at time zero outside the initial support and at spatial
infinity. Under radial reference structure, we obtain exact
coefficient formulas invariant under bounded changes of reference drift.
In the matching case, with reference covariance $a$, principal matrix
$a/2$ and normalized Hessian $a^{1/2}D^2u\,a^{1/2}$, the constant is $2$
for every initial law with time weight $t^\alpha$, $\alpha\ge1/2$;
point starts attain $2\sqrt2$ at $\alpha=0$.
Whenever the limiting constant is finite, increasing damping makes the
full-interval Hessian constant converge to it and the value and gradient
gains tend to zero. As applications, we obtain nonlinear solvability and
residual stability for small Hessian perturbations. For solutions and
candidates with a common regularity bound, secant estimates give linear
stability under a relative spectral gap, while covariance interpolation
extends residual stability beyond it.
}
\keywords{weighted Sobolev estimates; parabolic equations; diffusion marginals; heat kernel estimates}
\pacs[2020 Mathematics Subject Classification]{Primary 35B45, 35K10; Secondary 35B65, 60H10, 60H30}
\maketitle

\setcounter{tocdepth}{2}
\tableofcontents

\section{Introduction}
\label{sec:introduction}

\subsection{Problem and motivation}
\label{subsec:intro-setting}

Diffusion processes link stochastic models to second-order parabolic
partial differential equations (PDEs) through their generators. In stochastic control, optimizing
over diffusion covariances leads to Bellman equations with nonlinear
Hessian dependence \citep{fleming-soner-2006}.
Probabilistic approximation evaluates solution fields and derivatives
along diffusion paths, measuring errors in mean square under their law
\citep{guo-zhang-zhuo-2015,pham-warin-germain-2021,han-hu-long-zhao-2026}.
This motivates regularity and stability estimates under prescribed
diffusion marginals.

For Hessian-dependent equations, we seek to bound the full parabolic
Sobolev error by equation and terminal-data errors in this sampling norm.
Fix $T>0$, $I=(0,T)$, and consider the linear terminal problem
\begin{equation}
P_Au:=-u_t-A(t,x):D^2u=f,
\qquad \operatorname{Tr}_T u=g.
\label{eq:intro-equation}
\end{equation}
The observation law is supplied by a reference diffusion
\begin{equation}
\mathrm{d}X_t=b(t,X_t)\,\mathrm{d}t+a(t,X_t)^{1/2}\,\mathrm{d}W_t,
\qquad \Law(X_0)=\mu_0,
\label{eq:intro-diffusion}
\end{equation}
with marginals $\mu_t=\rho(t,\cdot)\mathcal L^d=\Law(X_t)$ for $t>0$.
The norm is fixed by $a,b,\mu_0$, independently of the
PDE matrix $A$. This permits \emph{coefficient mismatch}, $A\ne a/2$,
and choice of linear
parts for nonlinear iteration in the same space.
The standing assumptions on $a,A$ are boundedness, uniform ellipticity
and the regularity specified in Assumption~\ref{ass:coefficients};
the reference drift $b$ is bounded and Borel measurable.

The prescribed initial law $\mu_0$ may be singular: point-start
marginals, for example, concentrate as $t\downarrow0$.
The time weight $t^\alpha$, with $\alpha\ge0$ throughout,
adjusts integrability near this boundary. For $J=(T_1,T_2)\subseteq I$, define
\begin{equation}
\|f\|_{\alpha,J}^2
 =\int_J t^\alpha\mathbb E|f(t,X_t)|^2\,\mathrm{d}t
 =\int_J\int_{\R^d}t^\alpha\rho(t,x)|f(t,x)|^2\,\mathrm{d}x\,\mathrm{d}t.
\label{eq:intro-marginal-norm}
\end{equation}
Away from time zero, this time weight is equivalent to a constant.
Our regularity question is whether every $f$ with
$\|f\|_{\alpha,I}<\infty$ and every $g\in H^1(\mu_T)$ give a unique
solution of \eqref{eq:intro-equation} with $u,u_t,Du,D^2u$ in this
weighted $L^2$ space and the sum of their norms bounded by
$C(\|f\|_{\alpha,I}+\|g\|_{H^1(\mu_T)})$.

Concentration near time zero and rapid decay at infinity can obstruct weighted Hessian estimates even under uniform ellipticity.
We identify the governing density profiles and their interaction with the PDE coefficients.

\subsection{Optimal linear estimates}
\label{subsec:intro-results}

With zero terminal data, the Hessian bound is decisive:
Lemma~\ref{lem:direct-energy}\ref{item:energy-short-interval}
and the equation then control $u,Du,u_t$ by the source.
We measure the Hessian by
$\mathcal H_qu=q^{1/2}D^2u\,q^{1/2}$ in the Frobenius norm, with
$q$ bounded, uniformly elliptic and regular as in
Assumption~\ref{ass:coefficients}; $q=I_d$ gives the usual Hessian norm.
For $h>0$, consider the best constant $C$ in
\begin{equation*}
\|\mathcal H_qu\|_{\alpha,J}\le C\|P_Au\|_{\alpha,J},
\qquad \operatorname{Tr}_{T_2}u=0,
\end{equation*}
required to hold for every $J=(T_1,T_2)\subseteq I$ of length at most
$h$ and every $u$ in the corresponding weighted Sobolev class.
Its limit as $h\downarrow0$ is the limiting Hessian gain
$G_\alpha(I)\in[0,\infty]$ (see \eqref{eq:short-time-constants}).
For every initial probability law, finiteness of this gain is
equivalent to full-interval well-posedness in the weighted parabolic
Sobolev space (Theorem~\ref{thm:linear-theory}).

Three density profiles already arise for one-dimensional
Brownian motion started at zero, with
$\rho(t,x)=(2\pi t)^{-1/2}\mathrm{e}^{-x^2/(2t)}$.
Rescaling at the starting point on the heat scale retains the heat
profile; at a fixed positive-time point, the normalized density
becomes constant. On the smaller scale $r=t/x$, for fixed $x>0$,
\begin{equation*}
\frac{\rho(t+r^2s,x+rz)}{\rho(t,x)}\longrightarrow \mathrm{e}^{s/2-z}
\qquad\text{as }t\downarrow0.
\end{equation*}
The same locally uniform limit holds for fixed $t\in I$
as $x\to+\infty$. In both cases, the initial boundary recedes to
$-\infty$, giving an exponential profile on the whole time line.

Assumption~\ref{def:initial-laws} covers compactly supported doubling
probability laws, including point masses, and finite mixtures of
nondegenerate Gaussian laws. Under this assumption,
Theorem~\ref{thm:sharp-limit} makes the preceding classification exact
for general reference diffusions:
\(
G_\alpha(I)
 =\max\{G_\alpha^{\mathrm{init}}(I),
         G^{\mathrm{flat}}(I),G^{\mathrm{phase}}(I)\}
\)
in
\(
[0,\infty].
\)
The three contributions come from heat evolutions of rescaled initial
measures, constant weights and exponential-mixture weights, respectively.
Each weight is paired with the coefficient limits along the same
rescaling sequence. The flat term is finite by ellipticity; additional
loss and infinite gain come from the initial or phase terms.

For a phase model with coefficient limits $(a_*,A_*,q_*)$, the spatial
density-ratio limit is $\int_{\R^d}\mathrm{e}^{-2\zeta\cdot z}\,\lambda(\mathrm{d}\zeta)$
for a compactly supported probability measure $\lambda$ with
$0\notin\supp\lambda$
\eqref{eq:spatial-phase-identification}.
Proposition~\ref{prop:whole-line-gains} gives its contribution to
$G^{\mathrm{phase}}(I)$ as
\begin{equation*}
\sup_{\zeta\in\supp\lambda}
 \frac{\zeta^{\mathsf T}q_*\zeta}
      {\zeta^{\mathsf T}(a_*-A_*)\zeta},
\end{equation*}
where a nonpositive denominator gives $+\infty$.
For fixed coefficient limits, only the directions in $\supp\lambda$
matter; the full phase contribution is the supremum over all jointly
realized models.

In the matching case $q=2A=a$, this quotient equals $2$ in every
direction. Since the flat contribution is also $2$
(Proposition~\ref{prop:whole-line-gains}), only the initial contribution
can raise the gain above this value.
Comparison with point sources gives the matching estimates for
arbitrary initial probability laws
(Theorem~\ref{thm:matching-arbitrary-laws}):
$G_\alpha((T_1,T_2))=2$ when $0<T_1<T_2\le T$, while on the full interval
\begin{equation*}
2\le G_0(I)\le2\sqrt2,\qquad
G_\alpha(I)=2\quad\text{for }\alpha\ge1/2.
\end{equation*}
Point starts maximize the gain and attain $2\sqrt2$ when $\alpha=0$.
Proposition~\ref{prop:smooth-initial-absorption} gives
$G_\alpha(I)=2$ for nondegenerate Gaussian initial laws at every
$\alpha$; the matching theorem extends this to their finite mixtures.

For mismatched coefficients, maximizing over all directions
gives phase upper bounds in terms of coefficients at time zero outside the
initial support and at spatial infinity
(Proposition~\ref{prop:coefficient-upper-bound}, under
Assumption~\ref{def:initial-laws}). For every initial probability law,
uniform point-kernel comparison gives finite gain if, for some
$\delta\in(0,1)$, $\tau\in(0,T)$ and $R>0$,
\(
A\preceq(1-\delta)a
\,\text{on }([0,T]\times\R^d)
 \setminus\bigl([\tau,T]\times\overline B_R\bigr)
\)
(Corollary~\ref{cor:uniform-point-bound}).
A global spectral gap alone does not suffice for Borel
principal matrices: the gain can still be infinite
(Example~\ref{ex:rough-principal}).

Radial structure allows these phase bounds to be replaced by exact
formulas involving radial coefficient ratios.
Let $\mu_{\mathrm{rad}}$ be a radial initial law satisfying
Assumption~\ref{def:initial-laws}, and assume that the reference
covariance satisfies
\begin{equation*}
a(t,x)x=\kappa_{\mathrm{rad}}(t,|x|)x,
\qquad |D_xa(t,x)|\le C(1+|x|)^{-\sigma}
\end{equation*}
for some $\sigma>0$, uniformly in time.
The radial direction is an eigenvector; tangential diffusion may remain
anisotropic and depend on angle.
Under these conditions, Theorem~\ref{thm:radial-formulas} determines
the phase contributions exactly from the radial components of $a-A$
and $q$, at time zero outside the initial support and at spatial infinity.
The full gain is independent of the bounded Borel reference drift
and unchanged under the bounded changes of initial density
specified in Assumption~\ref{ass:radial-structure}.

For nonlinear applications, whenever $G_\alpha(I)<\infty$, the Hessian
gain for $P_A+\kappa$ on $I$ tends to $G_\alpha(I)$ as
$\kappa\to\infty$, while the value and gradient gains tend to zero
(Proposition~\ref{prop:poisson}).
Bounded Borel drifts in the PDE preserve this finite limiting gain
(see \eqref{eq:borel-drift-invariance}).
The limiting gain thus sets the contraction threshold for Hessian
perturbations on the full interval.

\subsection{Nonlinear solvability and stability}
\label{subsec:intro-nonlinear}

For an arbitrary initial probability law, consider a globally Lipschitz
perturbation of a linear part with finite gain.
The condition $L_2G_\alpha(I)<1$, where
$L_2$ is the Hessian Lipschitz constant in the $q$ normalization, makes
source iteration contractive in a sufficiently damped norm.
The value and gradient Lipschitz constants may be arbitrarily large.
Proposition~\ref{prop:nonlinear-terminal-problem} yields a unique
solution in the weighted parabolic Sobolev space and controls every
candidate's error in that space by its equation residual and terminal
$H^1(\mu_T)$ error.

Additional regularity allows comparison through the principal matrices
generated by the nonlinearity itself. Suppose the solution and
candidates lie in a bounded convex subset of
$C_b^{1+\gamma/2,2+\gamma}([0,T]\times\R^d)$, with $0<\gamma<1$.
Under the derivative and ellipticity assumptions of
Theorem~\ref{thm:tangent-secant-stability}, secant linearization gives
a linear residual estimate when these matrices lie uniformly below
the reference covariance outside a compact set separated from time zero.
Without this gap, a linear residual estimate can fail even on a fixed
H\"older ball (Example~\ref{ex:nonlinear-gap-necessity}).
An interpolation estimate remains available: enlarging the reference
covariance and comparing the resulting marginal densities bounds the
error by a power of the residual
(Theorem~\ref{thm:nonlinear-interpolation}).
The admissible exponent depends on the relative size of the principal
matrices, while both error and residual are measured under the original
diffusion.

For equations with convex Hessian level sets,
Proposition~\ref{prop:quasiconcave-terminal-problem} gives a unique bounded
viscosity solution with global second-order parabolic H\"older bounds
up to the terminal time, under its structural and data assumptions and
independently of the reference diffusion and initial law.
The proof adapts bounded-domain viscosity existence
\citep{crandall-kocan-lions-swiech-1999} and frozen interior estimates
\citep{goffi-2024} to full-jet dependence and the terminal boundary on $\R^d$.
The stability estimates additionally require the respective
principal-matrix conditions of
Theorems~\ref{thm:tangent-secant-stability} and~\ref{thm:nonlinear-interpolation}.

\subsection{Related literature}
\label{subsec:intro-literature}

Gaussian evolution laws for stable nonautonomous Ornstein--Uhlenbeck
dynamics extending to the infinite past yield $L^2$ theory with arbitrary
$L^2$ forcing and $H^1$ terminal data \citep{GeissertLunardi2008},
and an $L^p$ extension \citep{geissert-lorenzi-schnaubelt-2010}, allowing
linearly growing drifts. For state-dependent coefficients, related work
studies regularity under an invariant measure \citep{lorenzi-lunardi-2006},
evolution families and gradient bounds for $C_b$ data
\citep{kunze-lorenzi-lunardi-2010}, and the relation between terminal
fractional smoothness and derivative bounds for homogeneous backward
solutions along diffusions
\citep{GeissGobet2014}. Here the initial time is finite, its law may be
singular, the reference drift is bounded, and the PDE matrix is independent.

Classical Sobolev estimates and solvability use Lebesgue spaces
\citep{Krylov2008}. Weighted and mixed-norm estimates for fully nonlinear
equations allow relaxed convexity and almost VMO (vanishing mean oscillation)
dependence on space and time \citep{DongKrylov2019}; the parabolic weighted
estimates require $p>d+1$ and Muckenhoupt $A_{p/(d+1)}$ weights.
Abstract maximal regularity treats fixed spatial Banach spaces
\citep{DenkHieberPruess2003}, including interpolation and traces with temporal
power weights \citep{meyries-schnaubelt-2012-traces}.
Even Brownian point-start weights $t^\alpha\rho(t,x)$ lie outside these
global parabolic Muckenhoupt classes. We use density evolution and local
profiles to determine their interaction with the PDE.

Campanato's near-operator theory \citep{Campanato1994} transfers
invertibility through bounds on operator differences. Cordes conditions
give strong solvability for parabolic Hamilton--Jacobi--Bellman equations
on bounded convex domains in Lebesgue Sobolev spaces \citep{SmearsSuli2016}.
Here the limiting gain under diffusion marginals controls Hessian perturbations.

For probabilistic approximation, deep Picard iteration
\citep{han-hu-long-zhao-2026} raises the question of convergence with
Hessian dependence. The error-propagation analysis in
\citep{jia-ouyang-pham-zhou-2026} assumes a weighted Sobolev contraction
and verifies a sufficient condition for stationary Langevin dynamics.
Our gains give contraction conditions under finite-time diffusion marginals.
Our residual estimates also supply the PDE stability used with quadrature
control in generalization-error analysis \citep{mishra-molinaro-2023}.

Heat-kernel asymptotics relate short-time decay to intrinsic distance
\citep{Varadhan1967Heat,norris-1997}; kinetic-action estimates give
quantitative bounds \citep{norris-stroock-1991}. Kernel comparisons and
endpoint-conditioned path concentration give our density comparisons
to a factor $1+o(1)$, uniformly even over escaping centers.
We identify the classical exponential-mixture representations of
nonnegative ancient heat solutions \citep{widder-1963-appell,lin-zhang-2019}
jointly with coefficient limits under density rescaling.

Sections~\ref{sec:sharp-constants}--\ref{sec:proof} prove
the three-term formula and establish linear terminal theory.
Section~\ref{sec:explicit-constants} gives gain formulas and bounds;
Section~\ref{sec:nonlinear-stability} treats nonlinear solvability and residual stability.
The density estimates and model identification are proved in
Sections~\ref{sec:densities}--\ref{sec:phase-directions}.
Section~\ref{sec:conclusion} discusses open problems; the appendices provide
analytic tools, model computations and counterexamples.

\section{Local models and optimal constants}
\label{sec:sharp-constants}

We use $|\cdot|$ for the Euclidean norm of vectors and the Frobenius norm
of matrices, writing $|\cdot|_{\mathrm F}$ when distinction from other
matrix norms is needed.
We write $A_1:A_2=\operatorname{tr}(A_1^{\mathsf T}A_2)$,
use $\preceq$ for the order on symmetric matrices, and denote
the identity matrix by $I_d$.
For spacetime points $P=(t,x)$ and $Q=(s,y)$, we use the parabolic
distance $d_p(P,Q)=|x-y|+|t-s|^{1/2}$.
Unweighted $L^p$ and Sobolev spaces use Lebesgue measure.
We write $L^2(w\,\mathrm{d}x)$ for integration over $\R^d$ and
$L^2(w\,\mathrm{d}t\,\mathrm{d}x)$ for integration over $\R\times\R^d$;
restrictions to smaller domains are displayed explicitly.
$\|f\|_\infty$ denotes the $L^\infty$ norm on the stated
domain of $f$.
Positive constants $c,C$ may change from line to line.
Parameter dependence is left implicit when clear from the
statement or context. Dependence and uniformity needed for
parameter choices or limiting arguments are stated explicitly.
The same convention applies to $O,o,\asymp$.
We use the reference diffusion, PDE operator and marginal norms from
\eqref{eq:intro-diffusion}, \eqref{eq:intro-equation} and
\eqref{eq:intro-marginal-norm}, respectively, with
$\mathcal H_qu=q^{1/2}D^2u\,q^{1/2}$.

\subsection{Weighted spaces and gains}
\label{subsec:equation}

\begin{assumption}[Coefficients]\label{ass:coefficients}
The coefficients are defined on $[0,T]\times\R^d$.
The functions $a,b$ are Borel measurable, $a$ is symmetric, and
\begin{equation*}
0<\lambda_a I_d\preceq a(t,x)\preceq\Lambda_a I_d,
\qquad |b(t,x)|\le B_b,
\qquad
|a(t,x)-a(t,y)|\le L_a|x-y|.
\end{equation*}
There are a nondecreasing continuous modulus $\omega_a$ and $h_*>0$ such that
\begin{equation*}
\sup_x|a(t,x)-a(v,x)|\le\omega_a(|t-v|),\quad \omega_a(0) = 0,
\qquad
\int_0^{h_*}\frac{\sqrt{\omega_a(h)}}h\,\mathrm{d}h<\infty.
\end{equation*}
The matrices $A,q$ are symmetric and satisfy
\begin{equation*}
0<\lambda_A I_d\preceq A(t,x)\preceq\Lambda_A I_d,
\qquad
0<\lambda_q I_d\preceq q(t,x)\preceq\Lambda_q I_d.
\end{equation*}
There is a nondecreasing continuous modulus $\omega_{A,q}$,
with $\omega_{A,q}(0)=0$, such that
\begin{equation*}
|A(t,x)-A(v,y)|+|q(t,x)-q(v,y)|
\le \omega_{A,q}\bigl(d_p((t,x),(v,y))\bigr).
\end{equation*}
We collect the \emph{reference data} and the \emph{$A,q$ ellipticity data} as
\begin{equation*}
\mathfrak R=(d,T,\lambda_a,\Lambda_a,L_a,B_b,\omega_a,h_*),
\qquad
\mathfrak E=(\lambda_A,\Lambda_A,\lambda_q,\Lambda_q).
\end{equation*}
The modulus $\omega_{A,q}$ is listed separately when it is used.
\end{assumption}

Any uniform temporal H\"older modulus $\omega_a(h)\le Ch^\beta$,
with $\beta>0$, satisfies the square-root Dini condition.
By Lemma~\ref{lem:kernel-stability}\ref{item:kernel-estimates}, the reference diffusion has a
strictly positive, jointly continuous transition density
$\rho(r,y;t,x)$ with respect to $\mathrm{d}x$, for $0\le r<t\le T$.
The current initial law $\mu_0$ gives
\begin{equation*}
\rho(t,x)=\int_{\R^d}\rho(0,y;t,x)\,\mu_0(\mathrm{d}y),\qquad
\mu_t=\rho(t,\cdot)\mathcal L^d,\quad\forall t\in(0,T].
\end{equation*}
Here $\mathcal L^d$ denotes Lebesgue measure.
For an interval $U\subseteq\R$ and $h > 0$, write
\begin{equation*}
\mathcal J(U)
 =\{(T_1,T_2)\subseteq U:-\infty<T_1<T_2<\infty\},\quad
\mathcal J_h(U)
 =\{J\in\mathcal J(U):|J|\le h\}.
\end{equation*}
These intervals are nonempty, bounded and open, and may share endpoints with $U$.

The following definitions provide a common notation for spaces and
gains associated with the diffusion density and the model weights
introduced below. Fix $J=(T_1,T_2)\in\mathcal J(\R)$,
with $T_1\ge0$ if $\alpha>0$; for $\alpha=0$, $T_1$ and $T_2$ may be negative.
For a positive continuous weight $w$ on $(T_1,T_2]\times\R^d$, set
\(
L^2_\alpha(J;w)
 :=L^2(J\times\R^d,t^\alpha w(t,x)\,\mathrm{d}t\,\mathrm{d}x)
\)
with norm
\(
\|f\|_{\alpha,J;w}:=\|f\|_{L^2_\alpha(J;w)}.
\)

The weighted Sobolev space and its zero-terminal subspace are
\begin{equation*}
\begin{aligned}
\mathcal W_\alpha(J;w)
 &=\{u:u,u_t,Du,D^2u\in L^2_\alpha(J;w)\},\\
\mathcal W_{\alpha,0}(J;w)
 &=\{u\in\mathcal W_\alpha(J;w):
                  \operatorname{Tr}^{\mathrm{loc}}_{T_2}u=0\},
\end{aligned}
\end{equation*}
with norm $\|u\|_{\mathcal W_\alpha(J;w)}^2
 =\|u\|_{\alpha,J;w}^2+\|u_t\|_{\alpha,J;w}^2+\|Du\|_{\alpha,J;w}^2+\|D^2u\|_{\alpha,J;w}^2$.
All derivatives are distributional, and
$\operatorname{Tr}^{\mathrm{loc}}_{T_2}$ denotes the $H^1_{\mathrm{loc}}$
terminal trace in Lemma~\ref{lem:graph-core}.
No initial trace condition is imposed at $T_1$.
For approximation and localization, define the smooth test class
\begin{equation*}
\mathcal C_{\alpha,0}(J;w)=\{u\in C^\infty(J\times\R^d)\cap\mathcal W_\alpha(J;w):\exists R,\varepsilon>0,\ \supp(u)\subset(T_1,T_2-\varepsilon]\times\overline B_R\}
\end{equation*}
where $B_R=B(0,R)$,  $B(x,r)=\{z\in\R^d:|z-x|<r\}$,
and $\supp(u)$ is taken relative to $J\times\R^d$.
The radius $R$ and terminal gap $\varepsilon$ may depend on $u$.

Let $I=(T_-,T_+)\in\mathcal J(\R)$, with $T_-\ge0$ if $\alpha>0$.
For bounded measurable matrix fields $A,q$ on $I\times\R^d$ that
are symmetric and uniformly positive definite, and a positive
continuous weight $w$ on $(T_-,T_+]\times\R^d$, define,
for $J\in\mathcal J(I)$ and $h>0$,
\begin{equation}
\begin{aligned}
G_\alpha^{\mathrm{fix}}(J;w,A,q)
 &=\sup_{u\in\mathcal W_{\alpha,0}(J;w)}
       \frac{\|\mathcal H_qu\|_{\alpha,J;w}}{\|P_Au\|_{\alpha,J;w}}
  =\sup_{u\in\mathcal C_{\alpha,0}(J;w)}
       \frac{\|\mathcal H_qu\|_{\alpha,J;w}}{\|P_Au\|_{\alpha,J;w}},\\
G_{\alpha,h}(I;w,A,q)
 &=\sup_{J\in\mathcal J_h(I)} G_\alpha^{\mathrm{fix}}(J;w,A,q),
\quad
G_\alpha(I;w,A,q)=\lim_{h\downarrow0}G_{\alpha,h}(I;w,A,q).
\end{aligned}
\label{eq:short-time-constants}
\end{equation}
Thus $G_\alpha^{\mathrm{fix}}$ measures the gain on a fixed interval,
$G_{\alpha,h}$ allows that interval to vary with length at most
$h$, and $G_\alpha$ is the limiting gain as $h\downarrow0$.
The equality of the two suprema defining $G_\alpha^{\mathrm{fix}}$ follows from
Lemma~\ref{lem:graph-core} below.
For all norm quotients defining gains in this paper, we use
$0/0=0$ and assign $+\infty$ to a positive numerator
divided by zero.
The limit in \eqref{eq:short-time-constants} exists in $[0,\infty]$ by
monotonicity. Every subinterval inherits the same time weight
$t^\alpha$. The intervals may shrink toward any time in
$\overline I$.

\begin{lemma}\label{lem:graph-core}
The space $\mathcal W_\alpha(J;w)$ is Hilbert,
the trace map
$\operatorname{Tr}^{\mathrm{loc}}_{T_2}:\mathcal W_\alpha(J;w)\to H^1_{\mathrm{loc}}(\R^d)$
is continuous, and $\mathcal W_{\alpha,0}(J;w)$ is a closed subspace. Moreover,
\(
\mathcal C_{\alpha,0}(J;w)\subset\mathcal W_{\alpha,0}(J;w),
\,
\overline{\mathcal C_{\alpha,0}(J;w)}^{\,\mathcal W_\alpha(J;w)}
 =\mathcal W_{\alpha,0}(J;w).
\)
\end{lemma}

In the proof of this lemma and other proofs below, we use the following
cutoffs. Fix in each dimension a radial
$\chi\in C_c^\infty(B_2)$ with $0\le\chi\le1$ and $\chi=1$ on $B_1$.
For $R>0$, put $\chi_R(x)=\chi(x/R)$. Then
\begin{equation}
\chi_R=1\ \text{on }B_R,\qquad
\supp\chi_R\subset B_{2R},\qquad
\|D^m\chi_R\|_\infty\le C_dR^{-m},\quad m=1,2.
\label{eq:spatial-cutoffs}
\end{equation}

\begin{lemma}\label{lem:weight-operations}
Let $I,A,q$ be as in the definition \eqref{eq:short-time-constants}.
\begin{enumerate}
\item\label{item:positive-mixtures}
Let $(Y,\sigma)$ be a $\sigma$-finite measure space, and let
$w_y:(T_-,T_+]\times\R^d\to(0,\infty)$ be jointly measurable
in $(y,t,x)$ and continuous in $(t,x)$ for $\sigma$-a.e.
$y$, where $I=(T_-,T_+)$. If
$w=\int_Yw_y\,\sigma(\mathrm{d}y)$ is finite, positive and continuous, then
for every $J\in\mathcal J(I)$ and $h>0$,
\begin{equation*}
\begin{aligned}
G_\alpha^{\mathrm{fix}}(J;w,A,q)
 &\le\operatorname*{ess\,sup}_{y\in Y}G_\alpha^{\mathrm{fix}}(J;w_y,A,q),\\
G_{\alpha,h}(I;w,A,q)
 &\le\operatorname*{ess\,sup}_{y\in Y}G_{\alpha,h}(I;w_y,A,q).
\end{aligned}
\end{equation*}
More generally, fix $J\in\mathcal J(I)$, $E\subseteq\R^d$ and
$0\le C<\infty$. Suppose that, for $\sigma$-a.e. $y$,
every $u\in\mathcal C_{\alpha,0}(J;w_y)$ supported in $J\times E$ satisfies
\begin{equation*}
\|\mathcal H_qu\|_{\alpha,J;w_y}
 \le C\|P_Au\|_{\alpha,J;w_y}.
\end{equation*}
Then the same estimate holds with $w$ in place of $w_y$ for every
$u\in\mathcal C_{\alpha,0}(J;w)$ with $\supp u\subseteq J\times E$.
\item\label{item:time-weights}
For $\alpha'\ge\alpha$, a positive continuous weight $w$ and
$J=(T_1,T_2)\in\mathcal J(I)$ with $T_1\ge0$,
\begin{equation*}
G_{\alpha'}^{\mathrm{fix}}(J;w,A,q)\le\sup_{T_1<s<T_2}G_\alpha^{\mathrm{fix}}((s,T_2);w,A,q).
\end{equation*}
Consequently, if $I\subseteq(0,\infty)$, $\forall h > 0$,
\begin{equation*}
G_{\alpha',h}(I;w,A,q)\le G_{\alpha,h}(I;w,A,q),\qquad
G_{\alpha'}(I;w,A,q)\le G_\alpha(I;w,A,q).
\end{equation*}
\item\label{item:interval-gains}
Suppose $I=(0,T)$, $w=\rho$ for an arbitrary initial probability law,
and Assumption~\ref{ass:coefficients} holds. For
$J,K\in\mathcal J(I)$ with $J\subseteq K$, in $[0,\infty]$
\begin{equation*}
G_\alpha^{\mathrm{fix}}(J;\rho,A,q)
 \le G_\alpha^{\mathrm{fix}}(K;\rho,A,q),\quad
G_{\alpha,|K|}(K;\rho,A,q)
 =G_\alpha^{\mathrm{fix}}(K;\rho,A,q).
\end{equation*}
\end{enumerate}
\end{lemma}

Lemmas~\ref{lem:graph-core} and~\ref{lem:weight-operations} are proved
in Appendix~\ref{subsec:weighted-spaces}.
For the reference diffusion, take $w=\rho$ and the coefficient fields
$A,q$ from Assumption~\ref{ass:coefficients}.
On intervals $J\in\mathcal J((0,T))$, we omit the weight argument
from the spaces and norms; in particular,
$L^2_\alpha(J)=L^2_\alpha(J;\rho)$ and
$\|f\|_{\alpha,J}=\|f\|_{\alpha,J;\rho}$.
For these intervals and $h>0$, we also write
\begin{equation*}
G_\alpha^{\mathrm{fix}}(J)=G_\alpha^{\mathrm{fix}}(J;\rho,A,q),\quad
G_{\alpha,h}(J)=G_{\alpha,h}(J;\rho,A,q),\quad
G_\alpha(J)=G_\alpha(J;\rho,A,q).
\end{equation*}

\subsection{Initial law}
\label{subsec:initial-classes}

The following assumption combines a doubling condition on a reference measure
with a coercivity estimate on the resulting weighted probability law.
It allows singular initial laws and noncompact support.
Write $\langle x\rangle=(1+|x|^2)^{1/2}$.

\begin{assumption}[Initial law]\label{def:initial-laws}
The initial probability law has a representation
\begin{equation}
\mu_0(\mathrm{d}y)=\mathrm{e}^{-V(y)}\,\nu(\mathrm{d}y),\qquad
\int_{\R^d}\mathrm{e}^{-V(y)}\,\nu(\mathrm{d}y)=1,
\label{eq:weighted-initial-law}
\end{equation}
where $\nu$ is a nonzero nonnegative Radon measure.
With $S=\supp\nu$, there are a center $x_0\in\R^d$ and finite constants
$K_{\mathrm{dbl}}\ge1$, $K_{\mathrm{pot}}\ge0$ and
$K_{\mathrm{coer}}\ge1$ such that, for all $x\in\R^d$, $y\in S$ and $r>0$,
\begin{equation}
\begin{gathered}
\nu(B(y,2r))\le K_{\mathrm{dbl}}\nu(B(y,r)),\quad
V\in C^1(\R^d),\quad |DV(x)|\le K_{\mathrm{pot}}\langle x-x_0\rangle.
\end{gathered}
\label{eq:initial-doubling-potential}
\end{equation}
The law must also satisfy that $\forall f\in C_c^\infty(\R^d)$,
\begin{equation}
\int_{\R^d}\langle x-x_0\rangle^2|f(x)|^2\,\mu_0(\mathrm{d}x)
\le K_{\mathrm{coer}}\left(\int_{\R^d}|Df|^2\,\mathrm{d}\mu_0
                      +\int_{\R^d}|f|^2\,\mathrm{d}\mu_0\right).
\label{eq:initial-coercivity}
\end{equation}
Here $Df$ is the gradient of the smooth test function on $\R^d$,
also when $\mu_0$ is singular. We collect the \emph{initial constants} as
\(
\mathfrak I_0=(K_{\mathrm{dbl}},K_{\mathrm{pot}},K_{\mathrm{coer}}).
\)
\end{assumption}

A simultaneous spatial translation of $\mu_0,\nu,V$ and the coefficient
fields $a,b,A,q$ moves $x_0$ to the origin while preserving
$\mathfrak R,\mathfrak E,\omega_{A,q}$ and $\mathfrak I_0$.
The corresponding change of variables preserves the weighted norms and
gains.
Constants and thresholds depending only on these data are hence
independent of the location of $x_0$.
In the abstract theory below, we take $x_0=0$ in Assumption~\ref{def:initial-laws}.

The support $S=\supp\mu_0$ is closed and nonempty.
The measure $\nu$ may have infinite mass, and $S$ may be noncompact.
The doubling condition controls relative ball masses, and the gradient
condition gives
\begin{equation}
|V(y+u)-V(y)|
 \le K_{\mathrm{pot}}\bigl(\langle y\rangle|u|+\tfrac12|u|^2\bigr),
\qquad y,u\in\R^d.
\label{eq:initial-potential-increment}
\end{equation}
Condition~\eqref{eq:initial-coercivity} is a quadratic $U$-bound in the
terminology used in \citep[Section~2]{hebisch-zegarlinski-2010}.
It controls the energy cost of localization at infinity: for
$f\in C_c^\infty(\R^d)$ and $R >0$,
\begin{equation*}
\supp f\subseteq\{|x|\ge R\}
\quad\Longrightarrow\quad
\int |Df|^2\,\mathrm{d}\mu_0
 \ge\left(\frac{1+R^2}{K_{\mathrm{coer}}}-1\right)
      \int |f|^2\,\mathrm{d}\mu_0.
\end{equation*}
This radial localization cost enters the scale introduced in
Subsection~\ref{subsec:local-rescaling}.

Gaussian tails alone do not ensure \eqref{eq:initial-coercivity}:
the probability law on $\mathbb Z$ with masses proportional to
$\mathrm{e}^{-n^2}$ has such tails and satisfies
\eqref{eq:initial-doubling-potential} with counting measure and a
quadratic potential. A smooth test supported near $n$ and equal to
one there has zero gradient at every integer, so
\eqref{eq:initial-coercivity} would require
$1+n^2\le K_{\mathrm{coer}}$ for every $n$.

\begin{eexample}[Compactly supported doubling laws]
\label{ex:compact-initial-laws}
Every compactly supported doubling probability law satisfies
Assumption~\ref{def:initial-laws}: take $\nu=\mu_0$, $V=0$ and
$K_{\mathrm{coer}}=1+R^2$ when $S\subset\overline B_R$.
Only doubling remains to be checked. All balls below are ambient,
including for lower-dimensional supports.

Finite atomic probability laws $\mu_0=\sum_{i=1}^N c_i\delta_{y_i}$,
with distinct $y_i$ and $c_i>0$, are doubling with constant
$K_{\mathrm{dbl}}=(\min_i c_i)^{-1}$: a ball centered at $y_i$ has mass at
least $c_i$, while its double has mass at most one.
For $\mu_0=\delta_y$, choosing $x_0=y$ in
Assumption~\ref{def:initial-laws} before translating all spatial data
by $x\mapsto x-y$ gives $\langle y-x_0\rangle^2=1$ and
$K_{\mathrm{dbl}}=1$, $K_{\mathrm{pot}}=0$, $K_{\mathrm{coer}}=1$,
independently of $y$.

\label{item:doubling-manifolds}
Let $M\subset\R^d$ be a nonempty compact embedded
$k$-dimensional $C^1$ submanifold, possibly with boundary.
A probability law $\mu_0$ supported on $M$, with a positive continuous
density $p$ with respect to $k$-dimensional Hausdorff measure
$\mathcal H^k$, is also doubling.
Local bi-Lipschitz charts, using half-balls at boundary points, and compactness
give $c,C,r_0>0$, depending on $M,p$, such that
\begin{equation*}
cr^k\le\mu_0(B(y,r))\le Cr^k,
\qquad y\in M,\quad0<r\le r_0;
\end{equation*}
see \citep[Theorem~7.5]{mattila-1995} for the Hausdorff measure comparison
under Lipschitz maps.
The doubling ratio is at most $2^kC/c$ for $r\le r_0/2$
and $[c(r_0/2)^k]^{-1}$ otherwise: for larger radii,
the denominator is bounded below by the mass of $B(y,r_0/2)$,
while the numerator is bounded above by one.
Examples include normalized arc length on circles and
normalized surface measure on spheres.
\end{eexample}

\begin{eexample}[Weighted positive densities]
\label{ex:weighted-initial-densities}
This example covers nondegenerate Gaussian laws and their finite
positive mixtures, with arbitrary means and positive-definite
covariance matrices. We verify these examples through the following
sufficient condition on the potential.
Let $w$ be Borel measurable with $0<c_w\le w\le C_w<\infty$,
and let $V\in C^2(\R^d)$ satisfy, for some $c>0$ and $C<\infty$,
\begin{equation}
\int_{\R^d}\mathrm{e}^{-V(x)}w(x)\,\mathrm{d}x=1,\quad
|DV(x)|\le C\langle x\rangle,\quad
\tfrac14|DV(x)|^2-\tfrac12\Delta V(x)
 \ge c\langle x\rangle^2-C,
\label{eq:weighted-density-potential}
\end{equation}
for $x \in \R^d$.
Then $\nu=w\,\mathrm{d}x$ and $\mu_0=\mathrm{e}^{-V}w\,\mathrm{d}x$ satisfy
Assumption~\ref{def:initial-laws}.
Indeed, $\nu$ is doubling with constant $2^dC_w/c_w$.
For $f\in C_c^\infty(\R^d)$, integration by parts gives
\begin{equation*}
\int_{\R^d}|Df|^2\mathrm{e}^{-V}\,\mathrm{d}x
 =\int_{\R^d}|D(f\mathrm{e}^{-V/2})|^2\,\mathrm{d}x
  +\int_{\R^d}\left(\tfrac14|DV|^2-\tfrac12\Delta V\right)
                  |f|^2\mathrm{e}^{-V}\,\mathrm{d}x.
\end{equation*}
This proves \eqref{eq:initial-coercivity} for the law proportional
to $\mathrm{e}^{-V}\mathrm{d}x$; the identity is the ground-state calculation underlying
the Lyapunov estimate in \citep[Section~2]{cattiaux-guillin-wang-wu-2009}.
More generally, if $\mathrm{d}\widetilde\mu=w\,\mathrm{d}\mu$ with
$0<c_w\le w\le C_w$, an admissible coercivity constant for $\mu$
is multiplied by at most $C_w/c_w$ for $\widetilde\mu$.
This transfers the estimate to $\mu_0$ and allows measurable $w$.
The differential bounds in \eqref{eq:weighted-density-potential} hold if
$|D^2V(x)|=o(|x|^2)$ and
$c|x|\le|DV(x)|\le C|x|$ for sufficiently large $|x|$.

The potential conditions are also preserved by finite positive
mixtures. For normalized densities $\mathrm{e}^{-V_i}$ satisfying the two
differential bounds in \eqref{eq:weighted-density-potential}, set
\begin{equation*}
\mathrm{e}^{-V}=\sum_{i=1}^N\beta_i \mathrm{e}^{-V_i},\qquad
\pi_i(x)=\frac{\beta_i \mathrm{e}^{-V_i(x)}}{\sum_j\beta_j \mathrm{e}^{-V_j(x)}},
\qquad \beta_i>0,\quad \sum_i\beta_i=1.
\end{equation*}
Then $DV=\sum_i\pi_iDV_i$, so the gradient growth bound is preserved,
and direct differentiation gives
\begin{equation*}
\begin{aligned}
\tfrac14|DV|^2-\tfrac12\Delta V
 &=\sum_i\pi_i\left(\tfrac14|DV_i|^2-\tfrac12\Delta V_i\right)
   +\tfrac14\sum_i\pi_i|DV_i-DV|^2.
\end{aligned}
\end{equation*}
The last term is nonnegative, which preserves the quadratic lower bound.
For the Gaussian examples, each component has a normalized quadratic potential
$V_i(x)=\tfrac12(x-m_i)^{\mathsf T}\Sigma_i^{-1}(x-m_i)+c_i$.
Multiplication by any positive bounded $w$ bounded away from zero,
and then normalization, gives further examples.
\end{eexample}

For the remainder of this section, fix $I=(0,T)$, coefficients
satisfying Assumption~\ref{ass:coefficients}, and an initial law
with a representation as in Assumption~\ref{def:initial-laws}.

\subsection{Rescaling and the localization scale}
\label{subsec:local-rescaling}

To identify $G_\alpha(I)$, we apply the same parabolic rescaling
to the coefficients and density.
Choose centers $(t_j,x_j)\in[0,T]\times\mathbb R^d$ and
positive scales $r_j\to0$.
Introduce
\begin{equation*}
s=\frac{t-t_j}{r_j^2},\qquad z=\frac{x-x_j}{r_j},\qquad
M_j=\frac{t_j}{r_j^2},\qquad \tau=\frac{t}{r_j^2}=s+M_j.
\end{equation*}
Thus $s$ measures time from the center and $\tau$ from the initial time.
Writing
\begin{equation*}
I^{\mathrm{sc}}(t;r)=\left(-\frac{t}{r^2},\frac{T-t}{r^2}\right],
\qquad t\in[0,T],\quad r>0,
\end{equation*}
the available rescaled interval is
$I_j^{\mathrm{sc}}=I^{\mathrm{sc}}(t_j;r_j)$:
$s\in I_j^{\mathrm{sc}}$ exactly when $0<t_j+r_j^2s\le T$.
After extraction, $M_j$ converges in $[0,\infty]$, so the initial
boundary $s=-M_j$ either stays finite or recedes to $-\infty$.
The right endpoint may remain finite; a terminal center $t_j=T$
allows only $s\le0$.

First consider the exact change of variables on corresponding intervals.
Take $J_j\in\mathcal J(I)$ and write
$\widetilde J_j=\{(t-t_j)/r_j^2:t\in J_j\}$.
For a smooth test $v$ on $\widetilde J_j\times\R^d$, set
$u_j(t_j+r_j^2s,x_j+r_jz)=r_j^2v(s,z)$ and
\begin{equation*}
A_j^{\mathrm{sc}}(s,z)=A(t_j+r_j^2s,x_j+r_jz),\qquad
q_j^{\mathrm{sc}}(s,z)=q(t_j+r_j^2s,x_j+r_jz).
\end{equation*}
At corresponding points,
\begin{equation*}
P_Au_j=-v_s-A_j^{\mathrm{sc}}:D_z^2v,
\qquad
\mathcal H_qu_j=(q_j^{\mathrm{sc}})^{1/2}D_z^2v\,(q_j^{\mathrm{sc}})^{1/2}.
\end{equation*}
Both squared norms acquire the Jacobian $r_j^{d+2}$ and the full weight
\begin{equation*}
\varpi_j(s,z)=(t_j+r_j^2s)^\alpha\rho(t_j+r_j^2s,x_j+r_jz).
\end{equation*}
The Jacobian cancels in the quotient. This change of variables
preserves the terminal gap, compact spatial support and free initial trace of the smooth test classes, and hence
\begin{equation}
G_\alpha^{\mathrm{fix}}(J_j)
 =G_0^{\mathrm{fix}}(\widetilde J_j;\varpi_j,A_j^{\mathrm{sc}},q_j^{\mathrm{sc}}).
\label{eq:gain-rescaling}
\end{equation}
The zero exponent records that $\varpi_j$ includes the time weight.
Positive scalar normalization of $\varpi_j$ also leaves the gain unchanged.

The time factor has two forms relevant to the models below.
If $M_j$ has a finite limit, we use the initial-time coordinate
$\tau=s+M_j$: the centered interval $\widetilde J_j$ becomes
$\widetilde J_j+M_j$ in this coordinate.
Since $t^\alpha=r_j^{2\alpha}\tau^\alpha$, removing the scalar factor
leaves the weight $\tau^\alpha$ used by the initial models below.
At initial centers $t_j=0$, we have $s=\tau$.
If $M_j\to\infty$, then for each fixed $L>0$,
\begin{equation}
\sup_{|s|\le L}
\left|\frac{(t_j+r_j^2s)^\alpha}{t_j^\alpha}-1\right|
 =\sup_{|s|\le L}|(1+s/M_j)^\alpha-1|
\longrightarrow0.
\label{eq:tail-time-factor}
\end{equation}
Thus normalization by $t_j^\alpha$ removes the time-power weight in the limit.

For $L\ge0$, set
$\mathcal U_L=\{(s,z)\in\R\times\R^d:|s|,|z|\le L\}$.
Assumption~\ref{ass:coefficients} gives, for $(s,z)\in\mathcal U_L$
with $s\in\overline{I_j^{\mathrm{sc}}}$,
\begin{equation}
\begin{gathered}
|a(t_j+r_j^2s,x_j+r_jz)-a(t_j,x_j)|
 \le\omega_a(r_j^2L)+L_ar_jL,\\
r_j|b(t_j+r_j^2s,x_j+r_jz)|\le r_jB_b,\\
|A_j^{\mathrm{sc}}(s,z)-A(t_j,x_j)|
 +|q_j^{\mathrm{sc}}(s,z)-q(t_j,x_j)|
 \le\omega_{A,q}\bigl(r_j(\sqrt L+L)\bigr).
\end{gathered}
\label{eq:local-coefficient-freezing}
\end{equation}
These bounds are uniform in the centers, with dependence as in
Assumption~\ref{ass:coefficients}; those for $a,b$ use only $\mathfrak R$.
Ellipticity permits extraction of positive definite limits
$(a,A,q)(t_j,x_j)\to(a_*,A_*,q_*)$.
The rescaled matrices converge to these limits uniformly on each
$\mathcal U_L\cap(\overline{I_j^{\mathrm{sc}}}\times\R^d)$,
giving the operators $P_{A_*}$ and $\mathcal H_{q_*}$.
The rescaled reference drift converges uniformly to zero on
$\overline{I_j^{\mathrm{sc}}}\times\R^d$.
The normalized density limit retains the initial law's effect
under this rescaling.

The Brownian profiles in Subsection~\ref{subsec:intro-results}
exhibit the heat length $\sqrt t$ and the length $t/|x|$ associated
with distance from the starting point.
The radial cost in \eqref{eq:initial-coercivity} also enters localization,
including at centers that escape along the support. We use the
\emph{localization scale} for $0 < t \le T$:
\begin{equation}
\ell(t,x)^{-1}
 =t^{-1/2}+\inf_{y\in S}
          \left\{\frac{|x-y|}{t}+\langle y\rangle\right\},
\qquad
M(t,x)=\frac{t}{\ell(t,x)^2}.
\label{eq:weighted-tail-scale}
\end{equation}
Its geometric meaning is expressed by
\begin{equation}
\ell(t,x)^{-1}\asymp
 t^{-1/2}+\frac{\dist(x,S)}{t}+\langle x\rangle.
\label{eq:weighted-scale-comparison}
\end{equation}
Indeed, since $\langle\cdot\rangle$
is $1$-Lipschitz, the infimum in \eqref{eq:weighted-tail-scale}
is at least $\dist(x,S)/t$ and at least
$\langle x\rangle/\max\{1,t\}$.
At a nearest point $y\in S$, the bound
$\langle y\rangle\le\langle x\rangle+\dist(x,S)$ gives the reverse
inequality, with constants depending only on $T$.
The three terms account for heat diffusion, distance from the
initial support and radial localization, respectively.
For a point start at the origin the inverse scale is
$t^{-1/2}+|x|/t+1$; for $S=\R^d$ and $t\le1$ it is
$t^{-1/2}+\langle x\rangle$.
The depth $M(t,x)$ equals $t/r^2$ at $r=\ell(t,x)$.
For general $r>0$, $t/r^2=M(t,x)(\ell(t,x)/r)^2$, so smaller
scales can send the rescaled depth to infinity even for bounded $M(t,x)$.
The following estimate controls the gradient cost of spatial cutoffs
at the scale $\ell$.
\begin{lemma}
\label{lem:dynamic-hardy}
For $0<t\le T$, the inverse scale is locally Lipschitz in $x$, with
\begin{equation}
|D_x\ell(t,x)^{-1}|\le t^{-1}\quad\text{for a.e. }x\in\R^d.
\label{eq:weighted-scale-gradient}
\end{equation}
There is a constant $C$, depending only on $\mathfrak R$ and $\mathfrak I_0$,
such that, for every $0<t\le T$,
\begin{equation}
\|\ell(t,\cdot)^{-1}f\|_{L^2(\mu_t)}^2
 \le C\bigl(\|Df\|_{L^2(\mu_t)}^2
                       +t^{-1}\|f\|_{L^2(\mu_t)}^2\bigr), \quad \forall \,f,Df\in L^2(\mu_t).
\label{eq:weighted-hardy}
\end{equation}
\end{lemma}

The inverse-scale gradient bound is proved in
Subsection~\ref{subsec:initial-density-estimates}; the Hardy estimate is
deduced in
Subsection~\ref{subsec:density-hardy} from the phase approximation in
Proposition~\ref{prop:tail-compactness}.
At bounded depth, \eqref{eq:weighted-scale-comparison} gives
$\ell(t,x)\asymp\sqrt t$ and $\dist(x,S)=O(\sqrt t)$.
Propositions~\ref{prop:initial-heat-limits} and~\ref{lem:uniform-flat-approximation},
proved in Subsection~\ref{subsec:initial-density-estimates}, give initial
heat models at this scale and flat limits at finer scales
$r/\sqrt t\to0$, respectively.
At $M(t,x)\to\infty$, Proposition~\ref{prop:tail-compactness}, proved in
Subsection~\ref{subsec:density-hardy}, gives phase models at scale $\ell(t,x)$.

\subsection{Initial heat models}
\label{subsec:initial-models}

In both initial and whole-line model tuples, $a_*,A_*,q_*$ are
constant symmetric positive definite matrices.
Let $\theta$ be a nonzero nonnegative Radon measure such that
$\int_{\R^d}\mathrm{e}^{-c|z|^2}\,\theta(\mathrm{d}z)<\infty$ for every $c>0$.
The associated \emph{initial heat tuple} and its density are
\begin{equation}
\mathbf T=(\theta,a_*,A_*,q_*),\qquad
W_{\mathbf T}(\tau,z)=
 \int_{\R^d}\gamma_{\tau a_*}(z-y)\,\theta(\mathrm{d}y),\qquad \tau>0,
\label{eq:initial-heat-tuple}
\end{equation}
where, for $C\succ0$,
$\gamma_C(z)=(2\pi)^{-d/2}(\det C)^{-1/2}
\mathrm{e}^{-z^{\mathsf T}C^{-1}z/2}$
is the centered Gaussian density with covariance $C$.
Gaussian integrability makes $W_{\mathbf T}$ finite, positive and
smooth for $\tau>0$. It satisfies
$\partial_\tau W_{\mathbf T}=\tfrac12a_*:D_z^2W_{\mathbf T}$ and
$\int_{\R^d}W_{\mathbf T}(\tau,z)\,\mathrm{d}z=\theta(\R^d)$,
with infinite mass allowed.

The initial scaling in
Subsection~\ref{subsec:local-rescaling} leaves the time weight
$\tau^\alpha$. Every fixed model interval
$J\in\mathcal J((0,\infty))$ corresponds to a physical interval
$r_j^2J$ of length $r_j^2|J|\to0$.
We therefore retain the free initial trace and define the initial
gain over all such intervals by
\begin{equation}
\mathfrak G_\alpha(\mathbf T)
 =\sup_{J\in\mathcal J((0,\infty))}
 G_\alpha^{\mathrm{fix}}(J;W_{\mathbf T},A_*,q_*).
\label{eq:initial-model-gain}
\end{equation}

For $y\in S$ and $r>0$, rescale $\mu_0$ about $y$ and normalize
by its mass on $B(y,r)$:
\begin{equation}
m_{y,r}=\mu_0(B(y,r)),\qquad
\theta_{y,r}(E)=\frac{\mu_0(y+rE)}{m_{y,r}}
\quad\text{for every Borel }E\subset\R^d.
\label{eq:doubling-half-line-data}
\end{equation}
The masses of $\theta_{y,r}$ may diverge as $r\downarrow0$, so we use
vague convergence, tested against continuous compactly supported functions.
For $\mu_0=\tfrac12\mathbf1_{[-1,1]}\mathcal L^1$ and $0<r<1$,
the interior center $y=0$ gives
\(
m_{0,r}=r,\,
\theta_{0,r}=\tfrac12\mathbf1_{[-1/r,1/r]}\mathcal L^1
 \longrightarrow\tfrac12\mathcal L^1
.
\)
At the boundary center $y=1$, it gives
\(
m_{1,r}=r/2,\,
\theta_{1,r}=\mathbf1_{[-2/r,0]}\mathcal L^1
 \longrightarrow\mathbf1_{(-\infty,0]}\mathcal L^1
.
\)
Both limits are vague as $r\downarrow0$.
For $a_*=1$, their heat
densities are respectively the constant $1/2$ and the
Gaussian-smoothed step
$\int_{-\infty}^0\gamma_\tau(z-y)\,\mathrm{d}y$.

At a support point $y$, the localization scale satisfies
\begin{equation*}
\frac{r}{\ell(r^2,y)}=1+r\langle y\rangle,
\qquad 0<r\le1,\quad r^2\le T.
\end{equation*}
Thus $r\langle y\rangle$ bounded means that the initial scale $r$
is comparable to the localization scale at time $r^2$.
The same condition controls $V(y+rz)-V(y)$ on bounded $z$-sets
by \eqref{eq:initial-potential-increment}.
Write $a_0=a(0,\cdot)$, $A_0=A(0,\cdot)$ and $q_0=q(0,\cdot)$.
The actual initial family
$\mathfrak T(\mu_0)=\mathfrak T(\mu_0;a_0,A_0,q_0)$ consists of
all tuples $\mathbf T=(\theta,a_*,A_*,q_*)$ obtained along sequences
\begin{equation}
y_j\in S,\qquad r_j>0,\quad r_j\to0,\qquad
\sup_j r_j\langle y_j\rangle<\infty,
\label{eq:initial-admissible-scales}
\end{equation}
with the joint limits
\begin{equation}
\theta_{y_j,r_j}\longrightarrow\theta\quad\text{vaguely},\qquad
(a_0,A_0,q_0)(y_j)\longrightarrow(a_*,A_*,q_*).
\label{eq:initial-joint-limits}
\end{equation}
We call $(y_j,r_j)$ a \emph{defining sequence} for this tuple.
The centers may escape to infinity, and the finite bound in
\eqref{eq:initial-admissible-scales} may depend on the sequence.
Lemma~\ref{lem:initial-tangents} below gives such joint subsequential
limits and shows that they are initial heat tuples.
The corresponding initial contribution, which only depends on
$\mu_0,\alpha$ and $a_0,A_0,q_0$, is
\begin{equation}
G_\alpha^{\mathrm{init}}(I)
 =\sup_{\mathbf T\in\mathfrak T(\mu_0)}\mathfrak G_\alpha(\mathbf T).
\label{eq:initial-linear-formula}
\end{equation}

\begin{lemma}
\label{lem:initial-tangents}
Fix $B_0>0$. For all $y\in S$ and
$0<r\le1$ with $r\langle y\rangle\le B_0$, we have
\begin{equation}
\theta_{y,r}(B_R)
 \le C R^{\log_2K_{\mathrm{dbl}}}\mathrm{e}^{CB_0R+Cr^2R^2},\qquad R\ge1,
\label{eq:rescaled-local-mass}
\end{equation}
where $C$ depends only on $K_{\mathrm{dbl}},K_{\mathrm{pot}},B_0$.
For every $c>0$, there are $r_c\in(0,1]$ and $C_c<\infty$,
depending only on $c,K_{\mathrm{dbl}},K_{\mathrm{pot}},B_0$, such that
 when $r\le r_c$,
\begin{equation}
\int_{\R^d}\mathrm{e}^{-c|v|^2}\,\theta_{y,r}(\mathrm{d}v)\le C_c,\quad
\int_{|v|\ge R}\mathrm{e}^{-c|v|^2}\,\theta_{y,r}(\mathrm{d}v)
 \le C_c \mathrm{e}^{-cR^2/2},\,\, R\ge1.
\label{eq:rescaled-measure-mass}
\end{equation}
Every sequence satisfying \eqref{eq:initial-admissible-scales}
has a subsequence with the joint limits
\eqref{eq:initial-joint-limits}. Its measure limit $\theta$ is
nonnegative and Radon, satisfies $\theta(\overline B_1)\ge1$, and obeys
\begin{equation*}
\int_{\R^d}\mathrm{e}^{-c|v|^2}\,\theta(\mathrm{d}v)<\infty
\qquad\text{for every }c>0.
\end{equation*}
In particular, the limiting tuple is an initial heat tuple in
$\mathfrak T(\mu_0)$.
\end{lemma}

\begin{proof}
We first record two consequences of Assumption~\ref{def:initial-laws}
that will also be used below.
For $y\in S$, $z\in\R^d$ and $0<r\le R$, iterating the doubling
bound and using $B(z,R)\subset B(y,R+|z-y|)$ gives
\begin{equation}
\nu(B(z,R))
 \le K_{\mathrm{dbl}}
 \left(\frac{R+|z-y|}{r}\right)^{\log_2K_{\mathrm{dbl}}}
 \nu(B(y,r)).
\label{eq:initial-ball-comparison}
\end{equation}
For $y\in\R^d$ and $R>0$, if
$K_{\mathrm{pot}}(R\langle y\rangle+R^2/2)\le B$, then
\eqref{eq:initial-potential-increment} gives
\begin{equation}
\mathrm{e}^{-V(y)-B}\nu(E)\le\mu_0(E)\le \mathrm{e}^{-V(y)+B}\nu(E),
\qquad E\subset B(y,R)\text{ Borel}.
\label{eq:initial-local-mass-comparison}
\end{equation}

We now fix $y,r$ as in the statement.
The normalization \eqref{eq:doubling-half-line-data} gives $\theta_{y,r}(B_1)=1$.
For $R\ge1$, \eqref{eq:initial-ball-comparison} and
\eqref{eq:initial-local-mass-comparison} give
\begin{equation*}
\theta_{y,r}(B_R)
 \le K_{\mathrm{dbl}}R^{\log_2K_{\mathrm{dbl}}}
       \exp\!\left\{K_{\mathrm{pot}}\left[r\langle y\rangle(R+1)+\frac{r^2(R^2+1)}2\right]\right\},
\end{equation*}
which proves \eqref{eq:rescaled-local-mass}.

For a fixed $c>0$, split
$\int \mathrm{e}^{-c|v|^2/2}\,\mathrm{d}\theta_{y,r}$ into $B_1$ and
$B_{2^{k+1}}\setminus B_{2^k}$, $k\ge0$.
Put $d_{\mathrm{dbl}}=\log_2K_{\mathrm{dbl}}$ and choose
$r_c\in(0,1]$ so that $4Cr_c^2\le c/4$, with $C$ from
\eqref{eq:rescaled-local-mass}. For $r\le r_c$,
\begin{equation*}
\begin{aligned}
\int \mathrm{e}^{-c|v|^2/2}\,\mathrm{d}\theta_{y,r}
 \le 1+C\sum_{k\ge0}2^{(k+1)d_{\mathrm{dbl}}}
       \mathrm{e}^{CB_0 2^{k+1}-c4^k/4}
 \le C_c.
\end{aligned}
\end{equation*}
This gives both bounds in
\eqref{eq:rescaled-measure-mass}, using
$\mathrm{e}^{-c|v|^2}\le \mathrm{e}^{-cR^2/2}\mathrm{e}^{-c|v|^2/2}$ on $\{|v|\ge R\}$
for the tail bound.

Along a sequence in \eqref{eq:initial-admissible-scales},
\eqref{eq:rescaled-local-mass}
gives a vaguely convergent subsequence with a nonnegative
Radon limit $\theta$ \citep[Theorem~1.23]{mattila-1995}.
Since $\theta_{y_j,r_j}(B_1)=1$, the compact-set inequality for vague
convergence gives
$\theta(\overline B_1)\ge\limsup_j\theta_{y_j,r_j}(\overline B_1)\ge1$.
For each $c>0$, $r_j\le r_c$ for all sufficiently large $j$; compactly supported
approximations from below to $\mathrm{e}^{-c|v|^2}$
then give the integral bound in \eqref{eq:rescaled-measure-mass} for $\theta$.
Ellipticity permits a further
subsequence along which $(a_0,A_0,q_0)(y_j)$ converges jointly.
These are exactly the limits \eqref{eq:initial-joint-limits} defining $\mathfrak T(\mu_0)$.
\end{proof}

\begin{proposition}
\label{prop:initial-heat-limits}
For $y\in S$ and $r>0$, define
\begin{equation}
W_{y,r}(\tau,z)
 =\int_{\R^d}\gamma_{\tau a_0(y)}(z-\xi)\,\theta_{y,r}(\mathrm{d}\xi),
\qquad \tau>0,\quad z\in\R^d.
\label{eq:initial-heat-convolution}
\end{equation}
For every $\varepsilon>0$, $0<\tau_0<L$ and $B_0,B>0$, there is
$r_0>0$, depending only on $\mathfrak R$, $K_{\mathrm{dbl}},K_{\mathrm{pot}}$ and
$\varepsilon,\tau_0,L,B_0,B$, with $r_0^2L\le T$, such that
\begin{equation}
\sup_{\substack{y\in S\\r\langle y\rangle\le B_0}}
\ \sup_{\tau_0\le\tau\le L,\ |z|\le B}
\left|\log\frac{r^d\rho(r^2\tau,y+rz)}
                    {m_{y,r}W_{y,r}(\tau,z)}\right|
 \le\varepsilon,\qquad \forall\, 0<r\le r_0.
\label{eq:initial-heat-approximation}
\end{equation}
If $(y_j,r_j)$ satisfies
\eqref{eq:initial-admissible-scales}--\eqref{eq:initial-joint-limits}
with limiting tuple $\mathbf T=(\theta,a_*,A_*,q_*)$, then
\begin{equation}
\frac{r_j^d\rho(r_j^2\tau,y_j+r_jz)}{m_{y_j,r_j}}
 \longrightarrow W_{\mathbf T}(\tau,z), \quad \text{ locally uniformly on}\,\, (0,\infty)\times\R^d.
\label{eq:initial-realization}
\end{equation}
\end{proposition}

\subsection{Whole-line models and the three-term formula}
\label{subsec:tail-models}

For a compactly supported probability measure $\lambda$ on $\R^d$,
define the \emph{whole-line tuple} and its \emph{exponential-mixture weight} by
\begin{equation}
\begin{gathered}
\mathbf M=(\lambda,a_*,A_*,q_*),\qquad W_{\mathbf M}=W_{\lambda,a_*},\\
W_{\lambda,a}(s,z)=\int_{\R^d}
 \exp\bigl(2(s\zeta^{\mathsf T}a\zeta-\zeta\cdot z)\bigr)\,\lambda(\mathrm{d}\zeta),
\qquad a\succ0,\quad(s,z)\in\R\times\R^d.
\end{gathered}
\label{eq:whole-line-tuple}
\end{equation}
Here $\lambda$ is the mixing measure for the exponential weights.
$W_{\mathbf M}(s,\cdot)$ need not be a probability density.
The weight satisfies
$\partial_sW_{\mathbf M}=\tfrac12a_*:D_z^2W_{\mathbf M}$.
A tuple $\mathbf M$ is \emph{flat} if $\lambda=\delta_0$, so $W_{\mathbf M}=1$,
and is a \emph{phase tuple} if $0\notin\supp\lambda$.
It is a \emph{pure phase tuple} if $\lambda=\delta_\zeta$ with $\zeta\ne0$.

When $t_j/r_j^2\to\infty$, \eqref{eq:tail-time-factor} removes
the time power from the limiting weight.
For flat and phase tuples, we therefore define the whole-line gain by
\begin{equation*}
\mathfrak G(\mathbf M)=
 \sup_{v\in C_c^\infty(\R\times\R^d)}
 \frac{\|\mathcal H_{q_*}v\|_{L^2(W_{\mathbf M}\,\mathrm{d}s\,\mathrm{d}z)}}
      {\|P_{A_*}v\|_{L^2(W_{\mathbf M}\,\mathrm{d}s\,\mathrm{d}z)}}.
\end{equation*}
Corollary~\ref{cor:whole-line-intervals} below relates these compactly
supported time tests to the free initial trace in
\eqref{eq:short-time-constants}.

A sequence $(t_j,x_j,r_j)$ is a \emph{realizing sequence}
for a flat or phase tuple $\mathbf M$ if
\begin{equation}
\begin{gathered}
t_j\in(0,T],\, x_j\in\R^d,\, r_j>0,\,
r_j\to0,\, M_j=t_j/r_j^2\to\infty,\\
(a,A,q)(t_j,x_j)\longrightarrow(a_*,A_*,q_*),\\
\sup_{(s,z)\in\mathcal U_L,\,s\in I_j^{\mathrm{sc}}}
\left|\log\frac{\rho(t_j+r_j^2s,x_j+r_jz)}
                    {\rho(t_j,x_j)W_{\mathbf M}(s,z)}\right|
\longrightarrow0\quad\text{for every }L>0.
\end{gathered}
\label{eq:whole-line-convergence}
\end{equation}
Under the same center and scale conditions,
$(t_j,x_j,r_j)$ realizes the weight $W_{\lambda,a_*}$ if the
density convergence in \eqref{eq:whole-line-convergence} holds.
Realizing the full tuple also requires the joint coefficient limit.
Here ``whole-line'' refers to the model's time domain;
realization is restricted to $I_j^{\mathrm{sc}}$ from
Subsection~\ref{subsec:local-rescaling}.

For every fixed center $(t,x)\in I\times\R^d$ and
positive scales $r_j\to0$,
the sequence $(t,x,r_j)$ realizes the flat tuple
$(\delta_0,a(t,x),A(t,x),q(t,x))$ by positivity and continuity of $\rho$.
At moving centers with bounded localization depth $M(t_j,x_j)$,
observing at scales $r_j\ll\sqrt{t_j}$ sends $t_j/r_j^2$ to infinity
and gives the whole-line flat limit below.

\begin{proposition}
\label{lem:uniform-flat-approximation}
For every $B_0,B,\varepsilon>0$, there is $\eta_0>0$, depending
only on $\mathfrak R,\mathfrak I_0,B_0,B,\varepsilon$, such that
$M(t,x)\le B_0$ and $0<r\le\eta_0\sqrt t$ imply
\begin{equation*}
\sup_{\substack{(s,z)\in\mathcal U_B\\s\in I^{\mathrm{sc}}(t;r)}}
\left|\log\frac{\rho(t+r^2s,x+rz)}{\rho(t,x)}\right|
\le\varepsilon,\qquad 0<t\le T.
\end{equation*}
\end{proposition}

The actual phase family $\mathfrak M_+(I)$ consists of the
phase tuples for which \eqref{eq:whole-line-convergence}
holds along sequences with
\begin{equation}
t_j\in(0,T],\qquad x_j\in\R^d,\qquad
r_j=\ell(t_j,x_j),\qquad M_j=\frac{t_j}{r_j^2}\longrightarrow\infty.
\label{eq:doubling-tail-centers}
\end{equation}
We call such a realizing sequence a \emph{defining sequence}
for $\mathbf M\in\mathfrak M_+(I)$.
The scales in \eqref{eq:doubling-tail-centers} tend to zero
because $r_j^2\le T/M_j$.
The family depends on $(I;\mu_0,a,b,A,q)$; each tuple pairs the phase law
and coefficient limits along one sequence.

Exponential-mixture representations of nonnegative ancient heat
solutions are classical
\citep{widder-1963-appell,lin-zhang-2019}.
The next proposition constructs the approximating measures for
the actual marginals, with uniform relative error and
supports contained in a common annulus.
Its proof uses restarted posterior measures,
whose consistency at different restart ages is also used in
Section~\ref{sec:phase-directions}.

\begin{proposition}
\label{prop:tail-compactness}
There are constants $0<m_{\mathrm{ph}}\le M_{\mathrm{ph}}<\infty$,
depending only on $\mathfrak R$ and $\mathfrak I_0$, and probability
measures $\lambda_{t,x}$, for $(t,x)\in(0,T]\times\R^d$, supported on
\(
\mathcal Z_{\mathrm{ph}}:=\{\zeta\in\R^d:m_{\mathrm{ph}}\le|\zeta|\le M_{\mathrm{ph}}\}.
\)
For every $L>0$ and $\delta>0$, there is a threshold $M_0>0$,
depending only on $\mathfrak R$, $\mathfrak I_0,L,\delta$,
for which every center $(t,x)\in(0,T]\times\R^d$,
with $\ell=\ell(t,x)$ and $M(t,x)\ge M_0$, satisfies
\begin{equation*}
\sup_{(s,z)\in\mathcal U_L,\,s\in I^{\mathrm{sc}}(t;\ell)}
\left|\log\frac{\rho(t+\ell^2s,x+\ell z)}
                    {\rho(t,x)W_{\lambda_{t,x},a(t,x)}(s,z)}\right|
\le\delta.
\end{equation*}
Every sequence satisfying \eqref{eq:doubling-tail-centers} has a
subsequence realizing a tuple in $\mathfrak M_+(I)$ through
\eqref{eq:whole-line-convergence}. Every tuple in $\mathfrak M_+(I)$
has its phase law supported on $\mathcal Z_{\mathrm{ph}}$.
\end{proposition}

For a sequence satisfying \eqref{eq:doubling-tail-centers}, suppose
$(a,A,q)(t_j,x_j)\to(a_*,A_*,q_*)$ in matrix norm.
Proposition~\ref{prop:tail-compactness} and uniqueness of the Laplace
transform for compactly supported measures give the criterion
\begin{equation}
(t_j,x_j,r_j)\text{ realizes }(\lambda,a_*,A_*,q_*)
\qquad\Longleftrightarrow\quad
  \lambda_{t_j,x_j}\rightharpoonup\lambda,
\label{eq:phase-law-realization}
\end{equation}
where $\rightharpoonup$ denotes weak convergence against bounded
continuous functions.
In particular, a locally uniform limit $F$ of the spatial density
ratios determines the unique probability law $\lambda$ on
$\mathcal Z_{\mathrm{ph}}$ through
\begin{equation}
\frac{\rho(t_j,x_j+r_jz)}{\rho(t_j,x_j)}
 \longrightarrow F(z)
 =\int_{\R^d}\mathrm{e}^{-2\zeta\cdot z}\,\lambda(\mathrm{d}\zeta),
\label{eq:spatial-phase-identification}
\end{equation}
and hence the full realization \eqref{eq:whole-line-convergence}.

\begin{eexample}
\label{ex:two-point-model-families}
Let $d=1$, $a=1$, $b=0$,
$\mu_0=\tfrac12(\delta_{-1}+\delta_1)$, and $A,q>0$ be constant.
Write $D(x)=\dist(x,\{-1,1\})$.
The density and the localization scale \eqref{eq:weighted-tail-scale} are
\begin{equation*}
\rho(t,x)=(2\pi t)^{-1/2}\mathrm{e}^{-(x^2+1)/(2t)}\cosh(x/t),
\qquad
\ell(t,x)=\frac{t}{\sqrt t+D(x)+\sqrt2\,t}.
\end{equation*}
At either support point $y=\pm1$, the normalized measure in
\eqref{eq:doubling-half-line-data} is
$\theta_{y,r}=\delta_0+\delta_{-2y/r}$ for $r<1$, so its vague
limit is $\delta_0$.

For the phase family, take any sequence with
$0<t_j\le T$ and $M_j=M(t_j,x_j)\to\infty$, and write
$\ell_j=\ell(t_j,x_j)$.
The localization scale \eqref{eq:weighted-tail-scale} gives
\begin{equation*}
\frac{D(x_j)}{\sqrt{t_j}}\longrightarrow\infty,\qquad
\frac{\ell_jD(x_j)}{t_j}\longrightarrow1,\qquad
\frac{\ell_j^2}{t_j}=\frac1{M_j}\longrightarrow0.
\end{equation*}
After extraction, either $x_j/t_j\to c\in\R$ or
$|x_j|/t_j\to\infty$.

In the first case, $t_j\to0$ and $\ell_j/t_j\to1$.
Direct substitution in the density gives, locally uniformly in $z$,
\begin{equation*}
\frac{\rho(t_j,x_j+\ell_jz)}{\rho(t_j,x_j)}
 \longrightarrow\frac{\cosh(c+z)}{\cosh c}
 =p \mathrm{e}^{-z}+(1-p)\mathrm{e}^z,
\qquad p=\frac1{1+\mathrm{e}^{2c}}.
\end{equation*}
By \eqref{eq:spatial-phase-identification}, the phase law is
$\lambda_p=p\delta_{1/2}+(1-p)\delta_{-1/2}$.
Varying $c$ in the choice $x_t=ct$ as $t\downarrow0$ realizes
every $0<p<1$.

In the second, the nearer atom $y_j\in\{-1,1\}$ dominates the
Gaussian mixture at both $x_j$ and $x_j+\ell_jz$, uniformly for bounded $z$.
Consequently,
\begin{equation*}
\frac{\rho(t_j,x_j+\ell_jz)}{\rho(t_j,x_j)}
 =\exp\!\left\{-\frac{\ell_j(x_j-y_j)}{t_j}z
              -\frac{\ell_j^2}{2t_j}z^2\right\}(1+o(1)).
\end{equation*}
Since $|x_j-y_j|=D(x_j)$, a further subsequence gives the spatial
limit $\mathrm{e}^{-z}$ or $\mathrm{e}^z$.
The choices $x_t=\sqrt t$ and $x_t=-\sqrt t$ as $t\downarrow0$
realize $p=0$ and $p=1$, respectively.
Thus the complete actual families are
\begin{equation*}
\mathfrak T(\mu_0)=\{(\delta_0,1,A,q)\},
\qquad
\mathfrak M_+(I)=\{(\lambda_p,1,A,q):0\le p\le1\}.
\end{equation*}
At $c=0$, both initial atoms survive and the full weight is
$W_{\lambda_{1/2},1}(s,z)=\mathrm{e}^{s/2}\cosh z$.
This is why phase models must allow mixtures of distinct directions.
\end{eexample}

The whole-line gain formulas use the following directional ratio.
For $a,A,q\succ0$ and $v\ne0$, set
\begin{equation*}
\mathsf R(a,A,q;v)=\frac{v^{\mathsf T}(a-A)v}{v^{\mathsf T}qv},
\qquad
\Phi(c)=
\begin{cases}c^{-1},&c>0,\\+\infty,&c\le0.\end{cases}
\end{equation*}
The ratio $\mathsf R$ is unchanged by nonzero rescaling of $v$.

\begin{proposition}\label{prop:whole-line-gains}
For $\mathbf M=(\delta_0,a_*,A_*,q_*)$,
\begin{equation}
\mathfrak G(\mathbf M)=\lambda_{\max}(A_*^{-1/2}q_*A_*^{-1/2}).
\label{eq:flat-gain}
\end{equation}
For a phase tuple $\mathbf M=(\lambda,a_*,A_*,q_*)$,
\begin{equation}
\mathfrak G(\mathbf M)=
 \max\left\{\lambda_{\max}(A_*^{-1/2}q_*A_*^{-1/2}),
 \sup_{\zeta\in\supp\lambda}
   \Phi\bigl(\mathsf R(a_*,A_*,q_*;\zeta)\bigr)\right\}.
\label{eq:phase-gain-formula}
\end{equation}
\end{proposition}

In \eqref{eq:phase-gain-formula}, the first term is the gain of
the corresponding flat tuple. Only the directional term enters
the phase contribution defined below.

\begin{corollary}
\label{cor:whole-line-intervals}
Let $\mathbf M=(\lambda,a_*,A_*,q_*)$ be a flat or phase tuple.
\begin{enumerate}
\item\label{item:whole-line-intervals}
For every $J\in\mathcal J((0,\infty))$,
\begin{equation*}
G_\alpha^{\mathrm{fix}}(J;W_{\mathbf M},A_*,q_*)\le\mathfrak G(\mathbf M).
\end{equation*}
If $\alpha=0$, the same bound also holds
for every $J\in\mathcal J(\R)$.
Equality holds in this bound when
$\lambda=\delta_0$.

\item\label{item:whole-line-past}
The whole-line gain satisfies
\begin{equation*}
\mathfrak G(\mathbf M)=
\sup_{v\in C_c^\infty((-\infty,s_2)\times\R^d)}
 \frac{\|\mathcal H_{q_*}v\|_{L^2(W_{\mathbf M}\,\mathrm{d}s\,\mathrm{d}z)}}
      {\|P_{A_*}v\|_{L^2(W_{\mathbf M}\,\mathrm{d}s\,\mathrm{d}z)}},
\qquad\forall s_2\in\R.
\end{equation*}
\end{enumerate}
\end{corollary}

The interval bound applies with a free initial trace, and recovery on a past half-line
preserves the full gain even when the physical terminal time truncates
the rescaled interval.
The proofs of Proposition~\ref{prop:whole-line-gains} and
Corollary~\ref{cor:whole-line-intervals} are given in
Appendix~\ref{subsec:whole-line-gain-proofs}.

The flat contribution is
\begin{equation}
G^{\mathrm{flat}}(I)
 =\sup_{(t,x)\in I\times\R^d}
   \lambda_{\max}\!\left(A(t,x)^{-1/2}q(t,x)A(t,x)^{-1/2}\right)
 \le\frac{\Lambda_q}{\lambda_A}.
\label{eq:flat-contribution}
\end{equation}
By \eqref{eq:flat-gain}, the fixed-center realizations give all values
in this supremum, and the convergence of $(A,q)(t_j,x_j)$ in
\eqref{eq:whole-line-convergence} bounds the gain
of every flat realization by it; continuity also covers centers at $t_j=T$.
The flat term in \eqref{eq:phase-gain-formula}
for every $\mathbf M\in\mathfrak M_+(I)$ is likewise bounded by
$G^{\mathrm{flat}}(I)$. The phase contribution is
\begin{equation}
G^{\mathrm{phase}}(I)
 =\sup_{\substack{\mathbf M=(\lambda,a_*,A_*,q_*)\in
                              \mathfrak M_+(I)}}
   \ \sup_{\zeta\in\supp\lambda}
       \Phi\bigl(\mathsf R(a_*,A_*,q_*;\zeta)\bigr).
\label{eq:propagation-contribution}
\end{equation}

\begin{theorem}[Three-term formula]\label{thm:sharp-limit}
For each fixed problem on
$I=(0,T)$ satisfying Assumptions~\ref{ass:coefficients}
and~\ref{def:initial-laws},
the following identity holds in the extended real numbers:
\(
G_\alpha(I)=G_\alpha^{\mathrm{mod}}(I)
 :=\max\{G_\alpha^{\mathrm{init}}(I),
         G^{\mathrm{flat}}(I),G^{\mathrm{phase}}(I)\}.
\)
\end{theorem}

\section{Model comparison and terminal solvability}
\label{sec:proof}

We prove Theorem~\ref{thm:sharp-limit}, obtain common short-time bounds
for arbitrary initial laws, and develop the terminal and damping theory.
The partition tools and deferred proofs are in
Appendix~\ref{app:natural-realization}.

\subsection{Recovery of model gains}
\label{subsec:recovery}

Throughout Subsections~\ref{subsec:recovery}--\ref{subsec:localization},
we take $I=(0,T)$ and work under
Assumptions~\ref{ass:coefficients} and~\ref{def:initial-laws},
unless explicitly stated otherwise.

\begin{lemma}\label{lem:recovery}
In the extended real numbers, $G_\alpha(I)\ge G_\alpha^{\mathrm{mod}}(I)$.
\end{lemma}

\begin{proof}
Choose an initial tuple $\mathbf T\in\mathfrak T(\mu_0)$,
a phase tuple $\mathbf M\in\mathfrak M_+(I)$, or a flat tuple $\mathbf M$
at a fixed center in $I\times\R^d$.
For an initial tuple, choose a defining sequence satisfying
\eqref{eq:initial-admissible-scales}--\eqref{eq:initial-joint-limits}.
For a phase tuple, use \eqref{eq:whole-line-convergence} at the
scales \eqref{eq:doubling-tail-centers}.
For a flat tuple, take positive scales $r_j\to0$ at the fixed center.
For the initial tuple, set $x_j=y_j$.
Fix a finite $k_0>0$ below the gain of the chosen model.
For an initial heat tuple, \eqref{eq:initial-model-gain} gives
$\widetilde J=(s_1,s_2)\in\mathcal J((0,\infty))$ and
$v\in\mathcal C_{\alpha,0}(\widetilde J;W_{\mathbf T})$
whose norm quotient exceeds $k_0$.
By restriction and monotone convergence, we may take $s_1>0$
while preserving the strict inequality.
For a whole-line tuple, Corollary~\ref{cor:whole-line-intervals}\ref{item:whole-line-past} gives
$\widetilde J=(-L_v,0)$, $L_v>0$, and
$v\in C_c^\infty(\widetilde J\times\R^d)$ whose norm quotient exceeds
$k_0$. Write $\varpi=s^\alpha W_{\mathbf T}$ in the initial case
and $\varpi=W_{\mathbf M}$ in the whole-line case. On
$\widetilde J\times\R^d$ we therefore have
\begin{equation*}
\|\mathcal H_{q_*}v\|_{0,\widetilde J;\varpi}
 >k_0\|P_{A_*}v\|_{0,\widetilde J;\varpi}.
\end{equation*}

Along the chosen sequence, set
\begin{equation*}
(t_j^0,c_j)=
\begin{cases}
(0,m_{y_j,r_j}r_j^{2\alpha+2}),&\text{initial},\\
(t_j,t_j^\alpha\rho(t_j,x_j)r_j^{d+2}),&\text{whole-line},
\end{cases}
\qquad J_j=t_j^0+r_j^2\widetilde J,
\end{equation*}
and define $u_j(t_j^0+r_j^2s,x_j+r_jz)=r_j^2v(s,z)$.
For large $j$, $J_j\subset I$: in the initial case this follows
from $r_j\to0$; in the whole-line case,
$M_j\to\infty$ gives $\widetilde J\subset I_j^{\mathrm{sc}}$ for large $j$,
and $\sup J_j=t_j\le T$.
Let $\varpi_j$ be the full rescaled weight
used in \eqref{eq:gain-rescaling} at center $(t_j^0,x_j)$,
and fix a ball $B_R$ containing the spatial support of $v$.
The limit weight $\varpi$ has a positive lower bound on
$\overline{\widetilde J}\times\overline B_R$.
Thus \eqref{eq:initial-realization} in the initial case, and
\eqref{eq:tail-time-factor} together with
\eqref{eq:whole-line-convergence} in the whole-line case, give
\begin{equation*}
\sup_{(s,z)\in\widetilde J\times B_R}
\left|\frac{r_j^{d+2}\varpi_j(s,z)}{c_j\varpi(s,z)}-1\right|
\longrightarrow0.
\end{equation*}
Hence $v\in\mathcal C_{0,0}(\widetilde J;\varpi_j)$ for all sufficiently
large $j$, and the rescaling gives $u_j\in\mathcal C_{\alpha,0}(J_j)$.
Local convergence of the normalized weights and coefficients gives,
by \eqref{eq:gain-rescaling} and dominated convergence,
\begin{equation}
c_j^{-1}
\begin{pmatrix}
\|\mathcal H_qu_j\|_{\alpha,J_j}^2\\
\|P_Au_j\|_{\alpha,J_j}^2
\end{pmatrix}
\longrightarrow
\begin{pmatrix}
\|\mathcal H_{q_*}v\|_{0,\widetilde J;\varpi}^2\\
\|P_{A_*}v\|_{0,\widetilde J;\varpi}^2
\end{pmatrix}.
\label{eq:recovery-norm-limits}
\end{equation}
Here \eqref{eq:local-coefficient-freezing} gives coefficient convergence, so $\|\mathcal H_qu_j\|_{\alpha,J_j}
>k_0\|P_Au_j\|_{\alpha,J_j}$ for all large $j$.
Since $|J_j|=r_j^2|\widetilde J|\to0$, we have
$G_{\alpha,h}(I)\ge k_0$ for every $h>0$.
Letting $h\downarrow0$ and then taking the supremum over finite
$k_0$ recovers the gain of the chosen model, including infinite gains.
Taking the suprema over the initial, fixed-center flat and phase tuples,
using \eqref{eq:flat-gain} and \eqref{eq:phase-gain-formula}, proves the lemma.
\end{proof}

\subsection{Short-window upper bounds}
\label{subsec:comparison}

We now bound the gain from above by comparing localized tests
with the three density models. We first work on intervals whose
left endpoint stays away from zero relative to their right endpoint.
For $T_0\in I$, $0<\vartheta<1$ and $0<h\le T$, put
\begin{equation}
\mathcal J_h^+(T_0)=\mathcal J_h((T_0,T)), \quad
\mathcal J_h^0(\vartheta)
 =\{(T_1,T_2):0<T_2\le h,\ \vartheta T_2\le T_1<T_2\}.
\label{eq:comparison-time-windows}
\end{equation}
These are the positive-time and near-zero interval families,
respectively; $h$ bounds the length in the first and the right
endpoint in the second.

The cutoff errors will be controlled by
\begin{equation*}
\mathscr G_J(u)^2=\|D^2u\|_{\alpha,J}^2+\|P_Au\|_{\alpha,J}^2.
\end{equation*}
A \emph{square partition} on $J\times\R^d$
is a locally finite smooth real family $\{\chi_i\}$ with
$\sum_i\chi_i^2=1$. Each cutoff is assigned a cylinder
$Q_i=J_i\times B_i$, where $J_i\in\mathcal J(J)$, $B_i$ is a ball
and $\supp\chi_i\subseteq Q_i$, with support relative to
$J\times\R^d$. The identity $\sum_i\chi_i^2=1$ reproduces the
squared norms of the original test, up to terms containing
cutoff derivatives. We also require
$u_i=(\chi_i u)|_{J_i}\in\mathcal W_{\alpha,0}(J_i)$ for
$u\in\mathcal C_{\alpha,0}(J)$: the model interval bound in Corollary~\ref{cor:whole-line-intervals}\ref{item:whole-line-intervals} needs
zero terminal trace at the right endpoint of each $J_i$.
The following lemma supplies comparison cylinders and controls
the error in recovering the original squared norms. Each spatial
cutoff is assigned a support point $y_i\in S$, independently of
the test function.

\begin{lemma}
\label{lem:localization-partition}
Use $S=\supp\mu_0$ and the scale $\ell$ from
\eqref{eq:weighted-tail-scale}.
There is an absolute constant $c_0>0$ with the following property.
Fix $L\ge2$, $0<h\le T$ and either $T_0\in I$ or
$\vartheta\in(0,1)$, and take
\begin{equation}
J\in\mathcal J_h^+(T_0), \quad h\le T_0/(c_0L^2), \quad
\text{or}\quad
J\in\mathcal J_h^0(\vartheta), \quad h\le1.
\label{eq:initial-partition-windows}
\end{equation}
There is a square partition $\{\chi_i,Q_i\}_{i\in\mathcal I}$ on
$J\times\R^d$, chosen independently of the test function, with
indices $\mathcal I=\mathcal I_{\mathrm{sp}}\sqcup\mathcal I_{\mathrm{par}}$.
For the positive-time choice in \eqref{eq:initial-partition-windows},
$\mathcal I_{\mathrm{sp}}=\varnothing$.
\begin{enumerate}
\item\label{item:localization-spatial}
For $i\in\mathcal I_{\mathrm{sp}}$, $\chi_i$ depends only on $x$.
Set $\sigma=\sqrt{\sup J}\le\sqrt h$.
There is a constant $C_\vartheta>0$, depending only on $\vartheta$, such that
\begin{equation*}
y_i\in S,\qquad \sigma\langle y_i\rangle\le C_\vartheta L,\qquad
J_i=J,\qquad Q_i=J\times B(y_i,C_\vartheta L\sigma).
\end{equation*}
\item\label{item:localization-parabolic}
For $i\in\mathcal I_{\mathrm{par}}$, there are centers
$(t_i,x_i)\in J\times\R^d$ and scales $0<r_i\le\ell(t_i,x_i)$
with $r_i^2\le |J|\le h$ and $t_i/r_i^2\ge c_0L^2$, such that
\begin{equation*}
J_i=J\cap(t_i-9L^2r_i^2,t_i+9L^2r_i^2),\qquad
Q_i=J_i\times B(x_i,3Lr_i).
\end{equation*}
Moreover, $t_i/2\le t\le3t_i/2$ for $t\in J_i$.
\end{enumerate}
For every $u\in\mathcal C_{\alpha,0}(J)$, this partition satisfies,
with $u_i=(\chi_i u)|_{J_i}$,
\begin{equation}
\sum_{\mathcal T\in\{\mathcal H_q,P_A\}}
 \left|\sum_i\|\mathcal Tu_i\|_{\alpha,J_i}^2
                  -\|\mathcal Tu\|_{\alpha,J}^2\right|
 \le CL^{-2}\mathscr G_J(u)^2.
\label{eq:square-localization-error}
\end{equation}
Here $C$ depends only on $\mathfrak R,\mathfrak I_0,\alpha,
\Lambda_A,\Lambda_q$ and, in the near-zero case, $\vartheta$.
\end{lemma}

The construction and error estimates are proved in
Appendix~\ref{subsec:localization-tools}.
In the next lemma, a \emph{joint model limit} allows the initial
law and coefficients to vary along the same defining sequence.
Initial tuples use \eqref{eq:initial-admissible-scales}--\eqref{eq:initial-joint-limits};
flat tuples use \eqref{eq:whole-line-convergence}; phase tuples
use \eqref{eq:whole-line-convergence} at the scales
\eqref{eq:doubling-tail-centers}.

\begin{lemma}
\label{lem:short-window-bounds}
Let $\mathcal F$ be a family of problems satisfying
Assumptions~\ref{ass:coefficients} and~\ref{def:initial-laws},
with $x_0=0$ and common data
$\mathfrak R,\mathfrak E,\omega_{A,q},\mathfrak I_0$.
Fix $K>0$ and either $T_0\in I$ or $\vartheta\in(0,1)$,
and use the corresponding interval family in
\eqref{eq:comparison-time-windows}.
Consider joint model limits obtained along sequences of problems
in $\mathcal F$.
For positive-time windows, consider flat and phase limits with
$t_j\in[T_0,T]$. For near-zero windows, consider all initial limits
and flat and phase limits with $t_j\to0$.
Suppose every such limit satisfies
\begin{equation}
\mathfrak G_\alpha(\mathbf T)\le K
\quad\text{for initial tuples},\qquad
\mathfrak G(\mathbf M)\le K
\quad\text{for flat and phase tuples}.
\label{eq:common-model-bound}
\end{equation}
Then, for every $\varepsilon>0$, there is $h_0\in(0,T]$,
common to $\mathcal F$, such that
\begin{equation*}
G_\alpha^{\mathrm{fix}}(J)\le K+\varepsilon
\end{equation*}
for every member of $\mathcal F$ and every $J$ in the chosen
family at threshold $h_0$.
The threshold $h_0$ may depend on $\mathcal F,\alpha,K,\varepsilon$
and on $T_0$ or $\vartheta$, respectively.
\end{lemma}

\begin{proof}
Fix $K<K_1<K+\varepsilon$ and, for now, $L\ge2$.
For sufficiently small $h$, uniformly over $\mathcal F$, we claim
the following bound on each cylinder $Q_i$ of
Lemma~\ref{lem:localization-partition}:
\begin{equation}
\|\mathcal H_qv\|_{\alpha,J_i}
 \le K_1\|P_Av\|_{\alpha,J_i},
\qquad v\in\mathcal W_{\alpha,0}(J_i),\quad\supp v\subseteq Q_i.
\label{eq:local-cylinder-gain}
\end{equation}
If not, choose $h_j\downarrow0$, problems in $\mathcal F$ and
intervals $J^{(j)}$ in the chosen family with $h=h_j$.
In each partition choose a cylinder $Q_j=J_j\times B_j$ and
a test $v_j\in\mathcal W_{\alpha,0}(J_j)$ supported in $Q_j$
that violates \eqref{eq:local-cylinder-gain}.
Write $\mu_0^{(j)}$ for the initial laws and
$\rho_j,a_j,A_j,q_j$ for the corresponding densities and coefficients.
Pass to a subsequence of one cutoff type.

For spatial cutoffs, which occur only near zero, use
$(\tau,z)=(t/\sigma_j^2,(x-y_j)/\sigma_j)$ from
Lemma~\ref{lem:localization-partition}\ref{item:localization-spatial}.
Here $\sigma_j^2=\sup J^{(j)}\le h_j\to0$,
$\sigma_j\langle y_j\rangle\le C_\vartheta L$, and
$\vartheta\le\tau\le1$, $|z|\le C_\vartheta L$ on $Q_j$.
Lemma~\ref{lem:initial-tangents} gives a nonzero vague limit
of the normalized initial measures $\theta^{(j)}_{y_j,\sigma_j}$,
jointly with limits of $(a_j,A_j,q_j)(0,y_j)$.
Its common Gaussian tail bounds give locally uniform convergence
of the corresponding heat convolutions $W^{(j)}_{y_j,\sigma_j}$ in
\eqref{eq:initial-heat-convolution}. By \eqref{eq:initial-heat-approximation},
the normalized densities then converge uniformly on the rescaled $Q_j$
to the resulting positive heat density.

For parabolic cutoffs, use the centers $(t_j,x_j)$ and scales $r_j$
from Lemma~\ref{lem:localization-partition}\ref{item:localization-parabolic}, and put
$\ell_j=\ell^{(j)}(t_j,x_j)$, $\widetilde M_j=t_j/r_j^2$ and
$M_j=t_j/\ell_j^2$.
Then $r_j^2\le h_j$, $r_j\le\ell_j$ and
$\widetilde M_j\ge\max\{c_0L^2,M_j\}$.
For positive-time intervals $t_j\ge T_0$ and $\widetilde M_j\to\infty$;
for near-zero intervals $t_j\le h_j\to0$.
After passing to a subsequence, one of the following holds.
\begin{enumerate}
\item If $\widetilde M_j$ is bounded, choose a minimizing support point $y_j$
in \eqref{eq:weighted-tail-scale} and use
$(\tau,z)=(t/r_j^2,(x-y_j)/r_j)$. Then
\begin{equation*}
|x_j-y_j|/r_j\le \widetilde M_j,\,\, r_j\langle y_j\rangle\le1,\,\,
\widetilde M_j/2\le\tau\le3\widetilde M_j/2,\,\, |z|\le \widetilde M_j+3L
\,\,\text{on }Q_j.
\end{equation*}
The common bounds of Lemma~\ref{lem:initial-tangents} and
\eqref{eq:initial-heat-approximation} apply with $\sigma_j=r_j$,
giving a joint initial tuple and density convergence on the rescaled $Q_j$.

\item If $\widetilde M_j\to\infty$ and $M_j$ is bounded, the coordinates
$(s,z)=((t-t_j)/r_j^2,(x-x_j)/r_j)$ satisfy
$|s|\le9L^2$, $|z|\le3L$ on $Q_j$.
Since $r_j/\sqrt{t_j}\to0$,
the uniform estimate in Proposition~\ref{lem:uniform-flat-approximation} gives the flat
density limit. Take limits of $(a_j,A_j,q_j)(t_j,x_j)$ along a subsequence.

\item If $M_j\to\infty$, use
$(s,z)=((t-t_j)/\ell_j^2,(x-x_j)/\ell_j)$.
Since $r_j\le\ell_j$, these coordinates also satisfy
$|s|\le9L^2$, $|z|\le3L$ on $Q_j$.
The probability measures supplied by Proposition~\ref{prop:tail-compactness}
are supported on the common compact annulus $\mathcal Z_{\mathrm{ph}}$.
Take their weak limit jointly with the matrices at $(t_j,x_j)$.
The uniform estimate in Proposition~\ref{prop:tail-compactness} then gives
convergence to the corresponding phase weight.
\end{enumerate}

By \eqref{eq:common-model-bound}, the initial gain definition
\eqref{eq:initial-model-gain} and the whole-line interval bound
in Corollary~\ref{cor:whole-line-intervals}\ref{item:whole-line-intervals} give gain at most $K$ on each
rescaled interval, with free initial trace.
The density limits and
\eqref{eq:gain-rescaling}--\eqref{eq:tail-time-factor} give relative
convergence of the normalized full weights to $\tau^\alpha W_{\mathbf T}$
or $W_{\mathbf M}$. Equation~\eqref{eq:local-coefficient-freezing}
gives uniform matrix convergence.
Norm equivalence and continuity of the matrix square root therefore
give, back in physical variables,
\begin{equation}
\|\mathcal H_{q_j}v\|_j
 \le(K+\eta_j)\|P_{A_j}v\|_j+\eta_j\|D^2v\|_j,
\qquad \eta_j\longrightarrow0,
\label{eq:local-comparison-bound}
\end{equation}
where $\|\cdot\|_j=\|\cdot\|_{\alpha,J_j;\rho_j}$ and $\eta_j$
is controlled by the relative full-weight and matrix errors,
with a modulus depending only on $K,\mathfrak E$.
The comparison uses only times in $J_j$, including for centers approaching $T$.
Apply it to $v_j$. Since
$\lambda_q\|D^2v_j\|_j\le\|\mathcal H_{q_j}v_j\|_j$,
absorption gives \eqref{eq:local-cylinder-gain} for $v_j$ when $j$
is large, a contradiction.

For $u\in\mathcal C_{\alpha,0}(J)$, squaring
\eqref{eq:local-cylinder-gain} for each $u_i$ and summing gives,
by \eqref{eq:square-localization-error} and ellipticity,
\begin{equation*}
\|\mathcal H_qu\|_{\alpha,J}^2
 \le K_1^2\|P_Au\|_{\alpha,J}^2
 +C(1+K_1^2)L^{-2}
       \bigl(\lambda_q^{-2}\|\mathcal H_qu\|_{\alpha,J}^2
             +\|P_Au\|_{\alpha,J}^2\bigr).
\end{equation*}
Since $K_1<K+\varepsilon$, choose $L$ large enough that absorption
gives $\|\mathcal H_qu\|_{\alpha,J}
\le(K+\varepsilon)\|P_Au\|_{\alpha,J}$.
Then choose $h_0$ uniformly over $\mathcal F$ for
Lemma~\ref{lem:localization-partition} and \eqref{eq:local-cylinder-gain} to apply.
Equation~\eqref{eq:short-time-constants} gives the gain bound.
\end{proof}

\subsection{The three-term formula and bounds for arbitrary initial laws}
\label{subsec:localization}

Under Assumption~\ref{ass:coefficients} alone, the next lemma
extends the bounds assumed in \eqref{eq:two-window-input} to all short
intervals, including those starting at zero, for arbitrary initial laws.
Its proof is in Appendix~\ref{subsec:localization-tools}.

\begin{lemma}
\label{lem:time-window-extension}
For $0<k_0<k_1$, there is $\vartheta\in(0,1)$, depending only on
$k_0,k_1,\mathfrak R,\mathfrak E,\alpha$, with the following property.
If $0<b_\vartheta,h_+\le T$ and
\begin{equation}
G_\alpha^{\mathrm{fix}}(J)\le k_0,
\qquad
J\in\mathcal J_{b_\vartheta}^0(\vartheta)
 \cup\mathcal J_{h_+}^+(b_\vartheta/2),
\label{eq:two-window-input}
\end{equation}
then
\begin{equation*}
G_{\alpha,h}(I)\le k_1,
\qquad 0<h\le\min\{h_+,b_\vartheta/2\}.
\end{equation*}
\end{lemma}

\begin{proof}[Proof of Theorem~\ref{thm:sharp-limit}]
Put $k_*=G_\alpha^{\mathrm{mod}}(I)$.
Lemma~\ref{lem:recovery} proves the theorem when $k_*=\infty$.
Suppose $k_*<\infty$ and consider the singleton family consisting
of the fixed problem. Its initial and phase limits belong to
$\mathfrak T(\mu_0)$ and $\mathfrak M_+(I)$, respectively.
Equations~\eqref{eq:initial-linear-formula}, \eqref{eq:flat-gain},
\eqref{eq:phase-gain-formula}, \eqref{eq:flat-contribution} and
\eqref{eq:propagation-contribution} bound the initial, flat and phase gains
by $k_*$. Thus \eqref{eq:common-model-bound} holds with $K=k_*$
for either window family.

Fix $k_*<k_0<k_1$ and take $\vartheta$ from
Lemma~\ref{lem:time-window-extension}.
Lemma~\ref{lem:short-window-bounds}, with excess $k_0-k_*$,
gives a near-zero threshold $b_\vartheta$ and then, with
$T_0=b_\vartheta/2$, a positive-time threshold $h_+$.
Lemma~\ref{lem:time-window-extension} yields
$G_{\alpha,h}(I)\le k_1$ for small $h$.
Letting $h\downarrow0$ and $k_1\downarrow k_*$ proves the upper
bound; Lemma~\ref{lem:recovery} gives the reverse inequality.
\end{proof}

We now drop Assumption~\ref{def:initial-laws}
and allow an arbitrary initial probability law $\mu_0$.

\begin{proposition}
\label{prop:uniform-initial-law-bounds}
Fix $K>0$ and $\varepsilon>0$.
\begin{enumerate}
\item\label{item:uniform-positive-time}
Let $T_0\in I$ and $R\ge0$, and suppose
\begin{equation}
\begin{aligned}
q&\preceq KA &&\text{on }[T_0,T]\times\R^d,\\
q&\preceq K(a-A) &&\text{on }[T_0,T]\times\{|x|\ge R\}.
\end{aligned}
\label{eq:positive-time-matrix-bounds}
\end{equation}
For every initial probability law $\mu_0$, there is $h_0>0$ such that
\begin{equation*}
G_{\alpha,h_0}((T_0,T))\le K+\varepsilon.
\end{equation*}

\item\label{item:uniform-near-zero}
Let $S\subseteq\R^d$ be nonempty and closed. Suppose
$q\preceq KA$ and $q\preceq K(a-A)$ on $\{0\}\times\R^d$, and
\begin{equation}
\mathfrak G_\alpha(\delta_0,a(0,y),A(0,y),q(0,y))\le K,
\qquad y\in S.
\label{eq:point-initial-gain-bound}
\end{equation}
For each $\vartheta\in(0,1)$, there is $h_0>0$ such that
for every initial probability law with $\supp\mu_0\subseteq S$
\begin{equation*}
G_\alpha^{\mathrm{fix}}(J)\le K+\varepsilon,
\qquad J\in\mathcal J_{h_0}^0(\vartheta).
\end{equation*}
\end{enumerate}
The thresholds depend only on
$\mathfrak R,\mathfrak E,\omega_{A,q},\alpha,K,\varepsilon$,
with additional dependence on $T_0,R,\mu_0$ in
part~\ref{item:uniform-positive-time} and on $\vartheta$ in
part~\ref{item:uniform-near-zero}.
In part~\ref{item:uniform-positive-time}, dependence on $\mu_0$
can be omitted when $R=0$.
In part~\ref{item:uniform-near-zero}, the threshold is independent of $S$.
\end{proposition}

\begin{proof}
We first treat part~\ref{item:uniform-positive-time} with $R=0$
and part~\ref{item:uniform-near-zero}, fixing $T_0$ or $\vartheta$,
respectively.
For a point start at $y$, write $p_y(t,x)=\rho(0,y;t,x)$.

\paragraph*{Point starts and mixtures.}
Consider all starting points and coefficient fields with the
stated common data, satisfying
\eqref{eq:positive-time-matrix-bounds} with $R=0$ in the first case,
or the time-zero matrix inequalities in
part~\ref{item:uniform-near-zero} and
\eqref{eq:point-initial-gain-bound} at the starting point in the second.
Translate each starting point and its coefficient fields to place
the start at zero, and denote the resulting family by $\mathcal F$.
Every member has initial law $\delta_0$ and
$\mathfrak I_0=(1,0,1)$, by
Example~\ref{ex:compact-initial-laws}.
We verify the model bounds required by
Lemma~\ref{lem:short-window-bounds} for this family.

For flat and phase limits, the matrix inequalities pass to the
limit, using the common time continuity moduli for near-zero windows. Thus $q_*\preceq KA_*$ and
$q_*\preceq K(a_*-A_*)$, so
Proposition~\ref{prop:whole-line-gains} gives gain at most $K$.
For near-zero windows, the normalized initial measures are all $\delta_0$.
Along a joint matrix limit, the convergence of the norms of each
fixed test, as in the argument for \eqref{eq:recovery-norm-limits}, passes
\eqref{eq:point-initial-gain-bound} to the limiting point model.
Restriction and monotone convergence include intervals starting at zero.
Lemma~\ref{lem:short-window-bounds} now gives a common $h_0$
such that, after translating back,
\begin{equation*}
G_\alpha^{\mathrm{fix}}(J;p_y,A,q)\le K+\varepsilon
\end{equation*}
for all these point starts and all
$J\in\mathcal J_{h_0}^+(T_0)$ or
$J\in\mathcal J_{h_0}^0(\vartheta)$, respectively.
The family $\mathcal F$ was defined using only the listed data
and conditions, so $h_0$ has the stated dependence and is
independent of $S$.

Fix $a,b,A,q$ and an initial law $\mu_0$ allowed in the corresponding part.
The identity $\rho=\int p_y\,\mu_0(\mathrm{d}y)$ and
Lemma~\ref{lem:weight-operations}\ref{item:positive-mixtures}
give the same bound for $\rho$ at this $h_0$.
This proves part~\ref{item:uniform-positive-time} for $R=0$
and part~\ref{item:uniform-near-zero}.

\paragraph*{Positive-time bounds for $R>0$.}
Fix $K<K_1<K+\varepsilon$.
We first use the global case just proved to control tests supported
outside $B_{2R}$. Leave the reference coefficients $a,b$ unchanged
and set, with $\chi_R$ from \eqref{eq:spatial-cutoffs},
\begin{equation*}
\widetilde A=(1-\chi_R)A+\chi_R\frac a2,
\qquad
\widetilde q=(1-\chi_R)q+\chi_R\frac{Ka}{2}.
\end{equation*}
The marginal density $\rho$ is unchanged.
The modified matrices have common ellipticity and continuity data
depending only on the original data, $K,R$.
Since $\chi_R=1$ on $B_R$, convexity of
\eqref{eq:positive-time-matrix-bounds} gives
\begin{equation*}
\widetilde q\preceq K\widetilde A,\qquad
\widetilde q\preceq K(a-\widetilde A)
\quad\text{on }[T_0,T]\times\R^d.
\end{equation*}
The case $R=0$, applied with excess $K_1-K$, gives a common
threshold $h>0$, independently of the initial law.
Since $\widetilde A=A$ and $\widetilde q=q$ outside $B_{2R}$,
\begin{equation}
\|\mathcal H_qv\|_{\alpha,J}
 \le K_1\|P_Av\|_{\alpha,J},
\label{eq:uniform-exterior-bound}
\end{equation}
for $J\in\mathcal J_h^+(T_0)$ and
$v\in\mathcal W_{\alpha,0}(J)$ supported in
$J\times\{|x|\ge2R\}$.

Now fix $\mu_0$ and vary the coefficients with the stated common data.
Put $D=2R+1$ and $Q=[T_0,T]\times\overline B_D$.
Integrating the Gaussian lower bound \eqref{eq:kernel-gaussian}
over a ball of positive $\mu_0$-mass gives a common positive
lower bound for $\rho$ on $Q$.
For a common continuity modulus, split $\rho=\int p_y\,\mu_0(\mathrm{d}y)$
at $|y|=N$: the bounded part uses
Lemma~\ref{lem:kernel-stability}\ref{item:kernel-estimates},
and the Gaussian upper bound controls the tail by
$C\mu_0(B_N^c)\to0$.
Thus the full weights $\varpi=t^\alpha\rho$ have a common positive
lower bound and a common continuity modulus on $Q$.
Freeze $\varpi,A,q$ at a point of each cylinder $J\times B\subset Q$.
The relative weight and matrix errors tend uniformly to zero as
$|J|$ and the radius of $B$ tend to zero.
Since $q\preceq KA$, \eqref{eq:flat-gain} and
Corollary~\ref{cor:whole-line-intervals}\ref{item:whole-line-intervals}, with time exponent zero,
bound the frozen gain by $K$.
The comparison \eqref{eq:local-comparison-bound} and ellipticity
therefore give \eqref{eq:uniform-exterior-bound} also for tests
supported in these cylinders, after choosing a common $r<1/2$
and decreasing $h$, whenever $B$ has radius $r$ and $|J|\le h$.

Fix a spatial square partition
$\{\chi_i\}$ from
Lemma~\ref{lem:parabolic-partition}\ref{item:spatial-partition}
(with dilation parameter $2$), with $\supp\chi_i\subset B_i$
and balls $B_i$ of radius $r$.
For $J\in\mathcal J_h^+(T_0)$ and $u\in\mathcal C_{\alpha,0}(J)$,
each $\chi_i u$ satisfies \eqref{eq:uniform-exterior-bound}:
$B_i$ is either disjoint from $B_{2R}$ or contained in $B_D$.
Squaring and summing, \eqref{eq:fixed-spatial-localization-error} gives
\begin{equation*}
\|\mathcal H_qu\|_{\alpha,J}^2
 \le K_1^2\|P_Au\|_{\alpha,J}^2
+(1+K_1^2)C_\chi h\mathscr G_J(u)^2.
\end{equation*}
The partition is fixed with common derivative bounds, so $C_\chi$
is common to the coefficient family.
Ellipticity and $K_1<K+\varepsilon$ allow a further common decrease
of $h$ to absorb the error, proving part~\ref{item:uniform-positive-time}
with the stated dependence on $\mu_0$ and the coefficient data.
\end{proof}

\subsection{Linear terminal problems and damping}
\label{subsec:linear-theory}

Throughout this subsection,
$I=(0,T)$, Assumption~\ref{ass:coefficients} holds, and
$\mu_0$ is an arbitrary initial probability law.
The terminal problems below may be restricted to any subinterval of $I$;
the marginals $\mu_t$ and time weight $t^\alpha$ are retained.
For $s\in(0,T]$, let
\begin{equation*}
H^1(\mu_s)=\{g\in H^1_{\mathrm{loc}}(\R^d):g,Dg\in L^2(\mu_s)\},
\quad
\|g\|_{H^1(\mu_s)}^2
 =\|g\|_{L^2(\mu_s)}^2+\|Dg\|_{L^2(\mu_s)}^2.
\end{equation*}
For $0<t<s\le T$, let $U(t,s)g=v(t)$ denote the drift-free
evolution defined by $P_{a/2}v=0$, $v(s)=g$, in
Lemma~\ref{lem:reference-smoothing}.

\begin{theorem}[Terminal traces and linear solvability]\label{thm:linear-theory}\label{prop:realization}
Fix $J=(T_1,T_2)\in\mathcal J(I)$.
\begin{enumerate}
\item\label{item:linear-trace}
For every $u\in\mathcal W_\alpha(J)$, the local trace
$\operatorname{Tr}^{\mathrm{loc}}_{T_2}u$ from Lemma~\ref{lem:graph-core}
belongs to $H^1(\mu_{T_2})$.
Denote this trace by
$\operatorname{Tr}_{T_2}:\mathcal W_\alpha(J)\to H^1(\mu_{T_2})$.
It is bounded and surjective, with
\begin{equation*}
\|\operatorname{Tr}_{T_2}u\|_{H^1(\mu_{T_2})}
 \le C_{\mathrm{tr}}\|u\|_{\mathcal W_\alpha(J)},
 \qquad \forall u\in\mathcal W_\alpha(J).
\end{equation*}
It has the bounded right inverse $g\mapsto U(\cdot,T_2)g$, with
\(
\|U(\cdot,T_2)g\|_{\mathcal W_\alpha(J)}
 \le C_{\mathrm{ext}}\|g\|_{H^1(\mu_{T_2})}.
\)
Here $C_{\mathrm{tr}}$ depends only on $\mathfrak R$ in
Assumption~\ref{ass:coefficients}, $\alpha$ and $J$, while
$C_{\mathrm{ext}}$ depends only on $\mathfrak R$ and $\alpha$,
uniformly over initial laws and $J\in\mathcal J(I)$.

\item\label{item:linear-invertibility}
The following conditions are equivalent:
\begin{equation*}
\begin{gathered}
G_\alpha(J)<\infty;\qquad G_\alpha^{\mathrm{fix}}(J)<\infty;\qquad
P_A:\mathcal W_{\alpha,0}(J)\longrightarrow L^2_\alpha(J)
\quad\text{is bijective}.
\end{gathered}
\end{equation*}
When these conditions hold, denote the inverse by
$S_J:L^2_\alpha(J)\to\mathcal W_{\alpha,0}(J)$.
For every $K\in\mathcal J(J)$, let $\1_KF$ denote extension by zero to $J$. The inverse $S_K$ exists and satisfies
\begin{equation*}
S_KF=\bigl(S_J(\1_KF)\bigr)|_K,
\qquad F\in L^2_\alpha(K).
\end{equation*}

\item\label{item:linear-solvability}
Under the equivalent conditions in part~\ref{item:linear-invertibility},
for every $f\in L^2_\alpha(J)$ and $g\in H^1(\mu_{T_2})$,
there is a unique $u=\mathscr L_J(f,g)\in\mathcal W_\alpha(J)$ satisfying
$P_Au=f$ and $\operatorname{Tr}_{T_2}u=g$, with
\begin{equation*}
\|u\|_{\mathcal W_\alpha(J)}
 \le C_{\mathrm{sol}}
       \bigl(\|f\|_{\alpha,J}+\|g\|_{H^1(\mu_{T_2})}\bigr).
\end{equation*}
The constant $C_{\mathrm{sol}}$ depends only on $\mathfrak R$,
$\alpha$ and $\mathfrak E$, together with either a finite
upper bound for $G_\alpha^{\mathrm{fix}}(J)$, or $h_0>0$ and a finite
upper bound for $G_{\alpha,h_0}(J)$.
The solver is linear in $(f,g)$ and satisfies $\mathscr L_J(f,0)=S_Jf$.
\end{enumerate}
\end{theorem}

The proof is given in Appendix~\ref{subsec:zero-terminal-solvability}.
We next use damping to make the Hessian gain on a fixed interval
approach its short-time limit, while making the value and gradient
gains small.
For $J=(T_1,T_2)\in\mathcal J(I)$ and $\kappa\ge0$, define
\begin{equation}
\|f\|_{\alpha,\kappa,J}^2
 =\int_{T_1}^{T_2}\int_{\R^d}
   t^\alpha \mathrm{e}^{-2\kappa(T_2-t)}\rho(t,x)|f(t,x)|^2\,\mathrm{d}x\,\mathrm{d}t.
\label{eq:damped-interval-norm}
\end{equation}
This weight corresponds to damping: $v=\mathrm{e}^{-\kappa(T_2-t)}u$ satisfies
$(P_A+\kappa)v=\mathrm{e}^{-\kappa(T_2-t)}P_Au$.
The exponential factor is bounded above and below by positive constants
on $J$, so it leaves $\mathcal W_{\alpha,0}(J)$ unchanged.
The damped Hessian, value and gradient gains are
\begin{equation}
\begin{gathered}
G_{\alpha,\kappa}^{\mathrm{fix}}(J)
 =G_\alpha^{\mathrm{fix}}\bigl(J;\mathrm{e}^{-2\kappa(T_2-t)}\rho,A,q\bigr)
 =\sup_{u\in\mathcal W_{\alpha,0}(J)}
   \frac{\|\mathcal H_qu\|_{\alpha,\kappa,J}}
        {\|P_Au\|_{\alpha,\kappa,J}},\\
G_{\alpha,\kappa}^{\mathrm{val}}(J)
 =\sup_{u\in\mathcal W_{\alpha,0}(J)}
   \frac{\|u\|_{\alpha,\kappa,J}}{\|P_Au\|_{\alpha,\kappa,J}},\qquad
G_{\alpha,\kappa}^{\mathrm{grad}}(J)
 =\sup_{u\in\mathcal W_{\alpha,0}(J)}
   \frac{\|Du\|_{\alpha,\kappa,J}}{\|P_Au\|_{\alpha,\kappa,J}}.
\end{gathered}
\label{eq:damped-constants}
\end{equation}
The proof of the next proposition, based on
the zero-terminal inverse and interval averaging, is in
Appendix~\ref{subsec:damping-limit}.

\begin{proposition}\label{prop:poisson}
Fix $J=(T_1,T_2)\in\mathcal J(I)$ with $G_\alpha(J)<\infty$.
For every $K\in\mathcal J(J)$ and $\kappa\ge0$, each gain in
\eqref{eq:damped-constants} on $K$ is bounded by the corresponding gain on $J$.
\begin{enumerate}
\item\label{item:linear-damping}
Put $\varepsilon(h)=G_{\alpha,h}(J)-G_\alpha(J)$ for $h>0$.
For $\kappa>0$,
\begin{equation*}
0\le G_{\alpha,\kappa}^{\mathrm{fix}}(J)-G_\alpha(J)
 \le\int_0^\infty z\mathrm{e}^{-z}\varepsilon(z/\kappa)\,\mathrm{d}z.
\end{equation*}
In particular $G_{\alpha,\kappa}^{\mathrm{fix}}(J)\to G_\alpha(J)$ as $\kappa\to\infty$.
This convergence is uniform for families whose functions $\varepsilon$
are bounded by a common bounded function tending to zero at $h=0$.
\item\label{item:linear-lower-order}
Set
\begin{equation}
\kappa_{\mathrm{eff}}=\kappa-B_b^2/\lambda_a,\qquad
e_q=\left\|q^{-1/2}(A-a/2)q^{-1/2}\right\|_{L^\infty(\overline I\times\R^d)}.
\label{eq:linear-energy-data}
\end{equation}
For every $\kappa>B_b^2/\lambda_a$,
\begin{equation*}
G_{\alpha,\kappa}^{\mathrm{val}}(J)
 \le\frac{1+e_qG_{\alpha,\kappa}^{\mathrm{fix}}(J)}{\kappa_{\mathrm{eff}}}
   =O(\kappa^{-1}), \,\,
G_{\alpha,\kappa}^{\mathrm{grad}}(J)
 \le\frac{1+e_qG_{\alpha,\kappa}^{\mathrm{fix}}(J)}{\sqrt{\lambda_a\kappa_{\mathrm{eff}}}}
   =O(\kappa^{-1/2})
\end{equation*}
as $\kappa \rightarrow \infty$.
The $O$ constants depend only on $\mathfrak R$, $e_q$ and $G_\alpha^{\mathrm{fix}}(J)$.
\end{enumerate}
\end{proposition}

The decay of the gradient gain also gives invariance under bounded
PDE drifts. For a bounded Borel vector field $\beta$ on $I\times\R^d$,
write $G_{\alpha,h}^\beta,G_\alpha^\beta$ for the gains in
\eqref{eq:short-time-constants} with $P_A-\beta\cdot D$ in place of
$P_A$. The reference density $\rho$ and its drift $b$ remain fixed.
If $G_\alpha(I)<\infty$, then
\begin{equation}
G_\alpha^\beta(I;\rho,A,q)=G_\alpha(I).
\label{eq:borel-drift-invariance}
\end{equation}
Indeed, set
$\varepsilon_\kappa=\|\beta\|_\infty
G_{\alpha,\kappa}^{\mathrm{grad}}(I)$, which tends to zero by
Proposition~\ref{prop:poisson}\ref{item:linear-lower-order}.
The interval bounds in Proposition~\ref{prop:poisson} control
$\|\beta\cdot Du\|_{\alpha,\kappa,K}$ by
$\varepsilon_\kappa\|P_Au\|_{\alpha,\kappa,K}$ for
$K\in\mathcal J_h(I)$ and $u\in\mathcal W_{\alpha,0}(K)$.
Comparing the damped and undamped norms in
\eqref{eq:damped-interval-norm}, the triangle inequality gives,
for $\varepsilon_\kappa<1$,
\begin{equation*}
\frac{\mathrm{e}^{-\kappa h}G_{\alpha,h}(I)}{1+\varepsilon_\kappa}
 \le G_{\alpha,h}^\beta(I;\rho,A,q)
 \le\frac{\mathrm{e}^{\kappa h}G_{\alpha,h}(I)}{1-\varepsilon_\kappa}.
\end{equation*}
First let $h\downarrow0$ and then $\kappa\to\infty$ to obtain
\eqref{eq:borel-drift-invariance}.

\section{Formulas and bounds for limiting constants}
\label{sec:explicit-constants}

We first derive operator-norm formulas for initial models. We then give
coefficient upper bounds, quantitative estimates for arbitrary initial
laws, and matching constants. Radial reference structure yields exact
coefficient formulas in the final subsection.
Throughout this section, $I=(0,T)$, Assumption~\ref{ass:coefficients}
holds, and identities between optimal constants are understood in
$[0,\infty]$.

\subsection{The initial contribution}
\label{subsec:initial-convolution}

The initial contribution in \eqref{eq:initial-linear-formula} is a
supremum over initial heat models \eqref{eq:initial-heat-tuple},
whose initial measures may have infinite mass.

\begin{proposition}
\label{prop:finite-mass-initial}
For every initial heat tuple $\mathbf T=(\theta,a,A,q)$,
\begin{equation*}
\mathfrak G_\alpha(\mathbf T)
 \le\mathfrak G_\alpha(\delta_0,a,A,q).
\end{equation*}
If $\theta(\R^d)<\infty$, then equality holds.
\end{proposition}

\begin{proof}
Spatial translation leaves the point-model gain unchanged.
Applying Lemma~\ref{lem:weight-operations}\ref{item:positive-mixtures} to
$W_{\mathbf T}(s,z)=\int\gamma_{sa}(z-y)\,\theta(\mathrm{d}y)$
therefore gives the upper bound.

For the reverse inequality, let $m=\theta(\R^d)<\infty$ and assume
$C:=\mathfrak G_\alpha(\mathbf T)<\infty$. Set
\begin{equation*}
W_R(s,z)=\frac{R^d}{m}W_{\mathbf T}(R^2s,Rz)
 =\frac1m\int_{\R^d}\gamma_{sa}(z-y/R)\,\theta(\mathrm{d}y).
\end{equation*}
Dominated convergence gives $W_R/\gamma_{\cdot a}\to1$
locally uniformly on $(0,\infty)\times\R^d$ as $R\to\infty$.
Fix $J=(s_1,s_2)$ with $0<s_1<s_2$ and
$v\in\mathcal C_{\alpha,0}(J;\gamma_{\cdot a})$.
Since \eqref{eq:initial-model-gain} includes $R^2J$ for every $R>0$,
apply the estimate with constant $C$ to
$v_R(t,x)=R^2v(t/R^2,x/R)$. Changing variables gives
\begin{equation*}
\|\mathcal H_qv\|_{\alpha,J;W_R}
 \le C\|P_Av\|_{\alpha,J;W_R}.
\end{equation*}
Since $s_1>0$ and $v$ has compact spatial support, letting
$R\to\infty$ gives the same estimate with $\gamma_{\cdot a}$ in place of $W_R$.
Taking suprema over $v,J$ in \eqref{eq:initial-model-gain}, and including $s_1=0$
by restriction and monotone convergence, gives
$\mathfrak G_\alpha(\delta_0,a,A,q)\le C$.
\end{proof}

For homogeneous initial measures, logarithmic time and parabolic
rescaling express the gain through spatial operator norms. For
constant matrices $A,q\succ0$, define the Ornstein--Uhlenbeck operator
and its semigroup by
\begin{equation*}
\mathcal L_A=-A:D^2+\tfrac12x\cdot D,
\qquad S_A(\tau)=\mathrm{e}^{-\tau\mathcal L_A}.
\end{equation*}
For $\varphi\in C_c^\infty(\R^d;\mathbb C)$, the Mehler formula
\citep[Section~2]{lunardi-metafune-pallara-2020} gives
\begin{equation}
(S_A(\tau)\varphi)(x)
 =\int_{\R^d}\varphi(\mathrm{e}^{-\tau/2}x+z)
       \gamma_{2(1-\mathrm{e}^{-\tau})A}(z)\,\mathrm{d}z,\qquad \tau>0.
\label{eq:mehler-semigroup}
\end{equation}
The semigroup acts componentwise on matrix fields.
For $\nu>0$ and $\omega\in\R$, define
\begin{equation}
M_{A,q}(\nu-i\omega)\varphi
 :=\int_0^\infty \mathrm{e}^{-(\nu-i\omega)\tau}
                  S_A(\tau)\mathcal H_q\varphi\,\mathrm{d}\tau.
\label{eq:homogeneous-resolvent}
\end{equation}
We also write
$M_{A,q}(\nu-i\omega)=(\mathcal L_A+\nu-i\omega)^{-1}\mathcal H_q$,
with the inverse notation interpreted through this integral.
This operator maps the transformed source at frequency $\omega$
to the Hessian of the transformed solution.
For a positive continuous spatial weight $F$ and a linear operator $T$
from $C_c^\infty(\R^d;\mathbb C)$ to complex symmetric matrix fields, write
\begin{equation*}
\|T\|_{\mathrm{op},L^2(F\,\mathrm{d}x)}
 :=\sup_{0\ne\varphi\in C_c^\infty(\R^d;\mathbb C)}
 \frac{\left(\int|(T\varphi)(x)|_{\mathrm F}^2F(x)\,\mathrm{d}x\right)^{1/2}}
      {\left(\int|\varphi(x)|^2F(x)\,\mathrm{d}x\right)^{1/2}}.
\end{equation*}
The matrix norm in the numerator is the complex Frobenius norm.
The operator norm may be $+\infty$; when finite, it is the norm of the
bounded extension from scalar to matrix-valued $L^2(F\,\mathrm{d}x)$.

\begin{proposition}
\label{prop:homogeneous-initial-norm}
Let $\mathbf T=(\theta,a,A,q)$ be an initial heat tuple with
$\theta(rE)=r^k\theta(E)$ for every $r>0$ and Borel $E\subset\R^d$,
for some $k\ge0$. Set $F=W_{\mathbf T}(1,\cdot)$. Then
\begin{equation*}
\mathfrak G_\alpha(\mathbf T)
 =\sup_{\omega\in\R}
 \left\|
 M_{A,q}\!\left(\frac{\alpha+1}{2}+\frac{k}{4}-i\omega\right)
 \right\|_{\mathrm{op},L^2(F\,\mathrm{d}x)}.
\end{equation*}
\end{proposition}

Here $\omega$ is the frequency in logarithmic model time;
see Appendix~\ref{subsec:homogeneous-gain-proofs} for the proof.

\begin{eexample}
\label{ex:domain-initial-models}
Let $\Omega\subset\R^d$ be a nonempty bounded $C^1$ domain, and let
$\mu_0(\mathrm{d}x)=p(x)\1_\Omega(x)\,\mathrm{d}x$, where
$p\in C(\overline\Omega)$ is strictly positive and $\int_\Omega p=1$.
Then $\mu_0$ satisfies Assumption~\ref{def:initial-laws}, and
$G_\alpha^{\mathrm{init}}(I)$ is the supremum of the gains of the tuples
\begin{equation*}
\begin{aligned}
(\mathcal L^d,a_0(y),A_0(y),q_0(y)),\,\,
y\in\overline\Omega,\quad
(\1_{H_y}\mathcal L^d,a_0(y),A_0(y),q_0(y)),\,\,
y\in\partial\Omega,
\end{aligned}
\end{equation*}
where $H_y=\{z:z\cdot n(y)>0\}$ and $n(y)$ is the inward unit normal.
Both initial measures are homogeneous of degree $d$, so their gains
are given by Proposition~\ref{prop:homogeneous-initial-norm}.
In particular, the initial contribution is independent of $p$ within
this class.

Indeed, admissibility follows from Example~\ref{ex:compact-initial-laws}.
For any defining sequence $(y_j,r_j)$, pass to a subsequence such that
$y_j\to y$ and $\dist(y_j,\partial\Omega)/r_j\to\ell\in[0,\infty]$.
Boundary flattening and continuity of $p$ give
\begin{equation*}
\theta_{y_j,r_j}\xrightarrow{\mathrm{vague}}
\frac{\1_E\mathcal L^d}{|E\cap B_1|},\qquad
E=\begin{cases}
\R^d,&\ell=\infty,\\
\{z:z\cdot n(y)>-\ell\},&\ell<\infty.
\end{cases}
\end{equation*}
Here $y\in\partial\Omega$ in the second case, and the factor $p(y)$
cancels in the normalization. Spatial translation and multiplication
of the initial measure by a positive constant preserve the gain.
Conversely, fixed boundary centers realize the half-space models;
interior centers approaching any $y\in\overline\Omega$, with
$r_j=o(\dist(y_j,\partial\Omega))$, realize the full-space models.
Coefficient continuity and \eqref{eq:initial-linear-formula} prove the claim.
\end{eexample}

Point initial models have $k=0$ and $F=\gamma_a$ in
Proposition~\ref{prop:homogeneous-initial-norm}.
The next proposition, proved in Appendix~\ref{subsec:point-model-proofs},
reduces their supremum to zero frequency and shows that their gains
bound all whole-line gains with the same matrices.

\begin{proposition}
\label{prop:point-initial-models}
Let $a,A,q\succ0$ be constant matrices. By
Proposition~\ref{prop:whole-line-gains}, their maximal whole-line
model gain is
\begin{equation*}
G_{\mathrm{whole}}=\max\left\{
 \lambda_{\max}(A^{-1/2}qA^{-1/2}),
 \Phi\!\left(\lambda_{\min}
                  (q^{-1/2}(a-A)q^{-1/2})\right)\right\}.
\end{equation*}
\begin{enumerate}
\item\label{item:point-representation}
\(
\displaystyle
\mathfrak G_\alpha(\delta_0,a,A,q)
 =\left\|
   M_{A,q}\!\left(\frac{\alpha+1}{2}\right)
  \right\|_{\mathrm{op},L^2(\gamma_a\,\mathrm{d}x)}.
\)

\item\label{item:gaussian-finiteness}
For each fixed $\alpha$, the gain $\mathfrak G_\alpha(\delta_0,a,A,q)$ is continuous in $(a,A,q)$
throughout the region $a,A,q,a-A\succ0$. Moreover,
\begin{equation*}
\mathfrak G_\alpha(\delta_0,a,A,q)\ge G_{\mathrm{whole}};\qquad
\mathfrak G_\alpha(\delta_0,a,A,q)<\infty
\quad\Longleftrightarrow\quad a-A\succ0.
\end{equation*}

\item\label{item:matching-point-models}
For $A=a/2$ and $q=a$,
\begin{equation*}
\mathfrak G_\alpha(\delta_0,a,a/2,a)
 =C_\alpha^{\mathrm{pt}}
 :=\sup_{n\ge2}\frac{2\sqrt{n(n-1)}}{n-1+\alpha}
 \le2\sqrt2.
\end{equation*}
We have $C_0^{\mathrm{pt}}=2\sqrt2$,
$C_\alpha^{\mathrm{pt}}>2$ for $0\le\alpha<1/2$, and
$C_\alpha^{\mathrm{pt}}=2$ for $\alpha\ge1/2$.
\end{enumerate}
\end{proposition}

\begin{corollary}
\label{cor:constant-point-formula}
If $a,A,q$ are constant and $\mu_0=\delta_y$, then
\(
G_\alpha(I)=\mathfrak G_\alpha(\delta_0,a,A,q).
\)
\end{corollary}

\begin{proof}
Equation~\eqref{eq:initial-linear-formula} identifies
$G_\alpha^{\mathrm{init}}(I)$ with
$\mathfrak G_\alpha(\delta_0,a,A,q)$.
All coefficient limits equal $(a,A,q)$.
Proposition~\ref{prop:whole-line-gains} and
Proposition~\ref{prop:point-initial-models}\ref{item:gaussian-finiteness}
therefore give
\begin{equation*}
\max\{G^{\mathrm{flat}}(I),G^{\mathrm{phase}}(I)\}
 \le G_{\mathrm{whole}}
 \le\mathfrak G_\alpha(\delta_0,a,A,q).
\end{equation*}
The three-term formula in Theorem~\ref{thm:sharp-limit} proves the identity.
\end{proof}

For the following smooth densities, the initial contribution is
controlled by the whole-line models.
These laws satisfy Assumption~\ref{def:initial-laws} by
Example~\ref{ex:weighted-initial-densities}; they include all
nondegenerate Gaussian laws.
The proof of Proposition~\ref{prop:smooth-initial-absorption} is given in
Subsection~\ref{subsec:initial-absorption}, where nonconstant exponential
tangent profiles are recovered as phase models at large rescaled times.
\begin{proposition}
\label{prop:smooth-initial-absorption}
Let $\mu_0(\mathrm{d}x)=\mathrm{e}^{-V(x)}\,\mathrm{d}x$ be a probability law and $V\in C^2(\R^d)$ satisfy
\begin{equation}
|D^2V(x)|=o(|x|^2),\quad
c|x|\le|DV(x)|\le C|x|
\quad\text{as }|x|\to\infty,
\label{eq:smooth-potential}
\end{equation}
for some $c,C>0$. Then
\(
G_\alpha^{\mathrm{init}}(I)
 \le\max\{G^{\mathrm{flat}}(I),G^{\mathrm{phase}}(I)\}.
\)
Consequently,
\begin{equation*}
G_\alpha(I)=\max\{G^{\mathrm{flat}}(I),G^{\mathrm{phase}}(I)\}.
\end{equation*}
\end{proposition}

\subsection{Coefficient bounds and matching constants}
\label{subsec:coefficient-formulas}

Fix an initial law $\mu_0$ satisfying Assumption~\ref{def:initial-laws},
set $S=\supp\mu_0$, and write
$\mathsf R(t,x,v)=\mathsf R\bigl(a(t,x),A(t,x),q(t,x);v\bigr)$.
Denote by $\mathfrak M_+^{\mathrm{loc}}(I)$ and
$\mathfrak M_+^{\mathrm{esc}}(I)$ the tuples in $\mathfrak M_+(I)$
admitting a defining sequence with bounded $x_j$ or with
$|x_j|\to\infty$, respectively.
Their union is $\mathfrak M_+(I)$, although they may overlap.
For $\star\in\{\mathrm{loc},\mathrm{esc}\}$, set
\begin{equation}
\begin{aligned}
G^{\mathrm{phase},\star}(I)
 &=\sup_{(\lambda,a_*,A_*,q_*)\in\mathfrak M_+^\star(I)}
   \sup_{\zeta\in\supp\lambda}
       \Phi\bigl(\mathsf R(a_*,A_*,q_*;\zeta)\bigr),\\
G^{\mathrm{phase}}(I)
 &=\max\{G^{\mathrm{phase,loc}}(I),G^{\mathrm{phase,esc}}(I)\}.
\end{aligned}
\label{eq:local-escaping-contributions}
\end{equation}
An empty family contributes zero. In particular,
$\mathfrak M_+^{\mathrm{loc}}(I)$ is empty when $S=\R^d$.
Define the coefficient extrema
\begin{equation}
\mathsf r_0^{\mathrm{all}}(S)
 =\inf_{\substack{x\notin S\\|v|=1}}\mathsf R(0,x,v),\quad
\mathsf r_\infty^{\mathrm{all}}
 =\lim_{R\to\infty}
   \inf_{\substack{t\in[0,T],\ |x|\ge R\\|v|=1}}\mathsf R(t,x,v).
\label{eq:coefficient-upper-numbers}
\end{equation}
We use $\inf\varnothing=+\infty$ and $\Phi(+\infty)=0$.

\begin{proposition}
\label{prop:coefficient-upper-bound}
The local and escaping phase contributions satisfy
\begin{equation}
G^{\mathrm{phase,loc}}(I)
 \le\Phi\bigl(\mathsf r_0^{\mathrm{all}}(S)\bigr),
\qquad
G^{\mathrm{phase,esc}}(I)
 \le\Phi(\mathsf r_\infty^{\mathrm{all}}).
\label{eq:local-escaping-phase-upper-bounds}
\end{equation}
In particular,
\begin{equation}
G_\alpha(I)\le
\max\left\{G_\alpha^{\mathrm{init}}(I),G^{\mathrm{flat}}(I),
 \Phi\bigl(\mathsf r_0^{\mathrm{all}}(S)\bigr),
 \Phi(\mathsf r_\infty^{\mathrm{all}})\right\}.
\label{eq:coefficient-gain-upper-bound}
\end{equation}
\end{proposition}

\begin{proof}
Take $(\lambda,a_*,A_*,q_*)\in\mathfrak M_+^{\mathrm{loc}}(I)$
and a defining sequence with bounded centers.
Equation~\eqref{eq:weighted-scale-comparison} gives
$t_j\to0$ and $\dist(x_j,S)/\sqrt{t_j}\to\infty$, so
$x_j\notin S$ eventually. Time continuity and
\eqref{eq:coefficient-upper-numbers} imply
$\mathsf R(a_*,A_*,q_*;v)\ge\mathsf r_0^{\mathrm{all}}(S)$
for every $v\ne0$.

For a tuple in $\mathfrak M_+^{\mathrm{esc}}(I)$, choose a
defining sequence with $|x_j|\to\infty$.
Equation~\eqref{eq:coefficient-upper-numbers} then gives
$\mathsf R(a_*,A_*,q_*;v)\ge\mathsf r_\infty^{\mathrm{all}}$
for every $v\ne0$.
Since $\Phi$ is nonincreasing, these ratio bounds prove
\eqref{eq:local-escaping-phase-upper-bounds} by
\eqref{eq:local-escaping-contributions}.
The three-term formula in Theorem~\ref{thm:sharp-limit} then yields
\eqref{eq:coefficient-gain-upper-bound}.
\end{proof}

We now allow arbitrary initial probability laws and use the comparisons
of Proposition~\ref{prop:uniform-initial-law-bounds}.
The positive-time limiting bound depends only on the coefficients;
the full-interval bound also includes the point initial models.
For any nonempty closed $S\subseteq\R^d$, define the initial bound
\begin{equation*}
B_{\alpha,0}(S)
 =\sup_{y\in S}
       \mathfrak G_\alpha(\delta_0,a(0,y),A(0,y),q(0,y)).
\end{equation*}
For the coefficient bounds, set
\begin{equation*}
P_0=\sup_{x\in\R^d,\,|v|=1}
       \Phi\bigl(\mathsf R(0,x,v)\bigr),\qquad
B_+=\max\{G^{\mathrm{flat}}(I),\Phi(\mathsf r_\infty^{\mathrm{all}})\}.
\end{equation*}
The quantities $P_0,\Phi(\mathsf r_\infty^{\mathrm{all}})$ bound all coefficient directions
at time zero and at spatial infinity, respectively, independently of the
actual phase family. The time-zero bound is used in the near-zero
comparison for the full-interval estimate.
For finite-scale bounds, retain a radius $R\ge0$ and set
\begin{equation}
B_+(R)=\max\Bigl\{G^{\mathrm{flat}}(I),
 \sup_{\substack{0\le t\le T,\ |x|\ge R\\|v|=1}}
       \Phi\bigl(\mathsf R(t,x,v)\bigr)\Bigr\}.
\label{eq:finite-radius-coefficient-bound}
\end{equation}
Then $B_+(R)\downarrow B_+$ as $R\to\infty$.

\begin{corollary}
\label{cor:uniform-point-bound}
Fix $R\ge0$, $K>0$ and $\varepsilon>0$, with $B_+(R)\le K$.
\begin{enumerate}
\item\label{item:quantitative-gap-positive}
For every initial probability law $\mu_0$ and $0<T_0<T$, there is
$h_+ > 0$ such that
\begin{equation*}
G_{\alpha,h}((T_0,T))\le K+\varepsilon\quad\text{for }0<h\le h_+,
\qquad G_\alpha((T_0,T))\le B_+.
\end{equation*}
\item\label{item:quantitative-gap-full}
If $P_0\le K$ and $B_{\alpha,0}(S)\le K$ for a nonempty closed
$S\subseteq\R^d$, then every initial probability law supported in $S$
admits $h_0\in(0,T]$ such that
\begin{equation}
G_{\alpha,h}(I)\le K+\varepsilon,\qquad 0<h\le h_0,
\label{eq:finite-radius-full-bound}
\end{equation}
and satisfies
\begin{equation}
G_\alpha(I)\le\max\{B_{\alpha,0}(S),P_0,B_+\}.
\label{eq:localized-gap-bound}
\end{equation}
\end{enumerate}
The thresholds depend only on
$\mathfrak R,\mathfrak E,\omega_{A,q},\alpha,K,\varepsilon,R,\mu_0$,
and on $T_0$ in part~\ref{item:quantitative-gap-positive}.
Dependence on $\mu_0$ can be omitted when $R=0$.
\end{corollary}

\begin{proof}
By \eqref{eq:finite-radius-coefficient-bound} and
\eqref{eq:flat-contribution}, $B_+(R)\le K$ gives
$q\preceq KA$ everywhere and $q\preceq K(a-A)$ for $|x|\ge R$.
Proposition~\ref{prop:uniform-initial-law-bounds}\ref{item:uniform-positive-time}
therefore gives the finite-scale bound in
part~\ref{item:quantitative-gap-positive}, with the stated threshold
dependence.

For part~\ref{item:quantitative-gap-full}, set
$k_0=K+\varepsilon/2$, $k_1=K+\varepsilon$, and take
$\vartheta$ from Lemma~\ref{lem:time-window-extension}.
The assumptions $P_0,B_+(R),B_{\alpha,0}(S)\le K$ give the
time-zero matrix and point-model bounds required by
Proposition~\ref{prop:uniform-initial-law-bounds}\ref{item:uniform-near-zero},
which supplies a near-zero threshold $b_\vartheta\in(0,T]$
with constant $k_0$.
Apply part~\ref{item:quantitative-gap-positive} with
$T_0=b_\vartheta/2$ and excess $\varepsilon/2$ to obtain $h_+$.
Both bounds in \eqref{eq:two-window-input} now hold, so
Lemma~\ref{lem:time-window-extension} gives
\eqref{eq:finite-radius-full-bound} with
$h_0=\min\{h_+,b_\vartheta/2\}$ and the stated dependence.

Finally, take $K=B_+(R)$ in the positive-time estimate and
choose $K=\max\{B_{\alpha,0}(S),P_0,B_+(R)\}$ in the full-interval estimate.
Letting $h\downarrow0$, then $\varepsilon\downarrow0$, and finally $R\to\infty$
proves the limiting bounds in part~\ref{item:quantitative-gap-positive}
and \eqref{eq:localized-gap-bound}.
\end{proof}

The corollary allows the spectral gap to fail on a compact region
away from time zero: on that region the weight can be frozen and
only the flat coefficient bound is needed.
The exterior radius and the local variation of the density enter the
short-time threshold.
A convenient sufficient condition is, for fixed
$0<\tau<T$, $R\ge0$ and $\delta\in(0,1)$,
\begin{equation}
A\preceq(1-\delta)a
\quad\text{on }([0,T]\times\R^d)
 \setminus\bigl([\tau,T]\times\overline B_R\bigr).
\label{eq:compact-exception-spectral-gap}
\end{equation}
This bounds $B_+(R)$ and $P_0$ by continuity and ellipticity.
The time-zero matrix triples lie in a common compact subset of
$a,A,q,a-A\succ0$, so
Proposition~\ref{prop:point-initial-models}\ref{item:gaussian-finiteness}
also bounds $B_{\alpha,0}(\R^d)$.
Consequently, a common $K$ depending only on
$\mathfrak R,\mathfrak E,\alpha,\delta$ gives
\eqref{eq:finite-radius-full-bound} for every initial law.
On each fixed $J\in\mathcal J((T_0,T))$ in
Corollary~\ref{cor:uniform-point-bound}\ref{item:quantitative-gap-positive},
or $J\in\mathcal J(I)$ in part~\ref{item:quantitative-gap-full},
a common threshold and bound $K+\varepsilon$ give a common
terminal-solver estimate by
Theorem~\ref{thm:linear-theory}\ref{item:linear-solvability}.
For $A=a/2$, $q=a$, the gap is global and $B_+(0)=P_0=2$.
The next theorem gives the matching gains.

\begin{theorem}[Matching constants]
\label{thm:matching-arbitrary-laws}
Under Assumption~\ref{ass:coefficients}, suppose
$A=a/2$ and $q=a$ on $[0,T]\times\R^d$.
\begin{enumerate}
\item\label{item:matching-positive-time}
For every initial probability law $\mu_0$
and $0<T_1<T_2\le T$,
$G_\alpha((T_1,T_2))=2$.

\item\label{item:matching-full-interval}
Every initial probability law $\mu_0$ satisfies, on $I=(0,T)$,
\begin{equation*}
2\le G_\alpha(I)\le C_\alpha^{\mathrm{pt}},
\end{equation*}
with equality in the upper bound for every point mass and in the lower bound for $\alpha\ge1/2$.
If
\begin{equation*}
\mu_0=\sum_{i=1}^N\beta_i\mu_i,\qquad
\beta_i>0,\quad\sum_{i=1}^N\beta_i=1,
\end{equation*}
where each $\mu_i(\mathrm{d}x)=\mathrm{e}^{-V_i(x)}\,\mathrm{d}x$ satisfies
\eqref{eq:smooth-potential}, then $G_\alpha(I)=2$.
This includes finite mixtures of arbitrary nondegenerate Gaussian laws.
\end{enumerate}
\end{theorem}

\begin{proof}
For $0<T_1<T_2\le T$,
Proposition~\ref{prop:uniform-initial-law-bounds}\ref{item:uniform-positive-time},
with $K=2$, $R=0$ and $T_0=T_1$, gives
$G_\alpha((T_1,T_2))\le2$ after restriction to $(T_1,T_2)$
and letting $\varepsilon\downarrow0$.

Since $\rho$ is positive and continuous, the fixed-center case of
\eqref{eq:recovery-norm-limits} applies on every $J\in\mathcal J(I)$,
for any initial law. Together with \eqref{eq:flat-gain} and the matching
relations, it gives
\begin{equation}
G_\alpha(J)\ge2,\qquad J\in\mathcal J(I).
\label{eq:matching-flat-lower-bound}
\end{equation}
This proves the equality in part~\ref{item:matching-positive-time}.

The global matching relations and
Proposition~\ref{prop:point-initial-models}\ref{item:matching-point-models} give $P_0=B_+=2$ and
$B_{\alpha,0}(\R^d)=C_\alpha^{\mathrm{pt}}$.
The full-interval upper bound follows from \eqref{eq:localized-gap-bound}
with $S=\R^d$.
Since $C_\alpha^{\mathrm{pt}}=2$ for $\alpha\ge1/2$ by
Proposition~\ref{prop:point-initial-models}\ref{item:matching-point-models}, the full-interval bounds give
$G_\alpha(I)=2$ in this range.
For $\mu_0=\delta_y$, the initial family $\mathfrak T(\delta_y)$ contains
$(\delta_0,a(0,y),a(0,y)/2,a(0,y))$.
Lemma~\ref{lem:recovery} and Proposition~\ref{prop:point-initial-models}\ref{item:matching-point-models} give
$G_\alpha(I)\ge C_\alpha^{\mathrm{pt}}$, proving attainment
of the upper bound.

When the dependence on the initial law needs to be displayed, we write
$\rho^{\mu_0}$ for $\rho$.
For each smooth law $\mu_i$, the matching relations
and Proposition~\ref{prop:coefficient-upper-bound} give
$G^{\mathrm{flat}}(I)=2$ and $G^{\mathrm{phase}}(I)\le2$, so
Proposition~\ref{prop:smooth-initial-absorption} gives
$G_\alpha(I;\rho^{\mu_i},a/2,a)=2$.
For a finite mixture, Lemma~\ref{lem:weight-operations}\ref{item:positive-mixtures}
gives
\begin{equation*}
G_{\alpha,h}(I;\rho^{\mu_0},a/2,a)
 \le\max_{1\le i\le N}G_{\alpha,h}(I;\rho^{\mu_i},a/2,a).
\end{equation*}
Letting $h\downarrow0$ gives the upper bound $2$, because the
maximum is over finitely many laws. Equation~\eqref{eq:matching-flat-lower-bound} gives equality.
\end{proof}

\subsection{Radial structure and drift-invariant formulas}
\label{subsec:radial-formulas}
Under radial reference structure, we keep the covariance fixed and
choose an auxiliary bounded drift that gives a radial marginal density.
We compute its phase contributions and transfer them to the
prescribed drift by comparison.

\begin{assumption}[Radial reference structure]
\label{ass:radial-structure}
Let $\mu_{\mathrm{rad}}$ be a probability law invariant under every
orthogonal transformation,
satisfying Assumption~\ref{def:initial-laws}, and
\begin{equation*}
\mu_0=f\mu_{\mathrm{rad}},\qquad
0<c\le f\le C<\infty\quad\mu_{\mathrm{rad}}\text{-a.e.},
\qquad \int f\,\mathrm{d}\mu_{\mathrm{rad}}=1.
\end{equation*}
For every $B\ge1$, require
\begin{equation}
\lim_{\delta\downarrow0}
\sup_{\substack{y\in \supp\mu_{\mathrm{rad}},\ 0<r\le\delta\\r\langle y\rangle\le B}}
 \operatorname*{ess\,osc}_{B(y,r),\mu_{\mathrm{rad}}}\log f=0.
\label{eq:initial-flat-factor}
\end{equation}
For some $\sigma>0$,
the reference covariance satisfies
\begin{equation}
a(t,x)x=\kappa_{\mathrm{rad}}(t,|x|)x,
\qquad
C_\sigma:=\operatorname*{ess\,sup}_{(t,x)\in I\times\R^d}
 (1+|x|)^\sigma|D_xa(t,x)|<\infty.
\label{eq:radial-covariance}
\end{equation}
\end{assumption}

Throughout this subsection, write
$S=\supp\mu_0=\supp\mu_{\mathrm{rad}}$.
Here $\operatorname*{ess\,osc}_{E,\mu}g$ denotes the difference
between the $\mu$-essential supremum and infimum of $g$ on $E$.
On initial scales $r\downarrow0$ with $r\langle y\rangle$ bounded,
condition~\eqref{eq:initial-flat-factor} makes $f$ asymptotically
constant, so the normalized initial models are unchanged.
It holds whenever $\log f$ has a uniformly continuous representative
on $S$.

The covariance conditions ensure the radial density construction
and the comparison with the prescribed drift used below.
For $x\ne0$, set $e_x=x/|x|$.
The first condition in \eqref{eq:radial-covariance} allows the form
\begin{equation*}
a(t,x)=\kappa_{\mathrm{rad}}(t,|x|)e_x\otimes e_x
                 +B_{\mathrm{tan}}(t,x),
\qquad B_{\mathrm{tan}}(t,x)e_x=0.
\end{equation*}
Here $B_{\mathrm{tan}}$ is symmetric and uniformly positive on
$e_x^\perp$. Subject to the standing coefficient assumptions and
the derivative bound, the radial scalar and the entire tangential
matrix field can be chosen independently; the latter may depend on
time, radius and angle.
Bounded angular derivatives contribute $O(|x|^{-1})$ to the spatial
derivative. Thus the decay condition permits tangential anisotropy
that remains nontrivial at arbitrarily large radii.

Write $\rho_b^\mu$ for the marginal density with covariance $a$,
drift $b$ and initial law $\mu$, and
$G_\alpha^{\mathrm{init}}(\mu;a_0,A_0,q_0)$ for the initial
contribution in \eqref{eq:initial-linear-formula}.
The following lemma is proved in
Subsection~\ref{subsec:radial-local}.

\begin{lemma}
\label{lem:radial-reference}
Under Assumption~\ref{ass:radial-structure}, choose a bounded Borel
version of the weak row divergence and set
\begin{equation}
\beta=\tfrac12\operatorname{div}a,
\qquad (\operatorname{div}a)_i=\sum_{j=1}^d\partial_j a_{ij}.
\label{eq:radial-reference-drift}
\end{equation}
Then, there exists a positive continuous function $F$ such that
\(
\rho_\beta^{\mu_{\mathrm{rad}}}(t,x)=F(t,|x|).
\)
\end{lemma}

\begin{proposition}
\label{prop:model-contribution-comparison}
Under Assumption~\ref{ass:radial-structure}, for every bounded
Borel drift $b$,
\begin{equation*}
G_\alpha(I;\rho_b^{\mu_0},A,q)
 =G_\alpha(I;\rho_b^{\mu_{\mathrm{rad}}},A,q).
\end{equation*}
\end{proposition}

Proposition~\ref{prop:model-contribution-comparison},
proved in Subsection~\ref{subsec:model-comparison}, reduces the full
gain to that for $\mu_{\mathrm{rad}}$. To state the coefficient formulas
for its two phase contributions, define
\begin{equation*}
\mathsf r_0^{\mathrm{rad}}(S)
 =\inf_{\substack{x\notin S\\x\ne0}}
   \frac{e_x^{\mathsf T}(a(0,x)-A(0,x))e_x}
        {e_x^{\mathsf T}q(0,x)e_x},\,\,
\mathsf r_\infty^{\mathrm{rad}}
 =\lim_{R\to\infty}\inf_{\substack{t\in[0,T]\\|x|\ge R}}
   \frac{e_x^{\mathsf T}(a(t,x)-A(t,x))e_x}
        {e_x^{\mathsf T}q(t,x)e_x}.
\end{equation*}
These test the radial gap off the initial support at time zero and
at spatial infinity uniformly in time, respectively.
We use $\inf\varnothing=+\infty$ and $\Phi(+\infty)=0$.
See Subsection~\ref{subsec:radial-local} for the local formula below
and Subsection~\ref{subsec:radial-infinity} for the escaping formula.

\begin{proposition}
\label{prop:radial-local-phase}
Under Assumption~\ref{ass:radial-structure}, for every
bounded Borel drift $b$:
\begin{enumerate}
\item\label{item:radial-local-phase}
\(
G^{\mathrm{phase,loc}}(I;\rho_b^{\mu_{\mathrm{rad}}})
 =\Phi(\mathsf r_0^{\mathrm{rad}}(S)).
\)
\item\label{item:radial-escaping-phase}
\(
G^{\mathrm{phase,esc}}(I;\rho_b^{\mu_{\mathrm{rad}}})
 =\Phi(\mathsf r_\infty^{\mathrm{rad}}).
\)
\end{enumerate}
\end{proposition}

The proof of the following theorem is given in
Subsection~\ref{subsec:radial-infinity}.

\begin{theorem}[Radial coefficient formula and bounded-drift invariance]
\label{thm:radial-formulas}
Suppose Assumptions~\ref{ass:coefficients} and~\ref{ass:radial-structure}
hold.
For every bounded Borel reference drift $b$,
\begin{equation*}
G_\alpha(I;\rho_b^{\mu_0},A,q)
 =\max\Bigl\{G_\alpha^{\mathrm{init}}(\mu_{\mathrm{rad}};a_0,A_0,q_0),\,
              G^{\mathrm{flat}}(I),
             \Phi(\mathsf r_0^{\mathrm{rad}}(S)),\,
              \Phi(\mathsf r_\infty^{\mathrm{rad}})\Bigr\}.
\end{equation*}
In particular, for fixed $\mu_{\mathrm{rad}}$, this constant is
independent of $f$ and $b$.
\end{theorem}

\begin{eexample}
\label{ex:radial-finiteness}
Under the assumptions of Theorem~\ref{thm:radial-formulas},
fix a bounded Borel
reference drift $b$. Since $G^{\mathrm{flat}}(I)<\infty$,
Theorem~\ref{thm:radial-formulas} gives
\begin{equation*}
G_\alpha(I)<\infty
\quad\Longleftrightarrow\quad
G_\alpha^{\mathrm{init}}(\mu_{\mathrm{rad}};a_0,A_0,q_0)<\infty,\quad
\mathsf r_0^{\mathrm{rad}}(S)>0,\quad
\mathsf r_\infty^{\mathrm{rad}}>0.
\end{equation*}
For $\mu_0=\delta_0$, only the two radial inequalities remain:
$\mathsf r_0^{\mathrm{rad}}(\{0\})>0$ gives
$a(0,0)\succ A(0,0)$ by continuity, and
Proposition~\ref{prop:point-initial-models}\ref{item:gaussian-finiteness}
makes the initial contribution finite.

If $\mu_{\mathrm{rad}}$ is as in
Proposition~\ref{prop:smooth-initial-absorption}, its full support
makes the local contribution zero, and the proposition absorbs
the initial contribution. Theorem~\ref{thm:radial-formulas} therefore gives
\(
G_\alpha(I)=\max\{G^{\mathrm{flat}}(I),
                         \Phi(\mathsf r_\infty^{\mathrm{rad}})\}.
\)
Thus $G_\alpha(I)<\infty$ if and only if
$\mathsf r_\infty^{\mathrm{rad}}>0$, equivalently if and only if
there exist $\delta,R>0$ such that
\begin{equation*}
e_x^{\mathsf T}(A(t,x)+\delta q(t,x))e_x
 \le\kappa_{\mathrm{rad}}(t,|x|),
 \qquad \forall t\in[0,T],\ |x|\ge R.
\end{equation*}
\end{eexample}

\section{Nonlinear terminal problems and stability}
\label{sec:nonlinear-stability}

We treat small Hessian perturbations in $\mathcal W_\alpha(I)$
(Subsection~\ref{subsec:nonlinear-solvability}) and secant and interpolation
estimates on regular candidate classes (Subsection~\ref{subsec:tangent-stability}).
Convex Hessian level sets supply regular solutions
(Subsection~\ref{subsec:nonlinear-regularity}).

Throughout this section, $I=(0,T)$, $a,b$ satisfy Assumption~\ref{ass:coefficients}, and
$\mu_0$ is an arbitrary probability law. The spatial jet and its norm are
\begin{equation*}
J^2v=(v,Dv,D^2v),\qquad j=(z,p,\Gamma),\qquad
|j|=|z|+|p|+|\Gamma|,\qquad \Gamma\in\R_{\mathrm{sym}}^{d\times d}.
\end{equation*}
The jet does not include $v_t$. Derivatives with respect to
$\Gamma$ are identified with symmetric matrices using the
Frobenius inner product.
For a nonlinearity $\mathscr H$, write the terminal problem on $I$ as
\begin{equation}
\mathcal P(v)=-v_t-\mathscr H(t,x,J^2v),\qquad
\mathcal P(u)=f,\qquad \operatorname{Tr}_Tu=g.
\label{eq:nonlinear-terminal-equation}
\end{equation}
For fixed data $(f,g)$, define the residual size by
\begin{equation}
\mathcal E_\alpha(v)^2
 =\|\mathcal P(v)-f\|_{\alpha,I}^2
   +\|\operatorname{Tr}_Tv-g\|_{H^1(\mu_T)}^2.
\label{eq:nonlinear-residual-size}
\end{equation}
The hypotheses below ensure finiteness on the stated candidate class.
As in the weighted norms, the fixed reference density is suppressed;
it is displayed when different measures are compared.

\subsection{Nonlinear solvability and residual stability}
\label{subsec:nonlinear-solvability}

Fix $A,q$ satisfying Assumption~\ref{ass:coefficients} and write
the perturbation of the principal part as
\(
\mathcal N(t,x,j)=\mathscr H(t,x,j)-A(t,x):\Gamma.
\)
Damping allows arbitrary lower-order Lipschitz constants;
the Hessian gain sets the smallness condition.

\begin{proposition}[Nonlinear solvability and stability]
\label{prop:nonlinear-terminal-problem}
Assume $G_\alpha(I)<\infty$.
Suppose that $\mathcal N$ is Borel measurable and
$\mathscr H(\cdot,0)\in L^2_\alpha(I)$. Assume that, a.e. in
$(t,x)$ and for all jets $j=(z,p,\Gamma)$ and
$j'=(z',p',\Gamma')$,
\begin{equation}
|\mathcal N(t,x,j)-\mathcal N(t,x,j')|
 \le L_0|z-z'|+L_1|p-p'|
 +L_2|q^{1/2}(\Gamma-\Gamma')q^{1/2}|,
\label{eq:nonlinear-lipschitz}
\end{equation}
where $L_0,L_1,L_2\ge0$ and
\(
L_2G_\alpha(I)<1.
\)
For every $f\in L^2_\alpha(I)$ and $g\in H^1(\mu_{T})$,
the following hold.
\begin{enumerate}
\item\label{item:nonlinear-solvability}
There is a unique $u\in\mathcal W_\alpha(I)$ satisfying
\(
P_Au=f+\mathcal N(t,x,J^2u),
\, \operatorname{Tr}_{T}u=g.
\)
\item\label{item:nonlinear-iteration}
There is $\kappa_0>0$, independent of $f,g$, such that every
$\kappa\ge\kappa_0$ satisfies
\begin{equation*}
\vartheta_\kappa
 :=L_0G_{\alpha,\kappa}^{\mathrm{val}}(I)+L_1G_{\alpha,\kappa}^{\mathrm{grad}}(I)+L_2G_{\alpha,\kappa}^{\mathrm{fix}}(I)<1.
\end{equation*}
The choice of $\kappa_0$ may depend on the full short-interval gain
profile $h\mapsto G_{\alpha,h}(I)$ through Proposition~\ref{prop:poisson}\ref{item:linear-damping}.
For any $F_0\in L^2_\alpha(I)$, define
\begin{equation*}
u_n=\mathscr L_{I}(F_n,g),\qquad
F_{n+1}=f+\mathcal N(t,x,J^2u_n),
\end{equation*}
where $\mathscr L_{I}$ is the solver from
Theorem~\ref{thm:linear-theory}\ref{item:linear-solvability}.
Then $F_n\to F=P_Au$ geometrically in $\|\cdot\|_{\alpha,\kappa,I}$,
with contraction factor $\vartheta_\kappa$, and
$u_n\to u$ in $\mathcal W_\alpha(I)$.
\item\label{item:nonlinear-stability}
For fixed $\mathcal N$ and $\kappa \ge \kappa_0$, the solutions associated
with $(f,g)$ and $(\widetilde f,\widetilde g)$ satisfy
\begin{equation}
\begin{aligned}
\|\mathcal H_q(u-\widetilde u)\|_{\alpha,\kappa,I}
 &\le \frac{G_{\alpha,\kappa}^{\mathrm{fix}}(I)}{1-\vartheta_\kappa}
       \|f-\widetilde f\|_{\alpha,\kappa,I}
       +\frac{C}{1-\vartheta_\kappa}
        \|g-\widetilde g\|_{H^1(\mu_{T})}, \\
\|u-\widetilde u\|_{\mathcal W_\alpha(I)}
 &\le \frac{C}{1-\vartheta_\kappa}
 \left(\|f-\widetilde f\|_{\alpha,I}
       +\|g-\widetilde g\|_{H^1(\mu_{T})}\right).
\end{aligned}
\label{eq:nonlinear-data-stability}
\end{equation}
In particular, every $v\in\mathcal W_\alpha(I)$ satisfies
\begin{equation}
\|v-u\|_{\mathcal W_\alpha(I)}
 \le \frac{C}{1-\vartheta_\kappa}\mathcal E_\alpha(v).
\label{eq:natural-nonlinear-residual}
\end{equation}
Here $C$ depends only on $\mathfrak R$,
$\alpha,\kappa,G_\alpha^{\mathrm{fix}}(I)$, $\mathfrak E$, and $L_0,L_1,L_2$.
\end{enumerate}
\end{proposition}

\begin{proof}
The bounds in Proposition~\ref{prop:poisson}\ref{item:linear-damping}--\ref{item:linear-lower-order}
give $\vartheta_\kappa\to L_2G_\alpha(I)<1$.
Choose $\kappa_0>0$,  so that
$\vartheta_\kappa<1$ for all $\kappa\ge\kappa_0$, and fix such a $\kappa$.
Theorem~\ref{thm:linear-theory}\ref{item:linear-solvability}
supplies the bounded solver $\mathscr L_{I}$.
The zero-jet condition and \eqref{eq:nonlinear-lipschitz}
give $\mathcal N(\cdot,J^2u)\in L^2_\alpha(I)$ for every
$u\in\mathcal W_\alpha(I)$. Thus the map
$\mathcal T(F)=f+\mathcal N(\cdot,J^2\mathscr L_{I}(F,g))$
preserves $L^2_\alpha(I)$.
The difference
$v=\mathscr L_{I}(F,g)-\mathscr L_{I}(\widetilde F,g)=S_{I}(F-\widetilde F)$
has zero terminal value and satisfies $P_Av=F-\widetilde F$.
The gain definitions in \eqref{eq:damped-constants} and
\eqref{eq:nonlinear-lipschitz} give
\begin{equation}
\|\mathcal T(F)-\mathcal T(\widetilde F)\|_{\alpha,\kappa,I}
 \le\vartheta_\kappa\|F-\widetilde F\|_{\alpha,\kappa,I}.
\label{eq:nonlinear-source-contraction}
\end{equation}
Fixed points of $\mathcal T$ correspond exactly to solutions through
$F=P_Au$ and $u=\mathscr L_{I}(F,g)$.
Banach's theorem therefore gives existence, uniqueness and geometric
convergence of the source iteration in part~\ref{item:nonlinear-iteration}.
Theorem~\ref{thm:linear-theory}\ref{item:linear-solvability} and equivalence
of the damped and undamped norms on $I$ also give
$u_n\to u$ in $\mathcal W_\alpha(I)$.

For solutions associated with two data sets, put $F=P_Au$,
$\widetilde F=P_A\widetilde u$ and
$w=\mathscr L_{I}(0,g-\widetilde g)$.
Linearity and Theorem~\ref{thm:linear-theory}\ref{item:linear-solvability} give
$u-\widetilde u=S_{I}(F-\widetilde F)+w$ and
$\|w\|_{\mathcal W_\alpha(I)}\le C\|g-\widetilde g\|_{H^1(\mu_{T})}$.
Since $\widetilde u+w=\mathscr L_{I}(\widetilde F,g)$,
\eqref{eq:nonlinear-lipschitz} and \eqref{eq:nonlinear-source-contraction} yield
\begin{equation}
(1-\vartheta_\kappa)\|F-\widetilde F\|_{\alpha,\kappa,I}
 \le\|f-\widetilde f\|_{\alpha,\kappa,I}
       +C\|g-\widetilde g\|_{H^1(\mu_{T})}.
\label{eq:nonlinear-source-stability}
\end{equation}
Applying \eqref{eq:damped-constants} to $S_{I}(F-\widetilde F)$ and
using the bound on $w$ gives
\begin{equation*}
\|\mathcal H_q(u-\widetilde u)\|_{\alpha,\kappa,I}
 \le G_{\alpha,\kappa}^{\mathrm{fix}}(I)
       \|F-\widetilde F\|_{\alpha,\kappa,I}
       +C\|g-\widetilde g\|_{H^1(\mu_{T})}.
\end{equation*}
Together with \eqref{eq:nonlinear-source-stability}, this proves the first
inequality in \eqref{eq:nonlinear-data-stability}.
The second follows from Theorem~\ref{thm:linear-theory}\ref{item:linear-solvability},
\eqref{eq:nonlinear-source-stability} and equivalence of the source norms.
For $v\in\mathcal W_\alpha(I)$, the jet Lipschitz bound and
Theorem~\ref{thm:linear-theory}\ref{item:linear-trace} ensure that
$\mathcal P(v)\in L^2_\alpha(I)$ and
$\operatorname{Tr}_{T}v\in H^1(\mu_{T})$.
Applying the second inequality to $(\mathcal P(v),\operatorname{Tr}_{T}v)$
and $(f,g)$ gives \eqref{eq:natural-nonlinear-residual} by Cauchy--Schwarz.
\end{proof}

The structural condition $L_2G_\alpha(I)<1$
is essential to this argument.
For the matching choice $A=a/2$, $q=a$,
Theorem~\ref{thm:matching-arbitrary-laws} gives $L_2<1/2$ as a
sufficient condition for every initial law when $\alpha\ge1/2$,
and for the smooth initial laws and their finite mixtures specified
there for every $\alpha$.
For a point start at $\alpha=0$,
$G_0(I)=2\sqrt2$, so the condition is $L_2<1/(2\sqrt2)$.
These gains determine the perturbation threshold, but not the optimal
graph-norm constant in \eqref{eq:natural-nonlinear-residual}.

\begin{eexample}
\label{ex:borel-perturbations}
Let $\widetilde A$ be a bounded Borel symmetric uniformly elliptic
matrix field on $I\times\R^d$. Set
\(
\delta=\left\|q^{-1/2}(\widetilde A-A)q^{-1/2}\right\|_{L^\infty(I\times\R^d)}.
\)
If $G_\alpha(I)<\infty$ and $\delta G_\alpha(I)<1$,
Proposition~\ref{prop:nonlinear-terminal-problem}, with
$\mathcal N(t,x,j)=(\widetilde A-A):\Gamma$, gives a unique solution
$u\in\mathcal W_\alpha(I)$ to
$P_{\widetilde A}u=f$, $\operatorname{Tr}_Tu=g$
for every $f\in L^2_\alpha(I)$ and $g\in H^1(\mu_T)$.
Absorbing the principal-part perturbation in the short-interval estimate also gives
\begin{equation*}
G_\alpha(I;\rho,\widetilde A,q)
 \le\frac{G_\alpha(I)}{1-\delta G_\alpha(I)}.
\end{equation*}
In particular, take $a=q=I_d$, $b=0$ and $\mu_0=\delta_0$,
so that $\rho(t,x)=\gamma_{tI_d}(x)$.
The matching base $A=I_d/2$ has $G_\alpha(I)=C_\alpha^{\mathrm{pt}}$
by Theorem~\ref{thm:matching-arbitrary-laws}\ref{item:matching-full-interval}.
The sufficient condition is therefore the Frobenius bound
\begin{equation*}
\bigl\|\,|\widetilde A-\tfrac12I_d|_{\mathrm F}\,\bigr\|_{L^\infty(I\times\R^d)}
 <\frac1{C_\alpha^{\mathrm{pt}}}\le\frac12.
\end{equation*}
On the other hand, for $\widetilde A$ satisfying
Assumption~\ref{ass:coefficients},
Corollary~\ref{cor:uniform-point-bound}
gives
\begin{equation*}
\bigl\|\,\|\widetilde A-\tfrac12I_d\|_{\mathrm{op}}\,\bigr\|_{L^\infty(\overline I\times\R^d)}<\tfrac12
\quad\Longrightarrow\quad
G_\alpha(I;\rho,\widetilde A,I_d)<\infty.
\end{equation*}
The Frobenius norm can be $\sqrt d$ times the operator norm, so
the continuous-coefficient criterion extends beyond the matching
perturbation bound. The next example shows that the operator-norm
condition does not suffice for Borel matrices.
\end{eexample}

\begin{eexample}
\label{ex:rough-principal}
Let $d=3$, $a=q=I_3$, $b=0$ and $\mu_0=\delta_0$.
Set
\begin{equation*}
A_{\mathrm{rough}}(x)=\frac16I_3+
 \frac12\frac{(x-x_{\rm s})\otimes(x-x_{\rm s})}{|x-x_{\rm s}|^2}
 ,\,x\ne x_{\rm s},\qquad A_{\mathrm{rough}}(x_{\rm s})=\frac13I_3,
\end{equation*}
where $x_{\rm s} \neq 0$ is fixed.
This time-independent Borel matrix is smooth away from $x_{\rm s}$,
has no limit at $x_{\rm s}$, and satisfies
$\frac16I_3\preceq A_{\mathrm{rough}}\preceq\frac23I_3$ and
\begin{equation*}
\bigl\|\,|A_{\mathrm{rough}}-I_3/2|_{\mathrm F}\,\bigr\|_{L^\infty(\R^3)}
 =\frac12,\qquad
\bigl\|\,\|A_{\mathrm{rough}}-I_3/2\|_{\mathrm{op}}\,\bigr\|_{L^\infty(\R^3)}
 =\frac13.
\end{equation*}
For every $J\in\mathcal J(I)$,
\begin{equation}
G_\alpha^{\mathrm{fix}}(J;\rho,A_{\mathrm{rough}},I_3)
 =G_\alpha(I;\rho,A_{\mathrm{rough}},I_3)=+\infty.
\label{eq:rough-infinite-gain}
\end{equation}
The three frozen contributions for $A=A_{\mathrm{rough}}$,
with whole-line candidates defined by \eqref{eq:whole-line-convergence},
satisfy
\begin{equation}
G_\alpha^{\mathrm{mod}}(I;A_{\mathrm{rough}})
 =\max\{G_\alpha^{\mathrm{init}}(I),G^{\mathrm{flat}}(I),
         G^{\mathrm{phase}}(I)\}<\infty.
\label{eq:rough-frozen-candidate}
\end{equation}
The proof is in Appendix~\ref{app:rough-principal}.
\end{eexample}

Example~\ref{ex:rough-principal} shows that the strict condition in
Example~\ref{ex:borel-perturbations} cannot be replaced by
$\delta G_\alpha(I)\le1$: for $\alpha\ge1/2$, the matching base
$A=I_3/2$ has $G_\alpha(I)=C_\alpha^{\mathrm{pt}}=2$,
whereas $\delta=1/2$.
At the singular point $x_{\rm s}$, the rescaled matrix
$A_{\mathrm{rough}}(x_{\rm s}+\varepsilon z)$ retains its dependence
on the direction of $z$ as $\varepsilon\downarrow0$, even though
the center value is fixed. Thus convergence at the centers in
\eqref{eq:whole-line-convergence} does not imply convergence of
the rescaled coefficients to a constant matrix.

\subsection{Stability on regular candidate classes}
\label{subsec:tangent-stability}

For the nonlinear stability results in this subsection, fix
$0<\gamma<1$, $B>0$ and a nonempty convex candidate class satisfying
\begin{equation}
\mathcal V_B\subseteq
\{v\in C_b^{1+\gamma/2,2+\gamma}(\overline I\times\R^d):
          \|v\|_{C_b^{1+\gamma/2,2+\gamma}}\le B\}.
\label{eq:regular-candidate-class}
\end{equation}
We use a standard parabolic H\"older norm controlling each of
$v,Dv,D^2v,v_t$ in the supremum norm, so $|J^2v|\le3B$.

Assume that $\mathscr H$ is Borel in $(t,x)$ and $C^1$ in $j$
on a neighborhood of this jet ball, with jointly Borel derivatives.
For all $P,Q\in\overline I\times\R^d$ and $|j|,|\widetilde j|\le3B$, require
\begin{gather}
|\mathscr H_z|\le L_0,\qquad |\mathscr H_p|\le L_1,\qquad
\lambda_A I_d\preceq\mathscr H_\Gamma\preceq\Lambda_A I_d,
\label{eq:nonlinear-derivative-bounds}\\
|\mathscr H_\Gamma(P,j)-\mathscr H_\Gamma(Q,\widetilde j)|
 \le\omega\bigl(d_p(P,Q)+|j-\widetilde j|\bigr),
\label{eq:nonlinear-principal-modulus}
\end{gather}
where $0<\lambda_A\le\Lambda_A<\infty$ and $\omega$ is a
nondecreasing continuous modulus with $\omega(0)=0$.
Let $\mathcal K$ be either empty or
$[\tau,T]\times\overline B_R$, with $0<\tau<T$ and $R>0$.
For a fixed $\chi>0$, assume the relative upper bound
\begin{equation}
\mathscr H_\Gamma\bigl(P,J^2v(P)\bigr)\preceq\chi a(P),\qquad
P\in(\overline I\times\R^d)\setminus\mathcal K,\qquad v\in\mathcal V_B.
\label{eq:nonlinear-relative-upper-bound}
\end{equation}
Here $\mathcal K$ and $\chi$ are common to the candidate class.
This bound is imposed only outside $\mathcal K$ and along candidate
jets; the ellipticity and continuity conditions in
\eqref{eq:nonlinear-derivative-bounds}--\eqref{eq:nonlinear-principal-modulus}
hold on the whole domain and jet ball.

For $v,w\in\mathcal V_B$, define the tangent and secant operators by
\begin{equation*}
\mathcal P_v\phi
 =-\phi_t-\mathscr H_\Gamma(\cdot,J^2v):D^2\phi
       -\mathscr H_p(\cdot,J^2v)\cdot D\phi
       -\mathscr H_z(\cdot,J^2v)\phi,
\end{equation*}
and $\mathcal P_{v,w}\phi
 =\int_0^1\mathcal P_{w+s(v-w)}\phi\,\mathrm{d}s.$
The integral averages the coefficients, so
$\mathcal P_{v,v}=\mathcal P_v$. The fundamental theorem of calculus gives
\begin{equation}
\mathcal P(v)-\mathcal P(w)=\mathcal P_{v,w}(v-w).
\label{eq:nonlinear-secant-identity}
\end{equation}
If $u\in\mathcal V_B$ solves \eqref{eq:nonlinear-terminal-equation},
this identity and the derivative bounds make $\mathcal P(v)-f$
bounded for every $v\in\mathcal V_B$.
Also $\operatorname{Tr}_Tv-g\in W^{1,\infty}(\R^d)$, so
$\mathcal E_\alpha(v)$ is finite for every reference probability law.

\begin{theorem}[Residual stability under a spectral gap]
\label{thm:tangent-secant-stability}
Assume \eqref{eq:regular-candidate-class}--\eqref{eq:nonlinear-relative-upper-bound}
with $\chi<1$.
If $u\in\mathcal V_B$ solves \eqref{eq:nonlinear-terminal-equation},
then
\begin{equation*}
\|v-u\|_{\mathcal W_\alpha(I)}
 \le C\mathcal E_\alpha(v),\qquad v\in\mathcal V_B.
\end{equation*}
The constant depends only on
$\mathfrak R,\alpha,\gamma,B,\lambda_A,\Lambda_A,L_0,L_1,\omega,\chi$
when $\mathcal K=\varnothing$. For a nonempty exception it may
also depend on $\mu_0,\tau,R$.
\end{theorem}

\begin{proof}
The principal matrix of $\mathcal P_{v,w}$ is
\begin{equation*}
A_{v,w}(P)=\int_0^1\mathscr H_\Gamma
       \bigl(P,J^2(w+s(v-w))(P)\bigr)\,\mathrm{d}s.
\end{equation*}
By convexity, every interpolant $w+s(v-w)$ belongs to
$\mathcal V_B$.
Its jet has a common $\gamma$-H\"older modulus. Thus
\eqref{eq:nonlinear-principal-modulus} gives, for $d_p(P,Q)\le1$,
\begin{equation*}
|A_{v,w}(P)-A_{v,w}(Q)|
 \le\omega\bigl(d_p(P,Q)+CBd_p(P,Q)^\gamma\bigr).
\end{equation*}
Averaging preserves ellipticity and the spectral bound
\eqref{eq:nonlinear-relative-upper-bound} along these interpolants;
the gradient and value coefficients remain bounded by $L_1$ and $L_0$, respectively.

Take $w=u$, $A=A_{v,u}$ and $q=a$ in the linear estimates.
Corollary~\ref{cor:uniform-point-bound}\ref{item:quantitative-gap-full},
under \eqref{eq:compact-exception-spectral-gap} with $\delta=1-\chi$,
gives common $h_0>0$ and $K<\infty$ such that
$G_{\alpha,h_0}(I)\le K$ for every $v\in\mathcal V_B$.
Theorem~\ref{thm:linear-theory}\ref{item:linear-solvability}
therefore gives a common graph-norm bound for $P_A$ and a common
bound for $G_\alpha^{\mathrm{fix}}(I)$.
By Proposition~\ref{prop:poisson}\ref{item:linear-lower-order},
one can choose $\kappa$ uniformly so that
$L_0G_{\alpha,\kappa}^{\mathrm{val}}(I)
 +L_1G_{\alpha,\kappa}^{\mathrm{grad}}(I)\le1/2$.

Set $e=v-u$, $r=\mathcal P(v)-f$ and
$h=\operatorname{Tr}_Tv-g$.
The secant identity \eqref{eq:nonlinear-secant-identity} gives $\mathcal P_{v,u}e=r$.
By Theorem~\ref{thm:linear-theory}\ref{item:linear-trace},
$z=e-U(\cdot,T)h$ belongs to $\mathcal W_{\alpha,0}(I)$ and
$\|U(\cdot,T)h\|_{\mathcal W_\alpha(I)}\le C\|h\|_{H^1(\mu_T)}$.
The coefficient bounds and the choice of $\kappa$ yield
\begin{equation*}
\begin{aligned}
\|P_Az\|_{\alpha,\kappa,I}
 &\le\|\mathcal P_{v,u}z\|_{\alpha,\kappa,I}
       +L_1\|Dz\|_{\alpha,\kappa,I}+L_0\|z\|_{\alpha,\kappa,I}\\
 &\le\|r\|_{\alpha,\kappa,I}+C\|h\|_{H^1(\mu_T)}
       +\tfrac12\|P_Az\|_{\alpha,\kappa,I}.
\end{aligned}
\end{equation*}
Absorbing the last term, removing the exponential weight and using
the principal graph-norm bound, followed by the extension bound, gives
\begin{equation*}
\|e\|_{\mathcal W_\alpha(I)}
 \le C\bigl(\|r\|_{\alpha,I}+\|h\|_{H^1(\mu_T)}\bigr)
 \le C\mathcal E_\alpha(v).
\end{equation*}
The constants are independent of $\mu_0$ when $\mathcal K=\varnothing$;
otherwise the finite-scale threshold may also depend on $\mu_0,\tau,R$.
\end{proof}

To obtain an interpolation estimate for arbitrary $\chi>0$, retain
the reference drift and initial law and replace $a$ by
$\eta a$, where $\eta>1$. Denote the resulting marginal density
and law by $\rho^\eta,\mu_t^\eta$, with $\rho^1=\rho$.
This auxiliary covariance will be used only in the proof; the final
error and residual are measured under $\rho$.
The following lemma is
proved in Subsection~\ref{subsec:action-bridges}.

\begin{lemma}
\label{lem:covariance-interpolation}
For $\eta>1$ and $0<\theta<\eta^{-1}$,
\begin{equation}
\rho(t,x)\le C\rho^\eta(t,x),\qquad 0<t\le T,
\label{eq:enlarged-covariance-domination}
\end{equation}
and every $F\in L^2(I,t^\alpha \mathrm{d}t;L^\infty(\R^d))$ and
$h\in W^{1,\infty}(\R^d)$ satisfy
\begin{equation}
\begin{aligned}
\|F\|_{\alpha,I;\rho^\eta}
 \le C\|F\|_{\alpha,I;\rho}^{\theta}
       \|F\|_{L^2(I,t^\alpha \mathrm{d}t;L^\infty)}^{1-\theta},\qquad
\|h\|_{H^1(\mu_T^\eta)}
 \le C\|h\|_{H^1(\mu_T)}^\theta
       \|h\|_{W^{1,\infty}}^{1-\theta}.
\end{aligned}
\label{eq:covariance-interpolation}
\end{equation}
Here $C$ depends only on $\mathfrak R,\eta,\theta$,
uniformly over $\mu_0$ and $\alpha$.
\end{lemma}

\begin{theorem}[Residual stability with an interpolation exponent]
\label{thm:nonlinear-interpolation}
Assume \eqref{eq:regular-candidate-class}--\eqref{eq:nonlinear-relative-upper-bound}
with any $\chi>0$.
Let $u\in\mathcal V_B$ solve \eqref{eq:nonlinear-terminal-equation}.
For every $0<\theta<\min\{1,\chi^{-1}\}$,
\begin{equation*}
\|v-u\|_{\mathcal W_\alpha(I)}
 \le C\mathcal E_\alpha(v)^\theta,\qquad v\in\mathcal V_B.
\end{equation*}
The constant has the same dependencies as in
Theorem~\ref{thm:tangent-secant-stability}, with the additional
dependence on $\theta$.
\end{theorem}

\begin{proof}
Choose $\max\{1,\chi\}<\eta<\theta^{-1}$ depending only on
$\chi,\theta$. Under covariance $\eta a$, \eqref{eq:nonlinear-relative-upper-bound} holds with upper ratio
$\chi/\eta<1$.
Set
$e=v-u$, $r=\mathcal P(v)-\mathcal P(u)$ and
$h=\operatorname{Tr}_Tv-g$.
Apply Theorem~\ref{thm:tangent-secant-stability} with reference covariance
$\eta a$. Equation~\eqref{eq:enlarged-covariance-domination}
and then \eqref{eq:covariance-interpolation} give
\begin{equation*}
\begin{aligned}
\|e\|_{\mathcal W_\alpha(I;\rho)}
 &\le C\|e\|_{\mathcal W_\alpha(I;\rho^\eta)}
 \le C\bigl(\|r\|_{\alpha,I;\rho^\eta}
                        +\|h\|_{H^1(\mu_T^\eta)}\bigr)\\
 &\le C\bigl(\|r\|_{\alpha,I;\rho}
                        +\|h\|_{H^1(\mu_T)}\bigr)^\theta
 \bigl(\|r\|_{L^2(I,t^\alpha \mathrm{d}t;L^\infty)}
                        +\|h\|_{W^{1,\infty}}\bigr)^{1-\theta}.
\end{aligned}
\end{equation*}
The derivative bounds in \eqref{eq:nonlinear-derivative-bounds} yield
\begin{equation*}
|r|\le |e_t|+L_0|e|+L_1|De|+\sqrt d\,\Lambda_A|D^2e|.
\end{equation*}
The candidate budget bounds the factor with exponent $1-\theta$;
\eqref{eq:nonlinear-residual-size} then gives the estimate.
The dependence on $\mu_0,\tau,R$ comes only from the
application of
Theorem~\ref{thm:tangent-secant-stability}
when $\mathcal K\ne\varnothing$.
\end{proof}

The bound in \eqref{eq:nonlinear-derivative-bounds} gives
$\mathscr H_\Gamma\preceq\Lambda_A I_d
\preceq(\Lambda_A/\lambda_a)a$, hence the global choice
$\chi=\Lambda_A/\lambda_a$. A smaller ratio outside $\mathcal K$
can improve the exponent range. For $\chi<1$,
Theorem~\ref{thm:tangent-secant-stability} gives a linear estimate;
the next example shows its failure at the boundary.

\begin{eexample}
\label{ex:nonlinear-gap-necessity}
Let $a,A\succ0$ be constant, $b=0$, $\mu_0=\delta_0$ and
$\alpha=0$. Take $\mathscr H(P,j)=A:\Gamma$ and $f=g=0$,
so $u=0$. With $q=a$, the Hessian gain is
$G_0(I)=\mathfrak G_0(\delta_0,a,A,a)$.
Proposition~\ref{prop:point-initial-models}\ref{item:gaussian-finiteness}
and Theorem~\ref{thm:linear-theory} therefore show that
\begin{equation*}
\|v\|_{\mathcal W_0(I)}\le C\|P_Av\|_{0,I},
\qquad v\in\mathcal W_{0,0}(I),
\end{equation*}
holds if and only if $A\prec a$.
Necessity also holds on the full ball in
\eqref{eq:regular-candidate-class}: when $A\not\prec a$, the
proof of Proposition~\ref{prop:point-initial-models}\ref{item:gaussian-finiteness},
through \eqref{eq:recovery-norm-limits}, produces tests in
$C_c^\infty(I\times\R^d)$ with unbounded quotients.
Rescaling their amplitudes places them in that ball without changing
the quotients.
\end{eexample}

\subsection{Equations with convex Hessian level sets}
\label{subsec:nonlinear-regularity}

Theorems~\ref{thm:tangent-secant-stability} and
\ref{thm:nonlinear-interpolation} assume that the solution belongs
to $\mathcal V_B$. The next proposition supplies a solution and
its regularity budget from the derivative bounds in \eqref{eq:nonlinear-derivative-bounds},
imposed on all jets, together with H\"older regularity
and convexity of the Hessian level sets.

\begin{proposition}
\label{prop:quasiconcave-terminal-problem}
Fix $0<\sigma<1$.
Assume that $\mathscr H$ is $C^1$ in $j$ and satisfies
\eqref{eq:nonlinear-derivative-bounds} for all jets.
For some $K\ge0$, assume also that
\begin{equation}
|\mathscr H(P,j)-\mathscr H(Q,j)|
 \le K(1+|j|)d_p(P,Q)^\sigma,\qquad d_p(P,Q)\le1,
\label{eq:nonlinear-regularity-structure}
\end{equation}
for all $P,Q\in\overline I\times\R^d$ and all jets $j$.
Assume that every Hessian superlevel set
\begin{equation*}
\{\Gamma\in\R_{\mathrm{sym}}^{d\times d}:
                 \mathscr H(P,z,p,\Gamma)\ge c\}
\end{equation*}
is convex, for all $P,z,p,c$. For the data, assume
\begin{equation*}
f+\mathscr H(\cdot,0)\in C_b^{\sigma/2,\sigma}(\overline I\times\R^d),
\qquad g\in C_b^{2+\sigma}(\R^d).
\end{equation*}
Then \eqref{eq:nonlinear-terminal-equation} has a unique bounded
continuous viscosity solution $u$.
There is $\gamma_*\in(0,1)$, depending only on
$d,\lambda_A,\Lambda_A$, such that for every
$0<\gamma<\min\{\sigma,\gamma_*\}$,
\begin{equation*}
u\in C_b^{1+\gamma/2,2+\gamma}(\overline I\times\R^d),\qquad
\|u\|_{C_b^{1+\gamma/2,2+\gamma}}\le C.
\end{equation*}
The bound depends only on
$d,T,\lambda_A,\Lambda_A,L_0,L_1,K,\sigma,\gamma$ and the
two data norms above. It is independent of the reference diffusion
and its initial law. In particular, $u$ belongs to
$\mathcal W_\alpha(I)$ for every allowed reference probability
law and every $\alpha$.
\end{proposition}

The common convex-sublevel orientation reduces to this one by
replacing $(u,f,g,\mathscr H(P,j))$ with
$(-u,-f,-g,-\mathscr H(P,-j))$. Thus the usual concave and
convex Hessian nonlinearities are both included.

\begin{proof}
Use forward time $s=T-t$, put $w(s,x)=u(T-s,x)$, and set
$F(s,x,j)=\mathscr H(T-s,x,j)+f(T-s,x)$.
The terminal problem becomes $w_s=F(s,x,J^2w)$, $w(0)=g$.
The derivative bounds in \eqref{eq:nonlinear-derivative-bounds}
give uniform ellipticity and the lower-order Lipschitz constants
$L_0,L_1$ for $F$. Its Hessian superlevel sets remain convex.
Moreover, $F(\cdot,0)$ is the time reversal of
$f+\mathscr H(\cdot,0)$. Applying
\eqref{eq:nonlinear-regularity-structure} to
$\mathscr H(P,j)-\mathscr H(P,0)$ gives
\begin{equation*}
|F(P,j)-F(Q,j)|\le C(1+|j|)d_p(P,Q)^\sigma,
\qquad d_p(P,Q)\le1,
\end{equation*}
with $C\le2K+\|F(\cdot,0)\|_{C_b^{\sigma/2,\sigma}}$.
In particular, $F$ is jointly continuous, and its source norm
and structural constants are controlled by the stated data.

\paragraph*{Existence on the whole space.}
The change $\mathrm{e}^{-\kappa s}w$, with $\kappa>L_0$, makes the
equation proper in the value variable. The bounded-domain theorem
\citep[Theorem~2.1]{crandall-kocan-lions-swiech-1999} therefore
gives, after undoing this change, a continuous viscosity solution
$w_R$ on $[0,T]\times B_R$ with parabolic boundary data $g$.
For $d=1$, the same proof uses the local endpoint barrier
$1-\mathrm{e}^{-k\,\operatorname{dist}(x,\partial B_R)}$,
$k>L_1/\lambda_A$, in place of the exterior-cone barrier.

Let $h'=\|F(\cdot,J^2g)\|_\infty+L_0h$, $h(0)=0$.
The smooth barriers $g\pm h$ give $|w_R-g|\le h$ uniformly
in $R$. Together with parabolic interior H\"older compactness,
they give a locally uniform diagonal limit as $R\to\infty$,
including at $s=0$. Viscosity stability
\citep[Section~6]{crandall-ishii-lions-1992} yields a bounded
continuous solution. Comparison with spatially constant ordinary
differential equation (ODE) barriers
also gives, for every bounded solution,
\begin{equation}
\|w\|_\infty
 \le \mathrm{e}^{L_0T}\bigl(\|g\|_\infty+T\|F(\cdot,0)\|_\infty\bigr).
\label{eq:nonlinear-sup-bound}
\end{equation}

\paragraph*{Interior regularity.}
For a uniformly elliptic frozen operator $G(\Gamma)$ with convex
superlevel sets, \citep[Theorem~7.1]{goffi-2024} gives an interior
$C^{1+\gamma_*/2,2+\gamma_*}$ estimate controlled by the supremum
norm and $|G(0)|$. The exponent $\gamma_*$ depends only on
$d,\lambda_A,\Lambda_A$. Fix
$0<\gamma<\min\{\sigma,\gamma_*\}$ and write
$Q_r^-=(-r^2,0]\times B_r$.
Divide $w$ by
$2(1+\|w\|_\infty+\|F(\cdot,0)\|_{C_b^{\sigma/2,\sigma}})$
and rescale an interior cylinder of radius $r_0$ to $Q_1^-$.
The normalized equation $U_s=\Phi(P,J^2U)$ has
$\|U\|_\infty\le1/2$, the same ellipticity bounds and convex
Hessian superlevel sets. Its lower-order constants are
$r_0^2L_0,r_0L_1$, its zero-jet value at the center is
$O(r_0^2)$, and its time-space H\"older coefficient is at most
$C(K+\|F(\cdot,0)\|_{C_b^{\sigma/2,\sigma}})r_0^\sigma$.
For $\varepsilon>0$ to be fixed below, choose $r_0$
so that all these perturbation terms are at most $\varepsilon$.
For $\varepsilon\le1/2$, the polynomial $\pi_0(s,x)=\Phi(0,0)s$
starts the iteration at scale one, since
$\|U-\pi_0\|_{L^\infty(Q_1^-)}\le1/2+\varepsilon\le1$.

For the rescaling step, let $\pi$ be a polynomial of parabolic
degree at most two with coefficients in a fixed bounded set. Suppose that,
at a scale $0<r\le1$,
\begin{equation*}
\pi_s=\Phi(0,J^2\pi(0)),\qquad
\|U-\pi\|_{L^\infty(Q_r^-)}\le r^{2+\gamma}.
\end{equation*}
The normalized error $V_r(s,x)=r^{-2-\gamma}(U-\pi)(r^2s,rx)$
has supremum norm at most one. Write its equation as
$(V_r)_s=\Phi_r(P,J^2V_r)$ and set
\begin{equation*}
G_r(\Gamma)=r^{-\gamma}
 \bigl[\Phi(0,J^2\pi(0)+(0,0,r^\gamma\Gamma))-\pi_s\bigr].
\end{equation*}
The center condition gives $G_r(0)=0$. Hessian translation,
positive rescaling and constant shifts preserve ellipticity and
convex superlevel sets, so the frozen estimate for $G_r$ is
uniform in $r$ and $\pi$. The structural bounds give
\begin{equation*}
|\Phi_r(P,j)-G_r(\Gamma)|
 \le C\varepsilon
 \bigl(r^{\sigma-\gamma}+r^{1-\gamma}
       +r^\sigma|\Gamma|+r|p|+r^2|z|\bigr),
 \qquad P\in Q_1^-.
\end{equation*}
The first two terms come from the time-space oscillation and
the variation of the lower-order jets of $\pi$.
Here $C$ may depend on the fixed coefficient bound,
but is independent of $r$.

Since $\gamma<\sigma<1$, the perturbation is uniformly small
on bounded jet sets at all scales. The lower-order Lipschitz
constants of $\Phi_r$ are at most $\varepsilon r^2$ and
$\varepsilon r$, so both lower-order variables disappear in
any locally uniform limit as $\varepsilon\downarrow0$.
The compactness argument in
\citep[proof of Lemma~1.4]{wang-1992-parabolic-ii}
therefore extends to these operators by viscosity stability.
Together with the frozen estimate, it gives a fixed $\eta\in(0,1/4)$
and, for sufficiently small $\varepsilon$, a polynomial $\widehat\pi$
with uniformly bounded coefficients such that
\begin{equation*}
\widehat\pi_s=G_r(D^2\widehat\pi),\qquad
\|V_r-\widehat\pi\|_{L^\infty(Q_\eta^-)}
 \le\tfrac12\eta^{2+\gamma}.
\end{equation*}
Adjusting its time coefficient to $\Phi_r(0,J^2\widehat\pi(0))$
requires a correction of size $O(\varepsilon r)$, by the lower-order
Lipschitz bounds. This leaves the spatial jet at zero unchanged and
enforces the center equation. For small $\varepsilon$, the corrected
approximation error is at most $\eta^{2+\gamma}$.
Rescaling back gives the next polynomial at scale $\eta r$, with
coefficient increments $O(r^\gamma)$. Their geometric sum fixes a
common coefficient bound before $\varepsilon$ is chosen.
The iteration and convergence in
\citep[proofs of Lemma~1.2 and Theorem~1.1]{wang-1992-parabolic-ii},
applied at each center in $Q_{1/2}^-$, give
\begin{equation}
\|U\|_{C^{1+\gamma/2,2+\gamma}(Q_{1/2}^-)}\le C.
\label{eq:nonlinear-interior-estimate}
\end{equation}
Restoring scales gives a local estimate depending only on the
stated data and the distance from $s=0$, uniformly in the
spatial center.

\paragraph*{Regularity at the terminal face.}
To apply the interior estimate at $s=0$, we subtract an
extension of $g$ that also satisfies the equation there. Set
\begin{equation*}
c(x)=F(0,x,J^2g(x)),\qquad
\psi(s,x)=g(x)+s(\mathrm{e}^{s\Delta}c)(x).
\end{equation*}
The assumptions give $c\in C_b^\sigma$. Heat-kernel estimates yield
$\psi\in C_b^{1+\sigma/2,2+\sigma}$, with norm controlled by
the data, $\psi(0)=g$ and $\psi_s(0)=c$.
Consequently $v=w-\psi$ solves $v_s=\widetilde F(s,x,J^2v)$,
where
\begin{equation*}
\widetilde F(s,x,j)=F(s,x,j+J^2\psi(s,x))-\psi_s(s,x),
\qquad \widetilde F(0,x,0)=0.
\end{equation*}
Extend $v=0$ and
$\widetilde F(s,x,j)=\widetilde F(0,x,j)$ for $-1\le s<0$.
The extended operator retains the structural assumptions and the
Hessian level-set condition, with constants controlled by the data.
At $(0,x)$, an upper test $\phi$ with $\phi(0,x)=0$ satisfies
$\phi_s\le0$, $D\phi=0$ and $D^2\phi\succeq0$, by the zero
extension. Ellipticity therefore gives
\begin{equation*}
\phi_s-\widetilde F(0,x,\phi,D\phi,D^2\phi)
 \le-\widetilde F(0,x,0)=0.
\end{equation*}
Lower tests give the supersolution inequality. Thus the extension
is a viscosity solution across $s=0$.
Applying \eqref{eq:nonlinear-interior-estimate}, after normalization,
on fixed-radius cylinders in $(-1,T]\times\R^d$ now gives
the regularity bound on $[0,T]\times\R^d$.
Its constant is controlled by the data through
\eqref{eq:nonlinear-sup-bound} and the bound for $\psi$.

The estimates apply to every bounded continuous viscosity
solution. Such solutions are now classical, and their difference
satisfies a linear uniformly parabolic equation with gradient
and value coefficients bounded by $L_1$ and $L_0$.
For two such solutions $w_1,w_2$ and $\kappa>L_0$,
apply the maximum principle to
$\mathrm{e}^{-\kappa s}(w_1-w_2)$, using the penalty
$\delta \mathrm{e}^{Cs}(1+|x|^2)^{1/2}$ to control spatial infinity.
Letting $\delta\downarrow0$ gives $w_1\le w_2$;
interchanging the solutions proves uniqueness.
Finally, boundedness of $u,u_t,Du,D^2u$ and $g,Dg$, together with
$\mu_t(\R^d)=1$ and $\int_0^Tt^\alpha\,\mathrm{d}t<\infty$, gives
the graph-space membership and terminal trace.
\end{proof}

Choose an allowed $\gamma$ and take $B$ at least as large as the
bound in Proposition~\ref{prop:quasiconcave-terminal-problem}.
The full ball in \eqref{eq:regular-candidate-class} then contains
$u$; for a smaller convex class, verify $u\in\mathcal V_B$
separately. It remains to verify \eqref{eq:nonlinear-principal-modulus} on
$|j|,|\widetilde j|\le3B$ and
\eqref{eq:nonlinear-relative-upper-bound} along the chosen candidate jets.
Then Theorem~\ref{thm:tangent-secant-stability} applies when $\chi<1$,
and Theorem~\ref{thm:nonlinear-interpolation} applies for any $\chi>0$.
The candidate budget may be imposed directly or obtained from
common $C_b^{\sigma/2,\sigma}$ bounds on $\mathcal P(v)+\mathscr H(\cdot,0)$ and
$C_b^{2+\sigma}$ bounds on $\operatorname{Tr}_Tv$, by the same
proposition. Small weighted $L^2$ residual alone does not give
this budget.

\begin{eexample}
\label{ex:nonlinear-compact-gap}
Take $\mathcal K=[\tau,T]\times\overline B_R$,
$a=6I_d$, $b=0$, and a smooth $0\le\vartheta\le1$
supported compactly in $(\tau,T)\times B_R$, with
$\vartheta=1$ on a nonempty open set. Set
\begin{equation*}
\mathscr H(P,j)=(1+2\vartheta(P))
 \sum_{i=1}^d(2\Gamma_{ii}+\log\cosh\Gamma_{ii}).
\end{equation*}
This convex operator satisfies \eqref{eq:nonlinear-derivative-bounds}--\eqref{eq:nonlinear-principal-modulus} and
\eqref{eq:nonlinear-regularity-structure},
so Proposition~\ref{prop:quasiconcave-terminal-problem} applies after
the sign change. Its principal matrix is
\begin{equation*}
\mathscr H_\Gamma(P,j)
 =(1+2\vartheta(P))\operatorname{diag}(2+\tanh\Gamma_{ii}),
\qquad I_d\preceq\mathscr H_\Gamma\preceq9I_d.
\end{equation*}
Outside $\mathcal K$, $\mathscr H_\Gamma\preceq a/2$, giving
\eqref{eq:nonlinear-relative-upper-bound} with $\chi=1/2$ on every
full candidate ball. The global gap fails where $\vartheta=1$ and
all $\Gamma_{ii}>0$. This occurs along a solution by taking
$u\in C_c^\infty((0,T)\times\R^d)$ with $D^2u\succ0$ on a small
open subset of $\{\vartheta=1\}$ and setting $f=\mathcal P(u)$, $g=0$.

For $\mu_0=\delta_0$, $\alpha=0$ and $q=a$, the matching center
$A=3I_d$ has $G_0(I)=2\sqrt2$ and global perturbation strength
$L_2\ge\sqrt d$, by letting all $\Gamma_{ii}\to\infty$ where
$\vartheta=1$. Thus $L_2G_0(I)<1$ fails for this center, while
Theorem~\ref{thm:tangent-secant-stability} still applies.
\end{eexample}

\section{Density estimates and local model approximation}
\label{sec:densities}

We prove the initial, flat and phase density approximations of
Section~\ref{sec:sharp-constants} and the Hardy estimate used in
Section~\ref{sec:proof}. Kernel, action and bridge estimates supply
posterior concentration and bounded-drift comparisons for the model
identification in Section~\ref{sec:phase-directions}.

\subsection{Transition kernels and Gaussian comparison}
\label{subsec:kernel-comparison}

Throughout this subsection, $a,b$ are Borel on the finite time interval
under consideration, $a$ is symmetric, and
\begin{equation}
0<\lambda_a I_d\preceq a(t,x)\preceq\Lambda_a I_d,
\quad |b(t,x)|\le B_b, \quad
|a(t,x)-a(t,y)|\le L_a|x-y|.
\label{eq:kernel-coefficient-bounds}
\end{equation}

For the diffusion \eqref{eq:intro-diffusion} with coefficients $(a,b)$,
started from $(r,y)$, denote its transition density by $\rho_{a,b}$:
\begin{equation*}
\mathbb P_{r,y}(X_t\in \mathrm{d}x)=\rho_{a,b}(r,y;t,x)\,\mathrm{d}x,\qquad r<t.
\end{equation*}
We write $\rho=\rho_{a,b}$ when the coefficient pair is clear.

\begin{lemma}\label{lem:kernel-stability}
The constants below depend on
$d,T,\lambda_a,\Lambda_a,L_a,B_b$.
\begin{enumerate}
\item\label{item:kernel-estimates}
On $[0,T]$, the kernel $\rho$ has a strictly positive version, jointly continuous for
$0\le r<t\le T$, $x,y\in\R^d$.
It satisfies, for $0\le r<r+h\le T$,
\begin{equation}
c h^{-d/2}\mathrm{e}^{-C|x-y|^2/h}
 \le \rho(r,y;r+h,x)\le C h^{-d/2}\mathrm{e}^{-c|x-y|^2/h},
\label{eq:kernel-gaussian}
\end{equation}
and the Chapman--Kolmogorov identity
\begin{equation*}
\rho(r,y;t,x)=\int_{\R^d}\rho(r,y;s,z)\rho(s,z;t,x)\,\mathrm{d}z,\,\,
\forall\,0\le r<s<t\le T,\ x,y\in\R^d.
\end{equation*}
On each compact subset of
$\{(r,y,t,x):0\le r<t\le T,\ x,y\in\R^d\}$,
kernels with common constants in \eqref{eq:kernel-coefficient-bounds}
are equicontinuous.
\item\label{item:kernel-stability}
Let $(a_n,b_n)$ be Borel coefficient pairs on $[0,T]\times\R^d$,
with $a_n$ symmetric and kernels $\rho_n=\rho_{a_n,b_n}$.
Assume that \eqref{eq:kernel-coefficient-bounds} holds for $(a_n,b_n)$
with the constants $\lambda_a,\Lambda_a,L_a,B_b$ independent of $n$. If
$a_n\to a$, $b_n\to b$ in $L^1((0,T)\times B_R)$ for every $R>0$,
then $\rho_n\to \rho$ locally uniformly on
  $\{(r,y,t,x):0\le r<t\le T,\ x,y\in\R^d\}$.
\end{enumerate}
\end{lemma}

\begin{proof}
The bounded, elliptic and spatially Lipschitz coefficient
$\sigma=a^{1/2}$ and bounded drift $b$ satisfy the hypotheses of
\citep[Theorem~1.2(i)]{menozzi-pesce-zhang-2020}.
Thus \eqref{eq:intro-diffusion} has a unique weak solution,
whose transition density satisfies two-sided Gaussian bounds
centered at the flow $\theta_{t,r}(y)$ of a spatial mollification of $b$.
Since $|\theta_{t,r}(y)-y|\le B_b(t-r)$, these imply
\eqref{eq:kernel-gaussian}.

We first prove convergence of diffusion laws under the coefficient
approximation in part~\ref{item:kernel-stability}. Fix a finite $p>(d+2)/2$
and put $p'=p/(p-1)$.
Common boundedness upgrades the assumed local $L^1$ convergence
to local $L^p$ convergence.
For the law $\mathbb P_n$ started at $(r,y)$,
\eqref{eq:kernel-gaussian}
gives
\begin{equation*}
\int_r^T\int_{\R^d}\rho_n(r,y;s,z)^{p'}\,\mathrm{d}z\,\mathrm{d}s
 \le C\int_0^{T-r}u^{-d/(2(p-1))}\,\mathrm{d}u\le C_p.
\end{equation*}
Here $C_p$ also depends on $p$.
H\"older's inequality therefore yields
\begin{equation}
\mathbb E_n\int_r^T|f(s,X_s)|\,\mathrm{d}s
 \le C_p\|f\|_{L^p((r,T)\times\R^d)}.
\label{eq:kernel-occupation}
\end{equation}

The Burkholder--Davis--Gundy inequality and boundedness of the
coefficients give $\mathbb E_n|X_v-X_u|^4\le C|v-u|^2$ for
$r\le u<v\le T$, hence tightness on $C([r,T];\R^d)$.
\eqref{eq:kernel-occupation} passes to every weak limit, first for
continuous compactly supported $f$ and then by duality.
Let $L=\tfrac12a:D^2+b\cdot D$,
$L_n=\tfrac12a_n:D^2+b_n\cdot D$, and
$\varphi\in C_c^\infty(B_R)$.
With $Q_R=(r,T)\times B_R$, \eqref{eq:kernel-occupation} gives
\begin{equation*}
\mathbb E_n\int_r^T|(L_n-L)\varphi(s,X_s)|\,\mathrm{d}s
\le C_{p,\varphi}\bigl(\|a_n-a\|_{L^p(Q_R)}
                      +\|b_n-b\|_{L^p(Q_R)}\bigr)\longrightarrow0.
\end{equation*}
Since $L\varphi$ need not be continuous, approximate $a,b$ locally
in $L^p$ by bounded continuous coefficients. Applying
\eqref{eq:kernel-occupation} under both the approximating laws and
their weak limits passes the martingale identities to the limit;
weak uniqueness identifies the limit with the law $\mathbb P$ of
the diffusion with coefficients $(a,b)$ started at $(r,y)$. Thus
\begin{equation}
\mathbb P_n\Longrightarrow\mathbb P
\quad\text{on }C([r,T];\R^d).
\label{eq:kernel-law-convergence}
\end{equation}

To prove the continuity assertions in
part~\ref{item:kernel-estimates}, put $\beta=\tfrac12\operatorname{div}a-b$, with
$\|\beta\|_\infty\le C_dL_a+B_b$.
The forward and backward Kolmogorov equations take the weak divergence forms
\begin{equation*}
\partial_t\rho=\operatorname{div}_x(\tfrac12aD_x\rho+\beta \rho),\qquad
-\partial_r\rho=\operatorname{div}_y(\tfrac12aD_y\rho)-\beta\cdot D_y\rho .
\end{equation*}
Common spacetime mollification of $a,b$, after constant extension in time,
also mollifies $\beta$.
The smooth transition kernels converge to the weak fundamental solution
of \citep[Theorem~10 and its proof]{aronson-1968}, identified with $\rho$
by \eqref{eq:kernel-law-convergence}.
The interior H\"older estimate in \citep[Theorem~C]{aronson-1968},
applied forward in $(t,x)$ and backward in $(r,y)$, together with
\eqref{eq:kernel-gaussian}, gives a common local modulus of continuity
away from $r=t$.
The time extension covers $r=0$ and $t=T$;
continuity extends the Chapman--Kolmogorov identity to every terminal point.

Finally, \eqref{eq:kernel-gaussian} and
the equicontinuity in
part~\ref{item:kernel-estimates} give local precompactness of $\rho_n$
away from $r=t$. Equation~\eqref{eq:kernel-law-convergence}
identifies every subsequential limit with $\rho$.
\end{proof}

Far beyond the diffusive scale, an absolute kernel error can
dominate the density itself. The next lemma gives a logarithmic
comparison with a spatially constant covariance, requiring
closeness only in a tube joining the endpoints.
The long time intervals in this lemma arise
by parabolic rescaling of short physical bridge intervals in
Subsections~\ref{subsec:conditional-bridges} and~\ref{subsec:density-hardy}.

\begin{lemma}\label{lem:long-tube}
Fix $K_0,P_0>0$. There are $K,c,C>0$, depending only on
$d,\lambda_a,\Lambda_a,K_0,P_0$, with the following property.
On $[r,t]$, where $h=t-r\ge1$, let $B(s)$ be Borel and symmetric
with
$\lambda_a I_d\preceq B(s)\preceq\Lambda_a I_d$. Here coefficient norms use
$[r,t]\times\R^d$. Assume $\sqrt h\,\Lip_xa\le K_0$ and
$|x-y|\le P_0h$. Put $\Sigma=\int_r^tB(s)\,\mathrm{d}s$,
$\xi=\Sigma^{-1}(x-y)$ and $\gamma(s)=y+\int_r^sB(v)\xi\,\mathrm{d}v$.
For $R>0$, let
$\mathcal T_R(\gamma)=\{(s,z):r\le s\le t,\ |z-\gamma(s)|<R\}$.
Set $E_a=\|a-B\|_{L^\infty(\mathcal T_{2Kh}(\gamma))}$
and $E=E_a+\|b\|_\infty$. If $\sqrt h\,E\le1$, then
\begin{equation*}
\left|\log\frac{\rho(r,y;t,x)}{\gamma_\Sigma(x-y)}\right|
 \le C\bigl(hE+E_a^{1/2}+\mathrm{e}^{-ch}\bigr).
\end{equation*}
\end{lemma}

\begin{proof}
Within the stated coefficient class, \eqref{eq:kernel-occupation}
and weak uniqueness show that changes on spacetime null sets leave
the kernels unchanged. We may therefore take
$|b|\le\|b\|_\infty$ everywhere and, by spatial continuity,
$|a-B|\le E_a$ throughout the tube, changing $a$ only at a null
set of times.
Put $D_0=\Lambda_aP_0/\lambda_a$.
Then $|\xi|\le P_0/\lambda_a$ and
$|\gamma(s)-x|\le D_0h$, $\forall s\in[r,t]$.
\paragraph*{Localization to the tube.}
For $K\ge D_0+1$ to be fixed below, use the cutoffs
\eqref{eq:spatial-cutoffs} and define
\begin{equation*}
\widetilde a(s,z)=B(s)+\chi_{Kh}(z-\gamma(s))(a(s,z)-B(s)),\qquad
\widetilde b(s,z)=\chi_{Kh}(z-\gamma(s))b(s,z).
\end{equation*}
These coefficients have the same ellipticity bounds and satisfy
\begin{equation*}
\|\widetilde a-B\|_\infty\le E_a,\qquad
\|\widetilde b\|_\infty\le\|b\|_\infty,\qquad
\sqrt h\,\Lip_x\widetilde a\le K_0+\frac{CE_a}{K\sqrt h}\le K_0+C.
\end{equation*}
Since $\sqrt h\,\|b\|_\infty\le\sqrt h\,E\le1$,
the constants in \eqref{eq:kernel-gaussian}, after scaling
time by $h$ and space by $\sqrt h$, are independent of $K$.
Let $\widetilde\rho$ be the modified kernel. By weak uniqueness,
the stopped diffusions have the same law up to exit from
$\mathcal T_{Kh}(\gamma)$, where the exit point is at
distance at least $(K-D_0)h$ from $x$.
Let $\nu$ be their common first-exit distribution before $t$.
Subtracting the first-exit decompositions cancels the common
contribution from paths staying in the tube and gives
\begin{equation*}
\rho(r,y;t,x)-\widetilde\rho(r,y;t,x)
 =\int_{[r,t)\times\R^d}
       \bigl[\rho(s,z;t,x)-\widetilde\rho(s,z;t,x)\bigr]\,\nu(\mathrm{d}s,\mathrm{d}z).
\end{equation*}
This identity extends from a.e.\ terminal point to $x$ by kernel
continuity and the Gaussian upper bound \eqref{eq:kernel-gaussian},
since the exit locations stay a positive distance from $x$.
Since $\nu$ has mass at most one, the same bound gives
\begin{equation*}
|\rho(r,y;t,x)-\widetilde\rho(r,y;t,x)|
 \le C\sup_{0<q\le h}q^{-d/2}\mathrm{e}^{-c(K-D_0)^2h^2/q}
 \le Ch^{-d/2}\mathrm{e}^{-c(K-D_0)^2h}.
\end{equation*}
The lower bound in \eqref{eq:kernel-gaussian} also gives
$\widetilde\rho(r,y;t,x)\ge ch^{-d/2}\mathrm{e}^{-Ch}$.
Fix $K$ large enough that
$|\rho/\widetilde\rho-1|\le \mathrm{e}^{-h}/2$ for $h\ge1$.
Using $|\log(1+u)|\le2|u|$ for $|u|\le1/2$, we obtain
\begin{equation}
\left|\log\frac{\rho(r,y;t,x)}{\widetilde\rho(r,y;t,x)}\right|
 \le \mathrm{e}^{-h}.
\label{eq:long-tube-localization}
\end{equation}

\paragraph*{Exponential tilt.}
The tilt by $\xi$ centers the Gaussian reference at $x$.
Let $\widehat\rho$ correspond to
$(\widetilde a,\widetilde b+\widetilde a\xi)$.
Boundedness of $\widetilde a$ and $\xi$ gives Novikov's condition.
For the change of measure from drift
$\widetilde b+\widetilde a\xi$ to $\widetilde b$, Girsanov's weight
has terminal factor
$\exp\bigl(-\xi\cdot(X_t-y)+\tfrac12\xi^{\mathsf T}\Sigma\xi\bigr)$
and remaining factor
\begin{equation*}
\exp\left\{\int_r^t
 \left[\xi\cdot\widetilde b+
       \tfrac12\xi^{\mathsf T}(\widetilde a-B)\xi\right](s,X_s)\,\mathrm{d}s\right\}
 \in[\mathrm{e}^{-ChE},\mathrm{e}^{ChE}].
\end{equation*}
The resulting comparison of terminal laws gives a comparison
of $\widetilde\rho,\widehat\rho$ a.e.; their continuity extends it
to every terminal point.
Since $x-y=\Sigma\xi$ and
$\gamma_\Sigma(x-y)=\mathrm{e}^{-\xi^{\mathsf T}\Sigma\xi/2}\gamma_\Sigma(0)$,
\begin{equation}
\left|\log\frac{\widetilde\rho(r,y;t,x)}{\gamma_\Sigma(x-y)}
      -\log\frac{\widehat\rho(r,y;t,x)}{\gamma_\Sigma(0)}\right|
 \le ChE.
\label{eq:long-tube-tilt}
\end{equation}

\paragraph*{Centered kernel comparison.}
Moving with $\dot\gamma=B\xi$ removes the leading tilted drift.
In coordinates $v=r+hs$, $z=\gamma(r+hs)+\sqrt h\,w$, the
endpoints become $(0,0)$ and $(1,0)$, and the coefficients on
$[0,1]\times\R^d$ are
\begin{equation*}
a^\circ(s,w)=\widetilde a(v,z),\qquad
b^\circ(s,w)=\sqrt h\bigl[\widetilde b(v,z)
                     +(\widetilde a(v,z)-B(v))\xi\bigr].
\end{equation*}
Norms of rescaled fields use $[0,1]\times\R^d$.
Writing $B^\circ(s)=B(r+hs)$, we have
\begin{equation*}
\|a^\circ-B^\circ\|_\infty\le E_a,\qquad
\|b^\circ\|_\infty\le C\sqrt h\,E,\qquad
\Lip_wa^\circ\le K_0+C.
\end{equation*}
The reference covariance on $[0,1]$ is $\Sigma/h$.
Put $f(s,w)=(a^\circ(s,w))^{1/2}-(B^\circ(s))^{1/2}$.
Uniform ellipticity and the spatial Lipschitz bound give
$|f(s,w)|\le CE_a$ and $\Lip_w f(s,\cdot)\le C$ for every $s\in[0,1]$.
Since $E_a\le1$, interpolation yields
\begin{equation*}
\sup_{0\le s\le1}\|f(s,\cdot)\|_{C^{1/2}(\R^d)}
 \le C\bigl(E_a+E_a^{1/2}\bigr)\le CE_a^{1/2}.
\end{equation*}
Apply
\citep[Theorem~1.1]{konakov-kozhina-menozzi-2017}
to $(a^\circ,b^\circ)$ and $(B^\circ,0)$ on $[0,1]$,
with $q=\infty$, $\gamma=1/2$, and constants independent of $h$.
Returning to the original coordinates multiplies both densities
by $h^{-d/2}$. Since
$\gamma_\Sigma(0)=h^{-d/2}\gamma_{\Sigma/h}(0)\ge ch^{-d/2}$,
the comparison gives
\begin{equation*}
\left|\frac{\widehat\rho(r,y;t,x)}{\gamma_\Sigma(0)}-1\right|
 \le C\bigl(E_a^{1/2}+\sqrt h\,E\bigr).
\end{equation*}
Since $\|b^\circ\|_\infty\le C$,
\eqref{eq:kernel-gaussian} for $(a^\circ,b^\circ)$ and ellipticity give
$\widehat\rho(r,y;t,x)/\gamma_\Sigma(0)\ge c$.
Using $|\log u|\le C|u-1|$ on this range gives
\begin{equation}
\left|\log\frac{\widehat\rho(r,y;t,x)}{\gamma_\Sigma(0)}\right|
 \le C\left|\frac{\widehat\rho(r,y;t,x)}{\gamma_\Sigma(0)}-1\right|
 \le C\bigl(E_a^{1/2}+\sqrt h\,E\bigr).
\label{eq:long-tube-centered}
\end{equation}
Combining \eqref{eq:long-tube-localization},
\eqref{eq:long-tube-tilt} and \eqref{eq:long-tube-centered} gives
\begin{equation*}
\left|\log\frac{\rho(r,y;t,x)}{\gamma_\Sigma(x-y)}\right|
 \le \mathrm{e}^{-h}+ChE+C\bigl(E_a^{1/2}+\sqrt h\,E\bigr).
\end{equation*}
Since $h\ge1$, this proves
the claimed logarithmic bound.
\end{proof}

With only Borel time dependence, the Gaussian ratio expansion
retains the covariance increment in integral form.

\begin{lemma}\label{lem:gaussian-ratio}
Fix $H>0$ and $B_1,P>0$ with $B_1\le H/2$.
Let $B:[-H,H/2]\to\R^{d\times d}$ be Borel and symmetric,
with $\lambda_a I_d\preceq B(s)\preceq\Lambda_a I_d$, put
$\Sigma_H(s)=\int_{-H}^sB(v)\,\mathrm{d}v$, and suppose $|w|\le PH$.
For $\xi_H=\Sigma_H(0)^{-1}w$,
\begin{equation*}
\left|\log\frac{\gamma_{\Sigma_H(s)}(z-w)}
              {\gamma_{\Sigma_H(0)}(-w)}
 -\xi_H\cdot z-
  \tfrac12\xi_H^{\mathsf T}
          \left(\int_0^sB(v)\,\mathrm{d}v\right)\xi_H\right|
 \le C/H,
\end{equation*}
uniformly for $(s,z)\in\mathcal U_{B_1}$.
Here $C$ depends only on
$d,\lambda_a,\Lambda_a,B_1,P$.
\end{lemma}

\begin{proof}
Put $\Sigma=\Sigma_H(0)$ and $D_s=\int_0^sB(v)\,\mathrm{d}v$.
Write $Q=\Sigma^{-1}$ and $Q_s=(\Sigma+D_s)^{-1}$.
For $|s|\le H/2$, ellipticity gives
$|Q|+|Q_s|\le C/H$ and $|D_s|\le C|s|$.
The resolvent identity yields
\begin{equation*}
Q_s=Q-QD_sQ+R_s,\, R_s=QD_sQD_sQ_s, \,
|Q_s-Q|\le C|s|/H^2,\, |R_s|\le Cs^2/H^3.
\end{equation*}
Since $\Sigma+\vartheta D_s\succeq\lambda_a H I_d/2$ for
$0\le\vartheta\le1$, we also have
$|\log(\det(\Sigma+D_s)/\det\Sigma)|\le C|s|/H$.
Using $\xi_H=Qw$, expand the Gaussian logarithm as
\begin{equation*}
\begin{gathered}
\log\frac{\gamma_{\Sigma+D_s}(z-w)}{\gamma_\Sigma(-w)}
 =\xi_H\cdot z+\tfrac12\xi_H^{\mathsf T}D_s\xi_H+\mathcal R_{s,z},\\
\mathcal R_{s,z}
 :=-\tfrac12\log\frac{\det(\Sigma+D_s)}{\det\Sigma}
    -\tfrac12z^{\mathsf T}Q_sz
    +z^{\mathsf T}(Q_s-Q)w-\tfrac12w^{\mathsf T}R_sw .
\end{gathered}
\end{equation*}
The four remainder terms satisfy, using $|w|\le PH$,
\begin{equation*}
|\mathcal R_{s,z}|
 \le C\left(\frac{|s|+|z|^2}{H}
          +\frac{|s||z||w|}{H^2}+\frac{s^2|w|^2}{H^3}\right)
 \le \frac C H\bigl(|s|+|z|^2+|s||z|+s^2\bigr).
\end{equation*}
On $\mathcal U_{B_1}$ this is at most $C/H$, proving the claim.
\end{proof}

\subsection{Initial and flat density approximations}
\label{subsec:initial-density-estimates}

For the remainder of this section, $I=(0,T)$ and the reference
coefficients $a,b$ satisfy the conditions in Assumption~\ref{ass:coefficients}.
Kernel times satisfy $0\le r<t\le T$; marginal densities and the scale
$\ell$ are evaluated for $(t,x)\in(0,T]\times\R^d$.
In this subsection, assume Assumption~\ref{def:initial-laws} and use
$\ell$ from \eqref{eq:weighted-tail-scale}.
The weighted coercivity bound below
will also be used for posterior concentration in
Subsection~\ref{subsec:density-hardy}.

\begin{lemma}
\label{lem:initial-tail-bounds}
\leavevmode
\begin{enumerate}
\item\label{item:initial-coercivity-transfer}
Let $U\subseteq\R^d$ be open and $w:U\to(0,\infty)$ be locally
Lipschitz, with $|Dw|\le L_w w$ a.e. on $U$ for some $L_w<\infty$. Then,
for every $\varphi\in C_c^\infty(U)$,
\begin{equation*}
\int\langle y\rangle^2|\varphi|^2w\,\mathrm{d}\mu_0
 \le 2K_{\mathrm{coer}}\int|D\varphi|^2w\,\mathrm{d}\mu_0
+K_{\mathrm{coer}}\left(1+\tfrac12L_w^2\right)
                \int|\varphi|^2w\,\mathrm{d}\mu_0.
\end{equation*}
\item\label{item:initial-tail-bounds}
The second moment satisfies
\(
\int\langle y\rangle^2\,\mathrm{d}\mu_0\le K_{\mathrm{coer}}.
\)

\item\label{item:positive-time-density}
For each $t_0\in(0,T]$, there are $c,C>0$, depending only on
$\mathfrak R,K_{\mathrm{coer}},t_0$, such that
\begin{equation*}
c \mathrm{e}^{-C|x|^2}\le\rho(t,x),
\qquad t_0\le t\le T,\quad x\in\R^d.
\end{equation*}
For each $R>0$, $\rho$ and $\log\rho$ have a common modulus of
continuity on $[t_0,T]\times\overline B_R$, depending only on
$\mathfrak R,K_{\mathrm{coer}},t_0,R$.
\end{enumerate}
\end{lemma}

\begin{proof}
For $\varphi\in C_c^\infty(U)$, mollify $w$ near $\supp\varphi$
with a nonnegative kernel to obtain positive smooth $w_\epsilon$
with $|Dw_\epsilon|\le L_w w_\epsilon$ on $\supp\varphi$
and $w_\epsilon\to w$ uniformly there.
Applying \eqref{eq:initial-coercivity} to $\varphi\sqrt{w_\epsilon}$ and using
\begin{equation*}
|D(\varphi\sqrt{w_\epsilon})|^2
 \le 2w_\epsilon|D\varphi|^2
       +\tfrac12L_w^2w_\epsilon|\varphi|^2
\end{equation*}
gives part~\ref{item:initial-coercivity-transfer} with $w_\epsilon$ in place of $w$.
Uniform convergence on $\supp\varphi$ permits $\epsilon\downarrow0$,
also for singular $\mu_0$.

For the cutoffs in \eqref{eq:spatial-cutoffs},
\eqref{eq:initial-coercivity} gives
\begin{equation*}
\int\langle y\rangle^2\chi_R^2\,\mathrm{d}\mu_0
 \le K_{\mathrm{coer}}\left(1+\frac C{R^2}\right).
\end{equation*}
Fatou's lemma as $R\to\infty$ proves part~\ref{item:initial-tail-bounds}.

Finally, fix $t_0\in(0,T]$ and choose
$R_0$ by part~\ref{item:initial-tail-bounds} so that
$\mu_0(B_{R_0})\ge1/2$. Integrating the lower kernel bound in
\eqref{eq:kernel-gaussian} on this ball proves the lower bound in
part~\ref{item:positive-time-density}.
Fix $R>0$ and put $Q=[t_0,T]\times\overline B_R$, where
part~\ref{item:positive-time-density} gives $\rho\ge c_Q>0$.
For $L\ge1$, Lemma~\ref{lem:kernel-stability}\ref{item:kernel-estimates}
provides a common modulus $\omega_L$ for
$(t,x)\mapsto\rho(0,y;t,x)$ on $Q$, uniformly for $|y|\le L$.
For $(t,x),(s,z)\in Q$, set $h=|t-s|+|x-z|$.
Splitting the initial integral at $|y|=L$ and using
\eqref{eq:kernel-gaussian} and part~\ref{item:initial-tail-bounds} gives
\begin{equation*}
\begin{aligned}
|\log\rho(t,x)-\log\rho(s,z)|
 \le c_Q^{-1}|\rho(t,x)-\rho(s,z)|
 \le C\{\omega_L(h)+L^{-2}\}.
\end{aligned}
\end{equation*}
Since $C$ is independent of $L$, taking $L$ large and then $h$ small
proves both assertions.
\end{proof}

The bounded-depth flat approximation uses a minimizer in
\eqref{eq:weighted-tail-scale}, which exists because $S$ is closed
and the minimized expression tends to infinity as $|y|\to\infty$.
Each function under the infimum is $t^{-1}$-Lipschitz in $x$,
as is the infimum, proving
\eqref{eq:weighted-scale-gradient}.

The restarted-posterior estimates in
Subsection~\ref{subsec:density-hardy} compare nearby terminal points
at a common restart time. We therefore also record the stability
of $\ell$ and $M$ in physical windows.
For $(t,x)\in[0,T]\times\R^d$, $\varepsilon>0$ and $R\ge0$, set
\begin{equation*}
\mathcal U_R^{\varepsilon}(t,x)
 =\{(t+\varepsilon^2s,x+\varepsilon z):
       (s,z)\in\mathcal U_R,\ s\in I^{\mathrm{sc}}(t;\varepsilon)\}.
\end{equation*}
For $(t',x')\in\mathcal U_W^\ell(t,x)$, where
$\ell=\ell(t,x)$ and $M=M(t,x)\ge4W$, we have
\begin{equation}
\left|\frac{\ell(t',x')}{\ell(t,x)}-1\right|
 +\left|\frac{M(t',x')}{M(t,x)}-1\right|
 \le C\frac WM.
\label{eq:weighted-scale-stability}
\end{equation}
Indeed, $t'/t=1+O(W/M)$, while the spatial perturbation changes
$|x-y|/t$ by at most $W\ell/t=(W/M)\ell^{-1}$, uniformly in $y$.
Applying these comparisons in both directions under the infimum
gives the scale bound and then the depth bound, with an absolute
constant $C$.

\begin{proof}[Proof of Proposition~\ref{prop:initial-heat-limits}]
Fix $\varepsilon>0$, $0<\tau_0<L$ and $B_0,B>0$, and set
$K=[\tau_0,L]\times\overline B_B$.
For the finite-scale approximation, take $y\in S$ and
$r\langle y\rangle\le B_0$.
For $0<r\le1$ with $r^2L\le T$, write
\begin{equation*}
\begin{gathered}
\widehat\rho_{y,r}(0,v;\tau,z)
 =r^d\rho(0,y+rv;r^2\tau,y+rz),\\
\widehat\rho_{y,r}(\tau,z)=
\frac{r^d\rho(r^2\tau,y+rz)}{m_{y,r}}
 =\int_{\R^d}\widehat\rho_{y,r}(0,v;\tau,z)\,\theta_{y,r}(\mathrm{d}v).
\end{gathered}
\end{equation*}

To obtain a common threshold, take arbitrary sequences of initial
laws $\mu_0^{(j)}$ and coefficients $(a_j,b_j)$ with common
$\mathfrak R,K_{\mathrm{dbl}},K_{\mathrm{pot}}$, scales $r_j\to0$,
and centers $y_j\in\supp\mu_0^{(j)}$ with $r_j\langle y_j\rangle\le B_0$.
Write $\widehat\rho_j$ for the corresponding rescaled kernels.
Ellipticity permits extraction of $a_j(0,y_j)\to a_*$.
The bounds \eqref{eq:local-coefficient-freezing} then give locally
uniform convergence of the rescaled coefficients to $(a_*,0)$.
Lemma~\ref{lem:kernel-stability}\ref{item:kernel-stability} yields
$\widehat\rho_j(0,v;\tau,z)\to\gamma_{\tau a_*}(z-v)$ locally uniformly
for $\tau>0$.
Continuity of the Gaussian in its covariance therefore gives the
following estimate, with $o_R(1)$ uniform over all such initial laws
and coefficients with $R\ge 1$:
\begin{equation}
E_R(r):=\sup_{\substack{y\in S\\r\langle y\rangle\le B_0}}\sup_{\substack{(\tau,z)\in K\\v\in B_R}}
 \bigl|\widehat\rho_{y,r}(0,v;\tau,z)-\gamma_{\tau a_0(y)}(z-v)\bigr|
 =o_R(1)\quad\text{as }r\downarrow0.
\label{eq:initial-kernel-freezing}
\end{equation}

The total mass of $\theta_{y,r}$ need not stay bounded as $r\downarrow0$,
so local kernel convergence alone does not control the integral over
$\R^d$.
For $(\tau,z)\in K$, the Gaussian bound \eqref{eq:kernel-gaussian}
and ellipticity give
\begin{equation*}
\widehat\rho_{y,r}(0,v;\tau,z)+\gamma_{\tau a_0(y)}(z-v)
 \le C \mathrm{e}^{-c|v|^2},\qquad v\in\R^d.
\end{equation*}
For this Gaussian exponent, restrict $r\le r_c$ as in
\eqref{eq:rescaled-measure-mass}. The resulting Gaussian integral
and tail bounds are uniform over the admissible centers.

For $R\ge1$, split both density integrals with $\chi_R$ from
\eqref{eq:spatial-cutoffs} and $1-\chi_R$.
The mass bounds \eqref{eq:rescaled-local-mass}--\eqref{eq:rescaled-measure-mass}
and the local kernel error \eqref{eq:initial-kernel-freezing} give
\begin{equation}
\sup_{(\tau,z)\in K}
 \bigl|\widehat\rho_{y,r}(\tau,z)-W_{y,r}(\tau,z)\bigr|
 \le C_R E_{2R}(r)+C\mathrm{e}^{-cR^2/2}.
\label{eq:initial-density-comparison}
\end{equation}
Since $\theta_{y,r}(B_1)=1$, the uniform positive lower bound for
$\gamma_{\tau a_0(y)}(z-v)$ on $K\times\overline B_1$
gives $W_{y,r}\ge c_0>0$ on $K$.
The constants $C_R,C,c,c_0$ are uniform over initial laws with
the stated common data.

Choose $R$ and then $0<r_0\le r_c$, with $r_0^2L\le T$, so that
the right side of \eqref{eq:initial-density-comparison} is at most
$\tfrac{c_0}{2}\min\{1,\varepsilon\}$ for $0<r\le r_0$.
Then $\widehat\rho_{y,r}\ge c_0/2$ on $K$, and the logarithmic
relative error is at most $\varepsilon$, proving
\eqref{eq:initial-heat-approximation}.

To prove realization, fix a defining sequence of
$\mathbf T=(\theta,a_*,A_*,q_*)\in\mathfrak T(\mu_0)$ and choose
$B_0\ge\sup_j r_j\langle y_j\rangle$.
Once $r_j\le r_c$, the bounds
\eqref{eq:rescaled-local-mass}--\eqref{eq:rescaled-measure-mass}
apply to $\theta_{y_j,r_j}$ and their vague limit $\theta$.
For fixed $R$, the truncated Gaussian kernels indexed by $K$
form a compact family in the uniform norm. Vague convergence and
$a_0(y_j)\to a_*$ therefore give uniform convergence on $K$ of
the convolutions truncated by $\chi_R$.
The Gaussian tail bound in \eqref{eq:rescaled-measure-mass} then gives,
by first letting $j\to\infty$ and then $R\to\infty$,
\begin{equation*}
\sup_{(\tau,z)\in K}
 \bigl|W_{y_j,r_j}(\tau,z)-W_{\mathbf T}(\tau,z)\bigr|
 \longrightarrow0.
\end{equation*}
Together with \eqref{eq:initial-heat-approximation}, this proves \eqref{eq:initial-realization}
locally uniformly on $(0,\infty)\times\R^d$.
\end{proof}

\begin{proof}[Proof of Proposition~\ref{lem:uniform-flat-approximation}]
Suppose $M(t,x)\le B_0$, and put
$\sigma=\sqrt t$ and $\eta=r/\sigma$.
We take $(s,z)\in\mathcal U_B$ with $s\in I^{\mathrm{sc}}(t;r)$
and restrict to $\eta\le1$ and $B\eta^2\le1/2$.
For small $t$, choose a minimizer $y\in S$ in
\eqref{eq:weighted-tail-scale}.
Then $z_0=(x-y)/\sigma$ satisfies
\(
|z_0|\le\sqrt{B_0},\,
\sigma\langle y\rangle\le\sqrt{B_0}.
\)
The normalization $\theta_{y,\sigma}(B_1)=1$ and the Gaussian
integral bound \eqref{eq:rescaled-measure-mass} give
\begin{equation*}
W_{y,\sigma}\ge c,\qquad
|\partial_\tau W_{y,\sigma}|+|D_zW_{y,\sigma}|\le C
\quad\text{on }[1/2,2]\times\overline B_{\sqrt{B_0}+B},
\end{equation*}
for all sufficiently small $\sigma$.
Indeed, the lower bound follows by integrating over $B_1$,
and the Gaussian derivatives are bounded by $C\mathrm{e}^{-c|\xi|^2}$
on this set.
Choose $\delta\in(0,T/2]$ small enough that these bounds hold and
\eqref{eq:initial-heat-approximation} has error at most
$\varepsilon/4$ on the same set whenever $\sigma\le\sqrt\delta$.
For $t\le\delta$, this gives
\begin{equation*}
\left|\log\frac{\rho(t+r^2s,x+rz)}{\rho(t,x)}\right|
\le\frac{\varepsilon}{2}
 +\left|\log\frac{W_{y,\sigma}(1+\eta^2s,z_0+\eta z)}
                       {W_{y,\sigma}(1,z_0)}\right|
\le\frac{\varepsilon}{2}+C(\eta+\eta^2).
\end{equation*}

For $t\ge\delta$, \eqref{eq:weighted-scale-comparison}
bounds $|x|$. Thus both $(t,x)$ and $(t+r^2s,x+rz)$ lie in
$[\delta/2,T]\times B_R$ for a common $R$.
Lemma~\ref{lem:initial-tail-bounds}\ref{item:positive-time-density}
gives a common modulus of continuity for $\log\rho$ on this set.
Since $|rz|\le B\sqrt T\,\eta$ and $|r^2s|\le BT\eta^2$,
the absolute logarithmic ratio is at most $\varepsilon$ for all
sufficiently small $\eta$, uniformly over these centers.
Choose $\eta_0>0$ to satisfy this bound, the initial restrictions,
and $C(\eta_0+\eta_0^2)\le\varepsilon/2$.
The two time ranges prove the asserted flat approximation.
\end{proof}

\subsection{Kinetic action and kernel bounds}
\label{subsec:action-bridges}

Set $g=a^{-1}$, $h=t-r$ and $D=|x-y|$.
For a positive definite matrix $G$, write $|v|_G^2=v^{\mathsf T}Gv$.
For paths with fixed endpoints, define the zero-drift kinetic action
\begin{equation}
\begin{gathered}
\mathcal I^{g,0}_{r,t}(\gamma)
 =\frac12\int_r^t\dot\gamma(s)^{\mathsf T}
       g(s,\gamma(s))\dot\gamma(s)\,\mathrm{d}s,\quad
\mathcal A^{g,0}(r,y;t,x)
 =\inf_{\substack{\gamma\in H^1([r,t];\R^d)\\
                    \gamma(r)=y,\ \gamma(t)=x}}
       \mathcal I^{g,0}_{r,t}(\gamma).
\end{gathered}
\label{eq:action}
\end{equation}
Path spaces and norms in this subsection use
$[r,t]$ and Lebesgue time measure, unless a different interval is displayed.

The Gaussian bounds \eqref{eq:kernel-gaussian} have different
exponential rates. To resolve bridge displacements near an endpoint,
we need upper and lower kernel bounds with the same action.
The path estimates below supply that comparison and later permit
spatial freezing in Subsection~\ref{subsec:radial-infinity}.

\begin{lemma}
\label{lem:action-variations}
The infimum in \eqref{eq:action} is attained and satisfies
\begin{equation}
\frac{D^2}{2\Lambda_a h}\le\mathcal A^{g,0}(r,y;t,x)
 \le\frac{D^2}{2\lambda_a h}.
\label{eq:action-ellipticity}
\end{equation}
For $r<s<t$,
\begin{equation}
\mathcal A^{g,0}(r,y;t,x)
 =\inf_z\{\mathcal A^{g,0}(r,y;s,z)
                   +\mathcal A^{g,0}(s,z;t,x)\},
\label{eq:action-splitting}
\end{equation}
and every subpath of a minimizer is minimizing.
If $a'$ has the same ellipticity bounds, $g'=(a')^{-1}$,
and a ball $B_R$ contains a minimizer of each action, then
\begin{equation}
|\mathcal A^{g',0}(r,y;t,x)-\mathcal A^{g,0}(r,y;t,x)|
 \le C\frac{D^2}{h}
       \|a'-a\|_{L^\infty((r,t)\times B_R)}.
\label{eq:action-stability}
\end{equation}
Here $a'$ is spatially continuous and measurable in time, and
the constant in \eqref{eq:action-stability} depends only on
ellipticity.
\end{lemma}

\begin{proof}
Cauchy--Schwarz gives the lower bound in
\eqref{eq:action-ellipticity}; the straight path gives the upper bound.
The integrand is Carath\'eodory, convex in velocity and quadratically
coercive, and the straight path has finite action. Thus
\citep[Theorem~3.30]{dacorogna-2008-direct} gives a minimizer in
the affine $H^1$ class with endpoints $y,x$.
Restriction and concatenation give
\eqref{eq:action-splitting} and subpath optimality.
Each minimizer satisfies $\int_r^t|\dot\gamma|^2\le CD^2/h$;
Cauchy--Schwarz therefore gives $\|\gamma-y\|_\infty\le CD$.
Using one minimizer in $B_R$ for each action as a competitor for the other,
and $g'-g=g'(a-a')g$, proves \eqref{eq:action-stability}.
Spatial continuity makes the displayed norm control the
coefficient difference along each path for a.e. time.
\end{proof}

\begin{lemma}
\label{lem:action-speed}
Every minimizing path $\gamma$ satisfies
\begin{equation}
cD/h\le|\dot\gamma(s)|_{g(s,\gamma(s))}\le CD/h
\quad\text{a.e. on }[r,t].
\label{eq:action-speed}
\end{equation}
For $\delta_r,\delta_t\ge0$, $\delta_r+\delta_t\le h/2$,
and $z_r,z_t\in\R^d$,
\begin{equation}
\begin{gathered}
c(\delta_r+\delta_t)D^2/h^2
 \le\mathcal A^{g,0}(r+\delta_r,y;t-\delta_t,x)
                -\mathcal A^{g,0}(r,y;t,x)
 \le C(\delta_r+\delta_t)D^2/h^2,\\
|\mathcal A^{g,0}(r,y+z_r;t,x+z_t)
 -\mathcal A^{g,0}(r,y;t,x)|
 \le C h^{-1}\{D(|z_r|+|z_t|)+|z_r|^2+|z_t|^2\}.
\end{gathered}
\label{eq:action-endpoint-variation}
\end{equation}
The constants depend only on $\mathfrak R$.
\end{lemma}

\begin{proof}
Let $\gamma$ be a minimizer.
\paragraph*{Speed bounds.}
We use the frozen-metric
reparametrization argument from
\citep[proofs of Lemmas~A.2 and~A.4]{norris-stroock-1991},
with the square-root Dini modulus in place of a H\"older modulus.
On $J=[u,u+\delta]\subseteq[r,t]$, set
$v_*(s)=|\dot\gamma(s)|_{g(u,\gamma(s))}$ and
$v(s)=|\dot\gamma(s)|_{g(s,\gamma(s))}$.
The time modulus gives
$(1-\epsilon_\delta)v_*^2\le v^2\le(1+\epsilon_\delta)v_*^2$
with $\epsilon_\delta=C\omega_a(\delta)$ for small $\delta$.
Reparametrize the same spatial curve at constant $g(u,\cdot)$-speed.
It has the same endpoints, so subpath optimality gives
\begin{equation*}
\int_Jv_*^2\le(1+C\epsilon_\delta)\delta\bar v_{*,J}^2,
\qquad
\frac1\delta\int_J|v_*-\bar v_{*,J}|^2
 \le C\epsilon_\delta\bar v_{*,J}^2,
\end{equation*}
where bars denote interval averages.
Consequently, for either dyadic half $J'$ of $J$,
$|\bar v_{J'}-\bar v_J|\le C\sqrt{\epsilon_{|J|}}\bar v_J$.
Choose $\delta_0\le h_*$ so that
$C\sqrt{\epsilon_{\delta_0}}\le1/4$.
The square-root Dini condition implies
$\sum_{k\ge0}\sqrt{\epsilon_{2^{-k}\delta_0}}<\infty$.
Multiplying the bounds for successive dyadic averages gives
positive lower and upper products, uniform over $|J|\le\delta_0$.
Dyadic differentiation gives $c\bar v_J\le v(s)\le C\bar v_J$
for a.e. $s\in J$.
For longer intervals, cover $[r,t]$ by at most
$C(1+T/\delta_0)$ successive intervals of length at most $\delta_0$
whose neighbors overlap on at least a quarter of each interval.
The bounds on each overlap compare their averages, so
$v(s)\asymp h^{-1}\int_r^t v$ on the full interval.
Since $\int_r^t v^2=2\mathcal A^{g,0}$,
\eqref{eq:action-ellipticity} proves \eqref{eq:action-speed}.

\paragraph*{Endpoint variations.}
Apply the nonsmooth Euler condition
\citep[Theorem~18.1]{clarke-2013-functional} to
$L(s,z,v)=\tfrac12v^{\mathsf T}g(s,z)v$, with fixed endpoints and
velocity sets $\overline B(\dot\gamma(s),1)$.
The spatial Lipschitz bound on $g$ and \eqref{eq:action-speed}
give the required integrable Lipschitz bound for $L$.
The Weierstrass condition identifies the costate as
$\pi=g(\cdot,\gamma)\dot\gamma$, and the Euler inclusion gives
\begin{equation*}
\pi\in W^{1,\infty},\qquad
\|\pi\|_\infty\le CD/h,\qquad
|\dot\pi|\le CL_aD^2/h^2\quad\text{a.e.}
\end{equation*}
Since $\dot\gamma=a(\cdot,\gamma)\pi$ a.e., the velocity has a
continuous representative up to the endpoints.

The speed bound \eqref{eq:action-speed} gives
$|\gamma(s)-y|\le CD(s-r)/h$.
After translation and rescaling,
\citep[Lemma~A.5]{norris-stroock-1991} applies with $\beta=1$.
Applying it also after time reversal and exchanging the perturbed
and unperturbed data proves joint local Lipschitz continuity of $\mathcal A^{g,0}$
in $(r,y;t,x)$ for $r<t$.

At a differentiability point of $\Phi(t,x)=\mathcal A^{g,0}(r,y;t,x)$,
concatenation with a straight final segment gives
\begin{equation*}
\partial_t\Phi+D_x\Phi\cdot v
 \le\tfrac12v^{\mathsf T}g(t,x)v\qquad\text{for }v\in\R^d.
\end{equation*}
For an attaining path, continuity of $a,\pi$ and subpath optimality
give equality at $v=a(t,x)\pi(t)$ by taking its final segment.
Strict convexity in $v$ implies
$D_x\Phi=\pi(t)$ and
$\partial_t\Phi=-\tfrac12\pi(t)^{\mathsf T}a(t,x)\pi(t)$.
Time reversal gives the initial endpoint identities.
Thus, a.e. in the endpoint data,
\begin{equation*}
\begin{gathered}
cD^2/h^2\le\partial_r\mathcal A^{g,0}\le CD^2/h^2,
\quad cD^2/h^2\le-\partial_t\mathcal A^{g,0}\le CD^2/h^2,\\
|D_y\mathcal A^{g,0}|+|D_x\mathcal A^{g,0}|\le CD/h.
\end{gathered}
\end{equation*}
Integrate along the affine variations in the statement.
For time variations the duration stays at least $h/2$; for
spatial variations the distance is at most
$D+u(|z_r|+|z_t|)$, $0\le u\le1$.
Fubini justifies integration on almost every parallel translate
of a prescribed segment, and continuity passes the bound to that
segment, including boundary times. This proves
\eqref{eq:action-endpoint-variation}.
\end{proof}

Mollify constant time extensions of $a,b$ with a common nonnegative kernel.
The resulting $(a_n,b_n)$ preserve the reference bounds;
with $\rho_n=\rho_{a_n,b_n}$, we have
\begin{equation}
\begin{aligned}
a_n&\longrightarrow a
 &&\text{uniformly on }[0,T]\times\R^d,\\
b_n&\longrightarrow b
 &&\text{in }L^p_{\mathrm{loc}}((0,T)\times\R^d),\quad 1\le p<\infty,\\
\rho_n&\longrightarrow\rho
 &&\text{locally uniformly for }0\le r<t\le T.
\end{aligned}
\label{eq:reference-kernel-approximation}
\end{equation}
The kernel convergence follows from
Lemma~\ref{lem:kernel-stability}\ref{item:kernel-stability}.

\begin{proposition}\label{prop:kernel-action}
Put $M_y=D^2/h$ and
define $E_{\mathrm{ker}}=\log\{h^{d/2}\rho(r,y;t,x)\}+\mathcal A^{g,0}(r,y;t,x)$.
Uniformly for $0\le r<t\le T$ and $x,y\in\R^d$,
\begin{equation*}
|E_{\mathrm{ker}}|\le C(1+M_y^{2/3}).
\end{equation*}
Here $C$ depends only on $\mathfrak R$.
\end{proposition}

\begin{proof}
First take smooth coefficients with bounded derivatives and
common reference bounds. With
$\widetilde b=b-\tfrac12\operatorname{div}a$, the generator is
$\tfrac12\operatorname{div}(aD)+\widetilde b\cdot D$ and
$\|\widetilde b\|_\infty\le B_b+CL_a$.
\paragraph*{Upper kernel bound.}
Apply \citep[Theorem~2.7]{norris-stroock-1991} with coefficient
quadruple $(a/2,2a^{-1}\widetilde b,0,0)$.
In this proof only, write $\mathcal A^{g,\widetilde b}$ for the infimum
of $\tfrac12\int_r^t|\dot\gamma-\widetilde b(s,\gamma)|_{g(s,\gamma)}^2\,\mathrm{d}s$
over the fixed-endpoint paths in \eqref{eq:action}.
This is the energy in the cited theorem. The straight path bounds it
by $C(M_y+h)$, so the theorem yields
\begin{equation}
\rho(r,y;t,x)\le Ch^{-d/2}(1+M_y)^N
                  \mathrm{e}^{-\mathcal A^{g,\widetilde b}(r,y;t,x)}.
\label{eq:divergence-action-upper}
\end{equation}
This estimate holds for all $M_y\ge0$ and requires only uniform
continuity of $a,\widetilde b$; $C,N$ depend only on
$d,T$, ellipticity and $\|\widetilde b\|_\infty$, independently
of coefficient continuity moduli and higher derivatives.

Fix $\delta>0$. For a path whose energy for either action is at most
its infimum plus $\delta$, ellipticity and boundedness of
$\widetilde b$ give
$\int_r^t|\dot\gamma|^2\le C(M_y+h+\delta)$.
Comparing the two quadratic integrands and applying Cauchy--Schwarz
gives an error $C\sqrt{h(M_y+h+\delta)}+Ch$.
Use a $\delta$-minimizer for each action as a competitor for the other
and let $\delta\downarrow0$. Since $hM_y=D^2$, this proves
$|\mathcal A^{g,\widetilde b}-\mathcal A^{g,0}|\le C(D+h)$.
Together with \eqref{eq:divergence-action-upper}, this gives
\begin{equation}
E_{\mathrm{ker}}\le C\{1+D+\log(2+M_y)\}.
\label{eq:kernel-action-upper}
\end{equation}

\paragraph*{Lower kernel bound.}
For $M_y\le1$, \eqref{eq:kernel-gaussian} and
\eqref{eq:action-ellipticity} give the claimed remainder bound.
Assume $M_y>1$ for the lower bound. For $\tau>0$, average the
Lagrangian before minimizing $\overline{\mathcal E}_\tau(r,y;t,x):=$
\begin{equation*}
\inf_{\substack{\gamma\in H^1([r,t];\R^d)\\
                    \gamma(r)=y,\ \gamma(t)=x}}
       \frac12\int_r^t\int_{\R^d}
       |\dot\gamma(s)-\widetilde b(s,\gamma(s)+w)|_{g(s,\gamma(s)+w)}^2\,
       \gamma_{2\tau I_d}(w)\,\mathrm{d}w\,\mathrm{d}s.
\end{equation*}
Put $r_\varepsilon=r+\varepsilon h$ and
$t_\varepsilon=t-\varepsilon h$.
Equation~(3.6) in
\citep[proof of Theorem~3.5]{norris-stroock-1991}, with the same
coefficient quadruple and averaging density $\gamma_{2I_d}$,
keeps these truncated endpoints and the smoothing parameter
$\beta=\vartheta>0$ free. For $0<\varepsilon<1/4$, it gives
\begin{equation}
\log\{h^{d/2}\rho(r,y;t,x)\}
 \ge-\overline{\mathcal E}_{\vartheta h}
                 (r_\varepsilon,y;t_\varepsilon,x)
     -C\{1+\vartheta/\varepsilon+\vartheta^{-1}
                         +\varepsilon h+\sqrt{h/\vartheta}\}.
\label{eq:averaged-action-lower}
\end{equation}
Here the smoothing variance is $2\vartheta h I_d$.
The remaining auxiliary error is bounded by
$C\inf_{\varepsilon'>0}\{h\varepsilon'+(\vartheta\varepsilon')^{-1}\}=2C\sqrt{h/\vartheta}$.
Thus $C$ depends only on $d$, ellipticity and
$\|\widetilde b\|_\infty$.
By the time upper bound in \eqref{eq:action-endpoint-variation},
\begin{equation*}
\mathcal A^{g,0}(r_\varepsilon,y;t_\varepsilon,x)
 \le\mathcal A^{g,0}(r,y;t,x)+C\varepsilon M_y.
\end{equation*}
Evaluate the averaged energy on a minimizer $\gamma$ of this
truncated action. Its energy satisfies
$\int_{r_\varepsilon}^{t_\varepsilon}|\dot\gamma|^2\,\mathrm{d}s\le CM_y$.
Spatial Lipschitz continuity gives
$\int|g(s,\gamma(s)+z)-g(s,\gamma(s))|
       \gamma_{2\vartheta h I_d}(z)\,\mathrm{d}z\le C\sqrt{\vartheta h}$;
the terms containing $\widetilde b$ are bounded by $C(D+h)$.
Consequently,
\begin{equation*}
\overline{\mathcal E}_{\vartheta h}(r_\varepsilon,y;t_\varepsilon,x)
 \le\mathcal A^{g,0}(r,y;t,x)
      +C\{\varepsilon M_y+M_y\sqrt{\vartheta h}+D+h\}.
\end{equation*}
Inserting this in \eqref{eq:averaged-action-lower} and using $h\le T$ yields
\begin{equation*}
-E_{\mathrm{ker}}\le C\{1+D+\varepsilon M_y+M_y\sqrt{\vartheta h}
                +\vartheta/\varepsilon+\vartheta^{-1}
                +\sqrt{h/\vartheta}\}.
\end{equation*}
Choose
\(
\varepsilon=\tfrac14M_y^{-1/3},\,
\vartheta=M_y^{-2/3}.
\)
Since $M_y>1$ and $h\le T$, each error is at most $CM_y^{2/3}$.
Together with \eqref{eq:kernel-action-upper} and
$D\le\sqrt{TM_y}$, this proves the claimed remainder bound.

The constants are uniform for the smooth coefficient approximations in
\eqref{eq:reference-kernel-approximation}.
The action convergence in \eqref{eq:action-stability} and kernel
convergence in \eqref{eq:reference-kernel-approximation} pass the bound
to the reference coefficients.
\end{proof}

\begin{proof}[Proof of Lemma~\ref{lem:covariance-interpolation}]
The action estimate also gives
the covariance interpolation used in
Subsection~\ref{subsec:tangent-stability}.
Let $p_y=\rho(0,y;t,x)$ and let $p_y^\eta$ be the corresponding
kernel for $\eta a$. Set
$r=(1-\theta)^{-1}$, $M=|x-y|^2/t$ and
$\mathcal A_t=\mathcal A^{g,0}(0,y;t,x)$.
Covariance scaling divides the action by $\eta$, so
Proposition~\ref{prop:kernel-action} gives
\begin{equation*}
\begin{aligned}
\log\frac{p_y}{p_y^\eta}
 &=-\bigl(1-\eta^{-1}\bigr)\mathcal A_t+O(1+M^{2/3}),\\
\log\bigl[t^{d/2}(p_y^\eta)^r p_y^{1-r}\bigr]
 &=-r\bigl(\eta^{-1}-\theta\bigr)\mathcal A_t+O(1+M^{2/3}).
\end{aligned}
\end{equation*}
Both coefficients of $-\mathcal A_t$ are positive.
Since \eqref{eq:action-ellipticity} gives $\mathcal A_t\ge cM$,
absorbing the sublinear remainders yields, uniformly for $0<t\le T$,
\begin{equation}
p_y\le Cp_y^\eta,\qquad
(p_y^\eta)^r p_y^{1-r}\le Ct^{-d/2}\mathrm{e}^{-cM}.
\label{eq:covariance-kernel-comparison}
\end{equation}

The first inequality in \eqref{eq:covariance-kernel-comparison},
integrated against $\mu_0$, proves \eqref{eq:enlarged-covariance-domination}.
Joint convexity of $(u,v)\mapsto u^rv^{1-r}$, Tonelli's theorem
and the second inequality give
\begin{equation}
\int_{\R^d}(\rho^\eta)^r\rho^{1-r}\,\mathrm{d}x
 \le\int_{\R^d}\int_{\R^d}(p_y^\eta)^r p_y^{1-r}\,\mathrm{d}x\,\mu_0(\mathrm{d}y)
 \le C.
\label{eq:covariance-marginal-ratio}
\end{equation}
For any bounded scalar or vector function $f$, H\"older and
\eqref{eq:covariance-marginal-ratio} yield
\begin{equation*}
\|f\|_{L^2(\mu_t^\eta)}^2
 \le\|f\|_{L^2(\mu_t)}^{2\theta}
       \left(\int_{\R^d}|f|^2(\rho^\eta)^r\rho^{1-r}\,\mathrm{d}x\right)^{1-\theta}
 \le C\|f\|_{L^2(\mu_t)}^{2\theta}\|f\|_\infty^{2(1-\theta)}.
\end{equation*}
Taking $f=F(t)$ and applying H\"older with respect to $t^\alpha \mathrm{d}t$
proves the first estimate in \eqref{eq:covariance-interpolation}.
Taking $f=(h,Dh)$ at $t=T$ proves the second.
All constants depend only on $\mathfrak R,\eta,\theta$.
\end{proof}

\subsection{Bridge concentration and logarithmic gradients}
\label{subsec:conditional-bridges}

We use the kernel version of the conditional bridge law:
for $r=\tau_0<\tau_1<\cdots<\tau_n<\tau_{n+1}=t$, $z_0=y$, and $z_{n+1}=x$,
\begin{equation}
\mathbb P_{r,y;t,x}(X_{\tau_i}\in \mathrm{d}z_i,\ 1\le i\le n)
 =\frac{\prod_{i=0}^n \rho(\tau_i,z_i;\tau_{i+1},z_{i+1})}
        {\rho(r,y;t,x)}\,\mathrm{d}z_1\cdots \mathrm{d}z_n.
\label{eq:bridge-kernel-law}
\end{equation}
Conditional probabilities below refer to this version.
Put $h=t-r$, $D=|x-y|$ and $M_y=D^2/h$.

The Gaussian bounds \eqref{eq:kernel-gaussian} in
\eqref{eq:bridge-kernel-law} give, for every $R>0$,
\begin{equation}
\begin{aligned}
\mathbb P_{r,y;t,x}(|X_s-x|>R)
 &\le\min\{1,C\mathrm{e}^{CM_y-cR^2/(t-s)}\},
 &&r+h/2\le s<t,\\
\mathbb P_{r,y;t,x}(|X_s-y|>R)
 &\le\min\{1,C\mathrm{e}^{CM_y-cR^2/(s-r)}\},
 &&r<s\le r+h/2.
\end{aligned}
\label{eq:bridge-gaussian-tails}
\end{equation}
Indeed, on either half of the interval the longer-time kernel is
bounded by $Ch^{-d/2}$, and division by the denominator costs at
most $Ch^{d/2}\mathrm{e}^{CM_y}$. Integrating these tails, and using
$|X_s-x|\le D+|X_s-y|$ on the first half,
gives, with $C$ depending only on $\mathfrak R$,
\begin{equation}
\mathbb E_{r,y;t,x}|X_s-x|
 \le C(D+\sqrt h)\sqrt{\frac{t-s}{h}},\qquad r<s<t.
\label{eq:bridge-first-moment}
\end{equation}

To compare kernels at nearby endpoints, we need a common region
of intermediate points carrying most of the corresponding bridge mass
at a fixed observation time.
The Gaussian bounds \eqref{eq:bridge-gaussian-tails} do not resolve
the displacement scale $(t-s)D/h$ near the terminal endpoint,
because of the factor $\mathrm{e}^{CM_y}$.
The next lemma obtains annular concentration at this scale,
uniformly under endpoint perturbations.

\begin{lemma}\label{lem:bridge-annulus}
\label{lem:bridge-descent}
Fix $Q\ge1$, $0<\vartheta<1$ and $\varepsilon>0$.
For $0\le r<t\le T$, put $h=t-r$, $M=h/\varepsilon^2$, and assume
$Q^{-1}M\varepsilon\le |x-y|\le QM\varepsilon$.
There are $c,C,M_0,\eta_1>0$, depending only on $\mathfrak R$, $\vartheta,Q$, such that, for $M\ge M_0$,
$H=M^\vartheta$, the following estimates hold.
For $(t',x')\in\mathcal U_{\eta_1H}^{\varepsilon}(t,x)$,
\begin{equation}
\mathbb P_{r,y;t',x'}
 \{|X_{t-H\varepsilon^2}-x|\notin[cH\varepsilon,CH\varepsilon]\}\le C\mathrm{e}^{-cH}.
\label{eq:bridge-perturbed-endpoint}
\end{equation}
For $r'\ge0$ with $|y'-y|\le\eta_1H\varepsilon$
and $|r'-r|\le\eta_1H\varepsilon^2$,
\begin{equation}
\mathbb P_{r',y';t,x}
 \{|X_{r+H\varepsilon^2}-y|>CH\varepsilon\}\le C\mathrm{e}^{-cH}.
\label{eq:bridge-perturbed-initial}
\end{equation}
\end{lemma}

\begin{proof}
Put $b_\delta=\delta D^2/h^2$. We first prove that, for common positive
constants $\eta_0,c_0,C_0,K,c,C$ depending only on $\mathfrak R$,
$0<\delta\le\eta_0h$ and $b_\delta\ge K(1+M_y)^{2/3}$ imply
\begin{equation}
\mathbb P_{r,y;t,x}
 \left\{
 \bigl(|X_{t-\delta}-x|,\ |X_{r+\delta}-y|\bigr)
 \notin[c_0\delta D/h,C_0\delta D/h]^2
 \right\}
 \le C\mathrm{e}^{-cb_\delta}.
\label{eq:bridge-single-age}
\end{equation}
Set $d_z=|z-x|$ and
\begin{equation*}
G(z):=\mathcal A^{g,0}(r,y;t-\delta,z)
     +\mathcal A^{g,0}(t-\delta,z;t,x)-\mathcal A^{g,0}(r,y;t,x).
\end{equation*}
By \eqref{eq:action-splitting}, $G(z)\ge0$.
The short action is at least $d_z^2/(2\Lambda_a\delta)$ by
\eqref{eq:action-ellipticity}. Applying the time lower bound and
the spatial bound in \eqref{eq:action-endpoint-variation} to the
long action gives
\begin{equation*}
G(z)
 \ge cb_\delta-C(Dd_z+d_z^2)/h+d_z^2/(2\Lambda_a\delta)
 \ge cb_\delta-CDd_z/h+cd_z^2/\delta,
\end{equation*}
where the last inequality uses $\delta/h\le\eta_0$ with
$\eta_0$ sufficiently small. Put
$E_\delta=\{z:d_z\notin[c_0\delta D/h,C_0\delta D/h]\}$.
For
$d_z\le c_0\delta D/h$, choose $c_0$ small to absorb the
linear term into $cb_\delta$. For
$d_z\ge C_0\delta D/h$, choose $C_0$ large to absorb it
into $cd_z^2/\delta$. Hence $G(z)\ge c(b_\delta+d_z^2/\delta)$ on $E_\delta$.
By Proposition~\ref{prop:kernel-action} and
\eqref{eq:bridge-kernel-law}, the density of $X_{t-\delta}$
under $\mathbb P_{r,y;t,x}$ is at most
\begin{equation*}
C\delta^{-d/2}\exp\{-G(z)+C(1+M_y+G(z))^{2/3}\}
\le C\delta^{-d/2}\exp\{-\frac{G(z)}{2}+C(1+M_y)^{2/3}\},
\end{equation*}
since the sum of the two subinterval actions is at most
$CM_y+G(z)$ by \eqref{eq:action-ellipticity}.
Integration over $E_\delta$ gives
\begin{equation*}
\begin{aligned}
\mathbb P_{r,y;t,x}(X_{t-\delta}\in E_\delta)
&\le C\delta^{-d/2}
       \int_{\R^d}\exp\{-cb_\delta+C(1+M_y)^{2/3}-c|z-x|^2/\delta\}\,\mathrm{d}z\\
&\le C\mathrm{e}^{-cb_\delta},
\end{aligned}
\end{equation*}
where the last inequality uses $b_\delta\ge K(1+M_y)^{2/3}$
with $K$ sufficiently large.
For $X_{r+\delta}$, apply \eqref{eq:action-endpoint-variation}
at the initial endpoint and use the density
$\rho(r,y;r+\delta,z)\rho(r+\delta,z;t,x)/\rho(r,y;t,x)$.
Adding the failure probabilities for $X_{t-\delta}$ and $X_{r+\delta}$
proves \eqref{eq:bridge-single-age}.

\paragraph*{Iteration near the endpoints.}
For the target $\delta=H\varepsilon^2$,
the assumption $Q^{-1}M\varepsilon\le D\le QM\varepsilon$ gives
\begin{equation*}
Q^{-2}M\le M_y=M\left(\frac{|x-y|}{M\varepsilon}\right)^2\le Q^2M,
\qquad b_\delta=\frac HM M_y\asymp H.
\end{equation*}
If $\vartheta\le2/3$, the condition $b_\delta\ge K(1+M_y)^{2/3}$
may fail at this age. We therefore iterate \eqref{eq:bridge-single-age}
to bound $|X_{t-H\varepsilon^2}-x|$ and $|X_{r+H\varepsilon^2}-y|$
under $\mathbb P_{r,y;t,x}$.
For use when the bridge endpoints vary, we work on the wider range
$M^\vartheta/4\le H\le4M^\vartheta$.
Choose a fixed integer $J$ such that
$\gamma=\vartheta^{1/J}>2/3$. Set
$H_0=M,\quad H_j=M^{\gamma^j}$ for $1\le j<J$, $H_J=H$, and
$\sigma_j=t-\varepsilon^2H_j$ for $0\le j\le J$.
Let $\alpha_0=Q^{-1}$, $\beta_0=Q$,
$\alpha_{j+1}=c_0\alpha_j$, $\beta_{j+1}=C_0\beta_j$, and set
$E_j=\{\alpha_j\varepsilon H_j\le|X_{\sigma_j}-x|\le\beta_j\varepsilon H_j\}$.
Then $E_0$ has probability one. Inductively, on $E_j$, put
$D_j=|X_{\sigma_j}-x|$. The remaining bridge parameters satisfy
\begin{equation*}
\frac{D_j^2}{\varepsilon^2H_j}\asymp H_j,\qquad
\frac{\varepsilon^2H_{j+1}D_j^2}{(\varepsilon^2H_j)^2}
 \asymp H_{j+1},\qquad
\frac{H_{j+1}}{H_j^{2/3}}\longrightarrow\infty,\qquad
\frac{H_{j+1}}{H_j}\longrightarrow0.
\end{equation*}
Since $J$ is fixed, \eqref{eq:bridge-single-age} applies on $E_j$
to the bridge from $(\sigma_j,X_{\sigma_j})$ to $(t,x)$
with $\delta=\varepsilon^2H_{j+1}$, for every $j<J$ and all large $M$.
Conditioning on $X_{\sigma_j}$ and summing the failure probabilities gives
\begin{equation}
\begin{aligned}
&\mathbb P_{r,y;t,x}
 \{|X_{t-H\varepsilon^2}-x|
   \notin[\alpha_JH\varepsilon,\beta_JH\varepsilon]\}
 =\mathbb P_{r,y;t,x}(E_J^c)\\
&\quad\le\sum_{j<J}\mathbb P_{r,y;t,x}(E_j\cap E_{j+1}^c)
 \le\sum_{j<J}C\mathrm{e}^{-cH_{j+1}}
 \le C\mathrm{e}^{-cH}.
\end{aligned}
\label{eq:bridge-terminal-annulus}
\end{equation}
For $|X_{r+H\varepsilon^2}-y|$, take
$\sigma_j=r+\varepsilon^2H_j$ in decreasing order and replace
$|X_{\sigma_j}-x|$ by $|X_{\sigma_j}-y|$ in $E_j$.
Conditioned on $X_{\sigma_j}$, the path on $[r,\sigma_j]$ is the bridge
from $(r,y)$ to $(\sigma_j,X_{\sigma_j})$.
Applying the bound for $|X_{r+\delta}-y|$ in \eqref{eq:bridge-single-age}
with $\delta=\varepsilon^2H_{j+1}$ and summing over $j<J$ gives
\begin{equation}
\mathbb P_{r,y;t,x}
 \{|X_{r+H\varepsilon^2}-y|\notin[\alpha_JH\varepsilon,\beta_JH\varepsilon]\}
 \le C\mathrm{e}^{-cH}.
\label{eq:bridge-initial-annulus}
\end{equation}

\paragraph*{Varying the bridge endpoints.}
Increase $M_0$ so that $2M^\vartheta\le M/2$ for $M\ge M_0$,
and take $\eta_1\le1/2$.
The observation times in \eqref{eq:bridge-perturbed-endpoint}
and~\eqref{eq:bridge-perturbed-initial} then lie strictly between
their respective bridge endpoints.
To prove \eqref{eq:bridge-perturbed-endpoint}, fix
$(t',x')\in\mathcal U_{\eta_1H}^{\varepsilon}(t,x)$ and keep
$u=t-H\varepsilon^2$ fixed.
Put
$M'=(t'-r)/\varepsilon^2$, $H'=(t'-u)/\varepsilon^2$. For small $\eta_1$
and large $M$, the assumed perturbation bounds give
\begin{equation*}
\begin{gathered}
|M'-M|=|H'-H|\le\eta_1H,\qquad
(M')^\vartheta/4\le H'\le4(M')^\vartheta,\\
(2Q)^{-1}M'\varepsilon\le|x'-y|\le2QM'\varepsilon.
\end{gathered}
\end{equation*}
Apply \eqref{eq:bridge-terminal-annulus} with $(t,x,H,Q)$
replaced by $(t',x',H',2Q)$.
Under $\mathbb P_{r,y;t',x'}$, it bounds the probability of
$|X_u-x'|\notin[c_*H'\varepsilon,C_*H'\varepsilon]$
by $C\mathrm{e}^{-cH'}$.
Since $H'/H\in[1-\eta_1,1+\eta_1]$
and $|x'-x|\le\eta_1H\varepsilon$, choosing $\eta_1$ small
relative to $c_*$ proves \eqref{eq:bridge-perturbed-endpoint}.
For \eqref{eq:bridge-perturbed-initial}, apply
\eqref{eq:bridge-initial-annulus} with $(r,y,M,H,Q)$ replaced by
$(r',y',M',H',2Q)$, where $M'=(t-r')/\varepsilon^2$ and
$H'=(r+H\varepsilon^2-r')/\varepsilon^2$.
The observation time is $r'+H'\varepsilon^2=r+H\varepsilon^2$;
the bound $|y'-y|\le\eta_1H\varepsilon$ recenters the annulus at $y$.
\end{proof}

The covariance need not stay close to a spatially frozen one
along the whole bridge. For the starting-point comparison,
we use \eqref{eq:bridge-perturbed-initial} to localize to a short
initial segment.

\begin{lemma}
\label{lem:natural-kernel-ratio}
Put $h=t-r$, $D=|x-y|$ and $\ell_0=h/(D+\sqrt h)$.
There is $C$, depending only on $\mathfrak R$, such that
\begin{equation*}
\rho(r',y';t,x)\le C\rho(r,y;t,x),
\qquad r\le r'\le r+\tfrac12\ell_0^2,\quad |y'-y|\le\ell_0.
\end{equation*}
\end{lemma}

\begin{proof}
The time perturbation bound and $\ell_0^2\le h$ give
$h/2\le t-r'\le h$.
For bounded $D^2/h$, use \eqref{eq:kernel-gaussian}.
For large $M_y=D^2/h$, take $\varepsilon=h/D$,
$M=h/\varepsilon^2=M_y$, $H=M^{1/6}$ and
$u=r+H\varepsilon^2$. Since $\ell_0\le\varepsilon$ and $H\to\infty$,
\eqref{eq:bridge-perturbed-initial} gives a fixed $C_0$ such that
\begin{equation}
\mathbb P_{r',y';t,x}(|X_u-y|<C_0H\varepsilon)\ge\tfrac12.
\label{eq:natural-bridge-mass}
\end{equation}

Scale time by $\varepsilon^2$ and space by $\varepsilon$, and freeze
the covariance at $y$. The intervals $[r,u]$ and $[r',u]$ both have
rescaled length comparable to $H$.
On the tubes joining $(r,y)$ or $(r',y')$ to $(u,v)$,
$|v-y|<C_0H\varepsilon$,
$E_a\le CL_a\varepsilon H$, while the drift and spatial Lipschitz
bounds are $B_b\varepsilon$ and $L_a\varepsilon$.
Since $\varepsilon H^2\le\sqrt T M^{-1/6}\to0$,
Lemma~\ref{lem:long-tube} gives, in physical coordinates,
\begin{equation*}
\rho(r,y;u,v)\asymp\gamma_{\Sigma}(v-y),\qquad
\rho(r',y';u,v)\asymp\gamma_{\Sigma'}(v-y'),
\end{equation*}
where $\Sigma=\int_r^u a(s,y)\,\mathrm{d}s$ and
$\Sigma'=\int_{r'}^u a(s,y)\,\mathrm{d}s$.

Apply Lemma~\ref{lem:gaussian-ratio} with covariance profile
$\tau\mapsto a(r-\varepsilon^2\tau,y)$, using constant extension
in physical time, and parameters
$B_1=1$, $P=C_0$ and
\begin{equation*}
s=\frac{r-r'}{\varepsilon^2},\qquad
z=\frac{y'-y}{\varepsilon},\qquad w=\frac{v-y}{\varepsilon}.
\end{equation*}
The integrated covariances at $0$ and $s$ are
$\varepsilon^{-2}\Sigma$ and $\varepsilon^{-2}\Sigma'$.
Here $|s|,|z|\le1$, $|w|\le C_0H$, and $1\le H/2$ for large $M$.
Thus Lemma~\ref{lem:gaussian-ratio} and Gaussian symmetry bound
the ratio of these Gaussian densities above.
Consequently,
\begin{equation*}
\frac{\rho(r',y';u,v)}{\rho(r,y;u,v)}\le C,
\qquad |v-y|<C_0H\varepsilon.
\end{equation*}

Set $E=B(y,C_0H\varepsilon)$.
By \eqref{eq:natural-bridge-mass}, the bridge from $(r',y')$ to $(t,x)$
assigns probability at least $1/2$ to $\{X_u\in E\}$.

Applying the bridge law \eqref{eq:bridge-kernel-law} and the kernel ratio above gives
\begin{equation*}
\begin{aligned}
\rho(r',y';t,x)
 &\le2\int_E\rho(r',y';u,v)\rho(u,v;t,x)\,\mathrm{d}v\\
 &\le C\int_E\rho(r,y;u,v)\rho(u,v;t,x)\,\mathrm{d}v
 \le C\rho(r,y;t,x).
\end{aligned}
\end{equation*}
\end{proof}

The starting-point comparison in Lemma~\ref{lem:natural-kernel-ratio}
yields the following logarithmic gradient
bound. The same scaling is classical on smooth compact Riemannian
manifolds
\citep[Theorem~1.1]{hsu-1999}.

\begin{corollary}
\label{lem:kernel-log-gradient}
There exists $C$, depending only on $\mathfrak R$, such that
for every $0\le r<t\le T$ and
$x\in\R^d$, the function $y\mapsto\log\rho(r,y;t,x)$ is locally
Lipschitz and satisfies, for a.e. $y\in\R^d$,
\begin{equation*}
|D_y\log \rho(r,y;t,x)|
\le C\{(t-r)^{-1/2}+|x-y|/(t-r)\}.
\end{equation*}
\end{corollary}

\begin{proof}
First take smooth coefficients with common reference bounds.
Put $h=t-r$, $\ell_0=h/(|x-y|+\sqrt h)$ and choose a small fixed
$c>0$ with $4c^2\le1/2$. On
$Q=[r,r+4c^2\ell_0^2]\times B(y,2c\ell_0)$, the function
$u(s,z)=\rho(s,z;t,x)$ solves
$-\partial_su=\tfrac12a(s,z):D_z^2u+b(s,z)\cdot D_zu$.
Set $U(\tau,\eta)=u(r+\ell_0^2\tau,y+\ell_0\eta)$ and
$Q_R=(0,R^2)\times B_R$. The rescaled spatial Lipschitz
and drift bounds are at most $\sqrt T L_a$ and $\sqrt T B_b$.
The spatial Lipschitz bound gives a common VMO modulus.
For $p=2(d+2)$, apply \citep[Theorem~2.1]{krylov-2007-vmo} to
cutoff multiples of $U$ vanishing at the upper time and spatial
boundaries. Interpolation on nested cylinders absorbs the
first-derivative cutoff terms, giving
$\|U\|_{W_p^{1,2}(Q_c)}\le C\|U\|_{L^p(Q_{2c})}$.
The time trace at $\tau=0$ and the embedding
$W_p^{2-2/p}(B_c)\hookrightarrow C^1(B_c)$ give, after scaling back:
\begin{equation*}
|D_y \rho(r,y;t,x)|
 \le C\ell_0^{-1}\sup_{(s,z)\in Q}\rho(s,z;t,x).
\end{equation*}
The lower time trace also covers $r=0$, with constants independent of derivatives of $b$.
For $(s,z)\in Q$, we have $t-s\ge h/2$.
Thus Lemma~\ref{lem:natural-kernel-ratio} bounds the supremum by
$C\rho(r,y;t,x)$ and proves the asserted logarithmic gradient bound for smooth coefficients.
For the approximations in \eqref{eq:reference-kernel-approximation},
positivity gives
$\log \rho_n\to\log\rho$ locally uniformly off the time diagonal.
The logarithmic gradient bounds just proved for $\log\rho_n$
pass to the limit as local Lipschitz bounds in $y$,
giving the same inequality for a.e. $y$,
for each $0\le r<t\le T$ and $x\in\R^d$.
\end{proof}

\subsection{Posterior concentration, phase approximation and Hardy estimates}
\label{subsec:density-hardy}

In this subsection, assume
Assumption~\ref{def:initial-laws} and use the scale $\ell$ from
\eqref{eq:weighted-tail-scale}.
We first estimate the initial posterior and then transfer its
concentration to a restart time. Integrating the short kernel
comparisons against the restarted posterior gives the phase
approximation in Proposition~\ref{prop:tail-compactness}.
The resulting phase support lies in a fixed annulus away from zero,
which yields the Hardy estimate in Lemma~\ref{lem:dynamic-hardy}.

For a physical restart time $u$, write the posterior as
\begin{equation*}
\Pi_{u;t,x}(\mathrm{d}y)=\frac{\rho(u,y;t,x)\,\mu_u(\mathrm{d}y)}{\rho(t,x)},
\qquad 0\le u<t.
\end{equation*}
Here $\mu_u(\mathrm{d}y)=\rho(u,y)\,\mathrm{d}y$ for $u>0$, whereas $u=0$ uses
the prescribed initial law.

\begin{lemma}
\label{lem:initial-posterior-annulus}
There are $0<c_*<C_*$ and $c,C>0$, depending only on
$\mathfrak R,\mathfrak I_0$, such that, for
$\ell=\ell(t,x)$ and $M=t/\ell^2$,
\begin{equation}
\Pi_{0;t,x}\{\,|y-x|\notin[c_*M\ell,C_*M\ell]\,\}
 \le C\mathrm{e}^{-cM}.
\label{eq:initial-posterior-annulus}
\end{equation}

Moreover, for some $A_0\ge2$ and every $A\ge A_0$,
\begin{equation}
\Pi_{0;t,x}\{|y-x|>A R\}\le C \mathrm{e}^{-cA^2M},
\qquad R=t/\ell.
\label{eq:initial-posterior-outer-tail}
\end{equation}
The constants, including $A_0$, depend only on
$\mathfrak R,\mathfrak I_0$.
\end{lemma}

\begin{proof}
Put $R=t/\ell=M\ell$ and $k(y)=\rho(0,y;t,x)$.
For the inner bound, \eqref{eq:weighted-tail-scale} gives
$\langle y\rangle\ge\ell^{-1}(1-M^{-1/2}-2\eta)$ on
$S\cap B(x,2\eta R)$.
Take $\varphi\in C_c^\infty(B(x,2\eta R))$. By
Corollary~\ref{lem:kernel-log-gradient}, $k$ is locally Lipschitz
in its initial variable, with $|Dk|\le C\ell^{-1}(M^{-1/2}+\eta)k$
a.e. on $B(x,2\eta R)$.
Apply Lemma~\ref{lem:initial-tail-bounds}\ref{item:initial-coercivity-transfer} with
$U=B(x,2\eta R)$, $w=k$ and
$L_w=C\ell^{-1}(M^{-1/2}+\eta)$.
The lower bound for $\langle y\rangle$ gives, for small $\eta$ and large $M$,
\begin{equation*}
\ell^{-2}\int\varphi^2 k\,\mathrm{d}\mu_0
 \le C\int|D\varphi|^2 k\,\mathrm{d}\mu_0
 +C(\eta^2+M^{-1}+\ell^2)\ell^{-2}\int\varphi^2 k\,\mathrm{d}\mu_0.
\end{equation*}
Choose $\eta$ first and then $M_0$, using $\ell^2\le T/M$, to
absorb the last term. Dividing by $\rho(t,x)$ yields
\begin{equation*}
\int\varphi^2\,\mathrm{d}\Pi_{0;t,x}
 \le C\ell^2\int|D\varphi|^2\,\mathrm{d}\Pi_{0;t,x}.
\end{equation*}
Cutoffs with transition width $L_0\ell$ therefore give
\begin{equation*}
\Pi_{0;t,x}(B(x,r))
 \le (C/L_0^2)\Pi_{0;t,x}(B(x,r+L_0\ell)).
\end{equation*}
Choose $L_0$ with $C/L_0^2\le1/2$ and iterate from $r=\eta R$
through at least $cM$ nested balls contained in $B(x,2\eta R)$.
This proves the required bound on $\{|y-x|\le\eta R/2\}$,
including any mass on its boundary.

For the outer tail, let $y_0\in S$ attain the infimum in
\eqref{eq:weighted-tail-scale}.
Then $|x-y_0|\le R$ and $\ell\langle y_0\rangle\le1$.
The local measure comparison \eqref{eq:initial-local-mass-comparison}
and the lower kernel bound \eqref{eq:kernel-gaussian} on
$B(y_0,\ell)$ give
\begin{equation}
\rho(t,x)\ge c t^{-d/2}\mathrm{e}^{-V(y_0)}
              \nu(B(y_0,\ell))\mathrm{e}^{-CM}.
\label{eq:initial-posterior-denominator}
\end{equation}
Fix a small $t_1>0$.
For $t\le t_1$ and $d_y=|y-x|$,
\eqref{eq:initial-potential-increment} gives
\begin{equation*}
V(y_0)-V(y)
 \le C\{d_y/\ell+M+d_y^2+R^2\}.
\end{equation*}
Since $R^2=tM$, the quadratic term $Cd_y^2$ is absorbed by the
Gaussian exponent after reducing $t_1$; Young's inequality
absorbs $Cd_y/\ell$ with a remainder $CM$.
Consequently
\begin{equation*}
\mathrm{e}^{-V(y)}k(y)
 \le C t^{-d/2}\mathrm{e}^{-V(y_0)}
            \mathrm{e}^{CM-c|y-x|^2/t}.
\end{equation*}
Set $d_{\mathrm{dbl}}=\log_2K_{\mathrm{dbl}}$.
For $A\ge2$, \eqref{eq:initial-ball-comparison} with support center $y_0$
and \eqref{eq:initial-posterior-denominator} imply
\begin{equation*}
\Pi_{0;t,x}\{|y-x|>AR\}
 \le C\sum_{j\ge0}(2^jAM)^{d_{\mathrm{dbl}}}
       \mathrm{e}^{CM-c4^jA^2M}\le C\mathrm{e}^{-cA^2M},
\end{equation*}
uniformly for $A\ge A_0$, once $A_0$ is sufficiently large.
Here $M\ge1$ allows the polynomial factor and $CM$ to be
absorbed into a fixed fraction of $cA^2M$.

For $t\ge t_1$, \eqref{eq:weighted-scale-comparison} and
Lemma~\ref{lem:initial-tail-bounds}\ref{item:positive-time-density} give
$\rho(t,x)\ge c t^{-d/2}\mathrm{e}^{-CM}$.
The upper bound in \eqref{eq:kernel-gaussian} and $\mu_0(\R^d)=1$ give
\begin{equation*}
\int_{|y-x|>AR}k(y)\,\mathrm{d}\mu_0(y)
 \le Ct^{-d/2}\mathrm{e}^{-cA^2M}.
\end{equation*}
Increasing $A_0$ absorbs the denominator factor $\mathrm{e}^{CM}$ and
proves \eqref{eq:initial-posterior-outer-tail} also for $t\ge t_1$.
The outer estimate holds for all $M\ge1$. Together with the inner bound,
it proves \eqref{eq:initial-posterior-annulus} for $M\ge M_0$;
increasing $C$ covers $1\le M<M_0$ since probabilities are at most one.
\end{proof}

\begin{lemma}
\label{lem:restarted-posterior-annulus}
Fix $0<\vartheta<1$. There are $c,C,M_0,\eta_1>0$ and radii
$0<c_{\mathrm{pos}}\le C_{\mathrm{pos}}$, depending only on
$\mathfrak R,\mathfrak I_0,\vartheta$, with the following property.
For $\ell=\ell(t,x)$, $M=t/\ell^2\ge M_0$,
$H=M^\vartheta$, $u=t-H\ell^2$,
\begin{equation*}
\sup_{(t',x')\in\mathcal U_{\eta_1H}^\ell(t,x)}
\Pi_{u;t',x'}(\mathcal V_H^c)\le C\mathrm{e}^{-cH},
\qquad
\mathcal V_H=\{v:c_{\mathrm{pos}}H\ell\le|v-x|
                         \le C_{\mathrm{pos}}H\ell\}.
\end{equation*}
\end{lemma}

\begin{proof}
Increase $M_0$ so that $H\le M/2$, and take
$\eta_1\le1/2$. Then $0<u<t'$ and $\eta_1H\le M/4$ throughout
the stated window.
Integrating the bridge law \eqref{eq:bridge-kernel-law} against
$\Pi_{0;t',x'}$ gives, for every Borel $E\subset\R^d$,
\begin{equation*}
\Pi_{u;t',x'}(E)
 =\int\mathbb P_{0,y;t',x'}(X_u\in E)\,\Pi_{0;t',x'}(\mathrm{d}y).
\end{equation*}
By \eqref{eq:weighted-scale-stability}, the scale and depth at
$(t',x')$ have ratio $1+O(\eta_1H/M)$ to $\ell,M$.
Let $G'$ be the annulus in \eqref{eq:initial-posterior-annulus}
evaluated at $(t',x')$.
Since $\eta_1H/M=\eta_1M^{\vartheta-1}\to0$, for large $M$ we have
\begin{equation*}
G'\subset\{y:Q^{-1}M\ell\le|y-x|\le QM\ell\},
\end{equation*}
where $Q$ depends only on the initial annulus constants.
Choose $\eta_1$ no larger than the endpoint perturbation threshold
in Lemma~\ref{lem:bridge-annulus} for this $Q$.
Combining \eqref{eq:bridge-perturbed-endpoint} and
\eqref{eq:initial-posterior-annulus} gives
\begin{equation*}
\Pi_{u;t',x'}(\mathcal V_H^c)
 \le\Pi_{0;t',x'}((G')^c)
     +\sup_{y\in G'}\mathbb P_{0,y;t',x'}(X_u\notin\mathcal V_H)
 \le C\mathrm{e}^{-cH}.
\end{equation*}
\end{proof}

Fix a center $(t,x)$ and put $\ell=\ell(t,x)$, $M=t/\ell^2$.
For $0<H\ell^2\le t$, use the constant time extension of $a$ and set
\begin{equation}
\begin{gathered}
u_H=t-H\ell^2,\qquad
\Sigma_H(s)=\int_{-H}^s a(t+\ell^2\tau,x)\,\mathrm{d}\tau,
\qquad -H\le s\le H/2,\\
C_H=\ell^2\Sigma_H(0),\qquad
\pi_H=\ell C_H^{-1}(x-X_{u_H}).
\end{gathered}
\label{eq:bridge-phase-vectors}
\end{equation}
Here $\pi_H$ normalizes the restart displacement by its integrated covariance.

Set $H=M^{1/6}$.
This choice balances the bounds on
$H^{-1}$, $\ell H^2$ and $\sqrt{\ell H}$ in the proof of
Lemma~\ref{lem:posterior-finite-window}: since $\ell\le\sqrt{T/M}$,
all three are $O(M^{-1/6})$.
For all sufficiently large $M$,
Lemma~\ref{lem:restarted-posterior-annulus} with $\vartheta=1/6$
gives $\Pi_{u_H;t,x}(\mathcal V_H)\ge1/2$. Define
\begin{equation}
\lambda_{t,x}
 =\left(v\mapsto\tfrac\ell2 C_H^{-1}(x-v)\right)_\#
      \frac{\1_{\mathcal V_H}\Pi_{u_H;t,x}}
           {\Pi_{u_H;t,x}(\mathcal V_H)}.
\label{eq:posterior-phase-law}
\end{equation}
Here $T_\#\nu$ denotes pushforward by $T$; thus $\lambda_{t,x}$
is the law of $\pi_H/2$ under the posterior conditioned on $\mathcal V_H$.
For every bounded Borel test $\varphi$, conditioning gives
\begin{equation}
\left|\int\varphi\,\mathrm{d}\lambda_{t,x}
 -\int\varphi\!\left(\tfrac\ell2 C_H^{-1}(x-v)\right)
                 \Pi_{u_H;t,x}(\mathrm{d}v)\right|
\le2\|\varphi\|_\infty
                  \Pi_{u_H;t,x}(\mathcal V_H^c).
\label{eq:phase-conditioning-error}
\end{equation}
Ellipticity places the supports of $\lambda_{t,x}$
in a common compact annulus
$\mathcal Z_{\mathrm{ph}}$ with radii
$0<m_{\mathrm{ph}}\le M_{\mathrm{ph}}$ depending only on
$\mathfrak R,\mathfrak I_0$.
Elsewhere, use a fixed probability law on $\mathcal Z_{\mathrm{ph}}$ for $\lambda_{t,x}$.

\begin{lemma}
\label{lem:posterior-finite-window}
For each $B>0$, there are $C_B,M_B$, depending only on
$B,\mathfrak R,\mathfrak I_0$, such that, whenever
$M\ge M_B$ and $(s,z)\in\mathcal U_B$ with
$s\in I^{\mathrm{sc}}(t;\ell)$,
\begin{equation*}
\left|\log\frac{\rho(t+\ell^2s,x+\ell z)}
 {\rho(t,x)W_{\lambda_{t,x},a(t,x)}(s,z)}\right|
 \le C_B\{M^{-1/6}+\omega_a(B\ell^2)\}.
\end{equation*}
\end{lemma}

\begin{proof}
Put $H=M^{1/6}$, $(t',x')=(t+\ell^2s,x+\ell z)$ and
$\xi(v)=\tfrac\ell2 C_H^{-1}(x-v)$, $w=(v-x)/\ell$.
For $v\in\mathcal V_H$, $|\xi(v)|\le C$ and
$\Sigma_H(0)^{-1}w=-2\xi(v)$ by
\eqref{eq:bridge-phase-vectors}.
After scaling time by $\ell^2$ and space by $\ell$, the spatial
Lipschitz and drift bounds are $L_a\ell$ and $B_b\ell$.
Use $\tau\mapsto a(t+\ell^2\tau,x)$ as comparison covariance in
Lemma~\ref{lem:long-tube}. The tubes from $w$ to $0$ or $z$
have covariance error $E_a\le C\ell H$ and length comparable to $H$.
Since $\ell\le\sqrt{T/M}$ and $H\asymp M^{1/6}$,
$\sqrt H(E_a+B_b\ell)\le CM^{-1/4}$.
Lemma~\ref{lem:long-tube} therefore gives
\begin{equation*}
\left|\log\frac{\ell^d\rho(u_H,v;t',x')}
                         {\gamma_{\Sigma_H(s)}(z-w)}\right|
 +\left|\log\frac{\ell^d\rho(u_H,v;t,x)}
                         {\gamma_{\Sigma_H(0)}(-w)}\right|
 \le\mathcal E_H,
\qquad
\mathcal E_H+H^{-1}\le C_BM^{-1/6}.
\end{equation*}
Here the error before the final bound is at most
$C_B\{\ell H^2+\sqrt{\ell H}+H^{-1}+\mathrm{e}^{-cH}\}$.
Only the terminal covariance increment is frozen:
$|\Sigma_H(s)-\Sigma_H(0)-s a(t,x)|\le B\omega_a(B\ell^2)$.
Retaining $\Sigma_H(0)$ avoids an error proportional to
$H\omega_a(H\ell^2)$, which need not vanish.

The Gaussian ratio expansion in Lemma~\ref{lem:gaussian-ratio}, with
$B_1=B$, now gives
\begin{equation*}
\left|\log\frac{\rho(u_H,v;t',x')}{\rho(u_H,v;t,x)}
       -2\{s\xi(v)^{\mathsf T}a(t,x)\xi(v)-\xi(v)\cdot z\}\right|
 \le C_B\{M^{-1/6}+\omega_a(B\ell^2)\}.
\end{equation*}
Both endpoint posteriors discard at most $\eta=C\mathrm{e}^{-cH}$ outside
$\mathcal V_H$ by Lemma~\ref{lem:restarted-posterior-annulus}, once
$B\le\eta_1H$ and $\eta\le1/2$.
Use the identity
\begin{equation*}
\frac{\rho(t',x')}{\rho(t,x)}
 =\left(\int_{\mathcal V_H}
       \frac{\rho(u_H,v;t',x')}{\rho(u_H,v;t,x)}
       \frac{\Pi_{u_H;t,x}(\mathrm{d}v)}{\Pi_{u_H;t,x}(\mathcal V_H)}\right)
       \frac{\Pi_{u_H;t,x}(\mathcal V_H)}
            {\Pi_{u_H;t',x'}(\mathcal V_H)}.
\end{equation*}
The last factor has absolute logarithm at most $4\eta$.
Integrating the kernel ratio bounds and using
\eqref{eq:posterior-phase-law} proves
the fixed-window comparison.
\end{proof}

\begin{proof}[Proof of Proposition~\ref{prop:tail-compactness}]
Since $\ell^2\le T/M$, Lemma~\ref{lem:posterior-finite-window} gives
the finite-scale comparison with the stated dependence
of its threshold.

Along any sequence in \eqref{eq:doubling-tail-centers}, compactness
of $\mathcal Z_{\mathrm{ph}}$ and the matrix bounds give a subsequence
with joint limits $\lambda_{t_j,x_j}\rightharpoonup\lambda$ and
$(a,A,q)(t_j,x_j)\to(a_*,A_*,q_*)$.
The phase weights converge locally uniformly to $W_{\lambda,a_*}>0$,
so Lemma~\ref{lem:posterior-finite-window} gives the realization
\eqref{eq:whole-line-convergence} on the available physical time intervals.

For $(\lambda,a_*,A_*,q_*)\in\mathfrak M_+(I)$, take a weak
subsequential limit $\widetilde\lambda$ of $\lambda_{t_j,x_j}$
along a defining sequence.
By \eqref{eq:whole-line-convergence} and
Lemma~\ref{lem:posterior-finite-window} at $s=0$, $\lambda$ and
$\widetilde\lambda$ have the same Laplace transform.
Uniqueness gives $\lambda=\widetilde\lambda$, hence
$\supp\lambda\subseteq\mathcal Z_{\mathrm{ph}}$.
\end{proof}

For a compactly supported probability law $\lambda$ on $\R^d$,
denote the spatial profile of its phase weight at $s=0$, which is
independent of the covariance matrix $a$, by
\begin{equation*}
W_\lambda^0(z)=W_{\lambda,a}(0,z)
 =\int_{\R^d}\mathrm{e}^{-2\zeta\cdot z}\,\lambda(\mathrm{d}\zeta),
\qquad z\in\R^d.
\end{equation*}

\begin{proof}[Proof of Lemma~\ref{lem:dynamic-hardy}]
The inverse-scale bound \eqref{eq:weighted-scale-gradient} was
proved in Subsection~\ref{subsec:initial-density-estimates}.
Fix $t\in(0,T]$ and abbreviate $\ell(x)=\ell(t,x)$.
The separation of
$\mathcal Z_{\mathrm{ph}}$ from zero gives a common coercivity constant
for the exponential weights, and hence for their mixtures.
For $f\in C_c^\infty(\R^d)$ and a single slope $\zeta$,
integration by parts gives
\begin{equation*}
\int |Df|^2\mathrm{e}^{-2\zeta\cdot z}\,\mathrm{d}z
 =\int\bigl(|D(\mathrm{e}^{-\zeta\cdot z}f)|^2
                  +|\zeta|^2|\mathrm{e}^{-\zeta\cdot z}f|^2\bigr)\,\mathrm{d}z.
\end{equation*}
Integrating against a probability law $\lambda$ supported on
$\mathcal Z_{\mathrm{ph}}$ gives
\begin{equation}
m_{\mathrm{ph}}^2\|f\|_{L^2(W_\lambda^0\,\mathrm{d}z)}^2
 \le\|Df\|_{L^2(W_\lambda^0\,\mathrm{d}z)}^2.
\label{eq:phase-hardy}
\end{equation}
Local smooth approximation extends this inequality to compactly
supported $H^1$ functions.

Let $L\ge2$, to be fixed by absorption below. For this $L$, choose
$R$ with $R^2\ge16L$ so that the finite-scale comparison in Proposition~\ref{prop:tail-compactness} has logarithmic error at most one on $B(x,2L\ell(x))$ whenever
$M(t,x)\ge R^2$.
Set $r=\min\{\ell,\sqrt t/R\}$.
Since $\operatorname{Lip}(r)\le R^{-2}$ and $R^2\ge16L$,
\eqref{eq:admissible-spatial-scale} holds on $\R^d$.
Lemma~\ref{lem:parabolic-partition}\ref{item:spatial-partition}
provides a square partition supported in $B(x_i,2Lr_i)$,
where $r_i=r(x_i)$ and $\ell_i=\ell(x_i)$, with
\begin{equation}
\sum_i|D\chi_i|^2
 \le C_dL^{-2}r^{-2}
 \le C_dL^{-2}\ell^{-2}+C_dR^2t^{-1}.
\label{eq:hardy-partition-gradient}
\end{equation}

On $\supp\chi_i$, if $M(t,x_i)<R^2$, then
$\ell^{-1}\le R/\sqrt t+2Lr_i/t\le2R/\sqrt t$.
Otherwise $r_i=\ell_i$ and
$|\ell^{-1}/\ell_i^{-1}-1|\le2L/R^2\le1/8$ there.
For compactly supported $f\in H^1$, apply \eqref{eq:phase-hardy}
with $\lambda=\lambda_{t,x_i}$ to
$(\chi_i f)(x_i+\ell_i\,\cdot)$ on these latter supports.
Rescaling and using Proposition~\ref{prop:tail-compactness} at
$s=0$ gives
\begin{equation*}
\int\ell^{-2}|\chi_i f|^2\,\mathrm{d}\mu_t
 \le C\ell_i^{-2}\int|\chi_i f|^2\,\mathrm{d}\mu_t
 \le C\int|D(\chi_i f)|^2\,\mathrm{d}\mu_t.
\end{equation*}
The identity $\sum_i\chi_iD\chi_i=0$ gives
\begin{equation}
\sum_i|D(\chi_i f)|^2
 =|Df|^2+\sum_i|D\chi_i|^2|f|^2\quad\text{a.e.}
\label{eq:square-partition-gradient}
\end{equation}
Summing over both types of supports and using
\eqref{eq:hardy-partition-gradient} and
\eqref{eq:square-partition-gradient} yields
\begin{equation*}
\|\ell^{-1}f\|_{L^2(\mu_t)}^2
 \le C\|Df\|_{L^2(\mu_t)}^2
      +CL^{-2}\|\ell^{-1}f\|_{L^2(\mu_t)}^2
      +CR^2t^{-1}\|f\|_{L^2(\mu_t)}^2.
\end{equation*}
Here $C$ depends only on $d,m_{\mathrm{ph}}$.
Choose $L$ so that $CL^{-2}\le1/2$, and then $R$ as above.
The left side is finite for compactly supported $f$, so the middle
term is absorbed. These choices depend only on
$\mathfrak R,\mathfrak I_0$, giving \eqref{eq:weighted-hardy}.

For $f,Df\in L^2(\mu_t)$, positivity and continuity of $\rho(t,\cdot)$
give $f\in H^1_{\mathrm{loc}}$. Apply the estimate to
$\chi_nf$, with the cutoffs from \eqref{eq:spatial-cutoffs}.
Letting $n\to\infty$ and using
Fatou's lemma extends it to this class.
\end{proof}

\subsection{Comparison of bounded reference drifts}
\label{subsec:density-drift-comparison}

The logarithmic gradient and bridge moment estimates compare kernels
with different bounded drifts. The initial posterior tail then gives
the corresponding comparison of marginal densities.

\begin{lemma}
\label{lem:kernel-drift-comparison}
For the same covariance $a$, let $p_b$ and $p_\beta$ denote the
transition densities with bounded Borel drifts $b,\beta$.
Then
\begin{equation*}
\left|\log\frac{p_b(r,y;t,x)}{p_\beta(r,y;t,x)}\right|
 \le C\|b-\beta\|_\infty\bigl(|x-y|+\sqrt{t-r}\bigr).
\end{equation*}
The constant depends only on the common reference data, including
a common bound for $b,\beta$.
\end{lemma}

\begin{proof}
First take smooth coefficients and put
$b_\theta=\beta+\theta(b-\beta)$, $0\le\theta\le1$.
Duhamel differentiation of the transition operator, followed by
\eqref{eq:bridge-kernel-law}, gives
\begin{equation*}
\partial_\theta\log p_{b_\theta}(r,y;t,x)
 =\mathbb E^\theta_{r,y;t,x}\int_r^t
 (b-\beta)(s,X_s)\cdot D_z\log p_{b_\theta}(s,X_s;t,x)\,\mathrm{d}s.
\end{equation*}
Here $D_z$ differentiates the starting variable, and the expectation
uses the bridge with drift $b_\theta$.
For $h=t-r$ and $D=|x-y|$, Corollary~\ref{lem:kernel-log-gradient} and
\eqref{eq:bridge-first-moment} give
\begin{equation*}
|\partial_\theta\log p_{b_\theta}|
 \le C\|b-\beta\|_\infty
 \int_r^t\left\{(t-s)^{-1/2}
       +\frac{D+\sqrt h}{\sqrt h\sqrt{t-s}}\right\}\mathrm{d}s
 \le C\|b-\beta\|_\infty(D+\sqrt h).
\end{equation*}
The same integrable bound justifies the Duhamel formula by
truncation away from $s=t$.
Integrating over $\theta$ proves the estimate for smooth coefficients.
Use \eqref{eq:reference-kernel-approximation}, mollifying
$a,b,\beta$ with the same kernel. Then
$\|b_n-\beta_n\|_\infty\le\|b-\beta\|_\infty$.
The locally uniform kernel convergence and positivity pass the
inequality to the stated Borel drifts.
\end{proof}

\begin{lemma}
\label{lem:density-drift-comparison}
Fix an initial law $\mu$ satisfying Assumption~\ref{def:initial-laws}, and let $\ell$ be the scale
in \eqref{eq:weighted-tail-scale} with $S=\supp\mu$.
The two marginal densities with the same $a$ and bounded Borel drifts
$b,\beta$ satisfy
\begin{equation*}
\left|\log\frac{\rho_b^\mu(t,x)}{\rho_\beta^\mu(t,x)}\right|
 \le C\|b-\beta\|_\infty\frac t{\ell(t,x)}.
\end{equation*}
The constant depends only on the common reference and initial-law
data, including a common bound for both drifts.
\end{lemma}

\begin{proof}
Put $R=t/\ell$ and $M=t/\ell^2\ge1$.
For either initial posterior, \eqref{eq:initial-posterior-outer-tail} gives
\begin{equation*}
\int|y-x|\,\Pi_{0;t,x}(\mathrm{d}y)
 \le A_0R+CR\int_{A_0}^\infty \mathrm{e}^{-cA^2M}\,\mathrm{d}A
 \le CR.
\end{equation*}
By Jensen's inequality and Lemma~\ref{lem:kernel-drift-comparison},
\begin{equation*}
\begin{aligned}
\log\frac{\rho_b^\mu(t,x)}{\rho_\beta^\mu(t,x)}
 &=\log\int\frac{p_b(0,y;t,x)}{p_\beta(0,y;t,x)}
                         \,\Pi^\beta_{0;t,x}(\mathrm{d}y)
 \ge\int\log\frac{p_b(0,y;t,x)}{p_\beta(0,y;t,x)}
                         \,\Pi^\beta_{0;t,x}(\mathrm{d}y)\\
 &\ge-C\|b-\beta\|_\infty
          \int\bigl(|y-x|+\sqrt t\bigr)\,\Pi^\beta_{0;t,x}(\mathrm{d}y)
 \ge-C\|b-\beta\|_\infty R,
\end{aligned}
\end{equation*}
where $\sqrt t\le R$. Interchanging the drifts gives the opposite bound.
\end{proof}

\section{Recovery and identification of model contributions}
\label{sec:phase-directions}

The density estimates of Section~\ref{sec:densities} yield initial
absorption, model comparisons and the radial formulas.

Throughout this section, $I=(0,T)$ and
Assumptions~\ref{ass:coefficients} and~\ref{def:initial-laws} hold.
The initial-law comparison in Subsection~\ref{subsec:model-comparison}
uses only the conditions on $f$ in Assumption~\ref{ass:radial-structure}.
Radial symmetry of $\mu_{\mathrm{rad}}$ and the radial eigenvector
condition in \eqref{eq:radial-covariance} enter
Subsections~\ref{subsec:radial-local}--\ref{subsec:radial-infinity};
the derivative decay bound is used only in
Subsection~\ref{subsec:radial-infinity}.

\subsection{Absorption of the initial contribution}
\label{subsec:initial-absorption}

\begin{proof}[Proof of Proposition~\ref{prop:smooth-initial-absorption}]
Fix an actual initial tuple
$\mathbf T=(\theta,a_*,A_*,q_*)\in\mathfrak T(\mu_0)$ and a defining
sequence $(y_j,r_j)$ from
\eqref{eq:initial-admissible-scales}--\eqref{eq:initial-joint-limits}.
We first identify its density profile.
The bound on $DV$ in \eqref{eq:smooth-potential} permits extraction of
$r_jDV(y_j)/2\to\zeta$.
For each fixed $B$, the remainder after linearizing
$V(y_j+r_jz)$ at $y_j$ is bounded uniformly for $|z|\le B$ by
$C_Br_j^2\sup_{|w-y_j|\le Br_j}|D^2V(w)|$.
This tends to zero by \eqref{eq:smooth-potential} and boundedness of
$r_j\langle y_j\rangle$.
Thus $V(y_j+r_jz)-V(y_j)=2\zeta\cdot z+o(1)$ uniformly for $|z|\le B$.
The normalization in \eqref{eq:doubling-half-line-data} therefore identifies
\begin{equation*}
\theta(\mathrm{d}z)=c_\zeta \mathrm{e}^{-2\zeta\cdot z}\,\mathrm{d}z,\qquad
c_\zeta=\left(\int_{B_1}\mathrm{e}^{-2\zeta\cdot z}\,\mathrm{d}z\right)^{-1},\qquad
W_{\mathbf T}(\tau,z)=c_\zeta W_{\delta_\zeta,a_*}(\tau,z).
\end{equation*}
Multiplication by $c_\zeta$ does not change a gain. If $\zeta=0$,
the interval bound in Corollary~\ref{cor:whole-line-intervals}\ref{item:whole-line-intervals} and the supremum
in \eqref{eq:flat-contribution} give
$\mathfrak G_\alpha(\mathbf T)\le G^{\mathrm{flat}}(I)$,
using continuity of $A,q$ at time zero.

For $\zeta\ne0$, we construct an actual phase model in direction
$\zeta$ by moving to times large relative to $r_j^2$.
After further extraction,
$\beta_j=r_j\langle y_j\rangle\to\beta>0$;
positivity follows from the upper bound on $|DV|$.
A diagonal choice in \eqref{eq:initial-realization} gives
$\tau_j\to\infty$ with $t_j=r_j^2\tau_j\to0$ and
\begin{equation}
\sup_{1\le\tau\le2\tau_j,\ |z|\le\tau_j}
\left|\log\frac{r_j^d\rho(r_j^2\tau,y_j+r_jz)}
 {m_{y_j,r_j}c_\zeta W_{\delta_\zeta,a_*}(\tau,z)}\right|
 \longrightarrow0.
\label{eq:smooth-initial-diagonal}
\end{equation}
Since $\langle\cdot\rangle$ is $1$-Lipschitz,
\eqref{eq:weighted-tail-scale} gives
$\ell(t,y)^{-1}=t^{-1/2}+\langle y\rangle$ for $y\in S$ and $t\le1$.
Thus, with $\ell_j=\ell(t_j,y_j)$,
\begin{equation*}
\eta_j:=\frac{\ell_j}{r_j}
 =\frac1{\tau_j^{-1/2}+\beta_j}\longrightarrow\beta^{-1},\qquad
M(t_j,y_j)=\tau_j(\tau_j^{-1/2}+\beta_j)^2\longrightarrow\infty.
\end{equation*}

For each fixed $B$, the points
$(\tau_j+\eta_j^2s,\eta_jz)$, $(s,z)\in\mathcal U_B$, and
$(\tau_j,0)$ lie in $[1,2\tau_j]\times\overline B_{\tau_j}$
for all large $j$.
Taking the density ratio at these points in
\eqref{eq:smooth-initial-diagonal} gives
\begin{equation*}
\frac{\rho(t_j+\ell_j^2s,y_j+\ell_jz)}{\rho(t_j,y_j)}
 =W_{\delta_{\eta_j\zeta},a_*}(s,z)(1+o(1))
 \longrightarrow W_{\delta_{\zeta/\beta},a_*}(s,z)
\end{equation*}
uniformly on $\mathcal U_B$.
Since $t_j\to0$, the uniform temporal moduli in
Assumption~\ref{ass:coefficients} and the initial matrix limits in
\eqref{eq:initial-joint-limits} give
$(a,A,q)(t_j,y_j)\to(a_*,A_*,q_*)$.
Together with the density ratio limit, this verifies
\eqref{eq:whole-line-convergence} at the scales $\ell_j$ and proves
$(\delta_{\zeta/\beta},a_*,A_*,q_*)\in\mathfrak M_+(I)$.

It remains to compare the gains.
By homogeneity of $\mathsf R$,
Proposition~\ref{prop:whole-line-gains} assigns the same whole-line gain
to the pure phase tuples $(\delta_\zeta,a_*,A_*,q_*)$ and
$(\delta_{\zeta/\beta},a_*,A_*,q_*)$.
The interval bound in Corollary~\ref{cor:whole-line-intervals}\ref{item:whole-line-intervals}, which retains
the free initial trace and the weight $\tau^\alpha$, together with
\eqref{eq:flat-contribution}--\eqref{eq:propagation-contribution}, gives
\(
\mathfrak G_\alpha(\mathbf T)
 \le\max\{G^{\mathrm{flat}}(I),G^{\mathrm{phase}}(I)\}.
\)
Taking the supremum in \eqref{eq:initial-linear-formula} proves
the asserted bound on the initial contribution. Theorem~\ref{thm:sharp-limit}
then gives the two-term formula, whose right-hand side
does not depend on $\alpha$.
\end{proof}

\subsection{Drift and initial-law comparison}
\label{subsec:model-comparison}

Throughout this subsection, $a,A,q$ are fixed. The comparisons use the same
centers, scales and coefficient limits. For phase models,
comparability at all past model times is essential: spatial
exponential growth at a single time determines only the convex
hull of the support.

\begin{lemma}
\label{lem:phase-support-comparison}
Let $\lambda,\widetilde\lambda$ be compactly supported probability
measures and $a_*\succ0$. If
\begin{equation}
K^{-1}W_{\lambda,a_*}(s,z)
 \le W_{\widetilde\lambda,a_*}(s,z)
 \le K W_{\lambda,a_*}(s,z),
\qquad s\le0,\quad z\in\R^d,
\label{eq:comparable-phase-weights}
\end{equation}
for some $K\ge1$,
then $\supp\lambda=\supp\widetilde\lambda$.
\end{lemma}

\begin{proof}
For $\eta\in\R^d$, complete the square in
\eqref{eq:whole-line-tuple} to obtain
\begin{equation*}
\lim_{L\to\infty}\frac1{2L}
 \log W_{\lambda,a_*}(-L,-2La_*\eta)
 =|\eta|_{a_*}^2
      -\dist_{a_*}(\eta,\supp\lambda)^2,
\end{equation*}
where $\dist_{a_*}$ is the distance induced by $|\cdot|_{a_*}$.
Indeed, the exponent in the integral is
$2L(|\eta|_{a_*}^2-|\zeta-\eta|_{a_*}^2)$.
Its maximum on the compact support gives the upper bound;
every neighborhood of a maximizing support point has positive
mass and gives the matching lower bound.
The same limit formula holds with
$\lambda$ replaced by $\widetilde\lambda$.
Equation~\eqref{eq:comparable-phase-weights} makes the two limits
equal. The distance functions of the supports agree, hence so do
their zero sets.
\end{proof}

\begin{lemma}
\label{lem:bounded-drift-local-phase}
For every initial law $\mu$ satisfying Assumption~\ref{def:initial-laws}
and every pair of bounded Borel drifts $b,\widetilde b$,
\(
G^{\mathrm{phase,loc}}(I;\rho_b^\mu)
 =G^{\mathrm{phase,loc}}(I;\rho_{\widetilde b}^\mu).
\)
\end{lemma}

\begin{proof}
Put $S=\supp\mu$ and fix $y_0\in S$.
The scale definition \eqref{eq:weighted-tail-scale} gives
\begin{equation}
\frac t{\ell(t,x)}
 \le\sqrt t+|x-y_0|+t\langle y_0\rangle.
\label{eq:scale-fixed-support-point}
\end{equation}
Thus Lemma~\ref{lem:density-drift-comparison} gives, for every $R>0$,
a constant $C_R\ge1$ such that
\begin{equation*}
C_R^{-1}\rho_b^\mu(t,x)
 \le\rho_{\widetilde b}^\mu(t,x)
 \le C_R\rho_b^\mu(t,x),
\qquad 0<t\le T,\quad |x|\le R.
\end{equation*}

Fix a local phase tuple $\mathbf M=(\lambda,a_*,A_*,q_*)$ for
$\rho_b^\mu$, and choose a defining sequence $(t_j,x_j,\ell_j)$
with $\ell_j=\ell(t_j,x_j)$ and $|x_j|\le R$ for some $R>0$.
The scale $\ell$ depends only on the common initial support $S$,
so the same sequence is admissible for $\rho_{\widetilde b}^\mu$.
Applying Proposition~\ref{prop:tail-compactness} to this second density
and passing to a subsequence gives a local phase tuple
$\widetilde{\mathbf M}=(\widetilde\lambda,a_*,A_*,q_*)$.
The coefficient limits coincide because $a,A,q$ and the centers
$(t_j,x_j)$ are unchanged. By \eqref{eq:whole-line-convergence},
\begin{equation*}
\left(\frac{\rho_b^\mu(t_j+\ell_j^2s,x_j+\ell_jz)}{\rho_b^\mu(t_j,x_j)}, \frac{\rho_{\widetilde b}^\mu(t_j+\ell_j^2s,x_j+\ell_jz)}
     {\rho_{\widetilde b}^\mu(t_j,x_j)}\right)
 \longrightarrow (W_{\lambda,a_*}(s,z),
 W_{\widetilde\lambda,a_*}(s,z)).
\end{equation*}
For every fixed $L>0$, both convergences are uniform on
$\mathcal U_L\cap\{s\le0\}$; these model times lie in the available
rescaled interval for all large $j$ because
$M_j=t_j/\ell_j^2\to\infty$.

Since $\ell_j\to0$, the physical points
$(t_j+\ell_j^2s,x_j+\ell_jz)$ in each such fixed window eventually
belong to $(0,T]\times\overline B_{R+1}$.
Apply the density comparison with constant $C_{R+1}$ both at these
points and at $(t_j,x_j)$, and pass to the two limits above.
This gives \eqref{eq:comparable-phase-weights} for
$W_{\lambda,a_*}$ and $W_{\widetilde\lambda,a_*}$ with
$K=C_{R+1}^2$, independent of $L$.
Since $L$ is arbitrary, the comparison holds for every $s\le0$
and $z\in\R^d$.

Lemma~\ref{lem:phase-support-comparison} now gives
$\supp\lambda=\supp\widetilde\lambda$.
Together with the common coefficient limits, this shows that
$\mathbf M$ and $\widetilde{\mathbf M}$ give the same value of
$\sup_{\zeta\in\supp\lambda}\Phi(\mathsf R(a_*,A_*,q_*;\zeta))$
in \eqref{eq:local-escaping-contributions}.
Taking the supremum over the local phase tuples for $\rho_b^\mu$ yields
\begin{equation*}
G^{\mathrm{phase,loc}}(I;\rho_b^\mu)
 \le G^{\mathrm{phase,loc}}(I;\rho_{\widetilde b}^\mu).
\end{equation*}
Interchanging $b$ and $\widetilde b$ proves the reverse inequality.
\end{proof}

\begin{proof}[Proof of Proposition~\ref{prop:model-contribution-comparison}]
Put $\mu=\mu_{\mathrm{rad}}$ and $\widetilde\mu=\mu_0=f\mu$.
Write $\mu=\mathrm{e}^{-V}\nu$ and set $\widetilde\nu=f\nu$.
The two-sided bounds on $f$ preserve the Radon and doubling
properties and \eqref{eq:initial-coercivity}, with constants
multiplied by at most $C/c$. The potential is unchanged, so
$\widetilde\mu=\mathrm{e}^{-V}\widetilde\nu$ is admissible and has the
same support $S$ as $\mu$.

Fix an initial defining sequence $(y_j,r_j)$ from
\eqref{eq:initial-admissible-scales}, with
$r_j\langle y_j\rangle\le B$ for some $B\ge1$.
Let $\theta_j,\widetilde\theta_j$ be the normalized measures in
\eqref{eq:doubling-half-line-data}, and put
\begin{equation*}
c_j=\frac{\int_{B(y_j,r_j)}f\,\mathrm{d}\mu}{\mu(B(y_j,r_j))},
\qquad
\mathrm{d}\widetilde\theta_j(z)
 =\frac{f(y_j+r_jz)}{c_j}\,\mathrm{d}\theta_j(z).
\end{equation*}
For fixed $L\ge1$, condition~\eqref{eq:initial-flat-factor}, applied
at radius $Lr_j$ with bound $LB$, gives
\begin{equation*}
\varepsilon_j(L)
 :=\operatorname*{ess\,osc}_{B(y_j,Lr_j),\mu}\log f\longrightarrow0,
\qquad
\mathrm{e}^{-\varepsilon_j(L)}
 \le\frac{f(y_j+r_jz)}{c_j}
 \le \mathrm{e}^{\varepsilon_j(L)}
\end{equation*}
for $\theta_j$-a.e. $z\in B_L$.
The local mass bound \eqref{eq:rescaled-local-mass} implies
$\int\varphi\,\mathrm{d}(\widetilde\theta_j-\theta_j)\to0$ for every
$\varphi\in C_c(\R^d)$. The two rescalings therefore have the same
vague limits along the same coefficient sequences in
\eqref{eq:initial-joint-limits}. Their joint initial model families,
and hence their initial contributions, coincide.

Kernel positivity gives $c\rho_b^\mu\le\rho_b^{\widetilde\mu}\le C\rho_b^\mu$.
The common support gives the same localization scale.
Repeat the defining-sequence comparison in the proof of
Lemma~\ref{lem:bounded-drift-local-phase}, now with
$\rho_b^\mu$ and $\rho_b^{\widetilde\mu}$.
The global density bounds above replace the comparison on bounded
spatial sets, so the centers may also escape to infinity.
Normalization at each center gives
\eqref{eq:comparable-phase-weights} with $K=C/c$, uniformly on
all past model windows.
Thus every phase tuple for either density has a counterpart
for the other density with the same coefficient limits and phase
support. Their directional suprema in
\eqref{eq:propagation-contribution}, and hence their full phase
contributions, coincide.
The flat contribution depends only on $A,q$, so the three-term
formula in Theorem~\ref{thm:sharp-limit} proves the asserted identity.
\end{proof}

\subsection{Radial reference densities and local phases}
\label{subsec:radial-local}

\begin{proof}[Proof of Lemma~\ref{lem:radial-reference}]
We construct a radial solution using
the isotropic covariance $\kappa_{\mathrm{rad}}I_d$, then use the radial
eigenvector identity and uniqueness to identify it with the marginal
for $(a,\beta)$.

The spatial Lipschitz bound gives $\|\beta\|_\infty\le C_dL_a$.
By the radial eigenvector identity in \eqref{eq:radial-covariance},
continuity at the origin gives, for every unit $e$,
$a(t,0)e=\kappa_{\mathrm{rad}}(t,0)e$ with the same limiting
eigenvalue. Thus $a(t,0)=\kappa_{\mathrm{rad}}(t,0)I_d$.
Comparison along one fixed ray shows that
$\kappa_{\mathrm{rad}}$ inherits the spatial Lipschitz bound and
the time modulus of $a$. Therefore
$\widetilde a(t,x)=\kappa_{\mathrm{rad}}(t,|x|)I_d$ and
$\widetilde\beta=\tfrac12D_x\kappa_{\mathrm{rad}}(t,|x|)$
are also admissible reference coefficients.

For the drift \eqref{eq:radial-reference-drift}, the forward equation is
\begin{equation}
\partial_t U=\tfrac12\operatorname{div}(aDU)
\label{eq:radial-divergence-equation}
\end{equation}
in distributions.
For nonnegative initial densities in $L^1\cap L^2$,
\citep[Theorem~10(vi)]{aronson-1968} and the kernel identification
in the proof of Lemma~\ref{lem:kernel-stability} show that the
marginal densities for $(a,\beta)$ and
$(\widetilde a,\widetilde\beta)$ are weak solutions of their respective
divergence-form Cauchy problems.
The principal matrices $a/2$ and $\widetilde a/2$ are bounded and
uniformly elliptic, and all lower-order terms vanish.
Thus \citep[Corollary~4.1]{aronson-1968} gives uniqueness
among nonnegative weak solutions for each of these problems.

Now regularize the initial law by
$\mu_\delta=\gamma_{\delta I_d}*\mu_{\mathrm{rad}}$.
Its density is radial and belongs to $L^1\cap L^2$.
Let $U_\delta$ be the marginal density for
$(\widetilde a,\widetilde\beta)$ and initial law $\mu_\delta$.
Its divergence equation and initial density are invariant under
orthogonal transformations, so the preceding uniqueness result
makes $U_\delta$ radial.
Its weak spatial gradient is radial a.e., so
$aDU_\delta=\kappa_{\mathrm{rad}}DU_\delta$.
Consequently $U_\delta$ also solves
\eqref{eq:radial-divergence-equation}, and uniqueness gives
$U_\delta=\rho_\beta^{\mu_\delta}$.
For fixed $t>0,x$, the kernel for $(a,\beta)$ is a bounded
continuous function of its starting point, by
Lemma~\ref{lem:kernel-stability}\ref{item:kernel-estimates}.
As $\delta\downarrow0$, weak convergence
$\mu_\delta\rightharpoonup\mu_{\mathrm{rad}}$ gives
$U_\delta(t,x)=\rho_\beta^{\mu_\delta}(t,x)
\to\rho_\beta^{\mu_{\mathrm{rad}}}(t,x)$.
The limit is radial in $x$ and continuous, giving the asserted
identity everywhere, including at the origin.
Positivity follows from the kernel lower bound.
\end{proof}

\begin{proof}[Proof of Proposition~\ref{prop:radial-local-phase}\ref{item:radial-local-phase}]
Lemma~\ref{lem:radial-reference} and
Lemma~\ref{lem:bounded-drift-local-phase} reduce the claim to the drift
$\beta$. Work with its radial density
$\rho_\beta^{\mu_{\mathrm{rad}}}$ throughout the proof.

For the upper bound, fix
$(\lambda,a_*,A_*,q_*)\in\mathfrak M_+^{\mathrm{loc}}(I)$ and choose
a defining sequence with bounded centers $(t_j,x_j)$.
After extraction, $x_j\to x_*$.
As in the proof of Proposition~\ref{prop:coefficient-upper-bound},
\eqref{eq:weighted-scale-comparison} gives $t_j\to0$ and
$x_j\notin S$ eventually. Put $\ell_j=\ell(t_j,x_j)$ and $D_j=|x_j|$.
If $x_*\ne0$, then $D_j/\ell_j\to\infty$ and
$e_j=x_j/D_j\to e=x_*/|x_*|$.

For fixed $z$, the point on the ray through $x_j$ with the same
radius as $x_j+\ell_jz$ has rescaled coordinate
\begin{equation*}
z'_j=\frac{|x_j+\ell_jz|-D_j}{\ell_j}e_j
       \longrightarrow(e\cdot z)e.
\end{equation*}
Radial symmetry and locally uniform model convergence give
$W_{\lambda,a_*}(0,z)=W_{\lambda,a_*}(0,(e\cdot z)e)$.
For $v\perp e$, differentiating twice at $z=0$ yields
$\int(\zeta\cdot v)^2\,\mathrm{d}\lambda=0$.
Thus radial symmetry forces
\begin{equation}
\supp\lambda\subset\R e\setminus\{0\},
\qquad e=x_*/|x_*|,\quad x_*\ne0.
\label{eq:radial-local-support}
\end{equation}
For every $\zeta\in\supp\lambda$, homogeneity and the joint
coefficient limit give
\begin{equation*}
\mathsf R(a_*,A_*,q_*;\zeta)
 =\mathsf R(a_*,A_*,q_*;e)
 =\lim_{j\to\infty}\mathsf R(0,x_j,x_j)
 \ge\mathsf r_0^{\mathrm{rad}}(S).
\end{equation*}
Here the time modulus replaces $t_j$ by zero, and $x_j\notin S$
for all large $j$ gives the last inequality.
Applying $\Phi$ and taking the supremum over $\zeta\in\supp\lambda$
bounds this tuple's directional contribution in
\eqref{eq:local-escaping-contributions} by
$\Phi(\mathsf r_0^{\mathrm{rad}}(S))$.

If $x_*=0$, continuity gives
$(a_*,A_*,q_*)=(a,A,q)(0,0)$.
Since $x_j\notin S$ eventually, $0\in\overline{S^c}$.
Radial symmetry and openness of $S^c$ give radii $r_n\downarrow0$
with $r_n\mathbb S^{d-1}\subset S^c$.
For every unit vector $v$,
\begin{equation*}
\mathsf r_0^{\mathrm{rad}}(S)
 \le\lim_{n\to\infty}\mathsf R(0,r_nv,v)
 =\mathsf R(a_*,A_*,q_*;v).
\end{equation*}
Homogeneity and monotonicity of $\Phi$ give the same upper bound
on the directional contribution for $\zeta\in\supp\lambda$.

For the lower bound, fix $x\notin S$, $x\ne0$, and let
$t\downarrow0$. Then $M(t,x)\to\infty$ and $|x|/\ell(t,x)\to\infty$.
Proposition~\ref{prop:tail-compactness} gives, along a sequence
$t_j\downarrow0$, a local phase tuple
$(\lambda,a(0,x),A(0,x),q(0,x))$.
Equation~\eqref{eq:radial-local-support}, with $e=x/|x|$,
gives $\supp\lambda\subset\R x\setminus\{0\}$, so homogeneity gives
$\mathsf R(a(0,x),A(0,x),q(0,x);\zeta)=\mathsf R(0,x,x)$
for every $\zeta\in\supp\lambda$.
The directional contribution of this tuple in
\eqref{eq:local-escaping-contributions} is therefore
$\Phi(\mathsf R(0,x,x))$.
The bound on the spatial centers may depend on the tuple.
Taking the supremum over fixed $x$ therefore proves
part~\ref{item:radial-local-phase}, including an infimum
approached only as $|x|\to\infty$ and nonpositive infima.
If $S=\R^d$, both sides are zero.
\end{proof}

\subsection{Radial phases at infinity}
\label{subsec:radial-infinity}

We show that each escaping phase law is supported
on the limiting radial line of its centers.
The density comparison in Lemma~\ref{lem:density-drift-comparison} is not
uniform in the spatial center. To identify escaping models, we
first use a longer restart interval on which angular concentration
dominates that comparison error. We then shorten the interval,
retaining the same conditional path law, until we reach the age
used in the fixed-window phase representation.

For the next two lemmas and the following proposition, take
$\mu_0=\mu_{\mathrm{rad}}$, assume \eqref{eq:radial-covariance},
and fix a bounded Borel drift $b$.
The two lemmas use the derivative
decay bound; radial symmetry of $\mu_{\mathrm{rad}}$ and the radial
eigenvector condition enter
Proposition~\ref{prop:radial-escaping-directions}.
Decrease the decay exponent to $\min\{\sigma,1\}$, so
$0<\sigma\le1$. Consider a sequence $(t_j,x_j)$ with
$|x_j|\to\infty$ and $M(t_j,x_j)\to\infty$.
We suppress its index and write
$D=|x|$, $e=x/D$, $\ell=\ell(t,x)$ and $M=t/\ell^2$.
All limits below are along this sequence. Constants may also depend
on $\sigma,C_\sigma$, but not on its centers or depths.
Fix $y_0\in S$ and use \eqref{eq:scale-fixed-support-point} to obtain
\begin{equation}
M\ell\le CD,\qquad
\ell/D\le C/M,\qquad
\ell D^{-\sigma}
 = (\ell/D)^\sigma\ell^{1-\sigma}
 \le C M^{-(1+\sigma)/2}.
\label{eq:radial-escape-scales}
\end{equation}
The last inequality also uses $\ell\le\sqrt{T/M}$.
Choose the restart ages
\begin{equation}
H_+=M^{1/2+\sigma/4},\qquad H_-=M^{1/6}.
\label{eq:radial-restart-ages}
\end{equation}
They satisfy $\sqrt M\ll H_+\ll M^{(1+\sigma)/2}$.
The lower inequality leaves room for the $O(\sqrt M)$
density-comparison error; the upper one will make the
spatial-freezing error sublinear in age.
The smaller age $H_-$ is the one used in the fixed-window phase
construction \eqref{eq:posterior-phase-law}.
For $K\le H_+$, we have $K\ell=o(D)$ and $K\ell^2=o(t)$.

Use the restart vectors from \eqref{eq:bridge-phase-vectors}:
\begin{equation*}
u_K=t-K\ell^2,\qquad C_K=\int_{u_K}^t a(s,x)\,\mathrm{d}s,\qquad
\pi_K=\ell C_K^{-1}(x-X_{u_K}).
\end{equation*}
All ages are compared under the same path law
$\mathbb P_b(\,\cdot\mid X_t=x)$, with initial law $\mu_{\mathrm{rad}}$.
Lemma~\ref{lem:restarted-posterior-annulus}, applied with exponent
$1/2+\sigma/4$, gives a fixed $P$ such that
\begin{equation}
\Pi^b_{u_{H_+};t,x}\bigl(B(x,PH_+\ell)^c\bigr)\le C\mathrm{e}^{-cH_+}.
\label{eq:radial-large-age-tail}
\end{equation}

\begin{lemma}
\label{lem:radial-spatial-freezing}
Fix $P>0$. For all sufficiently large $M$, let $H_-\le K\le H_+$
and $[r,s]\subset[t-K\ell^2,t]$ with
$K\ell^2/4\le s-r\le K\ell^2$.
For $y,z\in B(x,PK\ell)$ and
$C_{r,s}(x)=\int_r^s a(v,x)\,\mathrm{d}v$, we have
\begin{equation*}
\left|\log\frac{p_b(r,y;s,z)}
                       {\gamma_{C_{r,s}(x)}(z-y)}\right|
 \le E_K,
\qquad
E_K=C_P\{1+K^{2/3}+K^2M^{-(1+\sigma)/2}\}.
\end{equation*}
The constant $C_P$ is independent of $K,r,s,y,z$.
\end{lemma}

\begin{proof}
We have $Q=|z-y|^2/(s-r)\le C_PK$.
Put $g(v,w)=a(v,w)^{-1}$ and $g_x(v,w)=a(v,x)^{-1}$.
We compare the minimum actions $\mathcal A^{g,0}(r,y;s,z)$ and
$\mathcal A^{g_x,0}(r,y;s,z)$ over paths from $y$ at time $r$
to $z$ at time $s$, as defined in \eqref{eq:action}.
By \eqref{eq:action-ellipticity}, minimizers for both values have
$\int_r^s|\dot\gamma|^2\le C_PK$ and remain in
$B(x,C_PK\ell)$.
This ball lies where $|w|\asymp D$, so the derivative bound in
\eqref{eq:radial-covariance} gives
$|a(v,w)-a(v,x)|\le C_PK\ell D^{-\sigma}$ for a.e. $v$.
Equation~\eqref{eq:action-stability} and
\eqref{eq:radial-escape-scales} bound
$|\mathcal A^{g,0}(r,y;s,z)-\mathcal A^{g_x,0}(r,y;s,z)|$
by $C_PK^2M^{-(1+\sigma)/2}$.
The frozen minimum $\mathcal A^{g_x,0}(r,y;s,z)$ equals
$\tfrac12(z-y)^{\mathsf T}C_{r,s}(x)^{-1}(z-y)$:
the minimizing velocity is $a(v,x)C_{r,s}(x)^{-1}(z-y)$.
Combining this identity with Proposition~\ref{prop:kernel-action}
and $\det C_{r,s}(x)\asymp(s-r)^d$ proves
the asserted spatial-freezing bound.
\end{proof}

\begin{lemma}
\label{lem:radial-age-transfer}
For ages in \eqref{eq:radial-restart-ages},
\(
\pi_{H_-}-\pi_{H_+}\longrightarrow0
\,\text{in probability under }\mathbb P_b(\,\cdot\mid X_t=x).
\)
\end{lemma}

\begin{proof}
Fix $P>0$ and $K\in[H_-,H_+]$.
Condition on $X_{u_K}=w$, with $|w-x|\le PK\ell$, and take
$\vartheta\in[1/4,1/2]$ such that $\vartheta K\ge H_-$.

\paragraph*{Comparison of adjacent ages.}
Let $\mathbb E_G[\,\cdot\mid w,x]$ and
$\operatorname{Cov}_G(\,\cdot\mid w,x)$ denote expectation and
covariance for the bridge from $(u_K,w)$ to $(t,x)$ of the
zero-drift Gaussian diffusion with covariance profile $s\mapsto a(s,x)$.
Under both this Gaussian bridge and the bridge for drift $b$,
$\pi_K=\ell C_K^{-1}(x-w)$ is deterministic.
The conditional Gaussian formula gives
\begin{equation}
\mathbb E_G[\pi_{\vartheta K}\mid w,x]=\pi_K,
\qquad
\operatorname{Cov}_G(\pi_{\vartheta K}\mid w,x)
 =\ell^2(C_{\vartheta K}^{-1}-C_K^{-1})
 \preceq CK^{-1}I_d.
\label{eq:radial-gaussian-bridge-vectors}
\end{equation}

Choose a sufficiently large fixed $L$, depending on $P$.
The bridge tail \eqref{eq:bridge-gaussian-tails} bounds the probability
of $X_{u_{\vartheta K}}\notin B(x,LK\ell)$ under
$\mathbb P_b(\,\cdot\mid X_{u_K}=w,X_t=x)$ by $C\mathrm{e}^{-cK}$:
the full bridge energy scale is at most
$CP^2K$, whereas the tail exponent is at most $-cL^2K$.
Inside this ball, apply Lemma~\ref{lem:radial-spatial-freezing} to
the kernels $p_b(u_K,w;u_{\vartheta K},z)$,
$p_b(u_{\vartheta K},z;t,x)$ and $p_b(u_K,w;t,x)$ in
\eqref{eq:bridge-kernel-law}, where $z\in B(x,LK\ell)$ is a possible
value of $X_{u_{\vartheta K}}$.
Each interval has length comparable
to $K\ell^2$, so the true bridge density divided by the Gaussian
bridge density lies between $\mathrm{e}^{-3E_K}$ and $\mathrm{e}^{3E_K}$,
after increasing $C_P$.
Equation~\eqref{eq:radial-gaussian-bridge-vectors} therefore yields
\begin{equation}
\mathbb P_b\{ |\pi_{\vartheta K}-\pi_K|>\delta
                       \mid X_{u_K}=w,X_t=x\}
 \le C\mathrm{e}^{-cK}+C\mathrm{e}^{3E_K-cK\delta^2}.
\label{eq:radial-adjacent-ages}
\end{equation}
With a sufficiently large $C_0$, depending on $P$, choose
$\delta_K=C_0\{K^{-1/6}+\sqrt{K M^{-(1+\sigma)/2}}\}$.
Then $K\delta_K^2$ absorbs $E_K$ with a remaining positive
multiple of $K^{2/3}$, and the right side of
\eqref{eq:radial-adjacent-ages} is at most $C\mathrm{e}^{-cK^{2/3}}$.

\paragraph*{Iteration to the shorter age.}
Define $K_0=H_+>K_1>\cdots>K_N=H_-$ by successively halving
until the current age lies in $[2H_-,4H_-]$, then taking the last
step to $H_-$.
Thus $K_{i+1}/K_i\in[1/4,1/2]$ for $0\le i<N$, and
$N\le C(1+\log M)$.
For the ages in \eqref{eq:radial-restart-ages}, the geometric sums
of these tolerances $\delta_K$ satisfy
\begin{equation}
\sum_{i=0}^{N-1}\delta_{K_i}
 \le C\{H_-^{-1/6}+\sqrt{H_+M^{-(1+\sigma)/2}}\}
 =C\{M^{-1/36}+M^{-\sigma/8}\}\longrightarrow0.
\label{eq:radial-age-sum}
\end{equation}
To iterate
\eqref{eq:radial-adjacent-ages}, we must keep each intermediate
restart point $X_{u_{K_i}}$ within $B(x,PK_i\ell)$.
By \eqref{eq:radial-large-age-tail} and ellipticity, there is a fixed
$B_0$ such that,
\begin{equation*}
\mathbb P_b(\mathcal E_0^c\mid X_t=x)\le C\mathrm{e}^{-cH_+}, \quad \mathcal E_0=\{|\pi_{K_0}|\le B_0\}.
\end{equation*}
Fix $B_1>B_0$ and choose $P\ge\Lambda_a B_1$.
Use the single-step estimates above with this $P$, which fixes
$C_0$ in the tolerances $\delta_K$.
For all large $M$, \eqref{eq:radial-age-sum} gives
$\sum_{i=0}^{N-1}\delta_{K_i}\le B_1-B_0$.
For $0\le i<N$, define
$\mathcal F_i=\{|\pi_{K_i}|\le B_1,\ |\pi_{K_{i+1}}-\pi_{K_i}|>\delta_{K_i}\}$.
On $\{|\pi_{K_i}|\le B_1\}$, ellipticity gives
$|X_{u_{K_i}}-x|=|C_{K_i}\pi_{K_i}|/\ell\le PK_i\ell$.
Conditioning on $X_{u_{K_i}}$, apply
\eqref{eq:radial-adjacent-ages} with $K=K_i$,
$\vartheta=K_{i+1}/K_i$ and $\delta=\delta_{K_i}$ to obtain
$\mathbb P_b(\mathcal F_i\mid X_t=x)\le C\mathrm{e}^{-cK_i^{2/3}}$.
On $\mathcal E_0\setminus\bigcup_{i<N}\mathcal F_i$, induction gives
$|\pi_{K_i}|\le B_0+\sum_{k<i}\delta_{K_k}\le B_1$ for $0\le i\le N$,
so every increment is at most its tolerance.
Consequently,
\begin{equation*}
\begin{aligned}
\mathbb P_b\left\{|\pi_{H_-}-\pi_{H_+}|
 >\sum_{i=0}^{N-1}\delta_{K_i}\,\middle|\,X_t=x\right\}
 &\le\mathbb P_b(\mathcal E_0^c\mid X_t=x)
       +\sum_{i=0}^{N-1}\mathbb P_b(\mathcal F_i\mid X_t=x)\\
 &\le C\mathrm{e}^{-cH_+}+C(1+\log M)\mathrm{e}^{-cH_-^{2/3}}
 \longrightarrow0.
\end{aligned}
\end{equation*}
Together with \eqref{eq:radial-age-sum}, this proves the asserted
convergence in probability.
\end{proof}

\begin{proposition}
\label{prop:radial-escaping-directions}
If $(t_j,x_j,\ell(t_j,x_j))$ realizes $(\lambda,a_*,A_*,q_*)$
and $x_j/|x_j|\to e_*$, then
$\supp\lambda\subset\R e_*\setminus\{0\}$.
\end{proposition}

\begin{proof}
For $d=1$, the claim follows from the nonzero phase annulus in
Proposition~\ref{prop:tail-compactness}. Assume $d\ge2$ and set
$H=H_+$, $h=H\ell^2$, $u=u_H$ and $\overline a_H=C_H/h$.
Set
$\overline\kappa_H=h^{-1}\int_u^t\kappa_{\mathrm{rad}}(s,D)\,\mathrm{d}s$.
The radial eigenvector condition implies
$\overline a_He=\overline\kappa_He$; radial and tangential
subspaces are preserved by $C_H$ and its inverse.
Choose $P$ as in \eqref{eq:radial-large-age-tail}.

\paragraph*{Radial density comparison.}
Write $\rho_\beta^{\mu_{\mathrm{rad}}}(u,v)=F(u,|v|)$ as in
Lemma~\ref{lem:radial-reference}, with
$\beta=\tfrac12\operatorname{div}a$.
On any fixed enlargement of $B(x,PH\ell)$,
\eqref{eq:weighted-scale-stability} gives $\ell(u,v)\asymp\ell$,
so $u/\ell(u,v)\le C\sqrt M$.
The density comparison in Lemma~\ref{lem:density-drift-comparison} and
Lemma~\ref{lem:radial-spatial-freezing} imply that with $E\le C M^{1/2}$,
\begin{equation}
\mathrm{e}^{-E}F(u,|v|)\gamma_{C_H}(x-v)
 \le\rho_b^{\mu_{\mathrm{rad}}}(u,v)p_b(u,v;t,x)
 \le \mathrm{e}^E F(u,|v|)\gamma_{C_H}(x-v).
\label{eq:radial-angular-density-comparison}
\end{equation}
Here $\sqrt M$, $H^{2/3}$ and
$H^2M^{-(1+\sigma)/2}$ are all at most a constant times $M^{1/2}$.

\paragraph*{Angular concentration.}
At each fixed radius, $F$ is
independent of the angular variable, so a comparison of angular masses
can be integrated against the same radial factor.
Write $v=\varrho\theta$ with $\theta\in\mathbb S^{d-1}$, and set
\begin{equation*}
Q(\varrho,\theta)
 =\frac1{2h}(De-\varrho\theta)^{\mathsf T}
               \overline a_H^{-1}(De-\varrho\theta).
\end{equation*}
For $|\varrho-D|\le PH\ell=o(D)$, the radial eigenvector identity gives
\begin{equation*}
Q(\varrho,\theta)-Q(\varrho,e)
 =\frac1{2h}\Bigl\{
 \varrho^2(\theta-e)^{\mathsf T}\overline a_H^{-1}(\theta-e)
 +\frac{\varrho(D-\varrho)}{\overline\kappa_H}|\theta-e|^2\Bigr\}.
\end{equation*}
Ellipticity and $|\varrho-D|/D\to0$ therefore imply
\begin{equation}
cD^2h^{-1}|\theta-e|^2
 \le Q(\varrho,\theta)-Q(\varrho,e)
 \le CD^2h^{-1}|\theta-e|^2.
\label{eq:radial-angular-cost}
\end{equation}
Integration in spherical polar coordinates yields
\begin{equation}
\begin{aligned}
\int_{|\varrho\theta-\varrho e|\le\sqrt h}
 \mathrm{e}^{-Q(\varrho,\theta)}\,\mathrm{d}\omega(\theta)
 &\ge c(\sqrt h/D)^{d-1}\mathrm{e}^{-Q(\varrho,e)},\\
\int_{|P_e(\varrho\theta)|>\eta H\ell}
 \mathrm{e}^{-Q(\varrho,\theta)}\,\mathrm{d}\omega(\theta)
 &\le C(\sqrt h/D)^{d-1}\mathrm{e}^{-Q(\varrho,e)}\mathrm{e}^{-c\eta^2H},
\end{aligned}
\label{eq:radial-angular-integrals}
\end{equation}
where $P_e=I_d-e\otimes e$, $\mathrm{d}\omega$ is surface measure,
$0<\eta\le1$, and $\eta\sqrt H$ is sufficiently large.
The first cap has area comparable to $(\sqrt h/D)^{d-1}$
and bounded excess cost by \eqref{eq:radial-angular-cost}.
For the second integral, $|P_e(\varrho\theta)|\le
\varrho|\theta-e|$ gives an angular lower threshold.
The same cost bound then reduces the integral to a Gaussian
radial tail in $d-1$ dimensions; its polynomial factor is absorbed
by decreasing the exponential constant.

Put $B_H=B(x,PH\ell)$ and
$E_\eta=\{v:|P_e(v-x)|>\eta H\ell\}$.
Let $\mathcal C_H$ be the union of the caps in the first line of
\eqref{eq:radial-angular-integrals} over
$\varrho\in[D-PH\ell,D+PH\ell]$.
Then $\mathcal C_H\subset B(x,(P+1)H\ell)$.
Splitting the posterior over $B_H$ and $B_H^c$, and applying
\eqref{eq:radial-large-age-tail} and
\eqref{eq:radial-angular-density-comparison}, gives
\begin{equation*}
\Pi^b_{u;t,x}(E_\eta)
\le C\mathrm{e}^{-cH}+\mathrm{e}^{2E}
 \frac{\displaystyle\int_{E_\eta\cap B_H}
        F(u,|v|)\gamma_{C_H}(x-v)\,\mathrm{d}v}
      {\displaystyle\int_{\mathcal C_H}
        F(u,|v|)\gamma_{C_H}(x-v)\,\mathrm{d}v}
\le C(\mathrm{e}^{-cH}+\mathrm{e}^{2E-c\eta^2H}).
\end{equation*}
For the second inequality, every $v\in B_H$ has
$|v|\in[D-PH\ell,D+PH\ell]$.
Enlarge the Gaussian numerator to all angles satisfying the
tangential constraint at these radii, and apply
\eqref{eq:radial-angular-integrals} after multiplying both angular
estimates by $F(u,\varrho)\varrho^{d-1}$ and integrating in $\varrho$.
For each fixed $\eta\in(0,1]$, the posterior bound tends to zero
because $E=O(\sqrt M)$ and $\sqrt M/H\to0$.
Since $C_H^{-1}$ preserves the tangential space and has norm at most $C/h$,
$|P_e\pi_H|\le C(H\ell)^{-1}|P_e(X_u-x)|$.
It follows that $P_e\pi_{H_+}\to0$ in probability under
$\mathbb P_b(\,\cdot\mid X_t=x)$.

\paragraph*{Identification of the phase support.}
Lemma~\ref{lem:radial-age-transfer} transfers this concentration
to $H_-$, giving $P_e\pi_{H_-}\to0$ in probability under the
same conditional path law.
Equation~\eqref{eq:posterior-phase-law} identifies $\lambda_{t,x}$
as the law of $\pi_{H_-}/2$ with the additional condition
$X_{u_{H_-}}\in\mathcal V_{H_-}$.
Lemma~\ref{lem:restarted-posterior-annulus} makes the discarded
posterior mass tend to zero, so \eqref{eq:phase-conditioning-error}
preserves the tangential concentration.
Along the given realizing sequence, \eqref{eq:phase-law-realization} gives
$\lambda_{t_j,x_j}\rightharpoonup\lambda$.
Their common compact annulus $\mathcal Z_{\mathrm{ph}}$ and
$x_j/|x_j|\to e_*$ then imply
$\supp\lambda\subset\R e_*\setminus\{0\}$.
\end{proof}

\begin{proof}[Proof of Proposition~\ref{prop:radial-local-phase}\ref{item:radial-escaping-phase}]
Fix a bounded Borel drift $b$ and work with $\rho_b^{\mu_{\mathrm{rad}}}$.
For the upper bound, take any escaping phase tuple
$(\lambda,a_*,A_*,q_*)$ and a defining sequence $(t_j,x_j,\ell_j)$
with $|x_j|\to\infty$.
After passing to a subsequence, $e_j=x_j/|x_j|\to e_*$.
Proposition~\ref{prop:radial-escaping-directions} gives
$\supp\lambda\subset\R e_*\setminus\{0\}$.
For every $\zeta\in\supp\lambda$, homogeneity and the joint
coefficient limit yield
\begin{equation*}
\mathsf R(a_*,A_*,q_*;\zeta)
 =\mathsf R(a_*,A_*,q_*;e_*)
 =\lim_{j\to\infty}\mathsf R(t_j,x_j,x_j)
 \ge\mathsf r_\infty^{\mathrm{rad}}.
\end{equation*}
Applying $\Phi$ and taking both suprema in
\eqref{eq:local-escaping-contributions} gives
$G^{\mathrm{phase,esc}}(I;\rho_b^{\mu_{\mathrm{rad}}})
\le\Phi(\mathsf r_\infty^{\mathrm{rad}})$.

For the opposite bound, choose $|x_j|\to\infty$ and
$t_j\in[0,T]$ with
$\mathsf R(t_j,x_j,x_j)\to\mathsf r_\infty^{\mathrm{rad}}$.
Replace $t_j$ by $\widehat t_j=\max\{t_j,|x_j|^{-1}\}$.
For large $j$, $0<\widehat t_j\le T$, and the common time modulus
of the directional ratios makes the change $o(1)$.
Equation~\eqref{eq:weighted-scale-comparison} gives
$M(\widehat t_j,x_j)\ge c\widehat t_j|x_j|^2\to\infty$.
Apply Proposition~\ref{prop:tail-compactness} along
$(\widehat t_j,x_j,\ell(\widehat t_j,x_j))$.
After passing to a common subsequence, these centers and scales
realize an escaping phase tuple $(\lambda,a_*,A_*,q_*)$, and
$e_j=x_j/|x_j|\to e_*$.
Proposition~\ref{prop:radial-escaping-directions} gives
$\supp\lambda\subset\R e_*\setminus\{0\}$, so for every
$\zeta\in\supp\lambda$,
\begin{equation*}
\mathsf R(a_*,A_*,q_*;\zeta)
 =\lim_{j\to\infty}\mathsf R(\widehat t_j,x_j,x_j)
 =\mathsf r_\infty^{\mathrm{rad}}.
\end{equation*}
The directional supremum for this tuple in
\eqref{eq:local-escaping-contributions} is
$\Phi(\mathsf r_\infty^{\mathrm{rad}})$, which proves the reverse
bound for $G^{\mathrm{phase,esc}}(I;\rho_b^{\mu_{\mathrm{rad}}})$.
This also covers zero or negative limiting ratios, which give
infinite contribution.
\end{proof}

\begin{proof}[Proof of Theorem~\ref{thm:radial-formulas}]
By Proposition~\ref{prop:model-contribution-comparison}, it suffices
to compute the gain for $\mu_{\mathrm{rad}}$.
Apply Theorem~\ref{thm:sharp-limit} to this law.
Proposition~\ref{prop:radial-local-phase} identifies its local and
escaping phase contributions, proving the asserted formula.
None of its terms depends on $f$ or $b$.
\end{proof}

\section{Open problems and future directions}
\label{sec:conclusion}

Outside the radial setting of Theorem~\ref{thm:radial-formulas},
identifying the phase directions realized jointly with each
coefficient limit remains open.

A question left by the three-term formula is whether uniform
continuity of $A$ can be weakened to vanishing mean oscillation,
as in classical parabolic
theory~\citep{krylov-2007-vmo}. Preserving the exact constant would require
localization and recovery from averaged coefficient limits on the
density-adapted scales, together with control of Hessian concentration
below those scales and near the initial boundary.
Appendix~\ref{app:rough-principal} illustrates the obstruction:
spatial variation surviving rescaling can produce infinite gain even
when all three frozen-coefficient candidate contributions are finite.

Globally Lipschitz drifts of linear growth raise a different difficulty:
transport can survive rescaling and change the limit model. Appendix~\ref{app:linear-growth-transport}
gives smooth coefficients and an invariant law for which shrinking
positive-time tests for the full matching operator $P_{a/2,b}$ have
quotients tending to $3/\sqrt2>2$.
In moving coordinates the density limit is flat, but the principal
matrix remains time-dependent. This points to a broader class of
limit models.
The invariant law fails the quadratic coercivity condition, so this
example does not settle the extension for initial laws satisfying it.

A further question is to obtain explicit damping thresholds from the
coefficients and initial law. The interval-averaging bound in
Proposition~\ref{prop:poisson} reduces this to estimating
$G_{\alpha,h}(I)-G_\alpha(I)$ as $h\downarrow0$.
Such estimates would quantify both convergence to the limiting Hessian
gain and the contraction margin for nonlinear source iteration.

\appendix
\section{Weighted spaces, localization and terminal theory}
\label{app:natural-realization}

\subsection{Weighted spaces and gains}
\label{subsec:weighted-spaces}

We prove the results of Subsection~\ref{subsec:equation}
with the weights and intervals specified there.

\begin{proof}[Proof of Lemma~\ref{lem:graph-core}]
Local upper and positive lower bounds on $t^\alpha w$ show that
limits in the four weighted norms preserve the distributional
derivative relations. Thus $\mathcal W_\alpha(J;w)$ is complete.
Near $T_2$, the weight is also bounded above and below by positive
constants on compact spatial sets.
For $u\in\mathcal W_\alpha(J;w)$, apply the trace theorem
\citep[Theorem~III.4.10.2]{amann-1995} to $\chi_Ru$ near $T_2$,
using the spatial cutoffs from \eqref{eq:spatial-cutoffs} and
$(L^2(\R^d),H^2(\R^d))_{1/2,2}=H^1(\R^d)$.
This gives continuity of $\operatorname{Tr}^{\mathrm{loc}}_{T_2}$
and hence closedness of $\mathcal W_{\alpha,0}(J;w)$.

Fix $u\in\mathcal W_{\alpha,0}(J;w)$.
Since $\chi_Ru\to u$ in $\mathcal W_\alpha(J;w)$ as $R\to\infty$,
we may assume that $u$ has compact spatial support.
Choose a smooth time cutoff $\theta$ on $[T_1,T_2]$,
equal to zero near $T_1$ and one near $T_2$.
For the terminal part $\theta u$, choose $T_0\in J$ with
$\theta=0$ on $[T_1,T_0]$. The local weight bounds and zero traces give
\(
\theta u\in H_0^1((T_0,T_2);L^2(\R^d))
              \cap L^2((T_0,T_2);H^2(\R^d)).
\)
Standard Sobolev density gives approximants in
$C_c^\infty((T_0,T_2)\times\R^d)$ converging in this space.
For the initial part $(1-\theta)u$, the Meyers--Serrin density
argument \citep{meyers-serrin-1964}, using the local upper and
positive lower bounds on $t^\alpha w$, gives smooth approximants
on $J\times\R^d$ converging in $\mathcal W_\alpha(J;w)$.
Multiplying the approximants to $(1-\theta)u$
by a fixed smooth time cutoff equal to one on $\supp(1-\theta)$
and zero near $T_2$ preserves convergence and the free initial trace.

Multiply both families by a fixed $\chi_R$ equal to one on
$\supp u$, and extend the approximants to $\theta u$ by zero to $J$.
The local weight bounds on $[T_0,T_2]\times\overline B_{2R}$
also give weighted convergence of the terminal approximants.
Both families now lie in $\mathcal C_{\alpha,0}(J;w)$,
and their sums converge to $u$ in $\mathcal W_\alpha(J;w)$.
\end{proof}

For diffusion weights, we use the following Sobolev input under
Assumption~\ref{ass:coefficients}, with an arbitrary initial probability
law. Fix $p_*>d+2$. For $J=(T_1,T_2)\in\mathcal J((0,T))$,
$f\in C_c^\infty(J\times\R^d)$ and $g\in C_c^\infty(\R^d)$,
applying \citep[Theorem~2.1]{krylov-2007-vmo} to $u-g$ gives
\begin{equation}
u\in W^{1,2}_{p_*}(J\times\R^d),\qquad P_Au=f,\qquad u(T_2)=g.
\label{eq:classical-terminal-solution}
\end{equation}
The spatial uniform continuity of $A$ gives the required $\mathrm{VMO}_x$
condition, and $f+A:D^2g\in L^{p_*}(J\times\R^d)$.
The Gaussian bound \eqref{eq:kernel-gaussian} and
$\mu_t(\R^d)=1$ give
\begin{equation}
\int_0^T\int_{\R^d}
 (t^\alpha\rho(t,x))^{p_*/(p_*-2)}\,\mathrm{d}x\,\mathrm{d}t
 \le C\int_0^T t^{(\alpha p_*-d)/(p_*-2)}\,\mathrm{d}t<\infty.
\label{eq:weighted-lp-embedding}
\end{equation}
Here $C$ depends only on $\mathfrak R$ and $p_*$.
H\"older's inequality therefore gives a continuous embedding
$W^{1,2}_{p_*}(J\times\R^d)\hookrightarrow\mathcal W_\alpha(J)$
for every $J\in\mathcal J((0,T))$, including $T_1=0$.

\begin{proof}[Proof of Lemma~\ref{lem:weight-operations}]
For part~\ref{item:positive-mixtures}, let
$C<\infty$ bound $G_\alpha^{\mathrm{fix}}(J;w_y,A,q)$ for $\sigma$-a.e.
$y$. For $u\in\mathcal W_{\alpha,0}(J;w)$, Tonelli gives
\begin{equation*}
\int_Y\|u\|_{\mathcal W_\alpha(J;w_y)}^2\,\sigma(\mathrm{d}y)
 =\|u\|_{\mathcal W_\alpha(J;w)}^2<\infty.
\end{equation*}
The local terminal trace is independent of the weight, so
$u\in\mathcal W_{\alpha,0}(J;w_y)$ for $\sigma$-a.e. $y$.
Integrating the squared estimates gives $G_\alpha^{\mathrm{fix}}(J;w,A,q)\le C$.
If $C<\infty$ bounds $G_{\alpha,h}(I;w_y,A,q)$ for $\sigma$-a.e. $y$,
this estimate holds uniformly over $J\in\mathcal J_h(I)$,
proving the finite-$h$ bound.
For $u\in\mathcal C_{\alpha,0}(J;w)$ supported in $J\times E$,
Tonelli gives $u\in\mathcal C_{\alpha,0}(J;w_y)$ for $\sigma$-a.e. $y$:
smoothness, support and terminal trace are weight-independent.
Integrating the assumed squared estimates proves the localized bound.

For part~\ref{item:time-weights}, first take $\alpha'>\alpha$.
The identity
$t^{\alpha'-\alpha}=\int_0^t(\alpha'-\alpha)s^{\alpha'-\alpha-1}\,\mathrm{d}s$
and Tonelli's theorem give, for $\mathcal T\in\{\mathcal H_q,P_A\}$,
\begin{equation}
\|\mathcal Tu\|_{\alpha',J;w}^2
 =\int_0^{T_2}(\alpha'-\alpha)s^{\alpha'-\alpha-1}
   \|\mathcal Tu\|_{\alpha,(\max\{T_1,s\},T_2);w}^2\,\mathrm{d}s,
\,\forall u\in\mathcal W_{\alpha',0}(J;w).
\label{eq:time-weight-mixture}
\end{equation}
Suppose $C:=\sup_{T_1<s<T_2}G_\alpha^{\mathrm{fix}}((s,T_2);w,A,q)<\infty$.
Since $t^{\alpha-\alpha'}$ is bounded on $(s,T_2)$ for $s>T_1\ge0$,
$u|_{(s,T_2)}\in\mathcal W_{\alpha,0}((s,T_2);w)$, so
\begin{equation*}
\|\mathcal H_qu\|_{\alpha,(s,T_2);w}
 \le C\|P_Au\|_{\alpha,(s,T_2);w},\qquad\forall s\in(T_1,T_2).
\end{equation*}
For $T_1>0$, monotone convergence as $s\downarrow T_1$ gives this
bound on $(T_1,T_2)$ as well. Applying these bounds inside
\eqref{eq:time-weight-mixture} proves the first inequality. For $\alpha'=\alpha$, restriction and
monotone convergence give the result directly.
Taking the suprema over $J\in\mathcal J_h(I)$, and then
$h\downarrow0$ gives the last two inequalities.

For part~\ref{item:interval-gains}, take $w=\rho$, fix $p_*>d+2$ and let
$J=(T_1,T_2)\subseteq K=(s_0,s_1)$.
It suffices to consider
$G_\alpha^{\mathrm{fix}}(K)<\infty$ and
$u\in\mathcal C_{\alpha,0}(J)$.
The test class imposes no smooth trace at $T_1$, so choose an interior
slice $s\in(T_1,T_2)$, where $u(s,\cdot)\in C_c^\infty(\R^d)$.
The Sobolev solution in \eqref{eq:classical-terminal-solution}, with
$p_*>d+2$, gives $v\in W^{1,2}_{p_*}((s_0,s)\times\R^d)$ satisfying
$P_Av=0$ and $v(s)=u(s,\cdot)$.
Join $v$ to $u$ on $(s,T_2)$ and to zero on $(T_2,s_1)$.
The matching traces and the terminal gap of $u$ give a function
$\widetilde u\in W^{1,2}_{p_*}(K\times\R^d)$.
By \eqref{eq:weighted-lp-embedding},
$\widetilde u\in\mathcal W_{\alpha,0}(K)$, with
$P_A\widetilde u=\1_{(s,T_2)}P_Au$. Consequently,
\begin{equation*}
\|\mathcal H_qu\|_{\alpha,(s,T_2)}
 \le\|\mathcal H_q\widetilde u\|_{\alpha,K}
 \le G_\alpha^{\mathrm{fix}}(K)\|P_Au\|_{\alpha,(s,T_2)}.
\end{equation*}
Letting $s\downarrow T_1$ and using the smooth-core characterization
in \eqref{eq:short-time-constants} proves interval monotonicity.
Since $K\in\mathcal J_{|K|}(K)$, taking the supremum over its
subintervals gives $G_{\alpha,|K|}(K)=G_\alpha^{\mathrm{fix}}(K)$,
also when the gain is infinite.
\end{proof}

\subsection{Energy estimates and reference evolution}
\label{subsec:reference-evolution}

For the remainder of this appendix, $I=(0,T)$ and
Assumption~\ref{ass:coefficients} holds.
The initial law is arbitrary; only the proof of
Lemma~\ref{lem:localization-partition} uses Assumption~\ref{def:initial-laws}.

\begin{lemma}\label{lem:direct-energy}
Let $J=(T_1,T_2)\in\mathcal J(I)$.
The following estimates hold for every $u\in\mathcal W_{\alpha,0}(J)$.
They also hold for arbitrary bounded Borel matrix fields $A$.
\begin{enumerate}
\item\label{item:energy-short-interval}
There is $C>0$, depending only on $\mathfrak R$ and
$\|A\|_{L^\infty(\overline I\times\R^d)}$,
such that, with $h=|J|$,
\begin{equation*}
h^{-2}\|u\|_{\alpha,J}^2+h^{-1}\|Du\|_{\alpha,J}^2
 \le C\{\|P_Au\|_{\alpha,J}^2+\|D^2u\|_{\alpha,J}^2\}.
\end{equation*}
\item\label{item:energy-damping}
For $\kappa>B_b^2/\lambda_a$, with the norm and parameters from
\eqref{eq:damped-interval-norm} and \eqref{eq:linear-energy-data},
\begin{equation*}
\max\{\kappa_{\mathrm{eff}}\|u\|_{\alpha,\kappa,J},
       \sqrt{\lambda_a\kappa_{\mathrm{eff}}}\|Du\|_{\alpha,\kappa,J}\}
 \le\|P_Au\|_{\alpha,\kappa,J}
             +e_q\|\mathcal H_qu\|_{\alpha,\kappa,J}.
\end{equation*}
\end{enumerate}
\end{lemma}

\begin{proof}
We first prove part~\ref{item:energy-damping}.
For $u\in\mathcal C_{\alpha,0}(J)$, set
$\eta(t)=t^\alpha \mathrm{e}^{-2\kappa(T_2-t)}$ and fix $T'\in(T_1,T_2)$.
Applying It\^o's formula to $\eta(t)u(t,X_t)^2$ on $[T',T_2]$,
taking expectations and using the zero terminal value, we obtain
\begin{equation*}
\begin{aligned}
\eta(T')\|u(T')\|_{L^2(\mu_{T'})}^2
&+\int_{T'}^{T_2}\eta(t)\left[
       \left(2\kappa+\frac\alpha t\right)\|u(t)\|_{L^2(\mu_t)}^2
       +\|a^{1/2}Du(t)\|_{L^2(\mu_t)}^2\right]\mathrm{d}t\\
&\qquad=2\int_{T'}^{T_2}\eta(t)\int_{\R^d}
       \bigl[u(P_{a/2}u-b\cdot Du)\bigr](t,x)\,\mu_t(\mathrm{d}x)\,\mathrm{d}t.
\end{aligned}
\end{equation*}
Use $2|ub\cdot Du|\le\lambda_a|Du|^2/2+2B_b^2|u|^2/\lambda_a$
and discard the nonnegative boundary and $\alpha/t$ terms
before letting $T'\downarrow T_1$. This gives
\begin{equation*}
2\kappa_{\mathrm{eff}}\|u\|_{\alpha,\kappa,J}^2
 +\tfrac12\lambda_a\|Du\|_{\alpha,\kappa,J}^2
 \le2\|u\|_{\alpha,\kappa,J}\|P_{a/2}u\|_{\alpha,\kappa,J}
 \le2\kappa_{\mathrm{eff}}\|u\|_{\alpha,\kappa,J}^2
    +\frac{\|P_{a/2}u\|_{\alpha,\kappa,J}^2}{2\kappa_{\mathrm{eff}}}.
\end{equation*}
Using $P_{a/2}u=P_Au+(A-a/2):D^2u$ and the definition \eqref{eq:linear-energy-data} of $e_q$,
we obtain part~\ref{item:energy-damping}
for $u\in\mathcal C_{\alpha,0}(J)$.
Lemma~\ref{lem:graph-core} extends this estimate
to $\mathcal W_{\alpha,0}(J)$.
Take $\kappa=h^{-1}+B_b^2/\lambda_a$. Then $\kappa_{\mathrm{eff}}=h^{-1}$ and
$\mathrm{e}^{-2(1+B_b^2T/\lambda_a)}\le \mathrm{e}^{-2\kappa(T_2-t)}\le1$ on $J$.
Together with $|(A-a/2):D^2u|\le C|D^2u|$, this proves
part~\ref{item:energy-short-interval}.
\end{proof}

\begin{lemma}
\label{lem:reference-smoothing}\label{lem:reference-terminal-energy}
For $0<t<s\le T$ and $g\in C_c^\infty(\R^d)$, define
$U(t,s)g=v(t)$, where $v$ is the bounded solution of
$P_{a/2}v=0$ with $v(s)=g$.
The evolution is drift-free; the measures $\mu_t$ below are those
of the reference diffusion with coefficients $(a,b)$.
All constants depend only on $\mathfrak R$, uniformly in $\mu_0$.
\begin{enumerate}
\item\label{item:reference-smoothing}
The operators $U(t,s)$ extend to bounded maps from $L^2(\mu_s)$ to
$L^2(\mu_t)$ and satisfy
\begin{equation*}
\|D^jU(t,s)g\|_{L^2(\mu_t)}
 \le C(s-t)^{-j/2}\|g\|_{L^2(\mu_s)},\qquad j=0,1,2,
\end{equation*}
for every $s\in(0,T]$, a.e. $t\in(0,s)$ and every $g\in L^2(\mu_s)$.
\item\label{item:reference-terminal-energy}
For $s\in(0,T]$ and $g\in H^1(\mu_s)$,
$v=U(\cdot,s)g\in\mathcal W_0((0,s))$,
$P_{a/2}v=0$ and $\operatorname{Tr}^{\mathrm{loc}}_s v=g$, with
\begin{equation}
\mathop{\rm ess\,sup}_{0<t<s}\|v(t)\|_{H^1(\mu_t)}^2
 +\int_0^s\|D^2v(t)\|_{L^2(\mu_t)}^2\,\mathrm{d}t
 \le C\|g\|_{H^1(\mu_s)}^2.
\label{eq:terminal-data-energy}
\end{equation}
For $E=(s-\delta,s)$ and $0<\delta\le s/2$,
\begin{equation}
\|g\|_{H^1(\mu_s)}^2
 \le\left(\frac{2}{s}\right)^\alpha(\delta^{-1}+C)
      \sum_{j=0}^2\|D^jv\|_{\alpha,E}^2.
\label{eq:terminal-trace-energy}
\end{equation}
\end{enumerate}
\end{lemma}

\begin{proof}
Fix $p_*>d+2$. For $0<s\le T$ and $g\in C_c^\infty(\R^d)$, apply
\eqref{eq:classical-terminal-solution} with $A=a/2$ and $f=0$. The Sobolev It\^o formula
\citep[Remark~2.2]{krylov-2007-vmo}
for the drift-free diffusion identifies the solution with
$v=U(\cdot,s)g$.
Apply the $W^{1,2}_{p_*}$ estimate in
\citep[Theorem~2.1]{krylov-2007-vmo} to spatial difference quotients
of $v-g$, which have zero terminal trace. The coefficient quotients are bounded
by $L_a$, so weak convergence gives
$Dv\in W^{1,2}_{p_*}((0,s)\times\R^d)$ and, for $w_i=\partial_i v$,
\begin{equation}
P_{a/2}w_i=\tfrac12\partial_i a:D^2v,\qquad
w_i(s)=\partial_i g,\qquad i=1,\ldots,d.
\label{eq:reference-gradient-equation}
\end{equation}
The embedding \eqref{eq:weighted-lp-embedding} with
$\alpha=0$ gives $v,Dv\in\mathcal W_0((0,s))$.

\paragraph*{Energy estimates and first-order smoothing.}
Set $V_j(t)=\|D^jv(t)\|_{L^2(\mu_t)}^2$, $j=0,1,2$.
Applying the Sobolev It\^o formula
to $v(t,X_t)^2$ and $\sum_i w_i(t,X_t)^2$ and taking expectations
shows that $V_0,V_1$ are absolutely continuous and, for a.e. $t$,
\begin{equation*}
\begin{aligned}
V_0'(t)
 &=\int_{\R^d}\bigl[(Dv)^{\mathsf T}aDv+2v\,b\cdot Dv\bigr]\,\mathrm{d}\mu_t,\\
V_1'(t)
 &=\int_{\R^d}\Bigl[\sum_{i=1}^d(Dw_i)^{\mathsf T}aDw_i
       +2Dv\cdot(D^2v\,b)
       -\bigl((Dv\cdot D)a\bigr):D^2v\Bigr]\,\mathrm{d}\mu_t.
\end{aligned}
\end{equation*}
The $b$-terms arise because the generator of $X$ contains $b\cdot D$,
whereas $P_{a/2}$ does not.
Ellipticity and Young's inequality, using
$|b|\le B_b$ and $|(Dv\cdot D)a|\le L_a|Dv|$, give, a.e.,
\begin{equation}
\begin{gathered}
V_0'\ge\tfrac12\lambda_aV_1-\frac{2B_b^2}{\lambda_a}V_0,\quad
V_1'\ge\tfrac12\lambda_aV_2-\frac{(2B_b+L_a)^2}{2\lambda_a}V_1,\\
|(V_0+V_1)'|\le C(V_0+V_1+V_2).
\end{gathered}
\label{eq:reference-energy}
\end{equation}
Multiplying the first two inequalities in \eqref{eq:reference-energy}
by $\mathrm{e}^{Ct}$, for a sufficiently large $C$,
and integrating from $t$ to $\tau$ gives
\begin{equation*}
V_j(t)+c\int_t^\tau V_{j+1}(r)\,\mathrm{d}r\le CV_j(\tau),
\qquad j=0,1,\quad 0<t<\tau\le s.
\end{equation*}
Taking $\tau=s$ and then letting $t\downarrow0$ proves
\eqref{eq:terminal-data-energy}.

The first two inequalities in \eqref{eq:reference-energy} give
\begin{equation*}
V_0(t)\le CV_0(s),\qquad
V_1(t)\le\frac C{s-t}\int_t^sV_1(r)\,\mathrm{d}r
        \le\frac C{s-t}V_0(s).
\end{equation*}
This proves part~\ref{item:reference-smoothing} for $j=0,1$.
Density of $C_c^\infty(\R^d)$ in $L^2(\mu_s)$ and closedness of
spatial differentiation extend $U(t,s)$ and $DU(t,s)$ to
$L^2(\mu_s)$ with these bounds.

\paragraph*{Hessian smoothing.}
To avoid second derivatives of $a$, we use \eqref{eq:reference-gradient-equation}.
Set $h(t)=\|D^2v(t)\|_{L^2(\mu_t)}$ and
$F_i=\tfrac12\partial_i a:D^2v$.
Since $v\in W^{1,2}_{p_*}((0,s)\times\R^d)$ and
$|F_i|\le L_a|D^2v|/2$, \eqref{eq:weighted-lp-embedding} gives
$F_i\in L^{p_*}((0,s)\times\R^d)\cap L^2_0((0,s))$.
For $0<m<s$, the Sobolev It\^o formula for the drift-free diffusion
gives the variation-of-constants identity for the
$W^{1,2}_{p_*}$ solution $w_i$:
\begin{equation*}
w_i(t)=U(t,m)w_i(m)
 +\int_t^m U(t,r)F_i(r)\,\mathrm{d}r.
\end{equation*}
For a.e. $t\in(0,m)$, the $j=0,1$ bounds in
part~\ref{item:reference-smoothing} and Young's inequality give
$\int_t^m\|U(t,r)F_i(r)\|_{H^1(\mu_t)}\,\mathrm{d}r<\infty$.
Spatial differentiation therefore commutes with the source integral.
The $j=1$ bound in part~\ref{item:reference-smoothing} gives
\begin{equation*}
h(t)\le C(m-t)^{-1/2}\|Dv(m)\|_{L^2(\mu_m)}
 +CL_a\int_t^m(\tau-t)^{-1/2}h(\tau)\,\mathrm{d}\tau.
\end{equation*}
By \eqref{eq:terminal-data-energy}, $h\in L^2(0,s)$.
The weakly singular Gronwall estimate
\citep[Lemma~7.1.1]{henry-1981}, applied to $\tau\mapsto h(m-\tau)$,
gives $h(t)\le C(m-t)^{-1/2}\|Dv(m)\|_{L^2(\mu_m)}$ for a.e. $t\in(0,m)$.
By Fubini, for a.e. $t\in(0,s)$ we may choose
$m\in((2t+s)/3,(t+2s)/3)$ for which this inequality and
part~\ref{item:reference-smoothing} for $j=1$ at $m$ both hold. They give
\begin{equation*}
h(t)\le\frac{C\|g\|_{L^2(\mu_s)}}{\sqrt{(m-t)(s-m)}}
      \le\frac{C\|g\|_{L^2(\mu_s)}}{s-t},
\end{equation*}
proving part~\ref{item:reference-smoothing} for $j=2$.
Density and closedness of spatial differentiation now give the
corresponding extension of $D^2U(t,s)$ to $L^2(\mu_s)$.

For $E=(s-\delta,s)$, $0<\delta\le s/2$, integrate the last
inequality in \eqref{eq:reference-energy} from $t$ to $s$ and
average over $t\in E$:
\begin{equation*}
V_0(s)+V_1(s)
 \le(\delta^{-1}+C)\int_E(V_0+V_1+V_2)(t)\,\mathrm{d}t.
\end{equation*}
Since $t\ge s/2$ on $E$, this gives \eqref{eq:terminal-trace-energy}.
For $g\in H^1(\mu_s)$, choose $g_n\in C_c^\infty(\R^d)$ with
$g_n\to g$ in $H^1(\mu_s)$ and set $v_n=U(\cdot,s)g_n$.
Since $\partial_t v_n=-a:D^2v_n/2$, estimate~\eqref{eq:terminal-data-energy}
applied to differences makes $(v_n)$ Cauchy in $\mathcal W_0((0,s))$.
The $j=0$ bound in part~\ref{item:reference-smoothing} identifies the
limit with $v=U(\cdot,s)g$.
Passing to the limit in $P_{a/2}v_n=0$ gives $P_{a/2}v=0$,
and Lemma~\ref{lem:graph-core} gives $\operatorname{Tr}^{\mathrm{loc}}_s v=g$.
Estimates~\eqref{eq:terminal-data-energy} and
\eqref{eq:terminal-trace-energy} pass to the limit as well.
\end{proof}

\subsection{Square partitions and time localization}
\label{subsec:localization-tools}

\begin{lemma}\label{lem:parabolic-partition}
Fix $L\ge2$.
The following partitions are locally finite, smooth,
take values in $[0,1]$, and satisfy $\sum_i\chi_i^2=1$ on their
domains. Each point belongs to at most $N_d$ supports; $N_d$ and
all constants below depend only on $d$.
\begin{enumerate}
\item\label{item:spatial-partition}
Let $r\in C(\R^d;(0,\infty))$. Suppose
\begin{equation}
\tfrac12 r(x)\le r(y)\le2r(x),
\qquad x\in\R^d,\quad y\in B(x,8Lr(x)).
\label{eq:admissible-spatial-scale}
\end{equation}
There are centers $x_i\in\R^d$, scales $r_i=r(x_i)$, and a
partition $\{\chi_i\}$ with
\begin{equation*}
\supp\chi_i\subseteq B(x_i,2Lr_i),\qquad
\sum_i|D^m\chi_i|^2\le C(Lr)^{-2m},\quad m=1,2.
\end{equation*}
For constant $r$, lattice translates with mesh comparable to $Lr$ suffice.

\item\label{item:parabolic-partition}
Write $\mathcal Q_R(t,x)=(t-R^2,t+R^2)\times B(x,R)$ for the
parabolic cylinder of radius $R>0$.
Let $\Omega\subseteq\R\times\R^d$ be open and
$r\in C(\Omega;(0,\infty))$.
Suppose that, for every $(t,x)\in\Omega$,
\begin{equation}
\tfrac12 r(t,x)\le r(s,y)\le2r(t,x),
\qquad\forall(s,y)\in\Omega\cap\mathcal Q_{8Lr(t,x)}(t,x).
\label{eq:admissible-local-scale}
\end{equation}
There exist centers $(t_i,x_i)\in\Omega$, scales
$r_i=r(t_i,x_i)$, and a partition $\{\chi_i\}$ such that for every $i$,
\(
\supp_\Omega\chi_i\subseteq\Omega\cap\mathcal Q_{2Lr_i}(t_i,x_i).
\)
On $\Omega$, the partition satisfies
\begin{equation*}
\sum_i|D\chi_i|^2\le CL^{-2}r^{-2},\qquad
\sum_i\bigl(|D^2\chi_i|^2+|\partial_t\chi_i|^2\bigr)
 \le CL^{-4}r^{-4}.
\end{equation*}
\end{enumerate}
\end{lemma}

\begin{proof}
For part~\ref{item:parabolic-partition},
\eqref{eq:admissible-local-scale} implies that intersecting
cylinders $\mathcal Q_{2Lr(X)}(X)$ and $\mathcal Q_{2Lr(Y)}(Y)$,
with $X,Y\in\Omega$, satisfy
$\tfrac12 r(X)\le r(Y)\le2r(X)$.
Choose a maximal disjoint family of cylinders
$\mathcal Q_{Lr_i/4}(X_i)$ in $\R^{1+d}$, with $X_i\in\Omega$
and $r_i=r(X_i)$. Every $\mathcal Q_{Lr(X)/4}(X)$, $X\in\Omega$,
then meets a chosen cylinder. The scale comparison gives $r(X)\le2r_i$, hence
$X\in\mathcal Q_{Lr_i}(X_i)$ and
$\Omega\subseteq\bigcup_i\mathcal Q_{Lr_i}(X_i)$.

If $X$ belongs to several $\mathcal Q_{2Lr_i}(X_i)$, fix one such
index $i_0$. Then $\tfrac12 r_{i_0}\le r_i\le2r_{i_0}$, and the
disjoint cylinders $\mathcal Q_{Lr_i/4}(X_i)$ all lie in
$\mathcal Q_{5Lr_{i_0}}(X)$. Comparing their volumes with that
of this containing cylinder bounds their number by $N_d$.
If $\mathcal Q_{2Lr_i}(X_i)$ meets a compact subset $K$ of $\Omega$,
then $\tfrac12\min_K r\le r_i\le2\max_K r$.
The disjoint $\mathcal Q_{Lr_i/4}(X_i)$ lie in a bounded set
and have uniformly positive volume, so only finitely many occur.

Products of translated and rescaled one-dimensional and spatial
cutoffs from \eqref{eq:spatial-cutoffs} give smooth
$0\le\varphi_i\le1$, equal to one on
$\mathcal Q_{Lr_i}(X_i)$ and supported in
$\mathcal Q_{2Lr_i}(X_i)$, with
$|D^m\varphi_i|\le C(Lr_i)^{-m}$ for $m=1,2$ and
$|\partial_t\varphi_i|\le C(Lr_i)^{-2}$.
On $\Omega$, set
\(
\chi_i=\varphi_i/(\sum_j\varphi_j^2)^{1/2}.
\)
The covering and overlap bounds give $1\le\sum_j\varphi_j^2\le N_d$;
local finiteness gives smoothness. Since $r_i/2\le r\le2r_i$ on
$\supp_\Omega\varphi_i$, differentiation proves
part~\ref{item:parabolic-partition}.

For part~\ref{item:spatial-partition}, repeat the covering construction
with balls and take $\varphi_i(x)=\chi_{Lr_i}(x-x_i)$.
Normalize as above; \eqref{eq:admissible-spatial-scale} gives the spatial
derivative bounds. When $r$ is constant, normalize lattice translates of $\chi_{Lr}$.
\end{proof}

\begin{lemma}
\label{lem:square-partition-errors}
Let $J\in\mathcal J(I)$ and let the initial probability law be arbitrary.
For a square partition $\{\chi_i,Q_i\}$ as in
Subsection~\ref{subsec:comparison} and $u\in\mathcal C_{\alpha,0}(J)$,
write $u_i=(\chi_i u)|_{J_i}$ and
$E_i^{\mathcal T}=\mathcal T(\chi_i u)-\chi_i\mathcal Tu$.
For $\mathcal T\in\{\mathcal H_q,P_A\}$ with
$\sum_i\|E_i^{\mathcal T}\|_{\alpha,J}^2<\infty$, one has
\begin{equation}
\left|\sum_i\|\mathcal Tu_i\|_{\alpha,J_i}^2
                 -\|\mathcal Tu\|_{\alpha,J}^2\right|
 \le\sum_i\|E_i^{\mathcal T}\|_{\alpha,J}^2
 +2\|\mathcal Tu\|_{\alpha,J}
       \left\|\sum_i\chi_iE_i^{\mathcal T}\right\|_{\alpha,J}.
\label{eq:localization-expansion}
\end{equation}
The cross terms satisfy
\begin{equation}
\sum_i\chi_iE_i^{\mathcal H_q}
 =-u\,q^{1/2}\Bigl(\sum_iD\chi_i\otimes D\chi_i\Bigr)q^{1/2},\qquad
\sum_i\chi_iE_i^{P_A}
 =u\sum_i(D\chi_i)^{\mathsf T}A D\chi_i.
\label{eq:localization-cancellation}
\end{equation}
If the partition is spatial, with $J_i=J$ and bounded sums
$\sum_i|D\chi_i|^2$ and $\sum_i|D^2\chi_i|^2$, then
there exists $C_\chi$ depending only on the derivative bounds and
$\mathfrak R,\alpha,\Lambda_A,\Lambda_q$ such that, for
$|J|\le h\le T$,
\begin{equation}
\sum_{\mathcal T\in\{\mathcal H_q,P_A\}}
 \left|\sum_i\|\mathcal Tu_i\|_{\alpha,J_i}^2
                 -\|\mathcal Tu\|_{\alpha,J}^2\right|
 \le C_\chi h\mathscr G_J(u)^2.
\label{eq:fixed-spatial-localization-error}
\end{equation}
\end{lemma}

\begin{proof}
Since $\supp\chi_i\subseteq Q_i$ and $\sum_i\chi_i^2=1$,
expanding the squared norms gives \eqref{eq:localization-expansion}.
The assumed square summability justifies summing the inner products.
Differentiating $\sum_i\chi_i^2=1$ cancels the terms containing
$Du$ or $\partial_t\chi_i$ in the cross term, giving
\eqref{eq:localization-cancellation}.
Lemma~\ref{lem:direct-energy}\ref{item:energy-short-interval} gives
\begin{equation}
|J|^{-1}\|u\|_{\alpha,J}+|J|^{-1/2}\|Du\|_{\alpha,J}
 \le C\mathscr G_J(u),
\label{eq:short-graph-budget}
\end{equation}
with $C$ depending only on $\mathfrak R,\alpha,\Lambda_A$.
For a spatial partition with the stated derivative bounds, the product
rule bounds $\sum_i\|E_i^{\mathcal T}\|_{\alpha,J}^2$ by
$C_\chi(\|Du\|_{\alpha,J}^2+\|u\|_{\alpha,J}^2)$, and
\eqref{eq:localization-cancellation} bounds
$\|\sum_i\chi_iE_i^{\mathcal T}\|_{\alpha,J}$ by
$C_\chi\|u\|_{\alpha,J}$.
Together with \eqref{eq:short-graph-budget} and
$\|\mathcal Tu\|_{\alpha,J}\le C\mathscr G_J(u)$, these bounds in
\eqref{eq:localization-expansion} give
\eqref{eq:fixed-spatial-localization-error}.
\end{proof}

\begin{proof}[Proof of Lemma~\ref{lem:localization-partition}]
Set $r_J=\min\{\sqrt{|J|},\ell\}$.
By \eqref{eq:weighted-scale-gradient}, $\ell(t,\cdot)^{-1}$ is
$t^{-1}$-Lipschitz in space. The definition \eqref{eq:weighted-tail-scale} also gives
\begin{equation*}
(t/s)\ell(t,x)^{-1}\le\ell(s,x)^{-1}\le\ell(t,x)^{-1},\qquad s\ge t>0.
\end{equation*}
These bounds show that $\ell$, and hence $r_J$, is continuous
on $J\times\R^d$.
For $c_0\ge128$ and a center $(t,x)$ with $t\ge c_0L^2r_J(t,x)^2$,
put $r=r_J(t,x)$.
For $(s,z)\in(J\times\R^d)\cap\mathcal Q_{8Lr}(t,x)$, they imply
\begin{equation*}
\left|\frac{s}{t}-1\right|\le\frac{64}{c_0},\qquad
\left|\frac{\ell(t,x)}{\ell(s,z)}-1\right|
 \le Cc_0^{-1/2},
\end{equation*}
using $\ell(t,x)^{-1}\ge t^{-1/2}$ and $Lr/\sqrt t\le c_0^{-1/2}$.
Fix $c_0$ so large that the last bound is at most $1/2$.
Thus $\tfrac12\ell(t,x)\le\ell(s,z)\le2\ell(t,x)$.
Since $r_J=\min\{\sqrt{|J|},\ell\}$, we also have
$\tfrac12 r_J(t,x)\le r_J(s,z)\le2r_J(t,x)$,
which is \eqref{eq:admissible-local-scale} at $(t,x)$.
For $J\in\mathcal J_h^+(T_0)$, $r_J^2\le h\le T_0/(c_0L^2)$,
so Lemma~\ref{lem:parabolic-partition}\ref{item:parabolic-partition}
applies throughout $J\times\R^d$.
All cutoffs are parabolic, so take
$\mathcal I_{\mathrm{sp}}=\varnothing$.

\paragraph*{Near-zero windows.}
The condition $t\ge c_0L^2r_J(t,x)^2$ may fail on these intervals,
so we use spatial cutoffs in that region. With $\sigma=\sqrt{\sup J}$, set
\(
F(x)=\dist(x,S)/\sigma+\sigma\langle x\rangle,\,
R_L=4\sqrt{c_0/\vartheta}\,L.
\)
The infimum in \eqref{eq:weighted-tail-scale} is at least $\dist(x,S)/t$
and, since $t\le1$, at least $\langle x\rangle$.
Hence, for $t\in J$,
\begin{equation*}
\frac{\sqrt t}{\ell(t,x)}
 \ge\frac12\left(\frac{\dist(x,S)}{\sqrt t}+\sqrt t\langle x\rangle\right)
 \ge\frac{\sqrt\vartheta}{2}F(x).
\end{equation*}
On $\Omega=J\times\{F>R_L/2\}$ this gives
$t/r_J^2\ge\vartheta R_L^2/16=c_0L^2$.
Thus Lemma~\ref{lem:parabolic-partition}\ref{item:parabolic-partition}
gives a partition $\{\chi_i^{\mathrm{par}}\}$ on $\Omega$.
Since $F$ has Lipschitz constant at most $2/\sigma$, convolution
with a fixed mollifier rescaled to $B(0,R_L\sigma/8)$ gives a smooth
function $\widetilde F$ with $|\widetilde F-F|\le R_L/4$.
Choose a fixed smooth $\psi:\R\to[0,\pi/2]$ equal to zero on
$(-\infty,5/4]$ and to $\pi/2$ on $[7/4,\infty)$, and set
\begin{equation*}
\chi_{\mathrm{in}}=\cos\psi(\widetilde F/R_L),\qquad
\chi_{\mathrm{out}}=\sin\psi(\widetilde F/R_L).
\end{equation*}
Then $\chi_{\mathrm{in}}^2+\chi_{\mathrm{out}}^2=1$, with
$\chi_{\mathrm{out}}=0$ on $\{F\le R_L\}$,
$\chi_{\mathrm{in}}=0$ on $\{F\ge2R_L\}$, and
\begin{equation}
|D^m\chi_{\mathrm{in}}|+|D^m\chi_{\mathrm{out}}|
 \le C(R_L\sigma)^{-m},\qquad m=1,2.
\label{eq:initial-cutoff-derivatives}
\end{equation}
Let $\{\chi_k^{\mathrm{sp}}\}$ be the spatial lattice partition of
Lemma~\ref{lem:parabolic-partition}\ref{item:spatial-partition}
with $r(x)\equiv\sigma$ and centers $z_k$.
Extend each $\chi_{\mathrm{out}}\chi_i^{\mathrm{par}}$ by zero
from $\Omega$ to $J\times\R^d$. This extension is smooth because
$\chi_{\mathrm{out}}=0$ on $\{F\le R_L\}$.
After discarding identically zero functions, the products
\(
\{\chi_{\mathrm{in}}\chi_k^{\mathrm{sp}}\}_k
 \;\cup\;\{\chi_{\mathrm{out}}\chi_i^{\mathrm{par}}\}_i
\)
form a locally finite square partition $\{\chi_i\}$ on $J\times\R^d$.
Index the first family by $\mathcal I_{\mathrm{sp}}$ and the second
by $\mathcal I_{\mathrm{par}}$. Their square sums are
$\chi_{\mathrm{in}}^2$ and $\chi_{\mathrm{out}}^2$, respectively.

If $\chi_{\mathrm{in}}\chi_k^{\mathrm{sp}}\not\equiv0$,
then $F(z_k)\le2R_L+4L$.
A nearest point $y_k\in S$ then satisfies
\begin{equation*}
\sigma\langle y_k\rangle
 \le\sigma\langle z_k\rangle+\sigma \dist(z_k,S)
 \le F(z_k)\le C_\vartheta L,\qquad
B(z_k,2L\sigma)\subseteq B(y_k,C_\vartheta L\sigma).
\end{equation*}
These give part~\ref{item:localization-spatial}.
In either window family, use each parabolic cutoff's center
$(t_i,x_i)$ and
$r_i=r_J(t_i,x_i)$ in part~\ref{item:localization-parabolic}
to define $J_i$ and $Q_i$.
Then $r_i^2\le |J|\le h$ and
$r_i\le\ell(t_i,x_i)$.
For $i\in\mathcal I_{\mathrm{par}}$,
$|t-t_i|\le9L^2r_i^2\le(9/c_0)t_i$ on $J_i$.
Our choice $c_0\ge128$ gives $t_i/2\le t\le3t_i/2$.

For the positive-time window family, Lemma~\ref{lem:parabolic-partition}\ref{item:parabolic-partition} gives the derivative bounds with $r=r_J$.
For the near-zero family, the same lemma and
\eqref{eq:initial-cutoff-derivatives}, with $r_J\le\sigma$,
give these bounds by the product rule, with $C$ depending only
on $d,\vartheta$ and independent of $L$.
For parabolic cutoffs, $\sup J_i=\min\{\sup J,t_i+9L^2r_i^2\}$,
and $\chi_i=0$ for $t\ge t_i+4L^2r_i^2$.
Since $u=0$ near $\sup J$, each $(\chi_i u)|_{J_i}$ vanishes near
$\sup J_i$; spatial cutoffs have $J_i=J$.
Since the cutoffs have bounded derivatives on each test's compact
spatial support over $J$, the product rule gives
$(\chi_i u)|_{J_i}\in\mathcal W_{\alpha,0}(J_i)$.

\paragraph*{Errors in the squared norms.}
Fix $u\in\mathcal C_{\alpha,0}(J)$ and use the errors
$E_i^{\mathcal T}$ from Lemma~\ref{lem:square-partition-errors}.
For the partition constructed here, use
Assumption~\ref{def:initial-laws} and put $r=r_J$.
For $J=(T_1,T_2)$, \eqref{eq:initial-partition-windows} gives
$T_1\ge T_2/2$ in the positive-time case and
$T_1\ge\vartheta T_2$ in the near-zero case.
In either case, $t^{-1}\le C|J|^{-1}$ for $t\in J$.
Since $r^{-1}=\max\{|J|^{-1/2},\ell^{-1}\}$,
\eqref{eq:weighted-scale-gradient} gives
$|D_xr^{-1}|\le C|J|^{-1}$, and \eqref{eq:weighted-hardy} gives
\begin{equation*}
\|r^{-1}f\|_{L^2(\mu_t)}
 \le C\bigl(\|Df\|_{L^2(\mu_t)}
             +|J|^{-1/2}\|f\|_{L^2(\mu_t)}\bigr),
\qquad f\in H^1(\mu_t),\quad t\in J.
\end{equation*}
Apply this estimate componentwise to $Du$, then to $u$ and $r^{-1}u$,
using $D(r^{-1}u)=r^{-1}Du+uD(r^{-1})$.
Integration against $t^\alpha\,\mathrm{d}t$ and \eqref{eq:short-graph-budget}
yield
\begin{equation}
\|r^{-1}Du\|_{\alpha,J}+\|r^{-2}u\|_{\alpha,J}
 \le C\mathscr G_J(u),
\label{eq:localization-budget}
\end{equation}
with $C$ independent of $L$.
The product rule, Lemma~\ref{lem:parabolic-partition}\ref{item:parabolic-partition} and
\eqref{eq:localization-budget} give
\begin{equation*}
\sum_i\|E_i^{\mathcal T}\|_{\alpha,J}^2
 \le C\bigl(L^{-2}\|r^{-1}Du\|_{\alpha,J}^2
             +L^{-4}\|r^{-2}u\|_{\alpha,J}^2\bigr)
 \le CL^{-2}\mathscr G_J(u)^2.
\end{equation*}
For the cross term in \eqref{eq:localization-expansion}, use
\eqref{eq:localization-cancellation} to obtain
\begin{equation*}
\left\|\sum_i\chi_iE_i^{\mathcal T}\right\|_{\alpha,J}
 \le CL^{-2}\|r^{-2}u\|_{\alpha,J}
 \le CL^{-2}\mathscr G_J(u).
\end{equation*}
Together with $\|\mathcal Tu\|_{\alpha,J}\le C\mathscr G_J(u)$,
these estimates in \eqref{eq:localization-expansion}, summed over
$\mathcal T\in\{\mathcal H_q,P_A\}$, prove
\eqref{eq:square-localization-error}.
\end{proof}

\begin{proof}[Proof of Lemma~\ref{lem:time-window-extension}]
Let $L\ge2$, to be chosen below, and set
$\vartheta=\mathrm{e}^{-4L}$.
Assume \eqref{eq:two-window-input}, and fix
$J=(T_1,T_2)\in\mathcal J((0,b_\vartheta))$ and
$u\in\mathcal C_{\alpha,0}(J)$.
Use the reference evolution $U(t,s)$ from
Lemma~\ref{lem:reference-smoothing}\ref{item:reference-smoothing}
and set $f=P_{a/2}u=P_Au+(A-a/2):D^2u$,
so $\|f\|_{\alpha,J}\le C\mathscr G_J(u)$.
Variation of constants on positive-time subintervals of $J$ gives,
for $t\in J$,
\(
u(t)=\int_t^{T_2}U(t,s)f(s)\,\mathrm{d}s
\).

\paragraph*{Localization of the source.}
Using the one-dimensional cutoff from \eqref{eq:spatial-cutoffs},
normalize the functions $\chi(L^{-1}\log(t/T_2)-k)$,
$k\in\mathbb Z$, by the square root of their squared sum.
The resulting functions satisfy
\begin{equation*}
T_k^\pm=T_2\mathrm{e}^{L(k\pm2)},\quad
\chi_k\in C_c^\infty((T_k^-,T_k^+)),\quad
\sum_k\chi_k^2=1,\quad
\sum_k|\chi_k'(t)|^2\le CL^{-2}t^{-2}.
\end{equation*}
Define
\begin{equation}
u_k(t)=\int_t^{T_2}U(t,s)\chi_k(s)f(s)\,\mathrm{d}s.
\label{eq:localized-source-evolution}
\end{equation}
For each nonempty $J_k=J\cap(T_k^-,T_k^+)$, set $v_k=u_k|_{J_k}$.
The backward evolution may make $u_k$ nonzero before $\inf J_k$;
restriction to $J_k$ is allowed because the initial trace is free.
The identity $P_{a/2}u_k=\chi_k f$ avoids the term $\chi_k'u$
that would arise from cutting off $u$ directly.
It remains to compare $D^2u_k$ with $\chi_kD^2u$.
Subtracting the two evolution representations gives
\begin{equation}
u_k(t)-\chi_k(t)u(t)
 =\int_t^{T_2}U(t,s)[\chi_k(s)-\chi_k(t)]f(s)\,\mathrm{d}s.
\label{eq:localized-source-difference}
\end{equation}

\paragraph*{Comparing the Hessians.}
Integration of the derivative bound and Minkowski's inequality give
\begin{equation*}
\left(\sum_k|\chi_k(s)-\chi_k(t)|^2\right)^{1/2}
 \le CL^{-1}\log(s/t),\qquad T_1<t<s<T_2.
\end{equation*}
Define the Hessian error
\(
E_k(t)=\int_t^{T_2}D^2U(t,s)[\chi_k(s)-\chi_k(t)]f(s)\,\mathrm{d}s.
\)
Set $F(y)=t^{(\alpha+1)/2}\|f(t)\|_{L^2(\mu_t)}$ for $t=\mathrm{e}^y\in J$
and extend $F$ by zero to $\R$, so $\|F\|_{L^2(\R)}=\|f\|_{\alpha,J}$.
With $s=t\mathrm{e}^v$ and
$v\mathrm{e}^v/(\mathrm{e}^v-1)\le v+1$, the Hessian bound in
Lemma~\ref{lem:reference-smoothing}\ref{item:reference-smoothing}
and the preceding cutoff estimate give, for a.e. $t\in J$,
\begin{equation*}
\begin{aligned}
&t^{(\alpha+1)/2}\int_t^{T_2}
 \Bigl(\sum_k|\chi_k(s)-\chi_k(t)|^2\Bigr)^{1/2}
 \|D^2U(t,s)f(s)\|_{L^2(\mu_t)}\,\mathrm{d}s\\
&\quad\le\frac{C}{L}\int_0^\infty
(v+1)\mathrm{e}^{-(\alpha+1)v/2}F(y+v)\,\mathrm{d}v.
\end{aligned}
\end{equation*}
Since this kernel is integrable on $(0,\infty)$,
the integrals defining $E_k$
are absolutely convergent in $L^2(\mu_t)$ for a.e. $t\in J$.
Minkowski's and Young's inequalities give
\begin{equation}
\left(\sum_k\|E_k\|_{\alpha,J}^2\right)^{1/2}
 \le CL^{-1}\|f\|_{\alpha,J}.
\label{eq:causal-commutator-square-sum}
\end{equation}

The $j=0,1$ bounds in Lemma~\ref{lem:reference-smoothing}\ref{item:reference-smoothing} and Young's inequality give
$u_k,Du_k\in L^2_\alpha(J)$, using $(t/s)^{\alpha/2}\le1$.
Absolute convergence of the integrals defining $E_k$ justifies
differentiating \eqref{eq:localized-source-difference} twice in space,
in the distributional sense. Thus
\begin{equation}
D^2u_k=\chi_kD^2u+E_k
\quad\text{in }\mathcal D'(J\times\R^d).
\label{eq:localized-source-hessian}
\end{equation}
Together with $P_{a/2}u_k=\chi_k f$ and
\eqref{eq:causal-commutator-square-sum}, this gives
$u_k\in\mathcal W_\alpha(J)$.

The endpoints satisfy
\(
\inf J_k\ge T_k^-=\vartheta T_k^+
\ge\vartheta\sup J_k.
\)
Thus $J_k\in\mathcal J_{b_\vartheta}^0(\vartheta)$.
The terminal gap of $f$ and the support of $\chi_k$ make
$\chi_k f$ vanish on a left neighborhood of $\sup J_k$ and at
all later times. The representation \eqref{eq:localized-source-evolution} gives the same
terminal gap for $v_k$, hence $v_k\in\mathcal W_{\alpha,0}(J_k)$.

On $J_k$, the source equation
$P_{a/2}v_k=\chi_k f$ and the Hessian identity \eqref{eq:localized-source-hessian} give
\begin{equation*}
P_Av_k
      =\chi_kP_Au-(A-a/2):E_k, \qquad
\mathcal H_qv_k=\chi_k\mathcal H_qu+q^{1/2}E_kq^{1/2}.
\end{equation*}
Since $\chi_k|_J$ is supported in $J_k$, one has
$\sum_k\|\chi_k\mathcal Tu\|_{\alpha,J_k}^2
=\|\mathcal Tu\|_{\alpha,J}^2$ for
$\mathcal T\in\{\mathcal H_q,P_A\}$.
Apply \eqref{eq:two-window-input} to each $v_k$ and sum the squared
inequalities. The two operator identities above and
\eqref{eq:causal-commutator-square-sum} then give, by the triangle inequality,
\begin{equation*}
\|\mathcal H_qu\|_{\alpha,J}
 \le k_0\|P_Au\|_{\alpha,J}
       +\frac{C(1+k_0)}{L}\mathscr G_J(u).
\end{equation*}
Here $C$ depends only on $\mathfrak R,\mathfrak E,\alpha$.
Since $\mathscr G_J(u)\le
\lambda_q^{-1}\|\mathcal H_qu\|_{\alpha,J}+\|P_Au\|_{\alpha,J}$,
we can absorb the Hessian norm by choosing $L\ge2$ so that
$C(1+k_0)L^{-1}(1+k_1\lambda_q^{-1})<k_1-k_0$.
This gives $\|\mathcal H_qu\|_{\alpha,J}\le k_1\|P_Au\|_{\alpha,J}$.
By \eqref{eq:short-time-constants}, this proves
$G_\alpha^{\mathrm{fix}}(J)\le k_1$ for every
$J\in\mathcal J((0,b_\vartheta))$, including intervals with left
endpoint zero. The choice of $L$, and hence of $\vartheta$, is
independent of the initial law.

If $h\le\min\{h_+,b_\vartheta/2\}$ and
$J=(T_1,T_2)\in\mathcal J_h(I)$ has $T_2>b_\vartheta$, then
$T_1>b_\vartheta/2$, so the positive-time part of
\eqref{eq:two-window-input} gives
$G_\alpha^{\mathrm{fix}}(J)\le k_0$.
Together, the two cases prove the asserted short-interval bound.
\end{proof}

\subsection{Terminal traces and linear solvability}
\label{subsec:zero-terminal-solvability}

\begin{proof}[Proof of Theorem~\ref{thm:linear-theory}]
Write $J=(T_1,T_2)$. The zero-terminal space $\mathcal W_{\alpha,0}(J)$
is defined by the local trace in Lemma~\ref{lem:graph-core}.
This allows us to establish invertibility before using the inverse
of $P_{a/2}$ to prove the global terminal trace bound.

\paragraph*{Finite gain and invertibility.}
Suppose $G_\alpha^{\mathrm{fix}}(J)<\infty$.
Apply Lemma~\ref{lem:direct-energy}\ref{item:energy-damping} with
$\kappa=1+B_b^2/\lambda_a$ and compare the damped and undamped
norms, whose equivalence factors are bounded by $\mathrm{e}^{\kappa T}$.
Together with \eqref{eq:short-time-constants}, the ellipticity of $q$
and $u_t=-P_Au-A:D^2u$, this gives
\begin{equation}
\|u\|_{\mathcal W_\alpha(J)}\le C\|P_Au\|_{\alpha,J},
\qquad u\in\mathcal W_{\alpha,0}(J),
\label{eq:local-full-graph}
\end{equation}
with the dependence on $G_\alpha^{\mathrm{fix}}(J)$ stated in
part~\ref{item:linear-solvability}.
Equations~\eqref{eq:classical-terminal-solution} with $g=0$ and
\eqref{eq:weighted-lp-embedding} give zero-terminal solutions for
sources in $C_c^\infty(J\times\R^d)$.
Estimate~\eqref{eq:local-full-graph}, density in $L^2_\alpha(J)$ and
completeness in Lemma~\ref{lem:graph-core} extend this solver to a
bounded inverse $S_J$.
Conversely, bijectivity implies $G_\alpha^{\mathrm{fix}}(J)<\infty$
by the bounded inverse theorem, and
Lemma~\ref{lem:weight-operations}\ref{item:interval-gains} gives
$G_\alpha(J)\le G_\alpha^{\mathrm{fix}}(J)$.

Now suppose $G_\alpha(J)<\infty$ and choose $h_0>0$ with
$G_{\alpha,h_0}(J)<\infty$.
For every $E\in\mathcal J_{h_0}(J)$,
$G_\alpha^{\mathrm{fix}}(E)\le G_{\alpha,h_0}(J)$, so
\eqref{eq:local-full-graph} holds on these intervals with a common constant.
If $|J|>h_0$, set $\delta=h_0/2$ and partition $J$ at
$T_1=t_0<\cdots<t_N=T_2$, with $E_j=(t_{j-1},t_j)$ and
$\delta/2\le |E_j|\le\delta$ for $j=1,\ldots,N$.
For $u\in\mathcal W_{\alpha,0}(J)$, estimate~\eqref{eq:local-full-graph}
applies directly to $u|_{E_N}$, which has zero terminal trace.
For $j<N$, apply \eqref{eq:local-full-graph} to $u$ times a cutoff
on $(t_{j-1},t_{j+1})$ equal to one on $E_j$, zero near $t_{j+1}$
and with derivative bounded by $C/\delta$. This gives
\begin{equation*}
\|u\|_{\mathcal W_\alpha(E_j)}
 \le C\bigl(\|P_Au\|_{\alpha,J}
                   +\delta^{-1}\|u\|_{\alpha,E_{j+1}}\bigr).
\end{equation*}
Iterating backwards from $E_N$ through at most $2T/\delta$
intervals proves \eqref{eq:local-full-graph} on $J$ with the
alternative dependence on $h_0,G_{\alpha,h_0}(J)$.
In particular, $G_\alpha^{\mathrm{fix}}(J)<\infty$, completing the
equivalence in part~\ref{item:linear-invertibility}.

Lemma~\ref{lem:weight-operations}\ref{item:interval-gains} and the
equivalence just proved in part~\ref{item:linear-invertibility}
give invertibility on every subinterval of $J$.
For $K=(s_1,s_2)\in\mathcal J(J)$ and $F\in L^2_\alpha(K)$,
put $v=S_J(\1_KF)$.
If $s_2<T_2$, uniqueness on $(s_2,T_2)$ gives $v=0$ there;
in either case $v|_K$ has zero terminal trace.
Uniqueness on $K$ gives the restriction identity in part~\ref{item:linear-invertibility}.

\paragraph*{Global terminal trace.}
For part~\ref{item:linear-trace}, let $J\in\mathcal J(I)$ be arbitrary.
For every $g\in H^1(\mu_{T_2})$,
Lemma~\ref{lem:reference-terminal-energy}\ref{item:reference-terminal-energy}
and $t^\alpha\le T^\alpha$
give $U(\cdot,T_2)g\in\mathcal W_\alpha(J)$ with the asserted
extension bound and local terminal trace $g$.

To bound the trace, set $E=((T_1+T_2)/2,T_2)$.
The uniform bound in
Proposition~\ref{prop:uniform-initial-law-bounds}\ref{item:uniform-positive-time},
with $A=a/2$, $q=a$, $K=2$, $R=0$ and $\varepsilon=1$, gives
$G_{\alpha,h_0}(E;\rho,a/2,a)\le3$ for an $h_0>0$ depending only
on $\mathfrak R$ and $\alpha,J$.
Part~\ref{item:linear-invertibility} and
\eqref{eq:local-full-graph} therefore give a bounded inverse
$S_E^{a/2}:L^2_\alpha(E)\to\mathcal W_{\alpha,0}(E)$ of $P_{a/2}$,
with norm controlled by the same data.

First let $u\in\mathcal W_\alpha(J)$ have compact spatial support.
Then $g=\operatorname{Tr}^{\mathrm{loc}}_{T_2}u\in H^1(\mu_{T_2})$,
and uniqueness for $P_{a/2}$ on $\mathcal W_{\alpha,0}(E)$ gives
\(
U(\cdot,T_2)g
 =u-S_E^{a/2}\bigl((P_{a/2}u)|_E\bigr)
\)
on $E$.
Equation~\eqref{eq:terminal-trace-energy} and the bound for
$S_E^{a/2}$ yield
\begin{equation*}
\|g\|_{H^1(\mu_{T_2})}
 \le C\|U(\cdot,T_2)g\|_{\mathcal W_\alpha(E)}
 \le C\|u\|_{\mathcal W_\alpha(J)}.
\end{equation*}
For general $u$, apply this estimate to $\chi_Ru$ with the cutoffs
from \eqref{eq:spatial-cutoffs}.
Since $\chi_Ru\to u$ in $\mathcal W_\alpha(J)$ and
$\operatorname{Tr}^{\mathrm{loc}}_{T_2}(\chi_Ru)
=\chi_R\operatorname{Tr}^{\mathrm{loc}}_{T_2}u$,
Fatou's lemma gives the asserted trace bound.
The identity $\operatorname{Tr}_{T_2}U(\cdot,T_2)g=g$ then proves
that the reference extension is a bounded right inverse.

Under the equivalent conditions in part~\ref{item:linear-invertibility},
set $w=U(\cdot,T_2)g$ and $u=w+S_J(f-P_Aw)$.
Then $P_Au=f$ and $\operatorname{Tr}_{T_2}u=g$.
Since $P_Aw=(a/2-A):D^2w$, estimate~\eqref{eq:local-full-graph}
and the extension bound in part~\ref{item:linear-trace} give
part~\ref{item:linear-solvability} with both stated controls of $C_{\mathrm{sol}}$.
Uniqueness follows from injectivity on $\mathcal W_{\alpha,0}(J)$;
the formula for $u$ gives linearity and $\mathscr L_J(f,0)=S_Jf$.
\end{proof}

\subsection{Interval averaging and the damping limit}
\label{subsec:damping-limit}

\begin{proof}[Proof of Proposition~\ref{prop:poisson}]
For the interval comparison, use the
extension construction in the proof of
Lemma~\ref{lem:weight-operations}\ref{item:interval-gains} with the
damped weight and with $\mathcal H_qu$, $u$ or $Du$ in the numerator.
On $K=(s_1,s_2)\in\mathcal J(J)$, replacing $s_2$ by $T_2$ in the
damping factor multiplies both norms in each ratio by
$\mathrm{e}^{-\kappa(T_2-s_2)}$, which cancels.

\paragraph*{Lower bound for the Hessian gain.}
For the Hessian gain, fix $\kappa>0$ and set
$\mathcal T=\mathcal H_qS_J$, a bounded operator from scalar functions
to matrix fields in $L^2_\alpha(J)$.
All operator norms below use the undamped norm $\|\cdot\|_{\alpha,J}$
on both spaces, so $\|\mathcal T\|_{\mathrm{op}}=G_\alpha^{\mathrm{fix}}(J)$.
For $F\in L^2_\alpha(K)$, the restriction identity in Theorem~\ref{thm:linear-theory}\ref{item:linear-invertibility} identifies $(\mathcal T(\1_KF))|_K$ with $\mathcal H_qS_KF$.
Thus
\begin{equation}
\|\1_K\mathcal T\1_K\|_{\mathrm{op}}=G_\alpha^{\mathrm{fix}}(K),
\qquad K\in\mathcal J(J).
\label{eq:hessian-compression}
\end{equation}

The isometry $F\mapsto \mathrm{e}^{\kappa(t-T_2)}F$ from the damped to the
undamped norm gives $\mathcal T_\kappa F=\mathrm{e}^{\kappa t}\mathcal T(\mathrm{e}^{-\kappa t}F)$
with
$\|\mathcal T_\kappa\|_{\mathrm{op}}=G_{\alpha,\kappa}^{\mathrm{fix}}(J).$
For $h>0$ and $K\in\mathcal J_h(J)$, the largest ratio between
two values of $\mathrm{e}^{\kappa t}$ on $K$ is at most $\mathrm{e}^{\kappa h}$.
The conjugation formula and \eqref{eq:hessian-compression} therefore give
\begin{equation}
G_{\alpha,\kappa}^{\mathrm{fix}}(J)
 \ge\sup_{K\in\mathcal J_h(J)}
       \|\1_K\mathcal T_\kappa\1_K\|_{\mathrm{op}}
 \ge \mathrm{e}^{-\kappa h}G_{\alpha,h}(J)
       \xrightarrow{h\downarrow0}G_\alpha(J).
\label{eq:hessian-damped-lower-bound}
\end{equation}

\paragraph*{Interval averaging and upper bound.}
For the upper bound, interval averaging relates the damped gain
to gains on short intervals.
For related norm bounds under monotone reweighting of causal operators,
see \citep[Theorem~3.2]{chalendar-partington-2014}.
For $r\in\R$ and $h>0$, let $Q_r$ and $P_{r,h}$ multiply by
$\1_{J\cap(-\infty,r)}$ and $\1_{J\cap(r,r+h)}$ on both spaces.
For $r\in J$ and $F\in L^2_\alpha(J)$, the function $S_JQ_rF$
has zero source and zero terminal trace on $(r,T_2)$.
Theorem~\ref{thm:linear-theory}\ref{item:linear-invertibility} gives
$S_JQ_rF=0$ there. Hence
$(\mathrm{Id}-Q_r)\mathcal T Q_r=0$ for every $r\in\R$.
Consequently,
\begin{equation*}
(\mathrm{Id}-Q_r)\mathcal T Q_s
 =
 \begin{cases}
   P_{r,s-r}\mathcal T P_{r,s-r},&r<s,\\
   0,&r\ge s.
  \end{cases}
\end{equation*}
Substitute $\mathrm{e}^{\kappa t}=\kappa\int_{-\infty}^t \mathrm{e}^{\kappa r}\,\mathrm{d}r$
and $\mathrm{e}^{-\kappa t}=\kappa\int_t^\infty \mathrm{e}^{-\kappa s}\,\mathrm{d}s$
in the definition of $\mathcal T_\kappa$.
Fubini's theorem and the change of variables $h=s-r$ give
\begin{equation}
\mathcal T_\kappa
 =\left(\kappa\int_\R \mathrm{e}^{\kappa r}(\mathrm{Id}-Q_r)\,\mathrm{d}r\right)
    \mathcal T
    \left(\kappa\int_\R \mathrm{e}^{-\kappa s}Q_s\,\mathrm{d}s\right)
 =\kappa^2\int_0^\infty \mathrm{e}^{-\kappa h}
       \int_\R P_{r,h}\mathcal T P_{r,h}\,\mathrm{d}r\,\mathrm{d}h.
\label{eq:hessian-interval-averaging}
\end{equation}
Apply each operator integrand to a fixed input and take the resulting
integrals in the scalar or matrix-valued space $L^2_\alpha(J)$.
The $L^2_\alpha(J)$ norms are integrable because $J$ is bounded
and $\mathcal T$ is bounded.

For scalar $F$ and matrix-valued $G$ in
$L^2_\alpha(J)$,
\eqref{eq:hessian-compression} and Cauchy--Schwarz give
\begin{equation*}
\begin{aligned}
\left|\int_\R
 \langle\mathcal T P_{r,h}F,P_{r,h}G\rangle_{\alpha,J}\,\mathrm{d}r\right|
&\le G_{\alpha,h}(J)
       \int_\R\|P_{r,h}F\|_{\alpha,J}\,\|P_{r,h}G\|_{\alpha,J}\,\mathrm{d}r\\
&\le hG_{\alpha,h}(J)\|F\|_{\alpha,J}\,\|G\|_{\alpha,J},
\end{aligned}
\end{equation*}
since $\int_\R\|P_{r,h}F\|_{\alpha,J}^2\,\mathrm{d}r=h\|F\|_{\alpha,J}^2$, and likewise for $G$.
Combining this bound with \eqref{eq:hessian-interval-averaging} gives
\begin{equation*}
G_{\alpha,\kappa}^{\mathrm{fix}}(J)
 \le\int_0^\infty z\mathrm{e}^{-z}G_{\alpha,z/\kappa}(J)\,\mathrm{d}z, \qquad z=\kappa h.
\end{equation*}
Subtracting $G_\alpha(J)$ and using $\int_0^\infty z\mathrm{e}^{-z}\,\mathrm{d}z=1$,
together with \eqref{eq:hessian-damped-lower-bound}, proves the gain inequality in part~\ref{item:linear-damping}.
Lemma~\ref{lem:weight-operations}\ref{item:interval-gains} gives
$0\le\varepsilon(h)\le G_\alpha^{\mathrm{fix}}(J)-G_\alpha(J)<\infty$, and
$\varepsilon(h)\to0$ as $h\downarrow0$, so dominated convergence applied to this inequality proves the limit. The same argument gives uniform convergence when all functions
$\varepsilon$ are bounded by a common bounded function tending to zero at $h=0$.

For $\kappa>B_b^2/\lambda_a$, apply
Lemma~\ref{lem:direct-energy}\ref{item:energy-damping}
to $u\in\mathcal W_{\alpha,0}(J)$ and take the suprema in
\eqref{eq:damped-constants} to obtain the two inequalities in
part~\ref{item:linear-lower-order}.
Lemma~\ref{lem:weight-operations}\ref{item:interval-gains} and
part~\ref{item:linear-damping} give
$G_{\alpha,\kappa}^{\mathrm{fix}}(J)\le G_\alpha^{\mathrm{fix}}(J)$.
Since $\kappa_{\mathrm{eff}}/\kappa\to1$, this also proves the $O$
bounds in part~\ref{item:linear-lower-order}.
\end{proof}

\section{Whole-line and initial model gains}
\label{app:model-gains}

The model tuples and their weights are defined in
\eqref{eq:initial-heat-tuple} and \eqref{eq:whole-line-tuple}.
We use the unitary Fourier transform
\begin{equation*}
\widehat f(\xi)=(2\pi)^{-n/2}\int_{\R^n}\mathrm{e}^{-ix\cdot\xi}f(x)\,\mathrm{d}x,
\qquad \xi\in\R^n.
\end{equation*}
The inverse transform uses the phase $\mathrm{e}^{ix\cdot\xi}$ with the same
normalization. Partial transforms are indicated by the variable subscript.

\subsection{Whole-line gains and interval bounds}
\label{subsec:whole-line-gain-proofs}

\begin{proof}[Proof of Proposition~\ref{prop:whole-line-gains}]
Complex tests give the same constants, since both squared norms
split into the sums for their real and imaginary parts.
For $\zeta\in\R^d$, set
$\mathbf M_\zeta=(\delta_\zeta,a_*,A_*,q_*)$; the flat case is
$\zeta=0$. To remove the weight from both norms, put
$\sigma=\zeta^{\mathsf T}a_*\zeta$ and
$c_\zeta=\zeta^{\mathsf T}(a_*-A_*)\zeta$.
\paragraph*{Pure phases and the flat case.}
For $u\in C_c^\infty(\R\times\R^d)$, set
$v=\mathrm{e}^{\sigma s-\zeta\cdot z}u$, $\widehat v=\mathcal F_zv$,
$N(\xi)=\xi^{\mathsf T}q_*\xi+\zeta^{\mathsf T}q_*\zeta$,
and $D_\xi=\xi^{\mathsf T}A_*\xi+c_\zeta$. Then
\begin{equation}
\mathcal F_z(\mathrm{e}^{\sigma s-\zeta\cdot z}P_{A_*}u)
 =(-\partial_s+D_\xi-2i\xi^{\mathsf T}A_*\zeta)\widehat v, \quad
\bigl|\mathcal F_z(\mathrm{e}^{\sigma s-\zeta\cdot z}\mathcal H_{q_*}u)\bigr|
 =N(\xi)|\widehat v|.
\label{eq:single-phase-conjugation}
\end{equation}
We claim that
\begin{equation}
\mathfrak G(\mathbf M_\zeta)
 =\operatorname*{ess\,sup}_{\omega,\xi}
   \frac{N(\xi)}{\sqrt{D_\xi^2+(\omega+2\xi^{\mathsf T}A_*\zeta)^2}}
 =\operatorname*{ess\,sup}_{\xi}\frac{N(\xi)}{|D_\xi|}.
\label{eq:single-phase-multiplier}
\end{equation}
By standard Sobolev density, the supremum over the conjugated
tests is unchanged if $C_c^\infty(\R\times\R^d)$ is replaced by
$H^1(\R;L^2(\R^d))\cap L^2(\R;H^2(\R^d))$:
both conjugated operators are continuous from this space into $L^2$.
Plancherel gives the first equality in \eqref{eq:single-phase-multiplier},
and minimizing the denominator over $\omega$ gives the second.
These identities also hold when the supremum is infinite.

For $\zeta=0$, the Rayleigh quotient gives \eqref{eq:flat-gain}.
When $c_\zeta>0$, for $\xi\ne0$ the quotient $N(\xi)/D_\xi$ is a weighted mean of
$\xi^{\mathsf T}q_*\xi/(\xi^{\mathsf T}A_*\xi)$ and
$\zeta^{\mathsf T}q_*\zeta/c_\zeta$.
Evaluating at $\xi=0$ and letting $\xi=r\eta$, $r\to\infty$,
with $\eta$ maximizing the Rayleigh quotient, gives
\eqref{eq:phase-gain-formula} for $\lambda=\delta_\zeta$.
If $\zeta\ne0$ and $c_\zeta\le0$, a frequency with
$\xi^{\mathsf T}A_*\xi=-c_\zeta$ has $D_\xi=0<N(\xi)$.
By continuity, $N(\xi)/|D_\xi|$ is essentially unbounded near
this frequency, so \eqref{eq:single-phase-multiplier} gives
$\mathfrak G(\mathbf M_\zeta)=\infty$.
In particular, whenever $\mathfrak G(\mathbf M_\zeta)<\infty$,
\begin{equation}
D_\xi\ge0,\qquad
N(\xi)\le\mathfrak G(\mathbf M_\zeta)D_\xi,
\qquad\forall\xi\in\R^d.
\label{eq:single-phase-finite-bound}
\end{equation}

\paragraph*{Mixtures and phase recovery.}
For a mixture, we show
\begin{equation}
\mathfrak G(\mathbf M)
 =\sup_{\zeta\in\supp\lambda}\mathfrak G(\mathbf M_\zeta).
\label{eq:whole-line-mixture-gain}
\end{equation}
Tonelli gives the upper bound by integrating the squared
estimates for $\mathbf M_\zeta$ against $\lambda$ for the same compact test.
The reverse bound must recover each direction in $\supp\lambda$,
including directions that carry no atom of $\lambda$.
Fix $\zeta_0\in\supp\lambda$, a finite
$0<k_0<\mathfrak G(\mathbf M_{\zeta_0})$, and a compact test $u$ with
\begin{equation*}
\|\mathcal H_{q_*}u\|_{L^2(W_{\mathbf M_{\zeta_0}}\,\mathrm{d}s\,\mathrm{d}z)}^2
 -k_0^2\|P_{A_*}u\|_{L^2(W_{\mathbf M_{\zeta_0}}\,\mathrm{d}s\,\mathrm{d}z)}^2>0.
\end{equation*}
To concentrate the phase law at $\zeta_0$, use the joint translation
$(s_R,z_R)=(-R,-2Ra_*\zeta_0)$ and define
\begin{equation*}
\lambda_R(\mathrm{d}\zeta)=
 \frac{\mathrm{e}^{-2R|\zeta-\zeta_0|_{a_*}^2}\,\lambda(\mathrm{d}\zeta)}
      {\displaystyle\int_{\R^d}\mathrm{e}^{-2R|\eta-\zeta_0|_{a_*}^2}
                               \,\lambda(\mathrm{d}\eta)}.
\end{equation*}
For every $\delta>0$, $\zeta_0\in\supp\lambda$ gives
$\lambda(\{|\zeta-\zeta_0|_{a_*}<\delta/2\})>0$, and
\begin{equation*}
\lambda_R\bigl(\{|\zeta-\zeta_0|_{a_*}\ge\delta\}\bigr)
\le
\frac{\mathrm{e}^{-3R\delta^2/2}}
     {\lambda\bigl(\{|\zeta-\zeta_0|_{a_*}<\delta/2\}\bigr)}
\longrightarrow0.
\end{equation*}
Thus $\lambda_R\rightharpoonup\delta_{\zeta_0}$ as $R\to\infty$.
Set $u_R(s,z)=u(s-s_R,z-z_R)$.
For each $\mathcal T\in\{\mathcal H_{q_*},P_{A_*}\}$, the map
$\zeta\mapsto\|\mathcal T u\|_{L^2(W_{\mathbf M_\zeta}\,\mathrm{d}s\,\mathrm{d}z)}^2$
is continuous and bounded on $\supp\lambda$, since $u$ and $\lambda$
have compact support. Changing variables, Tonelli's theorem and weak convergence give
\begin{equation}
\frac{\|\mathcal T u_R\|_{L^2(W_{\mathbf M}\,\mathrm{d}s\,\mathrm{d}z)}^2}
     {W_{\mathbf M}(s_R,z_R)}
 =\int_{\R^d}
    \|\mathcal T u\|_{L^2(W_{\mathbf M_\zeta}\,\mathrm{d}s\,\mathrm{d}z)}^2\,\lambda_R(\mathrm{d}\zeta)
 \longrightarrow\|\mathcal T u\|_{L^2(W_{\mathbf M_{\zeta_0}}\,\mathrm{d}s\,\mathrm{d}z)}^2.
\label{eq:phase-test-recovery}
\end{equation}
Thus the norm quotient of $u_R$ exceeds $k_0$ for all sufficiently large $R$.
Moreover, for every $s_2\in\R$,
$u_R\in C_c^\infty((-\infty,s_2)\times\R^d)$ for all sufficiently large $R$,
since $s_R=-R$ and $u$ has compact support.
Taking the supremum over finite $k_0$ and then over $\zeta_0$ proves
\eqref{eq:whole-line-mixture-gain}, including infinite values and
$\lambda=\delta_0$. For phase tuples, inserting the pure-phase values obtained from
\eqref{eq:single-phase-multiplier} into \eqref{eq:whole-line-mixture-gain}
gives \eqref{eq:phase-gain-formula}.
\end{proof}

\begin{proof}[Proof of Corollary~\ref{cor:whole-line-intervals}]
For part~\ref{item:whole-line-intervals}, it suffices to consider
$k_*:=\mathfrak G(\mathbf M)<\infty$.
Fix $J=(s_1,s_2)\in\mathcal J(\R)$ and $\zeta\in\supp\lambda$.
By \eqref{eq:single-phase-finite-bound} and \eqref{eq:whole-line-mixture-gain},
$D_\xi\ge0$ and $N(\xi)\le k_*D_\xi$.
For $u\in\mathcal W_{0,0}(J;W_{\mathbf M_\zeta})$, use the conjugation
$v=\mathrm{e}^{\sigma s-\zeta\cdot z}u$ and the symbols from
\eqref{eq:single-phase-conjugation}.
Then $v\in H^1(J;L^2(\R^d))\cap L^2(J;H^2(\R^d))$
and $v(s_2)=0$.
Thus $\widehat v=\mathcal F_zv$ satisfies
$\widehat v(\cdot,\xi)\in H^1(J)$ and $\widehat v(s_2,\xi)=0$
for a.e. $\xi$.
The initial trace is free, so we retain its boundary term when
integrating by parts:
\begin{equation*}
\begin{aligned}
\|(-\partial_s+D_\xi-2i\xi^{\mathsf T}A_*\zeta)\widehat v\|_{L^2(J)}^2
 &=\|\partial_s\widehat v+2i\xi^{\mathsf T}A_*\zeta \widehat v\|_{L^2(J)}^2
  +D_\xi^2\|\widehat v\|_{L^2(J)}^2\\
 &\quad+D_\xi|\widehat v(s_1,\xi)|^2\ge k_*^{-2}N(\xi)^2\|\widehat v\|_{L^2(J)}^2.
\end{aligned}
\end{equation*}
Spatial Plancherel gives
\(
G_0^{\mathrm{fix}}(J;W_{\mathbf M_\zeta},A_*,q_*)
 \le k_*
\)
for every $J\in\mathcal J(\R)$ and $\zeta\in\supp\lambda$.
Lemma~\ref{lem:weight-operations}\ref{item:positive-mixtures}
then gives $G_0^{\mathrm{fix}}(J;W_{\mathbf M},A_*,q_*)\le k_*$.
Since this bound holds on every terminal subinterval $(r,s_2)$,
Lemma~\ref{lem:weight-operations}\ref{item:time-weights} yields the interval bound in part~\ref{item:whole-line-intervals} when $s_1\ge0$.

For the flat tuple, fix $\alpha,J$ as in part~\ref{item:whole-line-intervals}.
For any $v\in C_c^\infty(\R\times\R^d)$ and $s_0\in J$, set
\(
u_r(s,z)=r^2v\bigl((s-s_0)/r^2,z/r\bigr).
\)
For all sufficiently small $r>0$, $u_r\in C_c^\infty(J\times\R^d)$.
The rescaled time factor converges uniformly on $\supp v$ to a positive
constant. Parabolic rescaling therefore gives
\begin{equation*}
\frac{\|\mathcal H_{q_*}u_r\|_{\alpha,J;1}}
     {\|P_{A_*}u_r\|_{\alpha,J;1}}
 \longrightarrow
\frac{\|\mathcal H_{q_*}v\|_{L^2(\mathrm{d}s\,\mathrm{d}z)}}
     {\|P_{A_*}v\|_{L^2(\mathrm{d}s\,\mathrm{d}z)}}
\qquad\text{as }r\downarrow0.
\end{equation*}
Taking the supremum over $v$ yields
$G_\alpha^{\mathrm{fix}}(J;1,A_*,q_*)\ge\mathfrak G(\mathbf M)$,
and hence equality in part~\ref{item:whole-line-intervals}.

For part~\ref{item:whole-line-past}, fix $s_2\in\R$.
The recovery tests in \eqref{eq:phase-test-recovery} eventually lie in
$C_c^\infty((-\infty,s_2)\times\R^d)$, so
\eqref{eq:whole-line-mixture-gain} gives the lower bound in
part~\ref{item:whole-line-past}, also for infinite gains.
The reverse bound follows by inclusion of the test classes.
\end{proof}

\subsection{Homogeneous initial models}
\label{subsec:homogeneous-gain-proofs}

\begin{proof}[Proof of Proposition~\ref{prop:homogeneous-initial-norm}]
Put $\nu=(\alpha+1)/2+k/4>0$.
Homogeneity gives $W_{\mathbf T}(s,z)=s^{(k-d)/2}F(z/\sqrt s)$.
Under $s=\mathrm{e}^r$, $z=\mathrm{e}^{r/2}x$, the weighted measure becomes
$s^\alpha W_{\mathbf T}(s,z)\,\mathrm{d}s\,\mathrm{d}z
 =\mathrm{e}^{2\nu r}F(x)\,\mathrm{d}r\,\mathrm{d}x$.
For $u\in\mathcal C_{\alpha,0}((s_1,s_2);W_{\mathbf T})$,
where $0<s_1<s_2<\infty$, put
$f=P_Au$ and $J_{\log}=(\log s_1,\log s_2)$. Set
\begin{equation*}
v(r,x)=\mathrm{e}^{(\nu-1)r}u(\mathrm{e}^r,\mathrm{e}^{r/2}x),
\qquad g(r,x)=\mathrm{e}^{\nu r}f(\mathrm{e}^r,\mathrm{e}^{r/2}x).
\end{equation*}
Then $(-\partial_r+\mathcal L_A+\nu-1)v=g$, with zero terminal data, and
\begin{equation*}
\frac{\|\mathcal H_qu\|_{\alpha,(s_1,s_2);W_{\mathbf T}}}
     {\|P_Au\|_{\alpha,(s_1,s_2);W_{\mathbf T}}}
 =\frac{\|\mathcal H_qv\|_{L^2(J_{\log}\times\R^d,F(x)\,\mathrm{d}r\,\mathrm{d}x)}}
       {\|g\|_{L^2(J_{\log}\times\R^d,F(x)\,\mathrm{d}r\,\mathrm{d}x)}}.
\end{equation*}

\paragraph*{The time-domain operator.}
The commutation identity
\begin{equation}
\mathcal H_q\mathcal L_A=(\mathcal L_A+1)\mathcal H_q
\label{eq:hessian-commutation}
\end{equation}
shows that the Hessian $h=\mathcal H_qv$ solves
$(-\partial_r+\mathcal L_A+\nu)h=\mathcal H_qg$ with zero terminal data.
Define $\mathcal K$ on $C_c^\infty(\R\times\R^d)$ by
\begin{equation}
(\mathcal Kg)(r)=\int_0^\infty \mathrm{e}^{-\nu\tau}
                    S_A(\tau)\mathcal H_qg(r+\tau)\,\mathrm{d}\tau.
\label{eq:homogeneous-time-operator}
\end{equation}
Let $\|\mathcal K\|_{\mathrm{op}}$ denote its norm from scalar to
matrix-valued $L^2(F\,\mathrm{d}r\,\mathrm{d}x)$, allowing $+\infty$.
When this norm is finite, $\mathcal K$ extends boundedly by density.
For $g$ extended by zero outside $J_{\log}$, the distributional Duhamel
formula gives $\mathcal Kg=h$ on $J_{\log}$.
Restricting the output norm and taking suprema in
\eqref{eq:initial-model-gain} gives
$\mathfrak G_\alpha(\mathbf T)\le\|\mathcal K\|_{\mathrm{op}}$,
which also holds when $\|\mathcal K\|_{\mathrm{op}}=\infty$.
Intervals starting at zero are included by restriction and monotone convergence.

Conversely, for any $g\in C_c^\infty(\R\times\R^d)$, set
$f(s,z)=s^{-\nu}g(\log s,z/\sqrt s)$ and choose $s_2$ beyond the time support of $f$.
Let $u$ be the zero-terminal heat potential of $f$ on $(0,s_2)$,
obtained by time reversal and a linear spatial change
\citep[Section~2.3.1, Theorem~2]{evans-2010}.
The Gaussian decay of $u$ and its derivatives, together with the
polynomial mass growth of $\theta$, gives
$u\in\mathcal W_{\alpha,0}((s_1,s_2);W_{\mathbf T})$ for every $0<s_1<s_2$.
Its transformed Hessian is $\mathcal Kg$, which vanishes for $r\ge\log s_2$.
Applying \eqref{eq:initial-model-gain} and letting $s_1\downarrow0$
proves the reverse inequality, hence
\begin{equation}
\mathfrak G_\alpha(\mathbf T)=\|\mathcal K\|_{\mathrm{op}}.
\label{eq:homogeneous-time-norm}
\end{equation}

\paragraph*{The frequency representation.}
The norm $\|\mathcal K\|_{\mathrm{op}}$ is unchanged for complex sources.
Homogeneity gives $\int_{B_R}F\le C(1+R)^k$.
For $R>0$ and $\varphi\in C_c^\infty(B_R;\mathbb C)$, the Mehler formula
\eqref{eq:mehler-semigroup} therefore yields
\begin{equation}
\|S_A(\tau)\mathcal H_q\varphi\|_{L^2(F\,\mathrm{d}x)}
 \le C_R \mathrm{e}^{k\tau/4}\|D^2\varphi\|_{L^\infty(\R^d)},
 \qquad\tau\ge0.
\label{eq:mehler-test-bound}
\end{equation}

To apply Plancherel without assuming $\|\mathcal K\|_{\mathrm{op}}<\infty$,
fix a finite-dimensional $E\subset C_c^\infty(\R^d;\mathbb C)$,
equipped with the $L^2(F\,\mathrm{d}x)$ norm.
Norm equivalence on $E$ gives
$\|D^2\varphi\|_\infty\le C_E\|\varphi\|_{L^2(F\,\mathrm{d}x)}$ for $\varphi\in E$.
Together with \eqref{eq:mehler-test-bound} and
$\nu-k/4=(\alpha+1)/2>0$, this makes
$\mathrm{e}^{-\nu\tau}S_A(\tau)\mathcal H_q|_E$ integrable in operator norm.
Fourier transformation of \eqref{eq:homogeneous-time-operator} in $r$ gives the continuous multiplier
$M(\omega)|_E$, where
$M(\omega)=M_{A,q}(\nu-i\omega)$.
Plancherel's theorem, applied componentwise
\citep[Chapter~I, Theorem~2.3]{stein-weiss-1971}, identifies the norm
of $\mathcal K$ on the subspace of sources with $g(r,\cdot)\in E$
for a.e. $r$ as
\begin{equation*}
\operatorname*{ess\,sup}_{\omega\in\R}\|M(\omega)|_E\|_{\mathrm{op}}
 =\sup_{\omega\in\R}\|M(\omega)|_E\|_{\mathrm{op}}.
\end{equation*}
Finite sums of products of smooth compactly supported functions of $r$
and of $x$ are dense in
$C_c^\infty(\R\times\R^d)$ in the test-function topology.
Estimate~\eqref{eq:mehler-test-bound} makes $\mathcal K$ continuous
from the test-function topology to $L^2(F\,\mathrm{d}r\,\mathrm{d}x)$.
Taking the supremum over $E$ and using \eqref{eq:homogeneous-time-norm} proves
the claimed operator-norm formula, including infinite values.
\end{proof}

\subsection{Point initial models}
\label{subsec:point-zero-frequency-proofs}
\label{subsec:point-model-proofs}

\begin{proof}[Proof of Proposition~\ref{prop:point-initial-models}]
For part~\ref{item:point-representation},
Proposition~\ref{prop:homogeneous-initial-norm} with $k=0$ reduces the
claim to showing that the frequency supremum is attained at zero.
Use $x=a^{1/2}y$ and write
$\widehat A=a^{-1/2}Aa^{-1/2}$ and
$\widehat q=a^{-1/2}qa^{-1/2}$.
This sends $\gamma_a(x)\,\mathrm{d}x$ to $\gamma_{I_d}(y)\,\mathrm{d}y$.
The polar decomposition of $q^{1/2}a^{-1/2}$ shows that the Hessian norm
is also preserved.
Put $\nu=(\alpha+1)/2$ and
$M(\omega)=M_{\widehat A,\widehat q}(\nu-i\omega)$.
Unless indicated otherwise, spatial operators act on $y$, and norms
use $L^2(\gamma_{I_d}\,\mathrm{d}y)$, with Euclidean norms for vectors and
Frobenius norms for matrices.

Let $\mathscr C_n$ be the complex Hermite space of total degree $n$.
The number operator $\mathcal N=-\Delta+y\cdot D$ acts as $nI$ on
$\mathscr C_n$ \citep[Section~3]{lunardi-pallara-2020}, and
$\mathcal L_{\widehat A}=\mathcal N/2+(I_d/2-\widehat A):D^2$.
Put $\mathscr P_N=\bigoplus_{n=2}^N\mathscr C_n$ for $N\ge2$, and let
$\Pi_{\ge2}$ project orthogonally onto the complement of the affine
polynomials. The Hessian annihilates those polynomials, so on each
$\mathscr P_N$ define the projected source operator
\begin{equation}
\mathsf M_\alpha
 =\Pi_{\ge2}\bigl(\mathcal L_{\widehat A}+(\alpha-1)/2\bigr).
\label{eq:gaussian-source-operator}
\end{equation}
Since second derivatives lower Hermite degree by two, this operator
is triangular, with diagonal $(n-1+\alpha)/2>0$ on $\mathscr C_n$.
All inverses of $\mathsf M_\alpha-i\omega$ below are taken on
$\mathscr P_N$. On finite polynomial spaces, the Mehler integral
\eqref{eq:homogeneous-resolvent} equals
$(\mathcal L_{\widehat A}+\nu-i\omega)^{-1}\mathcal H_{\widehat q}$.
Commuting the Hessian by \eqref{eq:hessian-commutation} shifts
$\nu-1=(\alpha-1)/2$ to $\nu$ and gives, for every polynomial $f$
of degree at most $N$,
\begin{equation}
M(\omega)f
 =\mathcal H_{\widehat q}(\mathsf M_\alpha-i\omega)^{-1}\Pi_{\ge2}f,
\qquad N\ge2.
\label{eq:gaussian-polynomial-operator}
\end{equation}

For $\ell\in\{0,1\}$ and an integer $m\ge0$, put
\begin{equation*}
X_j=\mathscr C_{\ell+2j+2},\qquad
Y_j=\mathscr C_{\ell+2j}\otimes\mathbb C_{\mathrm{sym}}^{d\times d},
\qquad 0\le j\le m.
\end{equation*}
On $\bigoplus_{j=0}^mX_j$, the operator
$\mathsf M_\alpha-i\omega$ is upper bidiagonal, with diagonal blocks
$(d_j-i\omega)I$, where $d_j=(\ell+2j+\alpha+1)/2>0$, and
superdiagonal blocks independent of $\omega$.
Since $\mathcal H_{\widehat q}$ maps $X_j$ into $Y_j$,
\eqref{eq:gaussian-polynomial-operator} and bidiagonal inversion give
an upper triangular restriction $T(\omega)$ of $M(\omega)$, with blocks
\begin{equation}
T(\omega)_{ij}
 =\left(\prod_{k=i}^j\frac{d_k}{d_k-i\omega}\right)T(0)_{ij},
 \qquad 0\le i\le j\le m.
\label{eq:gaussian-frequency-blocks}
\end{equation}

To compare operator norms, treat the factors in
\eqref{eq:gaussian-frequency-blocks} as independent variables
$z\in\mathbb C^{m+1}$ and set
$F(z)_{ij}=(\prod_{k=i}^jz_k)T(0)_{ij}$ for $i\le j$, zero otherwise.
When all $|z_k|=1$, put $p_{-1}=1$ and $p_j=\prod_{k=0}^jz_k$. Then
\begin{equation*}
F(z)=\operatorname{diag}_i(\overline{p_{i-1}})\,T(0)\,
      \operatorname{diag}_j(p_j),
\qquad \|F(z)\|_{\mathrm{op}}=\|T(0)\|_{\mathrm{op}}.
\end{equation*}
Applying the maximum modulus principle successively in each variable
to $\langle F(z)u,v\rangle$, for unit vectors $u,v$, gives
$\|F(z)\|_{\mathrm{op}}\le\|T(0)\|_{\mathrm{op}}$ whenever $|z_k|\le1$.
Taking $z_k=d_k/(d_k-i\omega)$ in \eqref{eq:gaussian-frequency-blocks}
proves $\|T(\omega)\|_{\mathrm{op}}\le\|T(0)\|_{\mathrm{op}}$.
Orthogonality of the even and odd input and output spaces now gives
\begin{equation}
\|M(\omega)|_{\mathscr P_N}\|_{\mathrm{op}}
 \le\|M(0)|_{\mathscr P_N}\|_{\mathrm{op}}.
\label{eq:gaussian-polynomial-comparison}
\end{equation}

To pass between polynomial and compactly supported tests,
choose $\sigma\ge1$ with $\sigma I_d\succeq2\widehat A$.
Jensen's inequality and Gaussian covariance comparison in
\eqref{eq:mehler-semigroup}, integrated in
\eqref{eq:homogeneous-resolvent}, give
\begin{equation*}
\|M(\omega)f\|
 \le C\|D^2f\|_{L^2(\gamma_{\sigma I_d}\,\mathrm{d}y)},
\qquad\omega\in\R.
\end{equation*}
Both test classes are dense in
$H^2(\gamma_{\sigma I_d}\,\mathrm{d}y)$; for polynomials, use rescaled
Hermite truncation and the Gaussian graph-norm characterization
\citep[Section~3]{lunardi-pallara-2020}.
The displayed bound and $\gamma_{I_d}\le C\gamma_{\sigma I_d}$ therefore
transfer estimates $\|M(\omega)f\|\le C\|f\|$ between the two classes
with the same constant.
Removing the affine components by
\eqref{eq:gaussian-polynomial-operator} gives
\begin{equation}
\|M(\omega)\|_{\mathrm{op}}
 =\sup_{N\ge2}\|M(\omega)|_{\mathscr P_N}\|_{\mathrm{op}},
\qquad\omega\in\R,
\label{eq:gaussian-hermite-core}
\end{equation}
including infinite values.
By \eqref{eq:gaussian-hermite-core}, taking the supremum over $N$ in
\eqref{eq:gaussian-polynomial-comparison} gives
$\|M(\omega)\|_{\mathrm{op}}\le\|M(0)\|_{\mathrm{op}}$.
Proposition~\ref{prop:homogeneous-initial-norm} gives the zero-frequency
formula. Combining it with \eqref{eq:gaussian-polynomial-operator}
and \eqref{eq:gaussian-hermite-core} also yields
\begin{equation}
\mathfrak G_\alpha(\delta_0,a,A,q)
 =\sup_{N\ge2}\|\mathcal H_{\widehat q}\mathsf M_\alpha^{-1}
                        |_{\mathscr P_N}\|_{\mathrm{op}}.
\label{eq:gaussian-polynomial-norm}
\end{equation}

For the lower bound in part~\ref{item:gaussian-finiteness},
work in the original coordinates and
consider the reference diffusion with constant covariance $a$, zero drift
and initial law $\delta_0$, whose density is $\rho(t,x)=\gamma_{ta}(x)$.
Fix $t_0\in I=(0,T)$ and $\zeta\ne0$.
Take $r_j\downarrow0$ with
$x_j=2t_0a\zeta/r_j$. Then
\begin{equation*}
\frac{\gamma_{(t_0+r_j^2s)a}(x_j+r_jz)}
     {\gamma_{t_0a}(x_j)}
 \longrightarrow \exp\bigl(2(s\zeta^{\mathsf T}a\zeta-\zeta\cdot z)\bigr)
\end{equation*}
locally uniformly in $(s,z)\in\R\times\R^d$.
The point-model definition \eqref{eq:initial-model-gain} and the
whole-line case of \eqref{eq:recovery-norm-limits} give
$\mathfrak G_\alpha(\delta_0,a,A,q)\ge\mathfrak G(\delta_\zeta,a,A,q)$.
Taking the supremum over $\zeta\ne0$ in \eqref{eq:phase-gain-formula}
gives the lower bound in part~\ref{item:gaussian-finiteness}.
Since $G_{\mathrm{whole}}=\infty$ unless $a-A\succ0$, this also proves necessity in
part~\ref{item:gaussian-finiteness}.

For sufficiency, we need a bound in
\eqref{eq:gaussian-polynomial-norm} uniform in the polynomial degree.
Return to the whitened coordinates and standard Gaussian norms.
Let $\Pi_n$ be the orthogonal projection onto $\mathscr C_n$.
For a polynomial $v$, the Hermite derivative formula and Parseval give
\begin{equation}
\|D^2v\|^2
 =\sum_{n\ge2}n(n-1)\|\Pi_nv\|^2.
\label{eq:hermite-hessian-norm}
\end{equation}
Assume $a-A\succ0$, so that $0\prec\widehat A\prec I_d$, and put
$\vartheta=2\|I_d/2-\widehat A\|_{\mathrm{op}}<1$.
Let $\mathsf D_\alpha$ be the diagonal part of $\mathsf M_\alpha$ and
define $K=(\mathsf M_\alpha-\mathsf D_\alpha)\mathsf D_\alpha^{-1}$
initially on $\bigcup_{N\ge2}\mathscr P_N$.
After an orthogonal change of coordinates, write
$I_d-2\widehat A=\operatorname{diag}(\beta_1,\ldots,\beta_d)$,
so $|\beta_i|\le\vartheta$.
In the orthonormal Hermite basis $h_{\boldsymbol\nu}$, write
$v=\sum_{|\boldsymbol\nu|=n}c_{\boldsymbol\nu}h_{\boldsymbol\nu}$.
For $n\ge4$, the coefficients of $Kv$ are
\begin{equation*}
(Kv)_{\mathbf m}
 =\frac1{n-1+\alpha}\sum_{i=1}^d
   \beta_i\sqrt{(m_i+1)(m_i+2)}\,c_{\mathbf m+2e_i},
\qquad |\mathbf m|=n-2,
\end{equation*}
where $e_i$ are the coordinate vectors.
Weighted Cauchy--Schwarz with weights $m_i+1$ contributes
$\sum_i(m_i+1)=n+d-2$; summing the remaining terms over
$|\mathbf m|=n-2$ gives
$\sum_{|\boldsymbol\nu|=n}\sum_{i:\nu_i\ge2}
\nu_i|c_{\boldsymbol\nu}|^2\le n\|v\|^2$.
Since $K$ vanishes on $\mathscr C_2\oplus\mathscr C_3$, this yields
\begin{equation*}
\|K|_{\mathscr C_n}\|_{\mathrm{op}}
 \le\vartheta\frac{\sqrt{n(n+d-2)}}{n-1+\alpha}
 \longrightarrow\vartheta<1.
\end{equation*}
Thus $K$ extends boundedly to $\Pi_{\ge2}L^2(\gamma_{I_d}\,\mathrm{d}y)$.
Choose $r\in(\vartheta,1)$, so all sufficiently high-degree blocks
of $K$ have norm at most $r$.
Since $K$ lowers degree by two, each product in $K^j$ contains
only a uniformly bounded number of lower-degree factors.
Orthogonality of the output degrees gives
$\|K^j\|_{\mathrm{op}}\le Cr^j$, with $C$ independent of $j$ and
the starting degree. Hence
$(I+K)^{-1}=\sum_{j\ge0}(-K)^j$ converges in operator norm.
Equation~\eqref{eq:hermite-hessian-norm} also bounds
$\mathcal H_{\widehat q}\mathsf D_\alpha^{-1}$.
Since the inverse preserves each $\mathscr P_N$,
\eqref{eq:gaussian-polynomial-norm} gives
\begin{equation*}
\mathfrak G_\alpha(\delta_0,a,A,q)
 =\bigl\|(\mathcal H_{\widehat q}\mathsf D_\alpha^{-1})
             (I+K)^{-1}\bigr\|_{\mathrm{op}}<\infty.
\end{equation*}
For fixed $\alpha$, applying the Hermite bounds to
coefficient differences gives operator-norm continuity of $K$ and
$\mathcal H_{\widehat q}\mathsf D_\alpha^{-1}$ in
$(\widehat A,\widehat q)$.
Continuity of inversion and whitening gives continuity in $(a,A,q)$.

For $A=a/2$ and $q=a$, we have
$\mathsf M_\alpha|_{\mathscr C_n}=(n-1+\alpha)I/2$ and
$\mathcal H_{\widehat q}=D^2$.
Orthogonality, \eqref{eq:gaussian-polynomial-norm} and
\eqref{eq:hermite-hessian-norm} give the matching gain formula in part~\ref{item:matching-point-models}.
The value $C_0^{\mathrm{pt}}=2\sqrt2$ is attained at degree two,
and monotonicity in $\alpha$ gives $C_\alpha^{\mathrm{pt}}\le2\sqrt2$.
The identity
$(n-1+\alpha)^2-n(n-1)=n(2\alpha-1)+(1-\alpha)^2$
and the limit of the ratios as $n\to\infty$ give $C_\alpha^{\mathrm{pt}}>2$ for $0\le\alpha<1/2$ and
$C_\alpha^{\mathrm{pt}}=2$ for $\alpha\ge1/2$.
\end{proof}

\section{A Borel principal matrix with infinite Hessian gain}
\label{app:rough-principal}

\begin{proof}[Proof of Example~\ref{ex:rough-principal}]
Write $r=|x-x_{\rm s}|$ and $n=(x-x_{\rm s})/r$ for $x\ne x_{\rm s}$.
The eigenvalues of $A_{\mathrm{rough}}$ are $2/3$ in the radial direction
and $1/6$ in each tangential direction. Thus the stated ellipticity
bounds hold, $I_3-A_{\mathrm{rough}}\succeq I_3/3$, and
$|A_{\mathrm{rough}}(x)-I_3/2|_{\mathrm F}=1/2$ for $x\ne x_{\rm s}$.
The value $A_{\mathrm{rough}}(x_{\rm s})=I_3/3$ also satisfies
the stated ellipticity and perturbation bounds.
For every $z\ne0$ and $\varepsilon>0$,
\begin{equation*}
A_{\mathrm{rough}}(x_{\rm s}+\varepsilon z)
 =\frac16I_3+\frac12\frac{z\otimes z}{|z|^2}.
\end{equation*}
Since the right-hand side depends on the direction of $z$,
$A_{\mathrm{rough}}$ has no limit at $x_{\rm s}$.

Since $\mu_0=\delta_0$ and $A_{\mathrm{rough}}$ is continuous near $0$,
the initial contribution is
\begin{equation*}
G_\alpha^{\mathrm{init}}(I)
 =\mathfrak G_\alpha(\delta_0,I_3,A_{\mathrm{rough}}(0),I_3)<\infty,
\end{equation*}
where finiteness follows from $I_3-A_{\mathrm{rough}}(0)\succ0$
and Proposition~\ref{prop:point-initial-models}\ref{item:gaussian-finiteness}.
The eigenvalues $2/3$ and $1/6$ of $A_{\mathrm{rough}}(x)$ for $x\ne x_{\rm s}$
and \eqref{eq:flat-contribution} give $G^{\mathrm{flat}}(I)=6$.
For every phase candidate, the limiting matrix $A_*$ satisfies
$I_3-A_*\succeq I_3/3$, so $\mathsf R(I_3,A_*,I_3;\zeta)\ge1/3$ for every $\zeta\ne0$.
Thus \eqref{eq:propagation-contribution} gives $G^{\mathrm{phase}}(I)\le3$, proving
\eqref{eq:rough-frozen-candidate}.

To prove \eqref{eq:rough-infinite-gain}, we truncate the singular
$A_{\mathrm{rough}}$-harmonic function $v(x)=r^{1/2}$.
On the punctured space, direct differentiation gives
\begin{equation}
D^2v=r^{-3/2}\left(\frac12I_3-\frac34n\otimes n\right),
\quad |D^2v|^2=\frac9{16}r^{-3}, \quad
A_{\mathrm{rough}}:D^2v=0.
\label{eq:rough-harmonic-identities}
\end{equation}
Fix $J\in\mathcal J(I)$,
$R>0$, and a nonzero $\psi\in C_c^\infty(J)$.
Using the cutoffs \eqref{eq:spatial-cutoffs}, for $0<\delta<R/4$ define
\begin{equation*}
v_\delta(x)=\bigl(1-\chi_\delta(x-x_{\rm s})\bigr)
                  \chi_R(x-x_{\rm s})r^{1/2},
\qquad u_\delta(t,x)=\psi(t)v_\delta(x),
\end{equation*}
with value zero at $x_{\rm s}$. Then
$u_\delta\in C_c^\infty(J\times\R^3)\subset
\mathcal W_{\alpha,0}(J)$.
On $\supp\psi\times\overline{B(x_{\rm s},2R)}$,
$t^\alpha\rho$ is bounded above and below by positive constants.
Since $v_\delta=v$ on $2\delta\le r\le R$,
\eqref{eq:rough-harmonic-identities} gives the Hessian lower bound with
$c,C>0$ independent of $\delta$:
\begin{equation}
\|D^2u_\delta\|_{\alpha,J}^2
 \ge c\int_{2\delta}^R\frac{\mathrm{d}r}{r}
 =c\log\frac{R}{2\delta},
\qquad
\|P_{A_{\mathrm{rough}}}u_\delta\|_{\alpha,J}^2\le C.
\label{eq:rough-cutoff-estimates}
\end{equation}
To check the source bound, use
$P_{A_{\mathrm{rough}}}u_\delta=-\psi'v_\delta
-\psi A_{\mathrm{rough}}:D^2v_\delta$.
The spatial $L^2$ norms of $v_\delta$ are uniformly bounded, and
the identity $A_{\mathrm{rough}}:D^2v=0$ in
\eqref{eq:rough-harmonic-identities} gives
\begin{equation*}
|A_{\mathrm{rough}}:D^2v_\delta|
 \le C\delta^{-3/2}\mathbf1_{\{\delta<r<2\delta\}}
       +CR^{-3/2}\mathbf1_{\{R<r<2R\}}.
\end{equation*}
The inner annulus has volume of order $\delta^3$, so its squared
$L^2$ contribution is bounded independently of $\delta$; the outer
annulus is fixed. This proves \eqref{eq:rough-cutoff-estimates}.
Since $q=I_3$, letting $\delta\downarrow0$ shows that the fixed-interval
gain is infinite.
Since $J$ was arbitrary, the definition
\eqref{eq:short-time-constants} gives \eqref{eq:rough-infinite-gain}.
\end{proof}

\section{Linear-growth drift and failure of the matching value}
\label{app:linear-growth-transport}

For the unbounded drift below, we use only
$P_{a/2,b}u=-u_t-(a/2):D^2u-b\cdot Du$ and its norm quotients
on smooth compactly supported tests.
Under Assumption~\ref{ass:coefficients}, for every $\alpha\ge1/2$,
the limiting gain on $I=(0,T)$ for $P_{a/2,b}$ with observation
$\mathcal H_a$ is $2$ for every initial probability law,
by \eqref{eq:borel-drift-invariance}
and Theorem~\ref{thm:matching-arbitrary-laws}\ref{item:matching-full-interval}.
The following example shows that this conclusion does not extend
to Lipschitz drifts of linear growth.
Its invariant law violates
\eqref{eq:initial-coercivity}, as verified after the proof.
Thus the example does not settle the extension when the initial law
satisfies Assumption~\ref{def:initial-laws}.

\begin{proposition}
\label{prop:linear-growth-transport}
There exist smooth, time-independent, globally
Lipschitz coefficients $a,b$ on $\R^2$ with $I_2\preceq a\preceq3I_2$
and $b$ of linear growth,
together with a smooth positive invariant probability density $\rho$,
such that the following holds.
Start the diffusion from $\mu_0(\mathrm{d}x)=\rho(x)\,\mathrm{d}x$,
so $\rho(t,x)=\rho(x)$. For every $T>0$ and every $\alpha$,
on $I=(0,T)$ there are intervals
$J_j\in\mathcal J(I)$ and nonzero functions
$u_j\in C_c^\infty(J_j\times\R^2)$ satisfying
\begin{equation*}
|J_j|\longrightarrow0,\qquad
\frac{\|\mathcal H_a u_j\|_{\alpha,J_j}}
     {\|P_{a/2,b}u_j\|_{\alpha,J_j}}
\longrightarrow\frac3{\sqrt2}>2.
\end{equation*}
\end{proposition}

\begin{proof}
We construct an invariant density that is constant on annuli
escaping to infinity. A rotational drift moves the test centers
through spatial variations of the covariance, producing a
time-dependent coefficient in the rescaled limit.

\paragraph*{Coefficients and invariant law.}
Set
\begin{equation*}
\mathsf J=\begin{pmatrix}0&-1\\1&0\end{pmatrix},
\qquad R_j=\mathrm{e}^{4j}.
\end{equation*}
Choose smooth $4$-periodic functions $\psi_0,\psi_1$ taking
values in $[0,1]$ such that
\begin{equation*}
\begin{aligned}
\psi_0(v)&=0 &&\text{if }\dist(v,4\mathbb Z)\le1,&
\psi_0(v)&=1 &&\text{if }\dist(v,4\mathbb Z)\ge3/2,\\
\psi_1(v)&=1 &&\text{if }\dist(v,4\mathbb Z)\le1/4,&
\psi_1(v)&=0 &&\text{if }\dist(v,4\mathbb Z)\ge1/2.
\end{aligned}
\end{equation*}
Set $V(x)=\chi_{\mathrm{ann}}(x)=0$ for $|x|\le2$, and, for $|x|>2$, set
\begin{equation*}
V(x)=\int_2^{|x|}r\psi_0(\log r)\,\mathrm{d}r,
\qquad \chi_{\mathrm{ann}}(x)=\psi_1(\log|x|).
\end{equation*}
Both definitions are smooth across $|x|=2$, since the two
periodic functions vanish near $\log2$. On each annulus
$\mathrm{e}^{-1}R_j<|x|<\mathrm{e}R_j$, the function $V$ is constant.
With $r=|x|$, integrating
$t\psi_0(\log t)$ over $[\mathrm{e}^{-4}r,r]$ and using the positive mean
of $\psi_0$ gives $C^{-1}r^2\le V(x)\le r^2/2$ for large $r$.
The radial and tangential eigenvalues of $D^2V$ outside $B_2$
are $\psi_0(\log r)+\psi_0'(\log r)$ and
$\psi_0(\log r)$, so $D^2V$ is bounded.

Normalize the integrable function $\mathrm{e}^{-V}$ to a probability density $\rho$, and set
\begin{equation*}
c(x)=2-\chi_{\mathrm{ann}}(x)\tanh x_2, \qquad
a(x)=c(x)I_2, \qquad b(x)=\mathsf Jx-DV(x)+\tfrac12Dc(x).
\end{equation*}
Here $1\le c\le3$, and all derivatives of
$c$ are bounded. The support of $\chi_{\mathrm{ann}}$ lies where $DV=0$,
so $cDV=2DV$ and
\(
Db=\mathsf J-D^2V+\tfrac12D^2c,\,
b\rho=\mathsf Jx\rho+\tfrac12D(c\rho).
\)
The first identity proves the Lipschitz bound for $b$.
Since $\rho$ is radial, $\operatorname{div}(\mathsf Jx\rho)=0$;
the second gives
\begin{equation*}
\tfrac12\Delta(c\rho)-\operatorname{div}(b\rho)=0.
\end{equation*}
The constant curve $t\mapsto\rho\mathcal L^2$ solves the
Fokker--Planck equation for $a,b$.
Boundedness of $a$, linear growth of $b$ and the Gaussian tail of $\rho$
give $\int_{\R^2}(|a|+|b|)\rho\,\mathrm{d}x<\infty$.
Thus a solution of the martingale problem for $(a,b)$
with these marginals exists by
\citep[Theorem~2.5]{trevisan-2016}. Since $a=cI_2\succ0$, martingale representation
gives a weak stochastic differential equation (SDE) solution with
diffusion matrix $\sqrt c\,I_2$.
Weak uniqueness for the globally Lipschitz
coefficients $b,\sqrt c\,I_2$ then gives $\rho(t,x)=\rho(x)$.

\paragraph*{Moving coordinates.}
Set $t_*=T/2$ and put
\begin{equation*}
r_j=R_j^{-1/2},\,\, x_j=R_j\mathbf e_1,\,\,
\gamma_j(s)=\mathrm{e}^{r_j^2s\mathsf J}x_j,\,\,
(t,x)=(t_*+r_j^2s,\gamma_j(s)+r_jz).
\end{equation*}
Since $R_jr_j^2=1$, rotation moves the centers a distance
of order one over a physical time interval of length $r_j^2$.
Fix $B\ge1$. For $(s,z)\in\mathcal U_B$
and all sufficiently large $j$, the physical
points lie where $V=V(x_j)$ and $\chi_{\mathrm{ann}}=1$. Thus
\begin{equation}
\begin{gathered}
\frac{\rho(\gamma_j(s)+r_jz)}{\rho(x_j)}=1,\\
c_j(s,z):=c(\gamma_j(s)+r_jz)
 =2-\tanh\bigl(R_j\sin(s/R_j)+r_jz_2\bigr).
\end{gathered}
\label{eq:linear-growth-density-identity}
\end{equation}
With $\bar c(s)=2-\tanh s$, local uniform convergence of
$R_j\sin(s/R_j)$ to $s$ gives
\begin{equation}
\sup_{(s,z)\in\mathcal U_B}|c_j(s,z)-\bar c(s)|\longrightarrow0.
\label{eq:linear-growth-coefficient-limit}
\end{equation}

For a smooth function $v(s,z)$ with $\supp v\subset
\mathcal U_B$,
define the test in moving coordinates by
\begin{equation}
u_j(t,x)=r_j^2v\left(\frac{t-t_*}{r_j^2},
              \frac{x-\mathrm{e}^{(t-t_*)\mathsf J}x_j}{r_j}\right).
\label{eq:linear-growth-rescaling}
\end{equation}
At $(t,x)=(t_*+r_j^2s,\gamma_j(s)+r_jz)$, the chain rule gives exactly
\begin{equation}
P_{a/2,b}u_j= -v_s-\tfrac12c_j\Delta_zv
   -\bigl(r_j^2\mathsf Jz+\tfrac12r_jDc(\gamma_j+r_jz)\bigr)
       \cdot D_zv,\quad
\mathcal H_a u_j=c_jD_z^2v.
\label{eq:linear-growth-chain-rule}
\end{equation}
The remaining drift is bounded by $C(Br_j^2+r_j)$
on $\mathcal U_B$.
Set
\(
\bar q(s)=\bar c(s)I_2,\,
\bar A(s)=\bar c(s)I_2/2,
\)
with limiting operators $P_{\bar A}=-\partial_s-\bar c(s)\Delta_z/2$
and $\mathcal H_{\bar q}=\bar c(s)D_z^2$.
Since $t_*\in I$, the rescaled time intervals exhaust $\R$.
The density ratio is $1$ and
$(t_*+r_j^2s)^\alpha/t_*^\alpha\to1$ locally uniformly, so we use
$L^2(\R\times\R^2)$ norms for the limiting operators.

\paragraph*{Compact tests for the limit.}
For spatial frequencies with $|\xi|=1$, take the time profile
$f(s)=(\cosh s)^{-1/2}$. It satisfies
\(
f'=-\tfrac12\tanh s\,f,\,
-f'+\tfrac12\bar c f=f.
\)
The substitution $y=\tanh s$ gives
\(
\int_{\R} f^2\,\mathrm{d}s=\pi,\,
\int_{\R}\bar c^2f^2\,\mathrm{d}s=9\pi/2.
\)
Hence
\begin{equation*}
\frac{\|\bar c f\|_{L^2(\R)}}
     {\|-f'+\tfrac12\bar c f\|_{L^2(\R)}}
 =\frac3{\sqrt2}>2.
\end{equation*}
Choose even $\varphi_n\in C_c^\infty(\R^2)$ with
$\|\varphi_n\|_{L^2(\R^2)}=1$ and
$\supp\varphi_n\subset B(\mathbf e_1,n^{-1})\cup B(-\mathbf e_1,n^{-1})$.
Set $w_n(s,z)=f(s)(\mathcal F_z^{-1}\varphi_n)(z)$.
Since $|\xi|\to1$ uniformly on $\supp\varphi_n$, spatial Plancherel
and the identities for $f$ give
\(
\|\mathcal H_{\bar q}w_n\|_{L^2}^2\longrightarrow\tfrac{9\pi}2,
\, \|P_{\bar A}w_n\|_{L^2}^2\longrightarrow\pi.
\)
Each $w_n$ belongs to $W_2^{1,2}(\R\times\R^2)$.
Both model operators are continuous from this space into $L^2$,
since $\bar c$ is bounded.
Standard Sobolev density then gives nonzero
$v_n\in C_c^\infty(\R\times\R^2)$ with
\begin{equation}
\frac{\|\mathcal H_{\bar q}v_n\|_{L^2}}{\|P_{\bar A}v_n\|_{L^2}}
 \longrightarrow\frac3{\sqrt2}.
\label{eq:linear-growth-model-tests}
\end{equation}

Fix $v\in C_c^\infty(\R\times\R^2)$ and choose
$B\ge1$ with $\supp v\subset\mathcal U_B$.
Let $u_j$ be defined by \eqref{eq:linear-growth-rescaling} and put
$J_j=(t_*-(B+1)r_j^2,t_*+(B+1)r_j^2)$.
For all sufficiently large $j$, $J_j\subset I$ and
$u_j\in C_c^\infty(J_j\times\R^2)$.
The density identity \eqref{eq:linear-growth-density-identity},
the operator limits from \eqref{eq:linear-growth-coefficient-limit}
and \eqref{eq:linear-growth-chain-rule}, and
$(t_*+r_j^2s)^\alpha\to t_*^\alpha$ give
\begin{equation*}
\frac{1}{r_j^4t_*^\alpha\rho(x_j)}\|\mathcal H_a u_j\|_{\alpha,J_j}^2 \longrightarrow \|\mathcal H_{\bar q}v\|_{L^2}^2
,\quad
\frac{1}{r_j^4t_*^\alpha\rho(x_j)}\|P_{a/2,b}u_j\|_{\alpha,J_j}^2
\longrightarrow
\|P_{\bar A}v\|_{L^2}^2
\end{equation*}
Apply this limit to each $v_n$ from \eqref{eq:linear-growth-model-tests},
choosing $j=j(n)$ so that the quotient is within $n^{-1}$ of its model
value and $J_{j(n)}$ has length less than $n^{-1}$.
Relabeling the resulting tests and intervals proves
the asserted limit.
\end{proof}

To verify the failure of \eqref{eq:initial-coercivity} directly, fix
$0\ne f\in C_c^\infty(B_1)$ and put
$f_j(x)=f(x-R_j\mathbf e_1)$.
For large $j$, the support of $f_j$ lies in the annulus
$\mathrm{e}^{-1}R_j<|x|<\mathrm{e}R_j$, where the density is constant.
For any fixed center $x_0$, applying
\eqref{eq:initial-coercivity} and cancelling this constant would give
\begin{equation*}
(R_j-|x_0|-1)^2\int_{\R^2}|f|^2\,\mathrm{d}x
\le K_{\mathrm{coer}}\left(
 \int_{\R^2}|Df|^2\,\mathrm{d}x+\int_{\R^2}|f|^2\,\mathrm{d}x\right),
\end{equation*}
which is impossible as $j\to\infty$.

\section*{Statements and Declarations}

\subsection*{Funding}
Zhenhua Zhao was supported by the Elite Ph.D. Program in Applied Mathematics
at Peking University. The funder had no role in the design of the study, the
analysis or interpretation of the results, the preparation of the manuscript,
or the decision to submit the article for publication.

\subsection*{Competing interests}
The authors have no relevant financial or non-financial interests to disclose.

\subsection*{Data availability}
No datasets were generated or analysed during this theoretical study.

\section*{Use of artificial intelligence}
The authors used OpenAI's GPT-5.6 Sol and GPT-6 Astra through Codex to assist
in constructing the counterexamples in
Appendices~\ref{app:rough-principal} and~\ref{app:linear-growth-transport}.
The authors independently verified the mathematical content of these
counterexamples, including their constructions, calculations and proofs,
and take full responsibility for the content of this paper.

\bibliography{references}
\end{document}